\documentclass[
11pt, 
english, 
headsepline, 
]{MastersDoctoralThesis} 
\usepackage{xcolor}
\usepackage{tcolorbox}
\usepackage[utf8]{inputenc} 
\usepackage[T1]{fontenc} 
\usepackage[shortlabels]{enumitem}
\usepackage{amsmath, amsthm, mathrsfs, mathrsfs, mathtools, amssymb,dsfont}
\RequirePackage{amsthm,amsmath,amsfonts,amssymb,nccmath}
\usepackage{pict2e} 
\usepackage{caption}
\usepackage{subcaption}
\usepackage{multirow} 
\usepackage{coverenpc}
\usepackage{bm}
\usepackage[maxnames=5, backend=bibtex]{biblatex} 
\newtheorem{theorem}{Theorem}[section]
\newtheorem{definition}[theorem]{Definition}
\newtheorem{prop}[theorem]{Proposition}

\newtheorem{corollary}[theorem]{Corollary}
\newtheorem{lemma}[theorem]{Lemma}
\newtheorem{ex}[theorem]{Example}
\newtheorem{remark}[theorem]{Remark}

\usepackage[autostyle=true]{csquotes} 
\usepackage{import}
\newcounter{hypoconbisfl}
\newcounter{saveconbisfl}
\newcommand\debutL{\begin{list} {\textbf{L\arabic{hypoconbisfl}}}{\usecounter{hypoconbisfl}}\setcounter{hypoconbisfl}{\value{saveconbisfl}}}
	\newcommand\finL{\end{list}\setcounter{saveconbisfl}{\value{hypoconbisfl}}}

\newcounter{hypoconbisf}
\newcounter{saveconbisf}
\newcommand\debutRU{\begin{list} {\textbf{RU\arabic{hypoconbisf}}}{\usecounter{hypoconbisf}}\setcounter{hypoconbisf}{\value{saveconbisf}}}
	\newcommand\finRU{\end{list}\setcounter{saveconbisf}{\value{hypoconbisf}}}

\newcounter{hypoconbisfm}
\newcounter{saveconbisfm}
\newcommand\debutI{\begin{list} {\textbf{I\arabic{hypoconbisfm}}}{\usecounter{hypoconbisfm}}\setcounter{hypoconbisfm}{\value{saveconbisfm}}}
	\newcommand\finI{\end{list}\setcounter{saveconbisfm}{\value{hypoconbisfm}}}

\usepackage{comment}
\newcounter{hypoconbisx}
\newcounter{saveconbisx}
\newcommand\debutTX{\begin{list}imsart.cls {\textbf{TX\arabic{hypoconbisx}}}{\usecounter{hypoconbisx}}\setcounter{hypoconbisx}{\value{saveconbisx}}}
	\newcommand\finTX{\end{list}\setcounter{saveconbisx}{\value{hypoconbisx}}}

\newcommand{\chapabstract}[1]{
	\begin{quote}
		\singlespacing\small
		\rule{14cm}{1pt}
		#1
		\vskip-4mm
		\rule{14cm}{1pt}
\end{quote}}

\newcommand{\cM}{\mathcal{M}}
\newcommand{\D}{\mathcal{D}}
\newcommand{\R}{{\mathbb R}}
\newcommand{\Z}{{\mathbb Z}}
\newcommand{\N}{{\mathbb N}}
\newcommand{\E}{{\mathbb E}}
\newcommand{\Var}{\mathrm{Var}}
\newcommand{\Cov}{\mathrm{Cov}}

\renewcommand{\P}{{\mathbb P}}
\renewcommand{\L}{{\mathbb L}}
\newcommand{\PRp}{{\mathfrak{P}(\R)}}
\newcommand{\PLRp}[1]{{\mathfrak{P}}_{#1}(\R)}
\newcommand{\pol}{\textup{pol}}
\newcommand{\Cpol}{{\mathcal C}^\infty_{\pol}(\R_+)}
\newcommand{\CpolK}[1]{{\mathcal C}^{#1}_{\pol}(\R_+)}
\newcommand{\CpolKL}[2]{{\mathcal C}^{#1,#2}_{\pol}(\R_+)}

\newcommand{\CP}[1]{{\mathcal C}^{#1}_{\pol}(\R\times \R_+)}
\newcommand{\CPL}[2]{{\mathcal C}^{#1,#2}_{\pol}(\R\times \R_+)}

\newcommand{\cPh}{\hat{\mathcal P}}
\newcommand{\USC}{\textsc{USC}}
\newcommand{\LSC}{\textsc{LSC}}
\newcommand{\bmcP}{\overline{\mathcal{P}}}

\def\bbar{\overline}
\def\brho{\bbar \rho}

\def\hP{\hat P}

\def\hx{{\hat x}}

\def\hf{{\hat f}}

\def\hX{\hat X}
\def\hY{\hat Y}

\def\bH{\overline H}

\def\tK{\tilde K}

\def\tX{\tilde X}
\def\tY{\tilde Y}

\def\tf{{\tilde f}}

\def\hu{\hat u}

\def\hetanu{{\hat \eta}_\nu}

\def\sjzL{\sum_{j= 0}^L}
\def\i{\mathrm{I}}

\def\mcA{{\mathcal A}}
\def\mcC{{\mathcal C}}
\def\mcD{{\mathcal D}}
\def\mcE{{\mathcal E}}

\def\cL{\mathcal L}
\def\mscrF{{\mathscr F}}
\def\mcK{{\mathcal K}}
\def\mcL{{\mathcal L}}
\def\mcF{{\mathcal F}}
\def\mcM{{\mathcal M}}
\def\mcN{{\mathcal N}}
\def\mcO{{\mathcal O}}
\def\mcP{{\mathcal P}}
\def\mcR{{\mathcal R}}
\def\mcS{{\mathcal S}}
\def\mcX{{\mathcal X}}
\def\mcY{{\mathcal Y}}
\def\bmcO{\bar{{\mathcal O}}}

\def\bfy{{\mathbf y}}

\def\tbfY{\tilde{\mathbf Y}}

\def\bDf{\overline{D_f}}

\def\bfy{{\mathbf y}}

\def\bff{{\mathbf f}}
\def\a{\alpha}
\def\b{\beta}
\def\g{\gamma}

\def\<{\langle}
\def\>{\rangle}

\def \pol{{\mathbf{pol}}}
\newcommand\ind[1]{\mathds{1}_{#1}}
\newcommand{\Dx}{{\Delta x}}
\newcommand{\PhDx}{\Pi^h_\Dx}

\date{}

\begin{document}

\frontmatter 



\title{Approximation and regularity results for the Heston model and
related processes}

\doctoralschool{MSTIC and Mathematics Tor Vergata}

\fieldthesis{Subject: Mathematics}

\laboratories{Thesis prepared at CERMICS,  and at the Mathematics department \\ of Università degli Studi di Roma Tor Vergata} 

\date{22/01/2025}

\author{Edoardo LOMBARDO}




\juryone{Cristina, CAROLI COSTANTINI}{Reviewer}{Professor}{Università degli Studi G. d'Annunzio Chieti--Pescara}
\jurytwo{Noufel, FRIKHA}{Reviewer}{Professor}{Université Paris 1 Panthéon Sorbonne}
\jurythree{Paolo, PIGATO}{Examiner}{Associate Professor}{Università degli studi di Roma Tor Vergata}
\juryfour{Michele, SALVI}{Examiner}{Associate Professor}{Università degli studi di Roma Tor Vergata}
\juryfive{Aurélien, ALFONSI}{Advisor}{Professor}{École Nationale des Ponts et Chaussées}
\jurysix{Lucia, CARAMELLINO}{Advisor}{Professor}{Università degli studi di Roma Tor Vergata}


\makecoverenpc

\clearpage
\section*{Summary}
This thesis investigates approximations and regularity results pertaining to Heston's stochastic volatility model. It comprises three papers organized across chapters 2, 3, and 4.

The first one tackles the challenge of developing high-order weak approximations for the Cox-Ingersoll-Ross (CIR) process, which is crucial in financial modelling. The standard Euler-Maruyama scheme fails due to the square root in the diffusion term, potentially leading to negative values. Additionally, theoretical frameworks that produce high-order approximations, like the Multistep Richardson Romberg approach (Pagès 2007), do not directly apply due to the CIR process's specific structure.
This work employs the random grid technique by Alfonsi and Bally (2021), which leverages an elementary scheme on random time grids to boost convergence order. We use Alfonsi's (2010) second-order CIR scheme as the elementary building block.
Rigorous proofs establish that weak approximations of any order can be achieved for smooth test functions with polynomial growth derivatives, given the condition $\sigma^2\le4a$ that is less restrictive than the well-known ``Feller Condition'' ($\sigma^2\le2a$). Numerical experiments validate these findings, showcasing convergence for both CIR and Heston models, with significant computational time improvements compared to lower-order schemes.
The limitation lies in the theoretical results being proven only under the less restrictive condition cited above. Numerical tests hint at effectiveness beyond this, but rigorous proof remains an open question.

The second work serves as a continuation of our first project. In this iteration, we apply our technique, which is based on random grids, to the log-Heston process. This process represents the logarithm of an asset price within the Heston model and the associated volatility process. The log transformation helps ensure that the moments of both variables remain bounded, thereby simplifying the mathematical analysis.
We start by proposing two second order schemes, built using splitting techniques. The first one uses an exact simulation for the volatility process; the other one uses the Ninomiya-Victoir splitting scheme for the log-Heston process (this one is valid only under the above-cited condition on parameters $\sigma^2\le 4a$). Rigorous proofs demonstrate convergence to any desired order.
Numerical experiments validate the theoretical results, showcasing the effectiveness of the high-order schemes for pricing European and arithmetic Asian options. The impact of different coupling choices on estimator variance is also investigated. Additionally, promising results are presented for the multifactor/rough Heston model, suggesting the potential of the random grid technique in this extended context.

The last work delves into the partial differential equation (PDE) associated with the log-Heston model, exploring classical and viscosity solutions.
Key contributions include extending classical solution results by incorporating linear and source terms in the PDE. In this work, we also prove the existence and uniqueness of viscosity solutions without relying on Feller's condition, a common assumption in the literature. Uniqueness is established even for initial data with specific discontinuities, which is relevant for financial applications like digital option pricing.
Furthermore, the chapter demonstrates the convergence of a hybrid numerical scheme (finite differences/tree scheme) to approximate the viscosity solution under relaxed regularity assumptions (continuity) on the initial data.
These results offer a more comprehensive understanding of the log-Heston PDE, particularly in scenarios where Feller's condition doesn't hold or the initial data is discontinuous. In the end, we prove a convergence result for a hybrid scheme provided that the initial data is just continuous.

In Appendix B, we collect other results for the CIR process that did not find space in these three articles: new regularity results,  ``High order approximation in high volatility regime'', a new proof of the CIR moment formula and ``polynomial schemes''.

\begin{acknowledgements}
My acknowledgements will be multilingual and tailored to each recipient.

I begin by thanking those who provided the tangible resources that enabled this thesis. I greatly thank the MIUR Excellence Department Project Math@TOV, which partially funded me and enabled me to enrol in Tor Vergata’s doctoral school in mathematics. Similarly, I extend my gratitude to the École des Ponts et Chaussées for funding my doctorate throughout my time at Champs-Sur-Marne.

I would like to express my gratitude to all the professors who have guided me throughout my academic journey.
J'adresse mes plus sincères remerciements à mon directeur de thèse, Aurélien Alfonsi, dont les conseils et le soutien ont été inestimables tout au long de mon doctorat. Son expertise a été déterminante pour l'élaboration de mes recherches et de cette thèse. J'espère avoir été un étudiant à la hauteur et avoir hérité d'une part de son esprit et de sa capacité de déduction. Esprimo la mia gratitudine alla mia direttrice di tesi, Lucia Caramellino, per avermi trasmesso la passione per la probabilità sin dalla laurea triennale e per il significativo supporto offerto durante tutto il mio percorso, dagli anni precedenti al dottorato fino ad oggi. Desidero ringraziare, inoltre, il Professor Alessio Porretta per il tempo dedicatomi, che ha contribuito significativamente all'avanzamento e all'affinamento delle mie ricerche sulle equazioni alle derivate parziali paraboliche. 
Finally, I wish to extend my thanks to Professor Noufel Frikha and Professor Cristina Caroli Costantini for accepting the responsibility of reviewing my thesis and for their valuable time spent proofreading the manuscript.

I now turn to expressing my gratitude to my colleagues and friends. 
Un grand merci à mes amis du troisième étage du CERMICS : Solal, Laurent, Eloïse, Alfred, Mathias, Simon, Alberic, Jean, Gaspard, Renato, Shiva, Régis et Epiphane, pour toutes les discussions, les bons moments passés ensemble et pour m'avoir accepté malgré que je sois ``un du deuxième''... Je plaisante, le deuxième étage règne ! J'adresse mes sincères remerciements à mes amis du deuxième étage, Dylan, Michel, Kacem, Hervé, Faten, Léo, Cyrille, Thomas, Sébastien, Louis, Vitor, Zoé et Hélène, pour la qualité de notre cohabitation et l'ambiance de travail que nous avons su créer. Leur compagnie a contribué à rendre plus légères les journées les plus chargées.
Un grand merci à mes amis de la ``Team Basic-Fit'' : Mohamad, Leo, Carlos et Clément, et surtout à Guillaume pour m'avoir aidé à rester en forme, me défouler pendant cette période intense et de ne pas oublier de manger mon poulet.
J'adresse un remerciement à mon ami Étienne pour son aide précieuse face à tous les problèmes imaginables concernant Linux et/ou la programmation. Un merci à mon ami Fabian qui a quitté le métavers pour apprendre aux machines à être intelligentes ! Ringrazio Luca per la sua enorme ospitalità e per aver condiviso tanti pasti al nostro amato Asian77.
Un ringraziamento speciale alla mia amica Roberta, una presenza costante e fidata, capace di offrire non solo supporto amicale, ma anche preziose consulenze psicologiche e matematiche, soprattutto nei momenti di maggiore difficoltà con i numeri.
Voglio ringraziare quel ``callone'' del mio amico Emanuele per avermi mostrato come affrontare le difficoltà con il sorriso sulle labbra e per le tante risate che mi ha fatto fare. Je suis reconnaissant envers Coco de m'avoir donné un avant-goût du chinois, de m'avoir offert de délicieux desserts chinois et pour sa précieuse amitié. Un grand remerciement à Anton mon acolyte de F1. J'espère qu'un jour on pourra de nouveau fêter un titre, que ce soit Ferrari ou Alpine.
Un ringraziamento va a Giacomo, con il quale ho condiviso l'ufficio. La sua compagnia e le nostre chiacchierate sono state di grande supporto durante il periodo a Roma.
Un ringraziamento affettuoso a Elia ``il Vecio'' per essere stato un fantastico compagno di passeggiate e un'illustre firma nei nostri fanta-articoli. Le nostre chiacchierate sono state una boccata d'aria fresca durante le giornate di lavoro più dure.
Rivolgo un doveroso ringraziamento ai miei nuovi colleghi Daniele e Katia, che con la loro saggezza mi hanno offerto un valido contributo nel periodo conclusivo di redazione della tesi.
I sincerely hope I have included everyone. If, however, anyone has been inadvertently omitted, I offer my sincerest apologies.

Vorrei ringraziare adesso le persone che più mi sono vicine.
Un ringraziamento speciale al mio ``Bro'' Ferdinando, che vive al di là della Manica. La distanza non intacca la nostra amicizia: è sempre presente e mi ricorda il vero significato di questa parola.
A Patrizio va un ringraziamento stratosferico, per tutti i suoi tentativi (falliti) di sabotaggio della mia esistenza fin dall'infanzia! Un grazie immenso anche a Lorenzo S., che per fortuna è lì a tenirci a bada. Siete due veri compagni di avventure.
A tutti gli altri ``Vendicatori'' grazie di esserci stati sempre nella vostra maniera.
Un grazie enorme al mio gigantesco amico Giacomino (``ino'' solo nel soprannome), che, pur essendo la mia ``custodia'', non manca mai di lusingare il mio ego, affermando che io sia ``grosso''.
Ringrazio il mio caro amico Andrea che con le giuste parole ha saputo ridarmi luce nel momento più buio della mia vita.

Con amore infinito, ringrazio Camilla, l'Amore della mia vita. Da quando le nostre vite si sono intrecciate, il tuo sostegno è stato una costante meravigliosa, una forza che mi sosterrà per sempre. Grazie di riempire il mio cuore di felicità e di rendere ogni giorno trascorso insieme un regalo inestimabile.

Esprimo la mia più sincera e profonda riconoscenza, accompagnata dal mio più sincero affetto,  a coloro che mi hanno offerto un sostegno concreto e determinante durante l'intero percorso che mi ha condotto al conseguimento del titolo di dottore: i miei genitori. A mia madre, che con un amore incondizionato si dedica quotidianamente alla felicità dei suoi figli, va il mio affetto più sincero e la mia profonda ammirazione. A mio padre, che ha dedicato la sua vita al bene della nostra famiglia, un pensiero colmo di nostalgia e amore: la sua presenza fisica ci manca profondamente, ma il suo spirito vive nei nostri cuori e ci guida. Spero un giorno di poter essere un padre degno del suo esempio.

Ringrazio di cuore tutta il resto della mia famiglia: mio fratello, mia sorella, Riccardo, i miei nipoti, i miei nonni e i miei zii, per essermi stati sempre vicini, ognuno a modo suo.

In the end a thought to all those who, in one way or another, have crossed my life and made me the person I am today.

\end{acknowledgements}


\tableofcontents 











\dedicatory{Al mio amato Papà} 


\mainmatter 

\pagestyle{thesis} 



\chapter{Introduction} 
\newcounter{my_steps}[subsubsection]

\label{Introduction} 


\def\a{\alpha}
\def\b{\beta}
\def\g{\gamma}
\def\D{{\mathbb D}}
\makeatletter
\newcommand{\bigcomp}{%
  \DOTSB
  \mathop{\vphantom{\sum}\mathpalette\bigcomp@\relax}%
  \slimits@
}
\newcommand{\bigcomp@}[2]{%
  \begingroup\m@th
  \sbox\z@{$#1\sum$}%
  \setlength{\unitlength}{0.9\dimexpr\ht\z@+\dp\z@}%
  \vcenter{\hbox{%
    \begin{picture}(1,1)
    \bigcomp@linethickness{#1}
    \put(0.5,0.5){\circle{1}}
    \end{picture}%
  }}%
  \endgroup
}
\newcommand{\bigcomp@linethickness}[1]{%
  \linethickness{%
      \ifx#1\displaystyle 2\fontdimen8\textfont\else
      \ifx#1\textstyle 1.65\fontdimen8\textfont\else
      \ifx#1\scriptstyle 1.65\fontdimen8\scriptfont\else
      1.65\fontdimen8\scriptscriptfont\fi\fi\fi 3
  }%
}
\makeatother


Stochastic processes, such as solutions of stochastic differential equations (SDEs), play a crucial role in modelling various random phenomena in fields such as finance, physics, biology, and engineering.  
Most of the time, exact simulation schemes for stochastic processes are often unachievable or have high computational costs, necessitating the use of approximation methods, possibly fast. We are interested in weak approximations of SDEs. Unlike strong approximation, which focuses on approximating sample paths closely, weak approximation aims to approximate the distribution of the process at specific time points. Here, we present a basic review of the literature on weak approximation of SDEs and acceleration techniques to boost the order of the approximation, given one.

\section{Weak approximation of solutions of SDEs}
We begin by providing a rigorous definition of Stochastic Differential Equation. Let $T>0$ and $d,d_W\in\N^*$. A Stochastic Differential Equation (hereafter SDE) is an equation of the form:
\begin{equation}\label{sde_autonomous_intro_gen}
dX^x_t = b(X^x_t) \, dt + \sigma(X^x_t) \, dW_t,  \quad X^x_0=x\in\R^d,
\end{equation}
where:
\begin{itemize}
    \item $(X^x_t, t\in[0,T])$ is the stochastic process in $\R^d$ that we want to simulate,
    \item $b:\R^d\rightarrow\R^d$ is the drift coefficient, which represents the deterministic ``trend'' and is a function of the current state $X^x_t$,
    \item $\sigma:\R^d\rightarrow\mcM(d,d_W,\R)$, where $\mcM(d,d_W,\R)$ are the real matrix with $d$ lines and $d_W$ columns, is the diffusion coefficient which represents the intensity of the random fluctuations and is also a function of the current state $X^x_t$,
    \item $(W_t,t\ge0)$ is a $d_W$-dimensional standard Wiener process  (or Brownian motion), representing the source of randomness.
\end{itemize}
We say that \eqref{sde_autonomous_intro_gen} has strong solutions if for every filtered probability space $(\Omega,\mscrF, (\mscrF_t)_t,\P)$ and Wiener process $(W_t)_t$ over it, there exists a process $(X^x_t)_t$ that verifies with probability 1 the following equality
\begin{equation}\label{sol_intro_gen}
    X^x_t = x + \int_0^t b(X^x_s) \, ds + \int_0^t \sigma(X^x_s) \, dW_s,
\end{equation}
where the second integral is an Itô integral.

The global existence and uniqueness of strong solutions $ (X^x_t)_t $ to the SDE \eqref{sde_autonomous_intro_gen} are guaranteed under certain conditions, often referred to as the locally Lipschitz and linear growth conditions:
\begin{itemize}
    \item For every compact set $K\subset \R^d$ exists a constant $ C_K>0$ such that for all $x,y\in K$,
    $$
    |b(x)-b(y)| + |\sigma(x)-\sigma(y)| \leq C_K |x-y|.
    $$
    \item There exists a constant $C>0$ such that for all $ x\in\R^d$,
    $$
    |b(x)| + |\sigma(x)| \leq C (1 + |x|).
    $$
\end{itemize}
We are interested in giving a notion of weak convergence for solutions of SDEs. If we consider a standard setting, given an $\R^d$-valued random variable $Z$ defined on a probability space $(\Omega,\mscrF,\P)$ and a sequence $(Z_n)_{n\in\N}$ defined over spaces $(\Omega_n,\mscrF_n,\P_n)$ we say that $Z_n$ converges weakly to $Z$ if
\begin{equation}\label{weak_conv_rv}
    \lim_{n\rightarrow\infty}  \E^n[f(Z_n)]\rightarrow\E[f(Z)], \text{ for all } f\in\mcC_b(\R^d),
\end{equation}
where $\E^n$ and $\E$ denote the expected values under the probabilities $\P^n$ and $\P$, and $\mcC_b(\R^d)$ is the space of real, continuous and bounded functions over $\R^d$.
So, fixed $x\in\R^d$ one could be interested to approximate the solution $X^x_T$ of \eqref{sde_autonomous_intro_gen} in the sense of \eqref{weak_conv_rv}, by constructing a sequence of random variables $(\hX^{n,x}_T)_{n\in\N}$ over spaces $(\Omega_n,\mscrF_n,\P_n)$. In the meantime, one could be interested in using a different vector space of  ``test'' functions $\mcF$ (to be specified) for which it is possible to give a rate of convergence or to use a more general definition that involves semigroups. Given $\mcF$ and the linear semigroup operator $P$ defined over it by $P_T f= \E[f(X^\cdot_T)]$, we say that a sequence of linear operators $\hP^n$ defined over $\mcF$ is an approximation of $P_T$ if
\begin{equation}\label{weak_conv_semigroup_point}
    \lim_{n\rightarrow\infty} |\hP^n f(x) - P_T f(x)|=0, \text{ for every }f\in\mcF \text{ and } x\in\R^d,
\end{equation}
or given a norm $\|\cdot\|$ over $\mcF$, if
\begin{equation}\label{weak_conv_semigroup_norm}
    \lim_{n\rightarrow\infty} \|\hP^n f - P_T f\|=0, \text{ for every }f\in\mcF.
\end{equation}
Weak approximation methods typically involve discretizing the continuous-time stochastic process into a sequence of random variables that are easier to handle computationally. The goal is to construct an approximation whose distribution closely matches that of the original process. Common techniques include Euler-Maruyama, Milstein, and higher-order schemes, which vary in complexity and accuracy. We now describe how to construct some approximations discretizing the SDE.


\subsection{Weak approximations schemes}\label{weak_approx_subsection_intro}
The most simple way of obtaining weak approximations is via weak approximation schemes.
The general idea of an approximation scheme is to create good approximations in law of $X^x_t$ for small $t$ such that composing these approximation schemes, the final law obtained is not so distant from the target one $X^x_T$. The most famous and used approximation scheme is the Euler-Maruyama scheme.

\subsubsection{The Euler-Maruyama scheme}
Let $T>0$, $n\in\N^*=\N\setminus\{0\}$ and consider the uniform grid $\Pi^n=\{t^n_k=kT/n\mid k=0,\ldots,n\}$ of step $T/n$.
The idea, like for the Euler scheme in the deterministic framework (ODEs), is to freeze the solution of the SDE between the regularly spaced discretization instants $t^n_k$. The discrete-time Euler-Maruyama scheme $\hX^{n,x}$ starting from $x$ is defined by 
\begin{equation}\label{Euler_scheme}
    \hX^{n,x}_0 = x, \qquad \hX^{n,x}_{t^n_{k+1}} = \hX^{n,x}_{t^n_k}+\frac{T}{n}b(\hX^{n,x}_{t^n_k})+\sigma(\hX^{n,x}_{t^n_k})\sqrt{\frac{T}{n}}(W_{t^n_{k+1}}-W_{t^n_{k}}), \quad k=0,\ldots,n-1. 
\end{equation}

In \cite{TT}, the following rate of convergence Theorem has been proved.
\begin{theorem}
    Let $b,\sigma$ be four times continuously differentiable on $\R^d$ with bounded partial derivatives. Assume $f:\R^d\rightarrow \R$ is four times differentiable such that $f$ and its derivatives have polynomial growth. Then there exists $C>0$ such that for every $x\in\R^d$, and $n\in\N^*$ large enough
    \begin{equation*}
        \E[f(\hX^{n,x}_T)]-\E[f(X^x_T)] \le \frac{C}{n}.
    \end{equation*}
\end{theorem}
What is shown in the previous Theorem guarantees under regularity of $b,\sigma$ and $f$ that the approximation $\E[f(\hX^{n,x}_T)]$ is a weak approximation of order one. 

\subsubsection{Weak error analysis and splitting operator techniques}
The Euler scheme, presented above, represents the simplest and most famous weak approximation scheme, and has order one. In practice, weak approximations converging with an order greater than one can be useful. One interesting way to get higher-order approximations is to build schemes with a higher order of convergence.
Alfonsi in \cite{AA_book} introduced a general framework to get, rigorously, approximation schemes of order 2. We consider the autonomous case \eqref{sde_autonomous_intro_gen} in which the solution $(X^x_t)_{t\in[0,T]}$ of the SDE is confined in a subset $\D\subset\R^d$, i.e.
\begin{equation}\label{domain_of_SDE_gen}
    \P(X^x_t\in\D,\forall t\in [0,T])=1.
\end{equation}
Given a multi index $\alpha=(\alpha_1,\ldots,\alpha_d)$ and $\partial_\alpha = \partial^{\alpha_1}_1\cdots \partial^{\alpha_d}_d$ the differential operator that differentiates $\alpha_i$ times in the $i$-th coordinates, we define the functional space 
$$
    \mcC^{\infty}_{\pol}(\D)=\{ f\in \mcC^\infty(\D,\R), \forall \alpha\in\N^d, \exists C_\alpha>0,e_\alpha\in\N^*, \forall x\in\D, \\
    |\partial_\alpha f(x)|\le C_\alpha (1+|x|^{e_\alpha}) \}
 $$
where $|\cdot|$ is the standard Euclidean norm. This is the space of smooth functions whose derivatives have a polynomial growth.
\begin{definition}\label{def_good_sequence}
    Let $f\in\mcC^\infty_\pol(\D)$. We say that $(C_\alpha,e_\alpha)_{\alpha\in\N^d}$ is a good sequence for $f$ if for any $\alpha\in\N^d$ and $x\in\D$ one has $|\partial_\alpha f(x)|\le C_\alpha(1+|x|^{e_\alpha})$. 
\end{definition}

Alfonsi \cite{AA_book} makes further assumption over the coefficient of the SDE \eqref{sde_autonomous_intro_gen}:  $b:\D\rightarrow\R^d$ and $\sigma:\D\rightarrow\mcM(d,d_W,\R)$ are such that
\begin{equation}\label{cpol_coeff_hp_intro}
    \forall 1\le i\le d, 1\le j\le d_W,\qquad x\in\D\mapsto b_i(x),\,x\in\D\mapsto(\sigma(x)\sigma^\top(x) )_{i,j}\in\mcC^\infty_\pol(\D),
\end{equation}
for instance, this assumption is fulfilled in the case of affine diffusion.
The infinitesimal generator $\mcL$ associated to the SDE is given by 
\begin{equation}\label{gen_inf_intro}
    f\in\mcC^2(\D), \quad \mcL f(x) = \sum_{i=1}^d b_i(x) \partial_i f(x) + \frac 12 \sum_{i,j=1}^d (\sigma(x)\sigma^\top(x) )_{i,j} \partial_i\partial_j f(x).
\end{equation}

\begin{definition}\label{def_resquired_assumptions_L}
    We say that $\mcL$, defined \eqref{gen_inf_intro}, satisfies the required assumptions over $\D$ if its SDE have coefficients $b$ and $\sigma$ that satisfies \eqref{cpol_coeff_hp_intro}, sub-linearity and has strong solutions that satisfy \eqref{domain_of_SDE_gen}.
\end{definition}

To study the weak error, we need to focus on the asymptotic behaviour of 
$$
\E[f(\hX^x_t)]-\E[f(X^x_t)] \qquad\text{for}\quad t\rightarrow0^+, \,f\in\mcC^\infty_\pol(\D).
$$
\begin{definition}\label{def_remainder_and_order_scheme}
    A function $\mcC^\infty_\pol(\D)\times(0,\infty)\times\D\ni (f,t,x)\mapsto Rf(t,x)\in\R$ is called a remainder of order $\nu\in\N$ if for any function $f\in\mcC^\infty_\pol(\D)$ with a good sequence $(C_\alpha,e_\alpha)_{\alpha\in\N^d}$, there exist $C,E,\eta>0$ depending only on the good sequence such that
    $$
    \forall t\in(0,\eta), \forall x\in\D, |Rf(t,x)|\le Ct^\nu(1+|x|^E).
    $$
    We say that $\hX^x_t$ is a potential weak $\nu$-th-order scheme for the operator $\mcL$ if $(f,t,x)\mapsto \E[f(\hX^x_t)]-\E[f(X^x_t)] $ is a remainder of order $\nu+1$. 
\end{definition}
It is important to remark that every exact simulation scheme is a potential weak $\nu$-th-order scheme for all $\nu\in\N$.
It is relatively straightforward to show the following result using Itô's formula.
\begin{prop}\label{diffusion_expansion_wea_Abook}
    Let $f\in\mcC^\infty_\pol(\D)$ and $\mcL$ that satisfies the required assumptions. Then for all $\nu\in\N$ and $t\ge0$
    \begin{equation}
        \E[f(X^x_t)] = \sum_{l=0}^\nu \frac{t^l}{l!} \mcL^lf(x)+\int_0^t\frac{(t-s)^\nu}{\nu!}\E[\mcL^{\nu+1}f(X^x_s)]ds
    \end{equation}
    and $(f,t,x)\mapsto \int_0^t\frac{(t-s)^\nu}{\nu!}\E[L^{\nu+1}f(X^x_s)]ds$ is a remainder of order $\nu+1$. 
\end{prop}
This last result implies one key equivalence that will be crucial in the weak error analysis in Chapter \ref{Chapter_CIR} and \ref{Chapter_Heston}.  
\begin{remark}
     $\hX^x_t$ is a potential weak $\nu$-th order scheme for $\mcL$ if, and only if 
    $$
    (f,t,x)\mapsto\E[f(\hX^x_t)]-\sum_{l=0}^\nu \frac{t^l}{l!}\mcL^l f(x) \text{ is a remainder of order } \nu+1.
    $$
\end{remark}

Alfonsi \cite{AA_book} proved the following result that is of key importance to build high order weak schemes.

\begin{prop}\label{scheme_composition_intro}
    Let $\mcL_1$ and $\mcL_2$ be two operators that satisfy the required assumptions on $\D$, and $\hX^1$ and $\hX^2$ be respectively potential weak second order discretization schemes on $\D$ for $\mcL_1$ and $\mcL_2$. Then, for $\lambda_1, \lambda_2>0$ and $f\in \mcC^\infty_\pol(\D)$ one has
    \begin{equation}\label{composition_expansion}
        \E\left[f(\hX^{2,\hX^{1,x}_{\lambda_1t}}_{\lambda_2t})\right] =\sum_{l_1+l_2\le 2} \frac{\lambda_1^{l_1} \lambda_2^{l_2}}{l_1!l_2!}t^{l_1+l_2}\mcL_1^{l_1}\mcL_2^{l_2} f(x) +Rf(t,x)
    \end{equation}
    where $Rf(t,x)$ is a remainder of order 3.
\end{prop}

We define the ordered composition of $k+1$ functions $\{f_0,f_1\ldots,f_k\}$ as follows
$$\bigcomp_{i=0}^{k}f_i=f_k\circ f_{k-1}\circ\cdots\circ f_0.$$

It is possible to prove the following result using the expansion \eqref{composition_expansion}.
\begin{corollary}
    Suppose the same hypotheses as in Proposition \ref{scheme_composition_intro} and 
    let $B$ be an independent Bernoulli random variable of parameter 1/2. Then, the two following schemes are potential second-order schemes for $\mcL_1+\mcL_2$
    $$
    \hX^{B,x}_t=B \hX^{2,\hX^{1,x}_{t}}_{t} + (1-B) \hX^{1,\hX^{2,x}_{t}}_{t}, \qquad \hX^{x}_t= \hX^{1,\hX^{2,\hX^{1,x}_{t/2}}_{t}}_{t/2}.
    $$
    More generally, let $k\in\N^*$, $\{\mcL_0,\mcL_1,\ldots,\mcL_k\}$ be operators that satisfy the required assumptions on $\D$, and $\{\hX^0,\hX^1,\ldots,\hX^k\}$ be respectively potential weak second order discretization schemes on $\D$ for them.
    Then
    \begin{align}
        \hX^{B,x}_t &=B \left(\bigcomp_{i=k}^{0} \hX^{i}_{t}(\cdot)\right) (x) +(1-B)\left(\bigcomp_{i=0}^{k}\hX^{i}_{t}(\cdot) \right) (x), \label{rand_leapfro_gen_scheme}\\
        \hX^{x}_t&=\left(\bigcomp_{i=k}^{1}\hX^{i}_{t/2}(\cdot)\right)\circ  \hX^{0}_{t}(\cdot) \circ \left(\bigcomp_{i=1}^{k}\hX^{i}_{t/2}(\cdot)\right) (x), \label{strang_splitting_gen_scheme}
    \end{align}
    are second order schemes for the operator $\mcL= \sum_{i=0}^k \mcL_i$. We call the first one the randomized leapfrog splitting scheme and the second one the Strang splitting scheme.
\end{corollary}

\subsubsection{The Ninomiya-Victoir scheme}
The theoretical tools introduced above allow us to demonstrate the convergence (and rate of speed) of the scheme proposed by Ninomiya and Victoir in \cite{NV}. The strength of this scheme is that it reduces the problem to the numerical approximation of Ordinary Differential Equations (ODEs). We consider an operator $\mcL$ that satisfies the required assumptions on $\D$, so it is defined by the formula \eqref{gen_inf_intro} for smooth coefficients $b$ and $\sigma\sigma^\top$. We define the following operators 
\begin{align*}
    V_0f(x)&=\sum_{i=1}^d b_i(x)\partial_i f(x) - \frac{1}{2}\sum_{k=1}^{d_W} \sum_{i,j=1}^d \partial_j \sigma_{i,k}(x)\sigma_{j,k}(x)\partial_i f(x)\\
    V_kf(x) &=\sum_{i=1}^d \sigma_{i,k}(x)\partial_i f(x), \quad\text{for}\quad k=1,\ldots,d_W.
\end{align*}
Then, one has the following identity 
$$
\mcL = \sum_{k=0}^{d_W}\mcL_k,
$$
where 
\begin{align*}
    &\mcL_k=\frac{1}{2} V^2_k= \sum_{i,j=1}^{d}\sigma_{j,k}(x)(\partial_j \sigma_{i,k}(x)\partial_i f(x)+\sigma_{i,k}(x)\partial_j\partial_if(x)) \quad\text{ for }k=1,\ldots,d_W, \\
    &\mcL_0=V_0,
\end{align*}
are well-defined and satisfy the required assumptions on $\D$. For all $k\in\{0,\ldots,d_W\}$, we call $v_k$ the vector field that verifies
$$
    V_k f(x) = v_k(x).\nabla f(x).
$$
We then consider the following ODEs:

\begin{align*}
    \partial_t X_0(t,x) &= v_0(X_0(t,x)), \quad t\ge 0,\, \qquad X_0(0,x)=x\in\D, \\
    \partial_t X_k(t,x) &= v_k(X_k(t,x)), \quad t\in \R, \qquad X_k(0,x)=x\in\D.
\end{align*}
One has the following result.
\begin{theorem}
    Let $\hX^k_t(x)=X_k(\sqrt{t}N,x)$, $N\sim \mcN(0,1)$ for all $k\in\{1,\ldots,d_W\}$ and $\hX^0_t(x)=X_0(t,x)$. 
    Under the above framework,  $\hX^{B,x}_t$ in \eqref{rand_leapfro_gen_scheme} and $\hX^{x}_t$  in \eqref{strang_splitting_gen_scheme} are potential second-order scheme for the operator $\mcL= \sum_{k=0}^{d_W}\mcL_k$.
\end{theorem}
Here, we give a remarkable example of the CIR model, for which the standard Euler scheme is not defined. 
\begin{ex}
    Consider the process
    \begin{equation}
        Y^y_t = (a-bY^y_t) dt + \sigma \sqrt{Y^y_t}dW_t, \quad Y^y_0=y\ge0,
    \end{equation}
    whose infinitesimal generator is given by 
    $$
        \mcL = \frac{1}{2}\sigma^2y\partial_y^2 + (a-by)\partial_y,\quad f\in\mcC^2,y\ge0.
    $$
    Following the Ninomiya-Victoir splitting, one has $\mcL=V_0+\frac{1}{2}V_1^2$ with
    \begin{equation}
        V_0f(y)=\Big(a-\frac{\sigma^2}{2}-by\Big)f'(y) \text{ and } V_1f(y) = \sigma\sqrt{y}f'(y),
    \end{equation}
    so, the two following ODEs have to be solved
    \begin{align*}
        \partial_t Y_0(t,y)&= (a-bY_0(t,y)), \hspace{20pt}t\ge 0, \hspace{1pt}\quad Y_0(0,y)=y\ge0,\\
        \partial_t\tilde{Y}_1(t,y) &= \sigma\sqrt{\tilde{Y}_1(t,y)},  \hspace{34pt}t\in \R, \quad Y_k(0,y)=y\ge0.
    \end{align*}
    The solutions are 
    \begin{align*}
        Y_0(t,y) &= e^{-bt}y +\psi_b(t)(a-\sigma^2/4), \quad \psi_b(t) = \frac{1-e^{-bt}}{b},  \\
        \tilde{Y}_1(t,y) &=((\sqrt{y}+t\sigma/2)_+)^2,
    \end{align*}
    but one can prove that substituting $\tilde{Y}_1$ with 
    \begin{equation}\label{Y1_intro}
        Y_1(t,y)=(\sqrt{y}+t\sigma/2)^2,
    \end{equation}
     gives the same development of the functional  $f\mapsto\E[f(\tilde{Y}_1(\sqrt{t}N,\cdot))]$ in terms of the operator $\frac{1}{2}V_1^2$.
    So one get using formulas \eqref{rand_leapfro_gen_scheme} and \eqref{strang_splitting_gen_scheme}
    \begin{align}
        \hY_t^{B,y} =& \,B\bigg(e^{-bt}\Big(\sqrt{y}+\frac{\sigma\sqrt{t}}{2}N_1\Big)^2+\psi_b(t)\Big(a-\frac{\sigma^2}{4}\Big)\bigg) \nonumber \\
        &+(1-B)\bigg(\sqrt{e^{-bt}y +\psi_b(t)(a-\frac{\sigma^2}{4})}+\frac{\sigma\sqrt{t}}{2}N_2\bigg)^2,  \label{rand_leapfro_gen_scheme_CIR} \\
        \hY_t^{y} =&\,e^{-bt/2}\bigg(\sqrt{e^{-bt/2}y +\psi_b(t/2)\Big(a-\frac{\sigma^2}{4}\Big)}+\frac{\sigma\sqrt{t}}{2}N\bigg)^2+\psi_b(t/2)\Big(a-\frac{\sigma^2}{4}\Big), \label{strang_splitting_scheme_CIR}
    \end{align}
    where we have exchanged the role of $Y_0$ and $Y_1$ in the Strang splitting to reduce the number of standard Gaussian random variables from 2 to 1. We remark that these schemes are defined only if $\sigma^2\le 4a$. 
\end{ex}

\subsection{Boosting techniques}
In the previous subsection, we saw how it is possible to construct weak approximations using schemes. Here, we show how it is possible to construct higher-order approximations using multiple times the same schemes calculated on different grids, either deterministic or random. We describe two similar but different approaches: Richardson-Romberg extrapolation and random grids techniques.

\subsubsection{Richardson-Romberg extrapolation}
Richardson-Romberg extrapolation, originally developed to improve the accuracy of numerical integration, can be adapted to enhance the convergence of approximation schemes. 
The application to weak order schemes has been introduced in the seminal paper \cite{TT}.
The technique consists of mixing schemes that evolve on different time-step grids. Let $T>0$, $\alpha,n\in\N^*$, and $\hX^{n,x}$ be a weak $\alpha$-th order scheme that runs on the grid $\Pi^n=\{t^n_k=kT/n\mid k=0,\ldots,n\}$ of step $T/n$. This approach relies on the existence of a development for the error $\E[f(\hX^{n,x}_T)]-\E[f(X^x_T)]$. If one can prove the following development
\begin{equation}\label{develop_order_1_romb}
    \E[f(\hX^{n,x}_T)]-\E[f(X^x_T)] = \frac{c_1}{n^\alpha} + O\left( \frac{1}{n^{\alpha+1}} \right)
\end{equation}
then there exist weights $w_1(\alpha)=1 -2^\alpha/(2^\alpha-1)$ and $w_2(\alpha)= 2^\alpha/(2^\alpha-1)$, $i\in{1,2}$ and one has
$$
    \E[w_1(\alpha) f(\hX^{n,x}_T) +w_2(\alpha) f(\hX^{2n,x}_T)] - \E[f(X^x_T)] = O\left( \frac{1}{n^{\alpha+1}} \right),
$$
so $\E[w_1(\alpha) f(\hX^{n,x}_T) +w_2(\alpha) f(\hX^{2n,x}_T)]$ is a weak approximation of order at least $\alpha+1$.
Under regularity of the drift and diffusion coefficients of the SDE, Talay and Tubaro proved in \cite{TT} the following expansion results for the Euler scheme \eqref{Euler_scheme}.
\begin{theorem}\label{euler_expansion}
    Let $b,\sigma\in\mcC^\infty_b(\R^d)$ and $f\in\mcC^\infty_\pol(\R^d)$. Let $n\in\N^*$ and $\hX^{n,x}$ be the Euler-Maruyama scheme starting from $x\in\R^d$, then for every integer greater than $\nu\in\N^*$ one has the following development
    \begin{equation}\label{develop_order_m_romb}
        E[f(\hX^{n,x}_T)]-\E[f(X^x_T)] = \sum_{i=1}^{\nu} \frac{c_i}{n^i} + O\left( \frac{1}{n^{\nu+1}} \right).
    \end{equation}
\end{theorem}
Under the hypotheses of the Theorem \ref{euler_expansion}, it is also possible to build weak approximations of order $\nu$ for every $\nu\in\N^*$ by systematically combining approximations computed with different step sizes of the Euler Scheme. In \cite{Pages}, Pagès has proven that if an expansion such \eqref{develop_order_m_romb} exists, then
\begin{equation}\label{multistep_RR_estimator}
    \E\left[ \sum_{i=1}^{\nu} \frac{(-1)^{\nu-i}i^\nu}{i!(\nu-i)!} f(\hat{X}^{in,x}_{T}) \right] = \E[f(X^x_{T})] + \frac{(-1)^{\nu}c_\nu}{\nu!}  n^{-\nu}+ O\left(\frac{1}{n^{\nu+1}}\right).
\end{equation}
\eqref{multistep_RR_estimator} proves that the linear operator $f\mapsto\E\left[ \sum_{i=1}^{\nu} \frac{(-1)^{\nu-i}i^\nu}{i!(\nu-i)!} f(\hat{X}^{in,\cdot}_{T}) \right]$ is a weak approximation of order $\nu$ of $P_T=\E[f(X^\cdot_{T})]$.  
Furthermore, Pagès has shown that if the Brownian increments for the Euler schemes are consistent, i.e. they are constructed using the same trajectories of the Wiener process, then one has 
\begin{equation*}
    \lim_{n\rightarrow\infty}\Var\left( \sum_{i=1}^{\nu} \frac{(-1)^{\nu-i}i^\nu}{i!(\nu-i)!} f(\hat{X}^{in,x}_{T}) \right) =\Var(f(X^x_T)).
\end{equation*}

\subsubsection{Random grids techniques}
Recently, Alfonsi and Bally \cite{AB} introduced a new technique to approximate general semigroups of linear operators $(P_t, t\ge 0)$ that works under a large framework. This technique permits building high order approximations by an intricate combination of elementary schemes running on random grids. In general, they consider a vector space $\mcF$ with a semigroup of linear operators $(P_t, t\ge 0)$ $P_t:\mcF\rightarrow\mcF$, and they equip the space with a family of seminorms $(\|\cdot\|_k)_{k\in\N}$ such that $\|f\|_k\le\|f\|_{k+1}$, for all $f\in \mcF$. Fixed a time horizon $T>0$, for all $n\in\N^*$ and $l\in\N$ they fix the time steps $h_l=\frac{T}{n^l}$. They consider a family of linear operators $Q_l:\mcF\rightarrow\mcF$, $l\in\N$, and denote, for $j\in\N^*$ $Q_l^{[j]}=Q_l^{[j-1]}Q_l$ as the operator obtained by composition ($Q^{[0]}= Id$). They suppose two conditions are met
\begin{equation}\tag{$\bbar{H_1}$}\label{H1_bar_intro}
    \begin{array}{c}\text{there exists } \alpha>0 \text{ and } \beta\in\N \text{ such that for any }  l,k\in\N, \text{ there exists } C>0,\text{ such that }\\ \|(P_{h_l}-Q_l)f\|_k \leq C\|f\|_{k+\beta} h_l^{1+\alpha} \text{ for all } f\in \mcF,
    \end{array}
\end{equation}

\begin{equation}\tag{$\bbar{H_2}$}\label{H2_bar_intro}
    \begin{array}{c}
        \text{for all } l,k\in \N, \text{ there exists } C>0 \text{ such that } \\ \max_{0\leq j\leq n^l}\|Q^{[j]}_l f\|_k + \sup_{t\leq T}\|P_t f\|_k\leq C\|f\|_k  \text{ for all } f\in \mcF.
      \end{array}
\end{equation}
The first quantifies how $Q_l$ approximates $P_{h_l}$, while the second one is a uniform bound with respect to all the seminorms.
Alfonsi and Bally show how one can construct, by mixing the operators $Q_l$, a linear operator $\cPh^{\nu,n}$ for which there exists $C>0$ and $k\in\N$ such that
\begin{equation}\label{rate_conv_nu_rg_intro}
    \|P_Tf-\cPh^{\nu,n}f\|_0\le C \|f\|_k n^{-\nu\alpha} \text{ for all } f\in\mcF.
\end{equation}
The general construction of $\cPh^{\nu,n}$  is described by trees \Cite[Section 3]{AB} and depends only on the value of the constant $\alpha$ and the desired boost $\nu$. Being quite complex and intricate, we do not reproduce here the construction in all its generality, but we describe only the procedure for $\nu=1,2$.
In the case $\nu=1$, one basically uses the same proof of Talay and Tubaro introduced in \cite{TT} for the weak error of the Euler scheme. Using the semigroup property, one has
\begin{equation}\label{dev_1_intro}
    P_Tf-Q_1^{[n]}f= P_{n h_1}f-Q_1^{[n]}f = \sum_{k=0}^{n-1}P_{(n-(k+1))h_1}(P_{h_1}-Q_1)Q_1^{[k]}f,
\end{equation}
and using twice \eqref{H2_bar_intro} and once \eqref{H1_bar_intro}
\begin{align*}
    \|P_Tf-Q^{[n]}_1f\|_0 & \le\sum_{k=0} ^{n-1} C \|[P_{h_1}-Q_1]Q_1^{[k]}f\|_0 \le \sum_{k=0} ^{n-1} C \|Q_1^{[k]}f\|_{\beta}h_1^{1+\alpha}  \\
                        & \le  C \|f\|_{\beta} n(T/n)^{1+\alpha}= C \|f\|_{\beta} n^{-\alpha},
\end{align*}
that proves $\cPh^{1,n} = Q_1^{[n]}$ is an approximation of order $\alpha$ of $P_T$.
To get the boost of order $\nu=2$, one use the same expansion used on $P_T=P_{nh_1}$ to $P_{(n-(k+1))h_1}$. One gets $P_{(n-(k+1))h_1}-Q_1^{[n-(k+1)]}=\sum_{k'=0}^{n-(k+2)}P_{(n-(k+k'+2))h_1}[P_{h_1}-Q_1]Q_1^{[k']}$ and then expand in~\eqref{dev_1_intro}:
\begin{align}
  P_Tf-Q^{[n]}_1f            & =\sum_{k=0}^{n-1}Q_1^{[n-(k+1)]}[P_{h_1}-Q_1]Q_1^{[k]}f + R(n)f, \label{dev_2_intro}                          \\
  \text{ with } R(n) & =\sum_{k=0}^{n-1}\sum_{k'=0}^{n-(k+2)}P_{(n-(k+k'+2))h_1}[P_{h_1}-Q_1] Q_1^{[k']}[P_{h_1}-Q_1] Q_1^{[k]} \nonumber.
\end{align}
This is not enough to produce our approximations because, in the extra terms on the right-hand side of \eqref{dev_2_intro}, there is $P_{h_1}$. This is solved by using again \eqref{dev_1_intro}, but this time over $P_{nh_2}-Q_2^{[n]}$ using the smaller time step $h_2$. One has 
\begin{align}
    P_Tf-Q^{[n]}_1f            & =\sum_{k=0}^{n-1}Q_1^{[n-(k+1)]}[Q_2^{[n]}-Q_1]Q_1^{[k]}f + \tilde{R}(n)f, \label{dev_2_mod_intro}                          \\
    \text{ with } \tilde{R}(n) & =\sum_{k=0}^{n-1}\sum_{j=0}^{n-1}Q_1^{[n-(k+1)]}P_{(n-(j+1))h_2}[P_{h_2}-Q_2] Q_2^{[j]} Q_1^{[k]} \nonumber\\
    +& \sum_{k=0}^{n-1}\sum_{k'=0}^{n-(k+2)}P_{(n-(k+k'+2))h_1}[P_{h_1}-Q_1] Q_1^{[k']}[P_{h_1}-Q_1] Q_1^{[k]} \nonumber.
\end{align}
As already done for $\nu=1$, using \eqref{H2_bar_intro} and \eqref{H1_bar_intro} over the terms in $\tilde{R}(n)$, one can prove $Q_1^{[n]} + \sum_{k=0}^{n-1}Q_1^{[n-(k+1)]}[Q_2^{[n]}-Q_1]Q_1^{[k]}$ to be a $2\alpha$ approximation of $P_T$. Unfortunately, simulating all the terms in \eqref{dev_2_mod_intro} would require a computational time in $O(n^2)$. Thus, the method would not be more efficient than using directly $Q_2^{[n^2]}$. To address this issue, Alfonsi and Bally introduce random grids and use a random variable $\kappa$ that is uniformly distributed on $\{0,\ldots,n-1\}$. One has
\begin{equation}\label{def_boost2_intro}
    \cPh^{2,n} = Q_1^{[n]} + n\E[Q_1^{[n-(\kappa+1)]}[Q_2^{[n]}-Q_1]Q_1^{[\kappa]}].
\end{equation}
In \cite{AB}, Alfonsi and Bally prove that the assumptions \eqref{H1_bar_intro} and \eqref{H2_bar_intro} are valid for several types of semigroups and approximation schemes that act over $C_b(\R^d)$, being equipped with the family of norms 
$$\|f\|_k =\sum_{0\le|\alpha|\le k}\sup_{x\in\R^d}|\partial^\alpha f(x)|.$$ 
Under regularity assumptions on the drift and diffusion coefficients, the following result was obtained for SDEs and the Euler Scheme.
\begin{prop}
    Consider $(X^x_t, t\ge0)$ the solution of \eqref{sde_autonomous_intro_gen} where $b,\sigma_j\in\mcC^\infty_b(\R^d)$ and the Euler Scheme \eqref{Euler_scheme}. Define $P_Tf(x)=\E[f(X^x_T)]$ and $Q_l=\E[f(X^{n^l,x}_T)]$. Then one has 
    \begin{equation}\tag{$H^E_1$}\label{H1_bar_euler}
        \begin{array}{c}\text{for any }  l,k\in\N, \text{ there exists } C>0,\text{ such that }\\ \|(P_{h_l}-Q_l)f\|_k \leq C\|f\|_{k+4} h_l^{2} \text{ for all } f\in C^\infty_b(\R^d),
        \end{array}
    \end{equation}
    
    \begin{equation}\tag{$H^E_2$}\label{H2_bar_euler}
        \begin{array}{c}
            \text{for all } l,k\in \N, \text{ there exists } C>0 \text{ such that } \\ \max_{0\leq j\leq n^l}\|Q^{[j]}_l f\|_k + \sup_{t\leq T}\|P_t f\|_k\leq C\|f\|_k  \text{ for all } f\in C^\infty_b(\R^d).
          \end{array}
    \end{equation}
\end{prop}
Under the same hypotheses, Alfonsi and Bally also proved that the estimator linked with $\cPh^{\nu,n}$ has a finite variance.

\section{Thesis contribution}
This section summarizes the main results obtained during the thesis.
These results are divided into three chapters related to three articles with Heston's model as a common link. This model describes the evolution of an asset and its volatility  Heston process that are the solutions respectively of $Y_t$ in \eqref{CIR_intro} and the couple $(S_t, Y_t)$ in \eqref{stock_heston_intro}-\eqref{CIR_intro}
\begin{align}
    dS^{s,y} _t &= rS^{s,y} _t dt +S^{s,y} _t\sqrt{Y^y_t} (\rho dW_t + \sqrt{1-\rho^2}dB_t), \quad S_0^{s,y} =s>0, \label{stock_heston_intro} \\
    dY^y_t &= (a-bY^y_t) dt + \sigma \sqrt{Y^y_t} dW_t, \quad Y^y_0 = y\ge 0. \label{CIR_intro}
\end{align}
In the first work (Chapter \ref{Chapter_CIR}), we deal with the construction of high order approximations of the volatility process \eqref{CIR_intro}. In the second one (Chapter \ref{Chapter_Heston}), we extend the results obtained to the couple $(S, Y)$. Instead, in the last work (Chapter \ref{Chapter_PDE}),
we study the PDE that describes the price of a European-type derivative under this model. Other minor results are listed in Appendix \ref{other_results_CIR}.
\subsection{Resume of Chapter \ref{Chapter_CIR}}
The goal is to build high order approximations of the CIR (Cox-Ingersoll-Ross) process \eqref{CIR_intro} and to prove rigorously a rate of convergence result. The theory developed for the Euler scheme in \cite{TT} and \cite{BT} and used in \cite{Pages} to create the multistep Richardson-Romberg approach, and the one developed in \cite{AB} do not cover this model. In fact, even if the drift coefficient is smooth (but not bounded), the diffusion coefficient is not even locally Lipschitz. Furthermore, because of this square root diffusion term, the standard Euler scheme is not well-defined: the increments are Gaussian distributed, so the positivity of the scheme is not achieved. 
In \cite{AA_MCOM}, Alfonsi proposed a second order scheme for the CIR that uses the Ninomiya-Victoir scheme coupled with a moment matching auxiliary scheme in a boundary of 0. Unfortunately, the techniques we develop cannot produce a proof for a general scheme that works even when $\sigma^2>4a$. Roughly speaking, this is because the analysis of the remainder in \eqref{def_boost2_intro} requires to be more elaborate: we do not only need to control its norm as in Subsection \ref{weak_approx_subsection_intro}, but we also have to handle its space regularity. Nevertheless, the following rate convergence Theorem has been proved.
\begin{theorem}\label{CIR_main_result_intro}
    Let $\hat{Y}^y_t$ be the scheme defined by~\eqref{strang_splitting_scheme_CIR} for $\sigma^2\le 4a$ and $Q_lf(y)=\E[f(\hat{Y}^y_{h_l})]$, for $l\ge 1$. 
    Then, for all $f\in \CpolK{18}$, we have $\cPh^{2,n}f(y)-P_Tf(y)=O(1/n^4)$ as $n\to \infty$.\\
    Besides, for  $f\in \Cpol$, we have $\cPh^{\nu,n}f(y)-P_Tf(y)=O(1/n^{2\nu})$.
\end{theorem}
Before giving details on how this Theorem has been proved (all details can be found in Chapter \ref{Chapter_CIR}), we mention that we later proved the same result under a slightly less demanding hypothesis on the regularity of $f$.  The result (Theorem \ref{thm_main_CIR_mod}) is in Appendix \ref{improvement_CIR_result}.

Theorem \ref{CIR_main_result_intro} is just an application of the random grids technique, so the key point is to prove a version of \eqref{H1_bar_intro} and \eqref{H2_bar_intro}.
In this work, we made the pedagogical choice of proving the main assumptions for the simpler space of polynomial functions. After that, we passed to the strictly larger class $$\CpolKL{k}{L}= \left\{f:\R_+ \to \R \text{ of class } \mathcal{C}^m \ :  \ \max_{j\in\{0,\ldots,m\}} \sup_{x\geq 0}\frac{|f^{(j)}(y)|}{1+y^L}<\infty \right\},$$
which we endowed with the following family of norms: for all $m\le k$ and $L'\ge L$
\begin{equation}\label{mLnorm_intro}
  \|f\|_{m,L'} = \max_{j\in\{0,\ldots,m\}} \sup_{x\geq 0} \frac{|f^{(j)}(y)|}{1+y^{L'}}.
\end{equation}
Across the proofs of \eqref{H1_bar_intro} and \eqref{H2_bar_intro} we consider an extension of the scheme in \eqref{strang_splitting_scheme_CIR}, in which the standard Gaussian random variable $N$ is replaced by a random variable $Z$ that satisfies the following criteria:
\begin{itemize}
    \item $Z$ is symmetric,
    \item $Z$ have all finite moments,
    \item $\E[Z^k]=\E[N^k]$ for $j\in\{2,4\}$.
\end{itemize}
We refer to these criteria as Assumption $(\mathcal{H}_Y)$, and the scheme is 
\begin{equation}\label{strang_splitting_scheme_CIR_Z}
    \hY_t^{y}=\varphi(t,y,\sqrt{t}Z) = e^{-bt/2}\bigg(\sqrt{e^{-bt/2}y +\psi_b(t/2)\Big(a-\frac{\sigma^2}{4}\Big)}+\frac{\sigma\sqrt{t}}{2}Z\bigg)^2+\psi_b(t/2)\Big(a-\frac{\sigma^2}{4}\Big).
\end{equation}
In the following, we point out when and why we must assume $Z=N$, $N\sim\mcN(0,1)$ to have the proof-machinery works.

\subsubsection{On the adapted version of \eqref{H2_bar_intro}}
In Section 4 of Chapter \ref{Chapter_CIR}, we prove this version of the assumption \eqref{H2_bar_intro}.
\begin{prop}\label{prop_H2_NV_intro}
    Let $\sigma^2\le 4a$ and $\hY^y_t=\varphi(y,t,\sqrt{t}N)$ be the scheme~\eqref{strang_splitting_scheme_CIR} with $N\sim \mathcal{N}(0,1)$. Let $T>0$ and  $m,L\in\N$ such that $L\ge m$. We define for $n\ge 1$ and $l\in \N$, $Q_lf(x)=\E[f(\hY^y_{h_l})]$with $h_l=\frac{T}{n^l}$. Then, there exists a constant $C\in \R_+^*$ such that for any $f \in \CpolKL{m}{L}$, $l\in \N$ and $t \in [0,T]$, 
    \begin{equation}\label{H2_NV_intro}   \left\|\E[f(Y^{\cdot}_t)]\right\|_{m,L} + \max_{0\le k\le n^l} \left\|Q_l^{[k]}f \right\|_{m,L} \le C \|f\|_{m,L}.
    \end{equation}
\end{prop}
We split the proof into two parts. We start by the upper bound for the CIR semigroup.
\begin{prop}\label{deriv_CIR_functionals_intro}
    Let $f\in\CpolKL{m}{L}$, $L\ge m$, $T>0$ and $t\in (0,T]$. Let  $Y^y$ be the CIR process starting from $y\ge 0$. Then, $\E[f(Y^{\cdot}_t)]\in\CpolKL{m}{L}$ and  we have the following estimate for some constant $C_\text{cir}(m,L,T)\in \R_+$:
    \begin{equation}\label{H0_estimate_CIR_intro}
      \left\|\E[f(Y^{\cdot}_t)]\right\|_{m,L} \le C_\text{cir}(m,L,T)\|f\|_{m,L}.
    \end{equation}
\end{prop}
The proof we gave in Chapter \ref{Chapter_CIR} relies heavily on the knowledge of the explicit form of the density of the CIR distribution.
To fully prove Proposition \ref{prop_H2_NV_intro}, we demonstrated the equivalent upper bound for the scheme.
\begin{prop}\label{prop_H2_sch_intro}
    Let $T>0, \sigma^2\le 4a$, $m,M \in \N$,  $Z$ be a symmetric random variable with density $\eta\in \mathcal{C}^M(\R)$ such that for all $i\in\{0,\ldots,M\}$, $|\eta^{(i)}(y)|=o(|z|^{-(2L+i)})$ for $|z|\rightarrow \infty$,  and $\eta^*_m\ge 0$ for all $1 \le m\le M$ (see Lemma~\ref{regular_density_intro} below for the definition of $\eta^*_m$).
    Let $Q_lf(x)=\E[f(\hat{Y}^y_{h_l})]$ with $\hat{Y}^y_t=\varphi(t,y,\sqrt{t}Z)$, $n\ge 1$, $l\in \N$ and $h_l=T/n^l$. Then, for any $L \in \N$, there exists $C\in \R_+$ such that: 
    $$  \max_{0\le j \le n^l}\|Q_l^{[j]}f\|_{m,L} \le C \|f\|_{m,L}, \ f\in \CpolKL{m}{L}, l \in \N.$$
  \end{prop}
  The key tool to prove Proposition \ref{prop_H2_sch_intro} is the following regularity result.
\begin{lemma}\label{regular_density_intro}
    Let $M,L\in \N$. Let $Y_1$ be defined in \eqref{Y1_intro} and $Z$ be a symmetric random variable with density $\eta\in \mathcal{C}^M(\R)$ such that for all $i\in\{0,\ldots,M\}$, $|\eta^{(i)}(z)|=o(|z|^{-(2L+i)})$ for $|z|\rightarrow \infty$. Then, for all function $f\in\CpolKL{M}{L}$, $m\in\{1, \ldots,  M\}$ and $t\in [0,T]$ one has the following representation
    \begin{equation}\label{repres_X1_functional_intro}
      \partial^m_y\E[f(Y_1(\sqrt{t}Z,y))] = \int_{-\infty}^\infty \int_0^1 (u-u^2)^{m-1} f^{(m)}(w(u,y,z)) \eta^*_m(z) dudz
    \end{equation}
    where  $w(u,y,z)=y+(2u-1)\sigma\sqrt{t}z\sqrt{y}+ \sigma^2tz^2/4$, $\eta^*_m(z)=(-1)^{m-1}  \left(\sum_{j=1}^m c_{j,m}  z^j\eta^{(j)}(z)\right)$, and the coefficients $c_{j,m}$  are defined by induction, starting from $c_{1,1}=-1$, through the following formula
    \begin{equation*}
      c_{j,m}   = \bigg(\frac{2j}{m-1}-4\bigg)  c_{j,m-1}\mathds{1}_{j<m} + \frac{2}{m-1}  c_{j-1,m-1}\mathds{1}_{j>1},     \  j\in\{1, \ldots, m\}, \, m\in\{2, \ldots,M\}.  \\
    \end{equation*}
   In particular, $c_{m,m}=-\frac{2^{m-1}}{(m-1)!}<0$. Furthermore, if the density $\eta$ is such that  $\eta^*_{m}(z)\geq0$ for all $z\in\R$, and all $ m\in\{1, \ldots,  M\}$, then there exists $C\in \R_+$ such that
    \begin{equation}\label{estimate_X1_functional_intro}
          \|\E[f(Y_1(\sqrt{t}Z,\cdot))]\|_{m,L} \leq (1+Ct) \|f\|_{m,L},\ t\in [0,T].
    \end{equation}
\end{lemma}
We stress here two things that are crucial in~\eqref{estimate_X1_functional_intro}: the same norm is used on both sides and the sharp time dependence of the multiplicative constant $(1+Ct)$. These properties are used in the proof of Proposition~\ref{prop_H2_sch_intro} to get~\eqref{H2_bar_intro}. 
We remark that the hypotheses are quite restrictive and do not allow using discrete random variables to continue our weak error analysis. Instead, we need a random variable $Z$ with a density that is regular enough and that satisfies the differential inequalities $\eta^*_m(z)\ge0$ up for all $M\in\N^*$, where $M$ depends on the order of the boost one wants to achieve with the random grids techniques.
In Chapter \ref{Chapter_CIR}, we show that $N\sim\mcN(0,1)$ satisfies all the required density hypotheses for all $M\in\N$ and, furthermore, that is the unique law that does that fixed the second and fourth moments. We proved the following characterization of the law $\mcN(0,1)$.
\begin{theorem}
    Let $Z$ be a symmetric random variable with a $\mcC^\infty$ probability density function $\eta$ such that $\E[Z^2]=1$, $\E[Z^4]=3$ and $\eta^*_m\ge 0$ for all $m\ge 1$. Then, $Z\sim\mcN(0,1)$.
\end{theorem}

\subsubsection{On the adapted version of \eqref{H1_bar_intro}}
In Section 4 of Chapter \ref{Chapter_CIR}, we prove this version of the assumption \eqref{H1_bar_intro}.
\begin{prop}
    Let $Z$ that satisfies $(\mathcal{H}_Y)$ and $\hY^y_t$ be the scheme \eqref{strang_splitting_scheme_CIR_Z}. Let $m,L\in\N$ such that $L+3\le m$ and $f\in\CpolKL{2(m+3)}{L}$. Then, there exists a constant $C>0$ such that for $t\in[0,T]$,
    $$ \|\E[f(\hat{Y}^\cdot_t)]- \E[f({Y}^\cdot_t)]\| _{m,L+3}\le C t^3 \|f\|_{2(m+3),L}.
    $$
\end{prop}
To prove this result, we compare $\E[f(\hat{Y}^\cdot_t)]$ and $\E[f(Y^\cdot_t)]$ with the expansion of order two $f+t\mcL f+\frac{t^2}{2}\mcL f$, as in the weak error analysis explained in \cite{AA_book} and resumed here.
We start by proving our framework's equivalent of Proposition \ref{diffusion_expansion_wea_Abook}. We need something more than show $(f,t,y)\mapsto\E[f({Y}^y_t)] - f(y)+t\mcL f(y)+\frac{t^2}{2}\mcL f(y)$ is a remainder of order three as in Definition \ref{def_remainder_and_order_scheme}, we need to prove also a norm estimate as specified in the next result, proved in Chapter \ref{Chapter_CIR}.
\begin{prop}\label{LCIR_Expansion_intro}
    Let $m,\nu,L\in \N$ such that $L+\nu+1\ge m$, $T>0$ and  $f\in\CpolKL{m+2(\nu+1)}{L}$. Let $Y^y$ be the CIR process and $\mcL$ its infinitesimal generator. Then, for  $t\in [0,T]$,  we have
    \begin{equation*}
      \E[f(Y^y_t)] =\sum_{i=0}^\nu \frac{t^i}{i!}\mcL^if(y) + t^{\nu+1}\int_0^1 \frac{(1-s)^\nu}{\nu!}\E[\mcL^{\nu+1}f(Y^y_{ts})]ds
    \end{equation*}
    where the function $y\mapsto \int_0^1 \frac{(1-s)^\nu}{\nu!}\E[\mcL^{\nu+1}f(Y^y_s)]ds$ belongs to $\CpolKL{m}{L}$, and we have the following estimate for all $t\in [0,T]$,
    \begin{equation}\label{rem_CIR_estimate_intro}
      \left\|\int_0^1 \frac{(1-s)^\nu}{\nu!}\E[\mcL^{\nu+1}f(Y^{\cdot}_{ts})]ds\right\|_{m,L+\nu+1} \le C\|f\|_{m+2(\nu+1),L},
    \end{equation}
    for some constant $C\in \R_+$ depending on $(a,b,\sigma,\nu,m,L,T)$.
\end{prop}
The norm estimate, given \eqref{H0_estimate_CIR_intro}, is just a trivial application of Itô's formula and the boundedness of the moments of $Y^y$.

As one can imagine, the trickiest part is to give the norm estimate for $\E[f(\hat{Y}^\cdot_t)] - f +t \cL f+\frac{t^2}{2} \cL^2f$. We proved the following result.
\begin{prop}\label{prop_H1sch_intro}
Let $Z$ that satisfies $(\mathcal{H}_Y)$, $\sigma^2\le 4a$ and $\hat{Y}^y_t$ be the scheme~\eqref{strang_splitting_scheme_CIR_Z}.  Let $m\in \N,L \in \N^*$ and $f \in \CpolKL{2(m+3)}{L}$. Then, we have for $t\in[0,T]$,
$$ \E[f(\hat{Y}^y_t)]=f(y) +t \cL f(y)+\frac{t^2}{2} \cL^2f(y) +\bar{R}f(t,y),$$
with $\|\bar{R}f(t,\cdot)\|_{m,L+3}\le C t^3 \|f\|_{2(m+3),L}$.
\end{prop}
The result is proved by composition of linear operators $P^0_t:f\mapsto f(Y_0(t,\cdot))$ and $P^1_t:f\mapsto \E[f(Y_1(\sqrt{t}Z,\cdot))]$ using the expansions proved in  Lemma~\ref{lem_expan_V} in Chapter \ref{Chapter_CIR}.
The first thing that jumps out is that to bound from above the $(M, L+3)$-norm of the remainder $\bar{R}f(t,\cdot)$ requires $2m+6$ derivatives, which is more than the $m+6$ derivatives needed for the CIR semigroup in Proposition~\ref{LCIR_Expansion_intro}, when $\nu=2$. The reason lies in the choice of the random variable $Z$. In Proposition~\ref{prop_H1sch_intro}, we consider variables $Z$ that only satisfies assumption $(\mathcal{H}_Y)$, so even discrete random variable (e.g. $Z$ such that $\P(Z=-\sqrt{3})=\frac{1}{6}=P(Z=\sqrt{3})$, and $P(Z=0)=\frac{2}{3}$) can be considered. So, to prove that $\bar{R}f(t,\cdot)$ is regular enough, we cannot use regularization techniques that use the existence of a regular density for the law of $Z$. In Chapter \ref{Chapter_CIR}, then, we prove the results Lemma~\ref{regular_rep} and Corollary~\ref{cor_psign} that exploit only the symmetry of $Z$ to prove the regularity of the remainder. We want to remark that in a later stage, we proved a regularity result, Lemma~\ref{regular_density_mod} that is an extension of Lemma~\ref{regular_density}, that permits using $Z=N$ or (fixed an order of desired boost $\nu\in\N^*$) possibly more general absolutely continuous random variables, to reduce the regularity demanded by $f$ (Proposition~\ref{prop_H1sch_regular_density}).

\subsubsection{Numerical experiments}
In Section 5 of Chapter \ref{Chapter_CIR}, we propose several numerical tests to validate our theoretical results and to push a little further the analysis ($\sigma^2>4a$) from an empirical point of view. First, we explain how to implement the approximations $\cPh^{2,n}$ and $\cPh^{3,n}$. We verify in the CIR model using the Ninomiya Victoir scheme (so $\sigma\le4a$) for $f$ smooth that one gets approximations of order 4 and 6 as expected from the theory. When $\sigma^2>4a$, we run similar tests for the Heston model with the second order scheme $(\exp(X^{NV,x,y},\hY^y_t))$ proposed \cite{AA_MCOM}
\begin{equation}\label{B_NV_heston_scheme}
    \hX^{NV,x,y}_t = x +(r-\frac{\rho}{\sigma}a)t +\frac{\rho}{\sigma}(\hY^y_t-y) +(\frac{\rho}{\sigma}b-\frac{1}{2})\frac{y+\hY^y_t}{2}t +\sqrt{y + B (1-\rho^2)(\hY^{,y}_t-y) t}N, 
\end{equation}
where $B\sim\mathcal{B}(1/2)$ is Bernoulli random variable independent of $N\sim\mcN(0,1)$. We price put options and get results similar to those obtained with the CIR. We ran simulations even in the case $\sigma^2>4a$ using for the CIR the general scheme proposed by Alfonsi in \cite{AA_MCOM}. Empirically, we observe that this kind of scheme that uses an auxiliary scheme in a neighbourhood of 0 is not well suited to be coupled with the random grids: the variance explodes in $n$. Finally, we propose a second order scheme $(\exp(X^{Ex,x,y}_t,Y^y_t))$
\begin{equation}\label{B_Ex_heston_scheme}
    \hX^{Ex,x,y}_t = x +(r-\frac{\rho}{\sigma}a)t +\frac{\rho}{\sigma}(Y^y_t-y) +(\frac{\rho}{\sigma}b-\frac{1}{2})\frac{y+Y^y_t}{2}t +\sqrt{y + B (1-\rho^2)(Y^{,y}_t-y) t}N,
\end{equation}
in which we simulate exactly the volatility $Y^y$. We show for this new scheme that the boost of order 2 works well and that the variance of the correction does not explode.

\subsection{Resume of Chapter \ref{Chapter_Heston}}
The material in Chapter \ref{Chapter_Heston} is an extension and refinement of what we proved in Chapter \ref{Chapter_CIR}. This time we build high order weak approximations for the logHeston process $(X^{x,y}_t, Y^y_t)$, solution of 
\begin{align}
    dX^{x,y} _t &= (r-\frac{1}{2}Y^y_t) dt +\sqrt{Y^y_t} (\rho dW_t + \sqrt{1-\rho^2}dB_t), \quad X_0^{x,y} =x\in\R, \label{log_stock_heston_intro} \\
    dY^y_t &= (a-bY^y_t) dt + \sigma \sqrt{Y^y_t} dW_t, \quad Y^y_0 = y\ge 0. \nonumber
\end{align}
for which we rigorously prove a rate of convergence result. We passed from equation \eqref{stock_heston_intro} to \eqref{log_stock_heston_intro}, applying the transformation $X^{x,y}_t=\log(S^{s,y_t})$, to have a couple of SDEs with bounded moments, since, now, their coefficients have at most linear growth.
As already done in the previous work, we want to approximate the semigroup $P_T:f(x,y)\mapsto\E[f(X^{x,y} _t, Y^y_t)]$. Once again, our primary tools are splitting techniques and the random grids approach.

\subsubsection{The second order schemes and the main Theorem}
In Section 5 of Chapter \ref{Chapter_CIR}, we did several numerical tests over the Heston process $(\exp(X^{x,y}_t), Y^y_t)$ using schemes \eqref{B_NV_heston_scheme} and \eqref{B_Ex_heston_scheme} in which we used the respectively Ninomiya-Victoir scheme and exact simulation for the CIR process. 
Here, two different schemes that do not use the Bernoulli random variables have been proposed. We split the infinitesimal generator 
$$
\mcL = \frac{y}{2}(\partial^2_x + 2\rho\sigma\partial_x\partial_y+\sigma^2\partial_y)+ (r-\frac{y}{2})\partial_x + (a-by)\partial_y,
$$
of the log-Heston SDE as $\mcL = \mcL_B+\mcL_W$, where $\brho=\sqrt{1-\rho^2}$ and
\begin{align*}
    \mcL_B &= \big((r-\frac{\rho a}{\sigma})+(\frac{\rho b}{\sigma}-\frac 12)y\big) \partial_x +\frac{y}{2}\brho^2\partial_x^2, \\
    \mcL_W &=\frac{y}{2}(\rho^2\partial^2_x + 2\rho\sigma\partial_x\partial_y+\sigma^2\partial^2_y) + (a-by)(\frac{\rho}{\sigma}\partial_x+\partial_y),
\end{align*}
that are infinitesimal generators of
\begin{equation*}
    \begin{cases}
      dX_t &= \big((r-\frac{\rho a}{\sigma})+(\frac{\rho b}{\sigma}-\frac 12)Y_t\big) dt+\brho\sqrt{Y_t} dB_t,\\
      dY_t &= 0,
    \end{cases}\quad\text{ and }\quad
    \begin{cases}
        dX_t &= (\frac{\rho a}{\sigma}-\frac{\rho b}{\sigma}Y_t) dt +\rho \sqrt{Y_t} dW_t,\\
      dY_t &= (a-bY_t) dt +\sigma \sqrt{Y_t} dW_t,
    \end{cases}
\end{equation*}
respectively. We emphasize that we have made (and will make again later) an abuse of notation by using the variables $(X, Y)$ in both systems; in fact, our goal here is only to associate the infinitesimal generators with the respective SDEs (not the solutions).
One should remark that the splitting is chosen to have in the second system $dX_t = \frac{\rho}{\sigma}dY_t$. So if one has an exact scheme for the CIR $Y^y_t$, has also an exact scheme for $\mcL_W$ given by
$$\varphi_W(t,x,y,Y^y_t)=(x+\frac{\rho}{\sigma}(Y^y_t-y),Y^y_t).$$
Instead, an exact scheme for $\mcL_B$ is given by 
$$
\varphi_B(t,x,y,N)=(x+(r-\rho a/\sigma)t -(1/2-\rho b/\sigma)yt+\brho\sqrt{ty}N,~y), \ \text{ with } N\sim\mcN(0,1),
$$
so, for all $f\in\CP{0}$, the semigroups associated to the two systems have the representation
$$
P^B_tf(x,y)=\E[f( \varphi_B(t, x, y, N ))], \qquad P^W_tf(x,y)=\E[f( \varphi_W(t,x,y,Y^y_t))].
$$
Composing schemes $\varphi_W$ and $\varphi_B$ as in \eqref{strang_splitting_gen_scheme} gives us the potential weak second order scheme $(\hX^{x,y}_t,Y^y_t)$, where the first component is given by
\begin{equation}\label{Xscheme_Ex_2N}
\hX^{x,y}_t = x +(r-\frac{\rho}{\sigma}a)t +\frac{\rho}{\sigma}(Y^y_t-y)  +(\frac{\rho}{\sigma}b-\frac{1}{2})\frac{y+Y^y_t}{2}t +\sqrt{(1-\rho^2)\frac{t}{2}}\bigg(\sqrt{y}N_1+\sqrt{Y^y_t}N_2\bigg).
\end{equation}
The linear operator $\hP^{Ex}_t$ associated to this scheme is, for all $\CP{0}$
\begin{equation}\label{def_PEx_intro}
\hP^{Ex}_t:f(x,y)\mapsto \E[f(\hX^{x,y}_t,Y^y_t)].
\end{equation}
This scheme is defined for all $\sigma>0$, but in practice simulating exactly the CIR process takes time, so we propose a second order scheme for the semigroup $P_W$ using the Ninomiya-Victoir splitting. One has $\mcL_W = \mcL_0+\mcL_1$ where
\begin{equation*}
  \mcL_0 =(a-\frac{\sigma^2}{4}-by)(\frac{\rho}{\sigma}\partial_x + \partial_y ),\qquad \mcL_1 =\frac{y}{2}(\rho^2\partial^2_x + 2\rho\sigma\partial_x\partial_y+\sigma^2\partial^2_y) + \frac{\rho \sigma}{4}\partial_x+ \frac{\sigma^2}4 \partial_y,
\end{equation*}
are the infinitesimal generators respectively associated to
$$\begin{cases} dX_t &= (\frac{\rho }{\sigma}(a-\sigma^2/4)-\frac{\rho b}{\sigma}Y_t) dt\\
dY_t &= (a-\sigma^2/4-bY_t) dt  \end{cases} \text{ and }
\begin{cases} dX_t &= \frac{\rho \sigma }{4} dt +\rho \sqrt{Y_t} dW_t\\
  dY_t &= \frac{\sigma^2}4 dt +\sigma \sqrt{Y_t} dW_t. \end{cases}$$
Let $\psi_b(t)=\frac{1-e^{-bt}}b$ (convention $\psi_b(t)=t$ for $b=0$) and define
\begin{align*}
  \varphi_0(t,x,y)&=\Big(x-\frac{\rho b}{\sigma}\psi_b(t)y +\frac{\rho}{\sigma}\psi_b(t)(a-\frac{\sigma^2}{4}),~e^{-bt}y+\psi_b(t)(a-\frac{\sigma^2}{4})\Big), \\
  \varphi_1(t,x,y)&=\Big(x+\frac{\rho }{\sigma}\big((\sqrt{y}+\frac{\sigma t}{2})^2-y\big),~(\sqrt{y}+\frac{\sigma t}{2})^2\Big). 
\end{align*}
We have for $f \in \CP{0}$,
\begin{equation*}
    P^0_t f(x,y)= f(\varphi_0(t,x,y)) \text{ and } P^1_t f(x,y)= \E[f(\varphi_1(\sqrt{t}G,x,y))], \text{ with } G\sim \mathcal{N}(0,1).  
\end{equation*}
The Ninomiya-Victoir scheme for $\mathcal{L}_W$ is then $P^0_{t/2}P^1_tP^0_{t/2}$ and is well-defined only for $\sigma^2\le 4a$. We define the linear operator
\begin{equation}\label{def_PNV_intro}
    \hat{P}_t^{NV}=P^B_{t/2}P^0_{t/2}P^1_tP^0_{t/2} P^B_{t/2},
\end{equation}
that is associated to the scheme $(\hX^{x,y}_t,\hY^y_t)$ where the first component is
\begin{equation}\label{Xscheme_NV_2N}
    \hX^{x,y}_t= x +(r-\frac{\rho}{\sigma}a)t +\frac{\rho}{\sigma}(\hY^y_t-y)  +(\frac{\rho}{\sigma}b-\frac{1}{2})\frac{y+\hY^y_t}{2}t +\sqrt{(1-\rho^2)\frac{t}{2}}\bigg(\sqrt{y}N_1+\sqrt{\hY^y_t}N_2\bigg),
\end{equation}
and the second one is the Ninomiya-Victoir scheme for the CIR \eqref{strang_splitting_scheme_CIR}.
We call $\mcC^{k}(\R\times\R_+)$ the space of continuous functions $f:\R \times \R_+ \to \R$ such that the partial derivatives $\partial^\a_x \partial^\b_y  f(x,y)$ exist and are continuous with respect to $(x,y)$ for all $(\a,\b)\in\N^2$ such that $\a+2\b\le k$. We define for every $L\in\N$
\begin{multline}\label{def_CpolkL_intro}
    \CPL{k}{L}=\{ f\in \mcC^{k}(\R\times\R_+) \mid \exists C>0 \text{ such that } \forall(\a,\b)\in\N^2, \a+2\b\le k,\\
    |\partial^\a_x \partial^\b_y  f(x,y)| \le C {\bf f}_L(x,y)   \},
\end{multline}
where ${\bf f}_L(x,y)=  (1+x^{2L}+y^{2L})$, for all $\ x\in \R$, $y\in \R_+$.
Furthermore, we set  $$\CP{k}= \cup_{L\in \N} \CPL{k}{L}.$$
The main result we proved is the following
\begin{theorem}\label{main_theorem_logHeston_intro}
    Let $\hat{P}_t$ be either $\hat{P}_t^{Ex}$ defined by~\eqref{def_PEx_intro} or $\hat{P}_t^{NV}$ by~\eqref{def_PNV_intro}. Let $T>0$, $n\in \N^*$ and $h_l=T/n^l$. Let $\cPh^{1,n}=\hat{P}_{h_1}^{[n]}$, $\cPh^{2,n}$ be defined by~\eqref{def_boost2_intro} and  $\cPh^{\nu,n}$ the further approximations developed in~\cite{AB}. Let $\nu \ge 1$.  For any $f\in \CP{12\nu}$ $x\in \R$ and $y\ge 0$, we have
     $$\cPh^{\nu,n}f(x,y)-P_T f(x,y) = O(1/n^{2\nu}).$$ 
\end{theorem}

\subsubsection{On the adapted version of (\ref{H1_bar_intro}) and (\ref{H2_bar_intro})}
As in Chapter \ref{Chapter_CIR}, to prove the required assumptions \eqref{H1_bar_intro} and \eqref{H2_bar_intro}, we need to fix a family of norms. We endow $\CPL{k}{L}$ with the following norm:
\begin{equation}\label{def_NormCpolKL_intro}
    \|f\|_{k,L}=\sum_{\alpha+2\beta \le k} \sup_{(x,y) \in \R \times \R_+} \frac{|\partial^\alpha_x\partial^\beta_y f(x,y)|}{\bff_L(x,y)}.
\end{equation}
In Chapter \ref{Chapter_Heston}, we do not repeat the analysis that allows the use of discrete random variables to get assumption \eqref{H1_bar_intro}; our second order schemes are obtained by composing exact schemes. This simplifies the analysis to get the regularity of our approximations and permits us to prove it all by studying the Cauchy problem of a slightly general log Heston SDE. The next proposition has been proved by Briani et al. in \cite{BCT}, the only additional result is the norm estimate \eqref{norm_estim_logHeston_semigroup}.
\begin{prop}\label{prop-rep-logHeston-estim_intro}
    Let $k,L \in \mathbb{N}$ and suppose that $f \in \CPL{k}{L}$. Let $\lambda \ge 0$, $c,d \in \R$. Let $(X^{t, x, y}, Y^{t, y})$ be the solution to the SDE, for $s\ge t$, 
    \begin{equation}\label{log-Heston-ext_SDEs_intro}
      \begin{cases}
        X^{t,x,y}_s &= x +\int_t^s (c+dY^y_r) dr + \int_t^s  \lambda \sqrt{Y^y_r} (\rho dW_r + \sqrt{1-\rho^2} dB_r)\\
        Y^{t,y}_s &= y+ \int_t^s (a-bY^y_r) dr +\sigma \int_t^s  \sqrt{Y^y_r} dW_r,
      \end{cases}
    \end{equation}
   and set
    $$ u(t, x, y)=\E[f(X_{T}^{t, x, y}, Y_{T}^{t, y})]=P_{T-t}f(x,y).$$
    Then, $u(t,\cdot,\cdot)\in \CPL{k}{L}$ and the following stochastic representation holds for $\a+2\beta \le k$,
    \small
    \begin{multline}\label{stoc_repr-new_intro}
    \partial_{x}^{\a} \partial_{y}^{\b} u(t, x, y)=\E\bigg[e^{-\b b(T-t)} \partial_{x}^{\a} \partial_{y}^{\b} f\big(X_{T}^{\b, t, x, y}, Y_{T}^{\b,t, y}\big) \\
     \quad+\b \int_{t}^{T}e^{-\b b(s-t)}\Big(\frac{\lambda^2}{2} \partial_{x}^{\a+2} \partial_{y}^{\b-1} u + d \partial_{x}^{\a+1} \partial_{y}^{\b-1} u \Big)\big(s, X_{s}^{\b, t, x, y}, Y_{s}^{\b,t, y}\big) d s\bigg],
    \end{multline}
    where $\partial_{x}^{\a} \partial_{y}^{\b-1}  u:=0$ when $\beta=0$ and $(X^{\b, t, x, y}, Y^{\b, t, y}), \beta \geq 0$, denotes the solution starting from $(x, y)$ at time $t$ to the SDE \eqref{log-Heston-ext_SDEs_intro}  with parameters
    \begin{equation*}
      \rho_{\b}=\rho,  \quad a_{\b}=a+\b \frac{\sigma^{2}}{2}, \quad b_{\beta}= b, \quad c_{\b}=r+\b \rho \sigma \lambda, \quad d_{\b}=d, \quad \sigma_{\b}=\sigma.
    \end{equation*}
    Moreover, one has the following norm estimation for the semigroup   
    \begin{equation}\label{norm_estim_logHeston_semigroup}
     \forall k, L \in \N, T>0, \  \exists C, \forall f \in  \CPL{k}{L}, t\in[0,T], \   \|P_tf\|_{k,L}\le  \|f\|_{k,L} e^{Ct}.
    \end{equation}
\end{prop}
The \eqref{H2_bar_intro} assumption can be obtained through the norm estimate \eqref{norm_estim_logHeston_semigroup}, just considering these sets of parameters
\begin{itemize}
    \item $\tilde{a}=a-\frac{\sigma^2}4$, $\tilde{b}
  =b$, $\tilde{c}=\frac{\rho}{\sigma}\left(a-\frac{\sigma^2}4\right) $, $\tilde{d}=-b\frac{\rho}{\sigma}$, $\tilde{\lambda}=0$, $\tilde{\sigma}=0$ for $P^0$,
  \item  $\tilde{a}=\frac{\sigma^2}4$, $\tilde{b}
  =0$, $\tilde{c}=\frac{\rho \sigma}4 $, $\tilde{d}=0$, $\tilde{\lambda}=\rho$, $\tilde{\sigma}=\sigma$, $\tilde{\rho}=1$ for $P^1$,
  \item $\tilde{a}=0$, $\tilde{b}
  =0$, $\tilde{c}=r-\frac{\rho a}{\sigma} $, $\tilde{d}=\frac{\rho b}{\sigma} -\frac 12$, $\tilde{\lambda}=\bar{\rho}$,  $\tilde{\sigma}=0$, $\tilde{\rho}=0$ for $P^B$,
  \item $\tilde{a}=a$, $\tilde{b}
  =b$, $\tilde{c}=\frac{\rho a}\sigma $, $\tilde{d}=-\frac{\rho b}\sigma$, $\tilde{\lambda}=\rho$, $\tilde{\sigma}=\sigma$, $\tilde{\rho}=1$ for $P^W$.
\end{itemize}
To get the \eqref{H1_bar_intro}, given the regularity results just shown, we prove a variant of Proposition \ref{scheme_composition_intro} that roughly tells the composition of schemes works as a composition of operators.
\begin{lemma}(Scheme composition) Let $\nu \in \N$ and $T>0$. Let $V_i$, $i\in \{1,\dots,I\}$,  be infinitesimal generators such that  there exists $k_i,L_i \in \N$ such that
    \begin{equation*}
      \forall k \in \N, \exists C\in \R_+,  \forall f \in \CPL{k+k_i}{L}, V_i f \in  \CPL{k}{L+L_i} \text{ and } \|V_if\|_{k,L+L_i}\le C \|f\|_{k+k_i,L}.
    \end{equation*}
    Let $k^\star=\max_{1\le i\le I} k_i$ and $L^\star=\max_{1\le i\le I} L_i$.
    We assume that for any $i$, $\hat{P}^i_t:\CPL{0}{L} \to \CPL{0}{L} $ is such that
    \begin{align*} 
        \forall k, L \in \N, 0\le  \bar{q}\le \nu+1, &\ \exists C,\ \forall f \in \CPL{k+\bar{q}k_i}{L}, \forall t\in[0,T], \notag \\& \|\hat{P}^i_t f -\sum_{q=0}^{\bar{q}-1}\frac{t^q}{q!} V_i^qf\|_{k, L+ \bar{q} L_i} \le C t^{\bar{q}} \|f\|_{k+\bar{q} k_i, L}. 
    \end{align*}
    Then, we have for $\lambda_1,\dots,\lambda_I\in [0,1]$, 
    \begin{align*}
     & \forall k, L \in \N, 0\le  \bar{q}\le \nu+1, \ \exists C, \forall f \in \CPL{k+\bar{q}k^\star}{L}, \forall t \in [0,T]  \notag \\
      & \left\|\hat{P}^I_{\lambda_I t} \dots \hat{P}^1_{\lambda_1 t} f -\sum_{q_1+\dots+q_I\le \bar{q}-1}\frac{\lambda_1^{q_1}\dots \lambda_I^{q_I} t^{q_1+\dots+q_I} }{q_1!\dots q_I!} V_I^{q_I}\dots V_1^{q_1}f\right\|_{k, L+ \bar{q} L^\star} \le C t^{\bar{q}} \|f\|_{k+\bar{q} k^\star, L}. 
    \end{align*}
\end{lemma}

\subsubsection{Numerical experiments}
In Section 3 of Chapter \ref{Chapter_Heston}, we remark the first component of second order schemes \eqref{Xscheme_Ex_2N} and \eqref{Xscheme_NV_2N} are normally distributed given respectively $Y^y_t$ and $\hY^y_t$, so we can save one Gaussian random variable and simulate instead
\begin{align}
    \hX^{EX,x,y}_t &= x +(r-\frac{\rho}{\sigma}a)t +\frac{\rho}{\sigma}(Y^y_t-y)  +(\frac{\rho}{\sigma}b-\frac{1}{2})\frac{y+Y^y_t}{2}t +\sqrt{(1-\rho^2)\frac{y+Y^y_t}{2}t}N,\label{Xscheme_Ex_1N}\\
    \hX^{NV,x,y}_t &= x +(r-\frac{\rho}{\sigma}a)t +\frac{\rho}{\sigma}(\hY^y_t-y)  +(\frac{\rho}{\sigma}b-\frac{1}{2})\frac{y+\hY^y_t}{2}t +\sqrt{(1-\rho^2)\frac{y+\hY^y_t}{2}t}N,\label{Xscheme_NV_1N}
\end{align}
that produce respectively the same laws of \eqref{Xscheme_Ex_2N} and \eqref{Xscheme_NV_2N}.
We run several tests (Put and Asian options), both in low volatility regime ($\sigma^2\le4a$) and in high volatility regime ($\sigma^2>4a$), that prove the effectiveness of the standard second order scheme $\cPh^{1,n}$ and of the boosted approximation $\cPh^{2,n}$: the empirical evidence confirm what proved in the theory giving approximation of order 2 and 4 respectively.
We run numerical experiments to study the variance of the estimators depending on $n$, testing two different couplings. Besides the choice of coupling, we also consider the schemes for the first component studied in Chapter \ref{Chapter_CIR} \eqref{B_NV_heston_scheme} and \eqref{B_Ex_heston_scheme} that use the Bernoulli random variable. We show the new schemes  \eqref{Xscheme_Ex_2N} and \eqref{Xscheme_NV_2N} are better suited to being used with random grids (besides, they require simulating less random variables) independently to the coupling chosen. Furthermore, the coupling studied in Cheng \cite{ZhengC} produce less variance.

In the end, we apply the random grids techniques to a modification of second order weak scheme proposed by Alfonsi \cite{AA_NPK} for the multifactor Heston model (known to be an excellent proxy of rough Heston Model see, for example, \cite{AEE19, AK24, BB23}). Unlike the Heston case, the approximation $\cPh^{2,n}$ is only defined when the parameters belong to a low volatility regime set. Nevertheless, we achieved good results showing the boost of random grids works.

\subsection{Resume of Chapter \ref{Chapter_PDE}}
Examining the PDE linked to the infinitesimal generator of an SDE offers a valuable method for demonstrating the regularity (smoothness) of the corresponding semigroup $P$, and Proposition \ref{prop-rep-logHeston-estim_intro} serves as a notable illustration of this principle.

Briani et al. \cite{BCT} obtained the formula \eqref{stoc_repr-new_intro} for the derivatives of the semigroup associated with the logHeston process studying this slightly general PDE
\begin{equation}\label{reference_PDE_intro}
    \begin{cases}
      \partial_tu(t,x,y) +\mcL u(t,x,y) +\varrho u(t,x,y) = h(t,x,y),\quad   t\in[0,T), &x\in\R, y\in\R_+, \\
      u(T,x,y) = f(x,y), \quad  &x\in\R, y\in \R_+,
    \end{cases}
\end{equation}
where $\mcL$ is the following differential operator 
\begin{equation}\label{reference_infinit_gen_intro}
\mcL = \frac{y}{2}(\lambda^2\partial^2_{x} + 2 \rho\lambda\sigma \partial_{x}\partial_{y} +  \sigma^2 \partial^2_{y}) +(c+ d y) \partial_x + (a-by) \partial_y,
\end{equation}
and  $b,c,d\in\R$, $a,\lambda,\sigma>0$ $\rho\in(-1,1)$, $\varrho\in\R$.
In this chapter, we want to delve deeper into the analysis of this PDE. In particular, we want to find minimal regularity hypotheses under which the function (produced via Feynman-Kac)
\begin{equation}\label{u_gen_sol_intro}
    u(t,x,y)=\E\bigg[e^{\varrho (T-t)}f(X^{t,x,y}_T,Y^{t,y}_T)-\int_t^Te^{\varrho (s-t)} h(s,X^{t,x,y}_s,Y^{t,y}_s)ds\bigg],
\end{equation}
is the unique classical or viscosity solution of \eqref{reference_PDE_intro}.
\subsubsection{Classical solutions results}
The first result concerns the resolution of \eqref{reference_PDE_intro}.
In \cite[Proposition 5.3]{BCT} Briani, Caramellino and Terenzi proved that if $h=0$, $\varrho=0$ and, for all $m+2n\le4 $, $f$ has partial derivatives  $\partial^m_x\partial^n_y f\in\mcC(\R\times\R_+)$  that have polynomial growth, then, for all $m+2n+4l\le4 $, the function $u$ in \eqref{u_gen_sol_intro} has partial derivatives $\partial^l_t\partial^m_x\partial^n_y u\in\mcC([0,T]\times\R\times\R_+)$ that have polynomial growth in $x$ and $y$ uniformly in $t$. In particular $u\in \mcC^{1,2}([0,T]\times(\R\times\R_+))$ and solves the reference PDE \eqref{reference_PDE_intro}.
In Section \ref{classical_section}, we give two refinements regarding classical solutions.

The first result concerns a verification result in which we prove that $u$ (as in \eqref{u_gen_sol_intro}) is a solution of the reference PDE \eqref{reference_PDE_intro} for more general $h$ and less regular $f$.  We prove, in fact, the following result. 
\begin{prop}\label{hp_minimal_u_sol_intro}
    Let $u$ be defined as in \eqref{u_gen_sol_intro}. Let $f$ and $h$ be such that, for all $m+2n\le2 $,  $f$ has partial derivatives $\partial^m_x\partial^n_y f\in\mcC(\R\times\R_+)$ with polynomial growth and $h$ has partial derivatives  $\partial^m_x\partial^n_y h\in\mcC([0,T)\times\R\times\R_+)$. Furthermore, suppose $h$ and $\partial_y h$ be such that $|h|^K_{\alpha,2}$, $|\partial_y h|^K_{\alpha,2}<\infty$ for all $K$ compact set contained in $[0,T)\times\R\times\R_+$.
    Then $u$ has partial derivatives $\partial_t u,\partial^m_x\partial^n_y u\in\mcC([0,T)\times\R\times\R_+)$ with $m+2n\le2 $ with polynomial growth in $x$ and $y$ uniformly in time, and solves
    \begin{equation}
      \begin{cases}
      \partial_t u(t,x,y)+ \mcL u(t,x,y) +\varrho u(t,x,y)= h(t,x,y),\qquad& t\in [0,T), \,(x,y)\in \R\times\R_+,\\
      u(T,x,y)= f(x,y),\qquad& (x,y)\in \R\times\R_+.
      \end{cases}	
    \end{equation}
\end{prop}

$|h|^K_{\alpha,2}$, $|\partial_y h|^K_{\alpha,2}$ in Proposition \eqref{hp_minimal_u_sol_intro} denote weighted Hölder seminorms that roughly measure the hölderianity of $h$ and $\partial_y h$ with a precise power of the weight. For a precise definition, we refer to \eqref{weighted_alfa_holder_norm_1} and \eqref{weighted_alfa_holder_norm_2}. Furthermore, we observe that Proposition \ref{hp_minimal_u_sol_intro} eases the requirements on the function $f$, demanding the condition $m+2n\le 2$ be satisfied rather than the stricter condition $m+2n\le4$.

The second contribution states sufficient conditions to ensure the uniqueness of classic solutions.
\begin{prop}
    There is at most one classical solution $u\in \mcC^{1,2}([0,T)\times(\R\times\R^*_+)) \cap \mcC^{1,1,1}([0,T)\times\R\times\R_+)\cap\mcC([0,T]\times\R\times\R_+)$ to PDE \eqref{reference_PDE_intro} such that the solution has polynomial growth in $(x,y)$ uniformly in $t$. So, in particular, under the hypothesis of Proposition \ref{hp_minimal_u_sol_intro}, $u$ defined as in \eqref{u_gen_sol_intro} is the unique solution.
\end{prop}
The importance of this proposition is particularly relevant when the Feller condition is not satisfied (cf. page \pageref{Page:Feller_condition_uniqueness}). Requiring $u$ to belong to $\mcC([0,T]\times\R\times\R_+)$ means giving conditions on the whole boundary, and for uniformly elliptic operators, it is quite natural. Here, the degeneracy and the fact that the associated diffusion can reach the boundary (where the log-Heston operator $\mcL$ is degenerate) impose an additional condition on the first-order derivatives to obtain uniqueness.
\subsubsection{Viscosity solutions results}
To get the regularity of $u$ necessary to satisfy in a classical sense the PDE \eqref{reference_PDE_intro}, we asked $f$ to be smooth enough.
In order to drop this regularity hypothesis, in Section \ref{viscosity_section} of Chapter \ref{Chapter_PDE}, we study the problem from the more general point of view of viscosity solutions. Being \eqref{reference_PDE_intro} a linear PDE, we used a standard smoothing argument (truncation and mollification) to prove the following viscosity verification theorem in which we consider the final data to be just continuous.
\begin{prop}\label{existence_viscsol_fcont_intro}
    Let $f\in \mcC(\R\times \R_+)$ and $h\in \mcC((0,T]\times\R\times\R_+)$ be such that for all compact set $\mcK_T\subset[0,T]\times\R\times\R_+$ there exists $p>1$ such that
    \begin{equation}\label{integrability_condition_intro}
      \sup_{(t,x,y)\in\mcK_T}\big\|f(X_T^{t,x,y},Y^{t,y}_T)\big\|_{L^p(\Omega)}, \sup_{(t,x,y)\in\mcK_T}\int_t^T \|h(s,X^{t,x,y}_s,Y^{t,y}_s)\|_{L^p(\Omega)}ds<\infty .
    \end{equation}
    Then, $$u(t,x,y)=\E\left[e^{\varrho (T-t)}f(X_T^{t,x,y},Y^{t,y}_T)-\int_t^T e^{\varrho (s-t)} h(s,X^{t,x,y}_s,Y^{t,y}_s)ds\right]$$ belongs to $\mcC([0,T]\times\R\times\R_+)$ and is a viscosity solution to the PDE \eqref{reference_PDE_intro}.
\end{prop}
Another key tool is a comparison principle (stated in Proposition \ref{comparison_principle}). Roughly speaking, it says that a sub-solution $w$ and super-solution $v$ that starts ordered ($w\le v$) stay ordered for any time. This is a key tool to prove the uniqueness of the viscosity solution. We call $\bDf$ the closure of the set of discontinuities of a function $f$ and state the main result.
\begin{theorem}\label{verification_theorem_intro}
    Let $f:\R\times\R_+\rightarrow\R$ be a function with polynomial growth, such that $\bDf$ has zero Lebesgue measure. Let $h\in \mcC([0,T)\times\R\times\R_+)$ be with polynomial growth in $(x,y)$ uniformly in $t$.
    Then $u$ in \eqref{u_gen_sol_intro} belongs to $\mcC\big(([0,T]\times\R\times\R_+)\setminus (\{T\}\times \bDf) \big)$, has polynomial growth in $(x,y)$ uniformly in $t$ and, in this class of functions, is the unique viscosity solution to the problem \eqref{reference_PDE_intro}. 
\end{theorem}
Let us sketch the main ideas of the proof.  Taken a general solution $v$ of the PDE \eqref{reference_PDE_intro} (with $f$ and $h$ as in the hypotheses), we sandwiched it between two sequences, $(u^+_n)_{n\in\N}$ of continuous super-solutions and $(u^-_n)_{n\in\N}$ of continuous super-solutions that satisfy two hypotheses:
\begin{align}
    &u^-_n(T,\cdot,\cdot)\le v_*(T,\cdot,\cdot)\le v^*(T,\cdot,\cdot)\le u^+_n(T,\cdot,\cdot), \label{squeeze_v_intro} \\
    \text{for any compact } &\text{set } \mcK_T\subset[0,T)\times\R\times\R_+, \text{ one has }\lim_{n\rightarrow\infty} \|u^\pm_n-u\|_{\infty,\mcK_T}  \label{local_uniform_conv_VT_intro}.
\end{align}
The first property guarantees $u^-_n\le v_*\le v^*\le u^+_n$ thanks to the comparison principle (Proposition \ref{comparison_principle}), then the second one guarantees $v_*=u=v^*$ in $[0,T)\times\R\times\R_+$.  

Prior research has established existence and uniqueness results for the Heston PDE and even more general jump-diffusion processes (as demonstrated in Costantini et al. \cite{CPD12}). However, the assumptions used in these studies, when applied to the Heston model, necessitate the adoption of the Feller condition.
The significance of this last result lies in its validity even when the Feller condition $\sigma^2\le2a$ is not satisfied. As far as we know, this is an original contribution to the existing literature on the Heston model.

The techniques we developed to deal with viscosity solution for the logHeston PDE \eqref{reference_PDE_intro} had been fruitfully used to study the convergence of numerical schemes.
In the closure of Chapter \ref{Chapter_PDE}, we apply this approach to prove the convergence of the hybrid scheme from \cite{BCT}. In \cite{BCT}, a rate of function has been proved under strong regularity assumptions on the test functions. 
In Section \ref{approximation_section}, we relax this request, and we prove (see Theorem \ref{convergence_theorem}) the convergence for functions that have just suitable continuity properties.
This result is confirmed empirically by the numerical experiment carried out in \cite{BCZ}, which computes the price of a European put option in the Heston model. Other numerical experiments that use 
the hybrid algorithm for the Bates and for the Heston-Hull-White models have been carried out in \cite{BCTZ} and \cite{BCZ2} respectively.



\chapter{High order approximations for the CIR process using random grids} \label{Chapter_CIR} 
The material for this chapter has been released in \cite{AL}.
\chapabstract{\\We present new high order approximations schemes for the Cox-Ingersoll-Ross (CIR) process that are obtained by using a recent technique developed by Alfonsi and Bally (2021) for the approximation of semigroups. The idea consists in using a suitable combination of discretization schemes calculated on different random grids to increase the order of convergence. This technique coupled with the second order scheme proposed by Alfonsi (2010) for the CIR leads to weak approximations of order $2k$, for all $k\in\N^*$. Despite the singularity of the square-root volatility coefficient, we show rigorously this order of convergence under some restrictions on the volatility parameters. We illustrate numerically the convergence of these approximations for the CIR process and for the Heston stochastic volatility model and show the computational time gain they give. }
\section*{Introduction}

The present paper develops approximations, of any order, of the semigroup $P_t f(x):=\E[f(X^x_t)]$ associated to the following Stochastic Differential Equation (SDE) known as the Cox-Ingersoll-Ross (CIR) process
\begin{equation}\label{CIR_SDE}
  X^x_t =x+ \int_0^t (a-kX^x_s)ds+ \int_0^t \sigma\sqrt{X^x_s}dW_s, \quad t\ge 0,
\end{equation}
where $W$ is a Brownian motion, $x,a\ge 0$, $k\in \R$ and $\sigma>0$. Let us recall that the process~\eqref{CIR_SDE} is   nonnegative and  the semigroup $(P_t)_{t\ge 0}$ is well-defined on the space of functions $f:\R \to \R$ with polynomial growth. The diffusion~\eqref{CIR_SDE} is widely used in financial mathematics, in particular because of its simple parametrization and the affine property that enables to use numerical methods based on Fourier techniques. We mention here the Cox-Ingersoll-Ross model~\cite{CIR} for the short interest rate and the Heston stochastic volatility model~\cite{Heston}, that have been followed by many other ones.   Developing efficient numerical methods for the process~\eqref{CIR_SDE} is thus of practical importance. 

To deal with the approximation of SDE's semigroups, a common approach is to consider stochastic approximations and the most standard one is the Euler-Maruyama scheme. The error between the approximated semigroup and the exact one is called the weak error, as opposed to the strong error that quantifies the error "omega by omega" on the probability space. 
The seminal work of Talay and Tubaro~\cite{TT} shows, under regularity assumptions on the SDE coefficients, that the weak error given by the Euler-Maruyama scheme is of order one, i.e. is proportional to the time step. They also obtain an error expansion that enables to use Richardson-Romberg extrapolations as developed by Pagès~\cite{Pages}. Higher order schemes for SDEs and related extrapolations have been proposed by Kusuoka~\cite{Kusuoka}, Ninomiya and Victoir~\cite{NV}, Ninomiya and Ninomiya~\cite{NiNi} and Oshima et al.~\cite{OTV} to mention a few. Recently, Alfonsi and Bally~\cite{AB} have given a method to construct weak approximation of general semigroups of any order by using random time grids. 

These general results on weak approximation of SDEs do not apply to the CIR~\eqref{CIR_SDE} process. This is due to the diffusion coefficient, namely the singularity of the square-root at the origin. Besides this, classical schemes such as the Euler-Maruyama scheme are not well-defined for~\eqref{CIR_SDE}, and one has to work with dedicated schemes.   Under some restrictions on the parameters, the weak convergence of order one for some discretization schemes of the CIR process has been obtained by Alfonsi~\cite{AA_MCMA}, Bossy and Diop~\cite{BoDi}, and more recently by Briani et al.~\cite{BCT} who also study the weak convergence of a semigroup approximation for the Heston model. We also mention the earlier work by Altmayer and Neuenkirch~\cite{AlNe} that precisely studies the weak error for the Heston model.
 Adapting ideas from Ninomiya and Victoir~\cite{NV} who developed a second order scheme for general SDEs, Alfonsi~\cite{AA_MCOM} has introduced second order and third order schemes for the CIR and proved their weak order of convergence, without any restriction on the parameters.

 The goal of the present paper is to boost the second order scheme developed in~\cite{AA_MCOM} and get approximations of any order. To do so, we rely on the method developed recently by Alfonsi and Bally~\cite{AB} to construct approximation of semigroups of any order. Roughly speaking, this method allows to get, from an elementary weak approximation scheme of order~$\alpha>0$, approximation schemes of any order by computing the elementary scheme on appropriate random grids. The method is illustrated in~\cite{AB} on the case of the Euler-Maruyama scheme for SDEs, under regularity assumptions on the coefficients that do not hold for the CIR process~\eqref{CIR_SDE}. This method is presented briefly in Section~\ref{Sec_nutshell}. It relies on an appropriate choice of a function space endowed with a family of seminorms. Section~\ref{Sec_O2CIR} then presents the second order scheme that is used as an elementary scheme to get higher order approximation. It states in Theorem~\ref{thm_main} the main result of this paper: we prove, when $\sigma^2\le 4a$, that we get weak approximations of any orders for smooth test functions~$f$ with derivatives having at most a polynomial growth.  
 Section~\ref{Sec_pol_fct} illustrates the boosting method when considering the space of polynomials function with their usual norm. In this simple case, proofs are quite elementary so that the method can be followed easily. Section~\ref{Sec_main} is more involved: it  first defines the appropriate family of seminorms on the space of smooth functions with derivative of polynomial growth and then proves Theorem~\ref{thm_main}. Last, we illustrate in Section~\ref{Simulations} the convergence of the high order approximations for different parameter sets. It validates our theoretical results and shows important computational gains given by the new approximations. We also test the method on the Heston model and obtain similar convincing results.


\section{High order schemes with random grids: the method in a nutshell}\label{Sec_nutshell}

In this paragraph, we recall briefly the method developed by Alfonsi and Bally in~\cite{AB} to construct approximations of any order from a family of approximation schemes. We consider $F$ a vector space endowed with a family of seminorms $(\| \|_k)_{k\in\N}$ such that $\|f\|_k\leq \|f\|_{k+1}$. We consider a time horizon $T>0$ and set, for $n\in \N^*$ and $l\in \N$,
\begin{equation}\label{def_hl}
  h_l=\frac{T}{n^l}.
\end{equation}
To achieve this goal, we consider a family of linear operators $(Q_l)_{l \in \N}$ on $F$.  For $l\in \N$,  we note $Q^{[0]}_l=I$ the identity operator and, for $j\in \N^*$, $Q^{[j]}_l = Q^{[j-1]}_l Q_l$ the operator obtained by composition. We suppose that the two following conditions are satisfied. The first quantifies how $Q_l$ approximates~$P_{h_l}$:
\begin{equation}\tag{$\bbar{H_1}$}\label{H1_bar}
  \begin{array}{c}\text{there exists } \alpha>0 \text{ such that for any }  l,k\in\N, \text{ there exists } C>0,\text{ such that }\\ \|(P_{h_l}-Q_l)f\|_k \leq C\|f\|_{\psi_{Q}(k)} h_l^{1+\alpha} \text{ for all } f\in F,
  \end{array}
\end{equation}
where $\psi_Q: \N \to \N$ is a function\footnote{Note that in~\cite{AB}, it is taken $\psi_Q(k)=k+\beta$ for some $\beta \in \N$, but is can be easily generalized to any function~$\psi_Q$. In this paper, we will work with a doubly indexed norm and take $\psi_Q(m,L)=(2(m+3),L-1)$.}. The second one is a uniform bound with respect to all the seminorms:
\begin{equation}\tag{$\bbar{H_2}$}\label{H2_bar}
  \begin{array}{c}
    \text{for all } l,k\in \N, \text{ there exists } C>0 \text{ such that } \\ \max_{0\leq j\leq n^l}\|Q^{[j]}_l f\|_k + \sup_{t\leq T}\|P_t f\|_k\leq C\|f\|_k  \text{ for all } f\in F.
  \end{array}
\end{equation}
Then, for any $\nu \in \N^*$, Alfonsi and Bally~\cite{AB} show how one can construct, by mixing the operators $Q_l$,  a linear operator $\cPh^{\nu,n}_T$ for which there exists $C>0$ and $k\in \N$  such that
\begin{equation}\label{const_Pnu}
   \|P_Tf-\cPh^{\nu,n}f\|_0\leq C\|f\|_k n^{-\nu \alpha} \text{ for all } f\in F.
\end{equation}

Let us explain how it works for $\nu=1$ and $\nu=2$. For $\nu=1$, we mainly repeat the proof of Talay and Tubaro~\cite{TT} for the weak error of the Euler scheme. From the semigroup property, we have
\begin{equation}\label{dev_1}P_Tf-Q^{[n]}_1f=P_{nh_1}f-Q^{[n]}_1f=\sum_{k=0} ^{n-1}P_{(n-(k+1))h_1}[P_{h_1}-Q_1]Q_1^{[k]}f.
\end{equation}
We get by using~\eqref{H2_bar}, then~\eqref{H1_bar} and then again~\eqref{H2_bar}
\begin{align}
  \|P_Tf-Q^{[n]}_1f\|_0 & \le\sum_{k=0} ^{n-1} C \|[P_{h_1}-Q_1]Q_1^{[k]}f\|_0 \le \sum_{k=0} ^{n-1} C \|Q_1^{[k]}f\|_{\psi_Q(0)}h_1^{1+\alpha} \notag \\
                        & \le  C \|f\|_{\psi_Q(0)} n(T/n)^{1+\alpha}= C \|f\|_{\psi_Q(0)} T^{1+\alpha}n^{-\alpha}.\label{nu_egal_1}
\end{align}
Here, and through the paper, $C$ denotes a positive constant that may change from one line to another. So, $\cPh^{1,n}=Q^{[n]}_1$ satisfies~\eqref{const_Pnu}  with $\nu=1$, $k=\psi_Q(0)$. The approximation scheme simply consists in using $n$ times the scheme~$Q_1$, which can be seen as a scheme on the regular time grid with time step~$h_1$.

We now present the approximation scheme~\eqref{const_Pnu} for $\nu=2$. To do so, we use again~\eqref{dev_1} to get $P_{(n-(k+1))h_1}-Q_1^{[n-(k+1)]}=\sum_{k'=0}^{n-(k+2)}P_{(n-(k+k'+2))h_1}[P_{h_1}-Q_1]Q_1^{[k']}$ and then expand further~\eqref{dev_1}:
\begin{align}
  P_Tf-Q^{[n]}_1f            & =\sum_{k=0}^{n-1}Q_1^{[n-(k+1)]}[P_{h_1}-Q_1]Q_1^{[k]}f + R_2^{h_1}(n)f, \label{dev_2}                          \\
  \text{ with } R_2^{h_1}(n) & =\sum_{k=0}^{n-1}\sum_{k'=0}^{n-(k+2)}P_{(n-(k+k'+2))h_1}[P_{h_1}-Q_1] Q_1^{[k']}[P_{h_1}-Q_1] Q_1^{[k]} \nonumber
\end{align}
Using~\eqref{H1_bar} three times and~\eqref{H2_bar} twice, we obtain $$\| R_2^{h_1}(n)f\|_0 \le C \| f\|_{\psi_Q(\psi_Q(0))}\frac{n(n-1)}2 h_1^{2(1+\alpha)}\le C \| f\|_{\psi_Q(\psi_Q(0))}\frac{T^{2(1+\alpha)}}{2} n^{-2\alpha}.$$
Thus, $Q^{[n]}_1 +\sum_{k=0}^{n-1}Q_1^{[n-(k+1)]}[P_{h_1}-Q_1]Q_1^{[k]}f $ is an approximation of order $2\alpha$, but it still involves the semigroup through $P_{h_1}$. To get an approximation that is obtained only with the operators $Q_l$, we use again~\eqref{dev_1} with time step~$h_2$ and final time~$h_1=nh_2$:
$$P_{h_1}f-Q_2^{[n]}f=\sum_{k=0} ^{n-1}P_{(n-(k+1))h_2}[P_{h_2}-Q_2]Q_2^{[k]}f. $$
We have $\|P_{h_1}f-Q_2^{[n]}f\|_0\le C \|f\|_{\psi_Q(0)}n h_2^{1+\alpha}$ by using again \eqref{H1_bar} and~\eqref{H2_bar}. We get from~\eqref{dev_2}
\begin{equation}\label{devt_erreur}
  P_Tf-Q^{[n]}_1f=\sum_{k=0}^{n-1}Q_1^{[n-(k+1)]}[Q_2^{[n]}-Q_1]Q_1^{[k]}f+ \sum_{k=0}^{n-1}Q_1^{[n-(k+1)]}[P_{h_1}-Q_2^{[n]}]Q_1^{[k]}f+ R_2^{h_1}(n)f,
\end{equation}
with  $\|\sum_{k=0}^{n-1}Q_1^{[n-(k+1)]}[P_{h_1}-Q_2^{[n]}]Q_1^{[k]}f\|_0\le C  \|f\|_{\psi_Q(0)}n^2 h_2^{1+\alpha}=C  \|f\|_{\psi_Q(0)}T^{1+\alpha}n^{-2\alpha}$. Therefore, the approximation
\begin{equation}\label{def_P_2}
  \cPh^{2,n}f:=Q^{[n]}_1f+\sum_{k=0}^{n-1}Q_1^{[n-(k+1)]}[Q_2^{[n]}-Q_1]Q_1^{[k]}f
\end{equation}
satisfies~\eqref{const_Pnu} with $\nu=2$ and is obtained only with the approximating operators~$Q_l$.  The first term~$Q_1^{[n]}$ corresponds to apply the scheme~$Q_1$ on the regular time grid with time step~$h_1$, while each term  $Q_1^{[n-(k+1)]}[Q_2^{[n]}-Q_1]Q_1^{[k]}$ is the difference between this scheme and the one where $Q_2^{[n]}$ is used instead of $Q_1$ for the $(k+1)$-th time step. This amounts to refine this time step and split it into $n$ time steps of size~$h_2$, and to use the scheme $Q_2$ on this time grid.

In practice, it is inefficient to calculate one by one the terms in $\cPh^{2,n}f$. In fact, each term requires a number of calculations that is proportional to~$n$, and the overall computation cost would be of the same order as~$n^2$. Since  the convergence is in $O(n^{-2\alpha})$ it would not be better asymptotically than using $\hat{P}^{1,n^2}f$. To avoid this, we use randomization. We sample a uniform random variable~$\kappa$ on $\{0,\dots,n-1\}$ and calculate $n\E[Q_1^{[n-(\kappa+1)]}[Q_2^{[n]}-Q_1]Q_1^{[\kappa]}f]=\sum_{k=0}^{n-1}Q_1^{[n-(k+1)]}[Q_2^{[n]}-Q_1]Q_1^{[k]}f$. This amounts to consider the regular time grid with time step~$h_1$, to select randomly one time step and to refine it, and then to compute the difference between the approximations on the (random) refined time-grid and on the regular time-grid. To be more precise, let us consider the case of an approximation scheme defined by $\varphi(x,h,V)$ where $\varphi$ is a measurable function, $x$ is the starting point, $h$ the time step and $V$  a random variable. The associated operators are $Q_lf(x)=\E[f(\varphi(x,h_l,V))]$, $l\in \N$. For a time-grid $\Pi=\{0=t_0<\dots<t_n=T\}$, we define $X^\Pi_0(x)=x$ and $X^\Pi_{t_{i}}(x)=\varphi(X^\Pi_{t_{i-1}}(x),t_i-t_{i-1},V_i)$ for $1\le i\le n$, where $(V_i)_{i\ge 1}$ is an i.i.d. sequence. Thus, we get on the uniform time grid $\Pi^0=\{ kT/n, 0\le k\le n \}$ $\E[f(X^{\Pi^0}_{T}(x))]=Q_1^{[n]}f(x)$. By taking the random grid  $\Pi^1 = \Pi^0 \cup  \{ \kappa T/n + k'T/n^2 , 1 \leq k' \leq n-1 \}$, where $\kappa$ is an independent uniform random variable on $\{0,\ldots,n-1\}$, we also get $\E[f(X^{\Pi^1}_{T}(x))]=\E[Q_1^{[n-(\kappa+1)]}[Q_2^{[n]}-Q_1]Q_1^{[\kappa]}f(x)]$, and then $\E[n(f(X^{\Pi^1}_{T}(x))-f(X^{\Pi^0}_{T}(x)))]=\sum_{k=0}^{n-1}Q_1^{[n-(k+1)]}[Q_2^{[n]}-Q_1]Q_1^{[k]}f(x)$. When using a Monte-Carlo estimator of this identity, one has thus to draw as many $\kappa$'s as trajectories.

We have presented here how to construct~$\cPh^{\nu,n}$ for $\nu=1$ and $\nu=2$, and it is possible by repeating the same arguments to construct by induction approximations of any order. Unfortunately, the induction is quite involved. It is fully described in~\cite[Theorem 3.10]{AB}. We do not reproduce it in this paper because it would require much more notation, and we will mainly use the scheme~\eqref{def_P_2}. Here, we give in addition the explicit form of   $\cPh^{3,n}$, $n\ge 2$:
\begin{align}\label{def_P_3}
  \cPh^{3,n}f:= & \cPh^{2,n}+  \sum_{0\le k_1<k_2<n}^{n-1}Q_1^{[n-(k_2+1)]}[Q_2^{[n]}-Q_1]Q_1^{[k_2-k_1-1]}[Q_2^{[n]}-Q_1]Q_1^{[k_1]}f     \\
                   & +\sum_{k=0}^{n-1}Q_1^{[n-(k+1)]}\left[\sum_{k'=0}^{n-1}Q_2^{[n-(k'+1)]}[Q_3^{[n]}-Q_2]Q_2^{[k']} \right] Q_1^{[k]}f. \notag
\end{align}
By similar arguments, it satisfies~\eqref{const_Pnu} with $\nu=3$.


\section{Second order schemes for the CIR process and main result}\label{Sec_O2CIR}

In this section, we focus on the approximation of the semigroup of the CIR process $P_tf(x)=\E[f(X^x_t)]$, where
$$X_t^x=x + \int_0^t(a-kX^x_s)ds+ \sigma \int_0^t\sqrt{X^x_s}dW_s, \ t\ge 0.$$
Equation~\eqref{nu_egal_1} shows that, necessarily, approximating operators $Q_l$ that satisfy both~\eqref{H1_bar} and~\eqref{H2_bar} lead to a weak error of order~$\alpha$. Therefore, we are naturally interested in approximation schemes of the CIR for which we know the rate of convergence~$\alpha$ for the weak error. \cite[Proposition 4.2]{AA_MCMA} gives a rate $\alpha=1$ for a family of approximation schemes that are basically obtained as a correction of the Euler scheme. Ninomiya and Victoir~\cite{NV} have developed a generic method to construct second order schemes ($\alpha=2$) for Stochastic Differential Equations with smooth coefficients. Applied to the Cox-Ingersoll-Ross process, their method leads to the following approximation scheme
\begin{equation}\label{NV_scheme}
  \hat{X}^x_t=\varphi(x,t,\sqrt{t}N),
\end{equation}
where  $N\sim \mathcal{N}(0,1)$ and  $\varphi:\R_+\times\R_+\times\R\to \R_+$ is defined by
\begin{align}\label{def_varphi}
  \varphi(x,t,w) & = e^{-kt/2}\left(\sqrt{(a-\sigma^2/4)\psi_k(t/2)+e^{-kt/2}x}+\sigma w/2 \right)^2 +(a-\sigma^2/4)\psi_k(t/2) \\
                 & =X_0(t/2,X_1(w,X_0(t/2,x))), \text{ with} \notag
\end{align}
\begin{align}\label{def_X0_and_psi_k}
  X_0(t,x) & = e^{-kt}x+\psi_k(t)(a-\sigma^2/4), \quad \psi_k(t)=\frac{1-e^{-kt}}k, \\ X_1(t,x)&=(\sqrt{x}+ t\sigma/2)^2, \label{def_X1}
\end{align}
with the convention that $\psi_0(t)=t$. This scheme corresponds to approximate $P_tf(x)$ by $\hat{P}_tf(x)=\E[f( \hat{X}^x_t)]$ for $x,t\ge 0$, and then to set $Q_l=\hat{P}_{h_l}$. Its construction comes from the splitting of the infinitesimal generator of the CIR process
\begin{equation}\label{def_LCIR}
  \cL f(x)=(a-kx)f'(x)+\frac 12 \sigma^2 x f''(x),\ f \in \mathcal{C}^2, x\ge 0,
\end{equation}
as $\cL=V_0+\frac 12 V_1^2$ with
\begin{equation}\label{def_V0_V1}
  V_0f(x)=\left(a-\frac{\sigma^2}4 -kx\right)f'(x) \text{ and } V_1f(x)=\sigma \sqrt{x} f'(x).
\end{equation}
The function $t\mapsto X_0(t,x)$ is the solution of the ODE $X_0'(t,x)=a-\frac{\sigma^2}4 -kX_0(t,x)$ such that $X_0(0,x)=x$, while $X_1(W_t,x)$ solves the SDE associated to the infinitesimal generator $V_1^2/2$.

The scheme~\eqref{NV_scheme} is well-defined for $\sigma^2\le 4a$. Instead, for $\sigma^2> 4a$, it is not well-defined for any $x\ge 0$ since the argument in the square-root is negative when $x$ is close to zero. To correct this, Alfonsi~\cite{AA_MCOM} has proposed the following scheme
\begin{align}\label{Alfonsi_scheme}
  \hat{X}^x_t =  (\mathds{1}_{x\geq K^Y_2(t)}\varphi(x,t,\sqrt{t}Y) + \mathds{1}_{x< K^Y_2(t)} \hat{X}^{x,d}_t),
\end{align}
where $Y$ is a random variable with compact support on $[-A_Y,A_Y]$ for some $A_Y>0$ such that $\E[Y^k]=\E[N^k]$ for $k\le 5$, and  $\hat{X}^{x,d}_t$ is a nonnegative random variable such that $\E[(\hat{X}^{x,d}_t)^i]=\E[(X^{x}_t)^i]$ for $i\in \{1,2\}$ and $K^Y_2(t)$ is a nonnegative threshold defined by
\begin{equation}\label{threshold_gen}
  K^Y_2(t) =\mathds{1}_{\sigma^2>4a} \left[ e^{\frac{kt}{2}}\left( (\sigma^2/4-a)\psi_k(t/2)+ \bigg( \sqrt{e^{\frac{kt}2}(\sigma^2/4-a)\psi_k(t/2)} + \frac \sigma 2 A_Y\sqrt{t} \bigg)^2 \right)\right].
\end{equation}
Note that when $\sigma^2\le 4a$, we have $K^Y_2(t)=0$ and thus $\hat{X}^x_t = \varphi(x,t,\sqrt{t}Y)$.
In~\cite{AA_MCOM}, it is taken $Y$ such that $\P(Y=\sqrt{3})=\P(Y=-\sqrt{3})=1/6$ and $\P(Y=0)=2/3$, and a discrete  random variable 
$\hat{X}^{x,d}_t$ such that $\P(\hat{X}^{x,d}_t=\frac{1}{2 \pi(t,x)})=\pi(t,x)$, $\P(\hat{X}^{x,d}_t=\frac{1}{2 (1-\pi(t,x))})=1-\pi(t,x)$ where $\pi(t,x)=\frac{1-\sqrt{1-\E[(X^{x}_t)]^2/\E[(X^{x}_t)^2]}}{2} \in (0,1/2)$.

We now restate~\cite[Theorem 2.8]{AA_MCOM} that analyzes the weak error. We introduce $\CpolK{k}$, the set of $\mathcal{C}^k$ functions $f:\R \to \R_+$ such that all its derivatives have polynomial growth. More precisely, this means that for all $k'\in\{0,\dots,k\}$, there exists $C_{k'},E_{k'}\in \R_+$ such that
$$|f^{(k')}(x)|\le C_{k'}(1+x^{E_k'}), \ x\ge 0.$$
We also set $\Cpol=\cap_{k\in \N}\CpolK{k}$.

\begin{theorem}\label{thm_MCOM}
  Let $\hat{X}^x_t$ be the scheme defined by~\eqref{NV_scheme} for $\sigma^2\le 4a$ or by~\eqref{Alfonsi_scheme} for any $\sigma>0$. 
  Then, for all $f\in \Cpol$, we have $Q_1^{[n]}f(x)-P_Tf(x)=O(1/n^2)$ where $Q_1f(x)=\E[f( \hat{X}^x_{h_1})]$.
\end{theorem}
The goal of this paper is to extend this result and prove the estimates~\eqref{H1_bar} and~\eqref{H2_bar} for a suitable space of functions and a suitable family of seminorms. We are able to prove such results only in the case $\sigma^2\le 4a$: the indicator function in~\eqref{Alfonsi_scheme} creates a singularity that is difficult to handle in the analysis. In Section~\ref{Sec_pol_fct}, we first prove~\eqref{H1_bar} and~\eqref{H2_bar} for polynomial test functions. Then, we deal in Section~\ref{Sec_main} with the much technical case of smooth test functions with derivatives of polynomial growth. We state here our main result, the proof of which is given in Section~\ref{Sec_main}.
\begin{theorem}\label{thm_main}
  Let $\hat{X}^x_t$ be the scheme defined by~\eqref{NV_scheme} for $\sigma^2\le 4a$ and $Q_lf(x)=\E[f(\hat{X}^x_{h_l})]$, for $l\ge 1$. 
  Then, for all $f\in \CpolK{18}$, we have $\cPh^{2,n}f(x)-P_Tf(x)=O(1/n^4)$ as $n\to \infty$.\\
 Besides, for  $f\in \Cpol$, we have $\cPh^{\nu,n}f(x)-P_Tf(x)=O(1/n^{2\nu})$.
\end{theorem}
Let us stress here that Theorem~\ref{thm_main}  gives an asymptotic result as $n\to \infty$. It thus might happen that for small values of $n$,  $\cPh^{2,n}$ is less accurate than $\cPh^{1,n}=Q_1^{[n]}$ for some $f\in \Cpol$ and $x\ge 0$. In practice, we have always noticed in our numerical experiments that $\cPh^{2,n}$ is more accurate than $\cPh^{1,n}$. However, the estimated rates of convergence obtained from relatively small values of $n$ may be different from the theoretical asymptotic ones, see Figures~\ref{CIR_Plot1},\ref{CIR_Plot2} and~\ref{CIR_Plot3} where are given the estimated rates for $\cPh^{1,n}$, $\cPh^{2,n}$ and $\cPh^{3,n}$.

\section{The case of polynomial test functions}\label{Sec_pol_fct}

In this section, we want to illustrate the method and consider test functions that are  polynomial test functions. We define for $L\in \N$
$$\PLRp{L}=\{ f: \R_+\to \R,  f(x)=\sum_{j=0}^L a_j x^j \text{ for some } a_0,\dots,a_L \in \R \},$$
the vector space of polynomial functions over $\R_+$ with degree less or equal to~$L$. We also define  $\PRp= \cup_{L\in \N} \PLRp{L}$ the space of polynomial functions.
We endow $\PRp$ with the following norm:
\begin{equation}\label{norm_pol}
  \|f\| = \sjzL |a_j|, \text{ for } f(x)=\sjzL a_j x^j.
\end{equation}

We consider the case $\sigma^2\le 4a$ and consider the scheme~\eqref{Alfonsi_scheme} for the CIR process with a time step~$t>0$, $\hat{X}^x_t =  \varphi(x,t,\sqrt{t}Y)$.  The approximation scheme $Q_l$ is then defined by $Q_lf=\E[f(\hat{X}^x_{h_l})]$. The goal of this section is to prove~\eqref{H1_bar} and~\eqref{H2_bar} for the norm~\eqref{norm_pol}. We make the following assumption on~$Y$.

\noindent {\bf Assumption~$(\mathcal{H}_Y)$:} $Y:\Omega\to \R$ is a symmetric random variable such that $\E[|Y|^k]<\infty$ for all $k\in \N$, and $\E[Y^k]=\E[N^k]$ for $k\in \{2,4\}$ with $N\sim \mathcal{N}(0,1)$.

We now state two lemmas that will enable us to prove that~\eqref{H2_bar} is satisfied by the scheme~\eqref{Alfonsi_scheme}. Lemma~\ref{estimates_pol} shows that polynomials functions are preserved by the approximation scheme, and gives short time estimate for the polynomial norm. Lemma~\ref{Moments_Formula_CIR} gives similar results for the CIR diffusion. The proofs of these lemmas are quite elementary and are postponed to Appendix~\ref{App_proof_sec_pol}.

\begin{lemma}\label{estimates_pol}
  Let $T\ge 0$, $t\in[0,T]$, $f\in\PLRp{L}$ and assume~$(\mathcal{H}_Y)$ and $\sigma^2\le 4a$.  Then, we have $f(X_0(t,\cdot)),\E[f(X_1(\sqrt{t}Y, \cdot ))] \in \PLRp{L}$ where $X_0$ and $X_1$ are defined by~\eqref{def_X0_and_psi_k} and~\eqref{def_X1}, and
  \begin{enumerate}
    \item $\|f(X_0(t,\cdot))\| \leq (1\vee e^{-kLt})(1+C_{X_0}^Lt) \|f\|$,
    \item $\| \E[f(X_1(\sqrt{t}Y, \cdot ))] \| \leq (1+\E[Y^{2L}]C_{X_1}^Lt) \| f \|,$
  \end{enumerate}
  for some constants $C_{X_0},C_{X_1}$ depending only on~$(a,\sigma,T)$.
\end{lemma}

\begin{lemma}\label{Moments_Formula_CIR}
  Let $(X^x_t,t\ge 0)$ be the CIR process starting from $x\in\R_+$. For $m\in \N$, we define $\tilde{u}_m(t,x):= \E[(X^x_t)^m]$. There exists $C^\infty$ functions $\tilde{u}_{j,m}:\R_+\to \R$ that depend on $(k,a,\sigma)$ such that:
  \begin{equation}
    \tilde{u}_m(t,x) = \sum_{j=0}^m \tilde{u}_{j,m}(t) x^j.
  \end{equation}
  If $f \in \PLRp{L}$, then we have $\E[f(X^\cdot_t)]\in \PLRp{L}$ and for $t\in [0,T]$,
  \begin{equation}\label{Pol_Estimate_CIR}
    \|\E[f(X^{\cdot}_t)]\| \leq C_{\text{cir}}(L,T) \|f\|,
  \end{equation}
  with $C_{\text{cir}}(L,T) = \max_{t\in[0,T],m\in \{0,\dots,L\} } \sum_{j=0}^m |\tilde{u}_{j,m}(t)|$.
\end{lemma}
We are now in position to prove the main result of this section, which is a weaker (but easier to prove) version of our main Theorem~\ref{thm_main}, since it only applies to polynomial test functions. Let us point however that it applies to a larger family of schemes, namely to the schemes $\varphi(x,t,\sqrt{t}Y)$ with $Y$ satisfying~$(\mathcal{H}_Y)$, while Theorem~\ref{thm_main} requires to take $Y\sim \mathcal{N}(0,1)$.

\begin{prop}
  Let $\sigma^2\le 4a$ and assume that $Y$ satisfies~$(\mathcal{H}_Y)$.  For any $L \in \N$, the properties~\eqref{H1_bar} and \eqref{H2_bar} are satisfied by the scheme~\eqref{Alfonsi_scheme} $\hat{X}^x_t=\varphi(x,t, \sqrt{t}Y)$ for $F=\PLRp{L}$ and the norm~\eqref{norm_pol}. Then, we have for any $f\in \PLRp{L}$,
  $$\|\E[f(X^x_T)]-\cPh^{\nu,n}f\| \le C_L \|f\| n^{-2\nu},$$
  for some constant $C_L$.
\end{prop}

\begin{proof}
  We first prove~\eqref{H2_bar}. The property $\sup_{t\in[0,T]}\|P_tf\|$ is given by Lemma~\ref{Moments_Formula_CIR}.
  Since $\hat{X}^x_t=X_0(t/2,X_1(\sqrt{t}Y,X_0(t/2,x)))$, we get by Lemma~\ref{estimates_pol}
  $$
    \|\E[f(\hat{X}^\cdot_t)]\|\le  [(1\vee e^{-kLt/2})(1+C_{X_0}^Lt/2)]^2(1+\E[Y^{2L}]C_{X_1}^Lt)  \|f\|.
  $$
  We now use that $1+x\le e^x$ to get
  \begin{equation}\label{maj_scheme}
    \| \E[f(\hat{X}^\cdot_t)] \| \leq e^{((-k)^+L+C_{X_0}^L+\E[Y^{2L}]C_{X_1}^L)t} \|f\|.
  \end{equation}
  Since $Q_lf(x)=\E[f(\hat{X}^x_{T/n^l})]$, this yields to $\max_{0\le j \le n^l}\|Q^{[j]}_lf\| \le  e^{((-k)^+L+C_{X_0}^L+\E[Y^{2L}]C_{X_1}^L)T} \|f\|$.

  We now prove~\eqref{H1_bar}. Let $m\in \{0,\dots,L\}$ and $0<x_0<\dots<x_L$ be fixed real numbers (one may take for example $x_\ell=\ell+1$).  Lemmas~\ref{estimates_pol} and~\ref{Moments_Formula_CIR} give that $v_m(t,x)=\E[(\hat{X}^x_t)^m]-\E[(X^x_t)^m]=\sum_{j=0}^m v_{j,m}(t)x^j $. By~\cite[Proposition 2.4]{AA_MCOM}, we know that there exists $C'_m,E'_m$ such that for all $t \in (0,1)$, $|v_m(t,x)|\le C'_m t^3(1+|x|^{E'_m})$. Therefore, there exists  $\tilde{C}_m\in \R_+$ such that for all $ \ell \in \{0,\dots,L\}$, $|v_m(t,x_\ell)|\le \tilde{C}_m t^3$. By using the invertibility of the Vandermonde matrix, we get the existence of $C_m \in \R_+$ such that
  $$ |v_{j,m}(t)|\le C_m t^3,  \ j\in \{0,\dots,m\}.$$
  Therefore, we get for $f \in \PLRp{L}$
  $$\|\E[f(\hat{X}^\cdot_t)]-\E[f(X^\cdot_t)]\| \le \sum_{m=0}^L|a_m|\sum_{j=0}^m C_m t^3\le L \max_{m\in \{0,\dots,L\}}C_m \|f\| t^3,$$
  that gives~\eqref{H1_bar}. We conclude by applying~\cite[Theorem 3.10]{AB}.
\end{proof}

\section{Proof of Theorem~\ref{thm_main}}\label{Sec_main}

In Section~\ref{Sec_pol_fct}, we have obtained the convergence for test functions that are polynomial functions. For these test functions, the choice of the norm is straightforward and the proofs are not very technical and quite easy. However, one would like to obtain the convergence result for a much larger class of test functions. This is the goal of this section.

We consider test functions that are smooth with polynomial growth, whose derivatives have a polynomial growth. Namely, we introduce for $m,L \in \N$,
\begin{equation}
  \CpolKL{m}{L}= \left\{f:\R_+ \to \R \text{ of class } \mathcal{C}^m \ :  \ \max_{j\in\{0,\ldots,m\}} \sup_{x\geq 0}\frac{|f^{(j)}(x)|}{1+x^L}<\infty \right\},
\end{equation}
which we endow with the norm
\begin{equation}\label{mLnorm}
  \|f\|_{m,L} = \max_{j\in\{0,\ldots,m\}} \sup_{x\geq 0} \frac{|f^{(j)}(x)|}{1+x^L}.
\end{equation}

To prove Theorem~\ref{thm_main}, we need to prove the estimates~\eqref{H1_bar} and~\eqref{H2_bar} for this family of norms. This is the goal of the two next subsections. More precisely, we will show respectively the estimates
  $$\|(P_{h_l}-Q_l)f\|_{m,L+3}\le C h_l^3  \|f\|_{2(m+3),L}, \  m\le L+3, f\in  \CpolKL{2(m+3)}{L} $$
in Proposition~\ref{H1_bar_mL}  and 
  $$\sup_{t\ge T} \|P_{t}f\|_{m,L}+ \max_{0\le j \le n^l} \|Q^{[j]}_l f\|_{m,L} \le \|f\|_{m,L} C h_l^3, \  m\le L, f\in  \CpolKL{m}{L} $$
in Proposition~\ref{prop_H2_NV} for $Q_l$ as in Theorem~\ref{thm_main}. Note that $L$ has to be large enough: this is not an issue for our purpose since $\CpolKL{m}{L}\subset \CpolKL{m}{L+1}$, and we can work with $L$ as large as needed. We refer to the proof of Theorem~\ref{thm_main} in Subsection~\ref{Subsec_thm_main} for further details.

Before, we summarize in the next lemma some properties of the norms defined in Equation~\eqref{mLnorm} that we will use later on. Its proof is postponed to Appendix~\ref{App_proof_sec_5}
\begin{lemma} \label{lem_estimnorm}
  Let $m,L\in\N$. We have the following basic properties:
  \begin{enumerate}
    \item $\|f\|_{m',L}= \max_{j\in\{0,\ldots, m'\}}\|f^{(j)}\|_{0,L}$ for $f\in\CpolKL{m}{L}$ and $m'\in\{0,\ldots,m\}$.
    \item $\CpolKL{m+1}{L}\subset \CpolKL{m}{L}$ and  $\|f\|_{m,L} \leq \|f\|_{m+1,L}$ for $f\in\CpolKL{m+1}{L}$.
    \item $\|f^{(i)}\|_{m,L} \leq \|f\|_{m+i,L}$ for  $i\in \N$ and $f\in \CpolKL{m+i}{L}$.
    \item $\CpolKL{m}{L}\subset \CpolKL{m}{L+1}$ and $\|f\|_{m,L+1} \leq 2\|f\|_{m,L}$ for $f\in\CpolKL{m}{L}$.
    \item Let $\mathcal{M}_1$ be the operator defined by $f\mapsto \mathcal{M}_1 f$, $\mathcal{M}_1 f(x)=xf(x)$. Then, $\mathcal{M}_1f\in  \CpolKL{m}{L+1}$ for $f\in  \CpolKL{m}{L}$ and $\|\mathcal{M}_1 f\|_{m,L+1} \leq (2m+3)\|f\|_{m,L}$.
    \item Let $\cL f (x)=(a-kx)f'(x)+\frac 12 \sigma^2 x f''(x)$ be the infinitesimal generator of the CIR process. Then, we have for $f \in \CpolKL{m+2}{L}$,
          $$\|\cL f\|_{m,L+1}\le  \left(2a  + (2m+3)(|k| +\sigma^2/2) \right) \|f\|_{m+2,L}.  $$
          We also have $\|(V_1^2/2) f\|_{m,L+1}\le  \sigma^2 (m+2) \|f\|_{m+2,L}$ and $\|V_0 f\|_{m,L+1}\le [2|a-\sigma^2/4|+(2m+3)|k|] \|f\|_{m+1,L}$, where $V_0$ and $V_1$ are defined by~\eqref{def_V0_V1}.
  \end{enumerate}
\end{lemma}

We also state the following elementary lemma that will be useful to prove both~\eqref{H1_bar} and~\eqref{H2_bar}.

\begin{lemma}\label{X0_inequalities}
  Let $T>0$, $\sigma^2\le 4a$ and $X_0$ be defined by~\eqref{def_X0_and_psi_k}. Then, there exists a constant $K\ge 0$ such that
  for any function $f\in\CpolKL{m}{L}$, we have
  $$ \|f(X_0(t,\cdot))\|_{m,L} \leq  e^{Kt} \|f\|_{m,L}, \ t \in [0,T].$$
\end{lemma}
\begin{proof}
  We first prove the following inequality
  $$1 + X_0(t,x)^L \leq  (1\vee e^{-Lkt})(1+\tilde{C}_{X_0}t)(1+x^L),$$
  for some constant $\tilde{C}_{X_0}$. To do so,  we develop the term $X_0(t,x)^L$ and get
  \begin{align*}
    1 + X_0(t,x)^L & = 1+\sum_{j=0}^L {L\choose j} e^{-(L-j)kt} x^{(L-j)}(\psi_k(t)(a-\sigma^2/4))^{j}                              \\
                   & = 1+e^{-Lkt}x^L +\psi_k(t)\sum_{j=1}^{L} {L\choose j} e^{-(L-j)kt} x^{(L-j)}\psi_k(t)^{j-1}(a-\sigma^2/4)^{j}.
  \end{align*}
  We remark that for $k\ge 0$, $0\le \psi_k(t)\leq t \leq 1\vee T$  for all $t\in[0,T]$. For $k<0$, we have $\psi_k(t)=e^{-kt}\psi_{-k}(t)$ and thus $\psi_k(t)\le e^{(-k)^+t}t$  for all $t\in[0,T]$ and $k\in \R$. Using $x^j\leq 1+x^L$ for all $j\in\{1,\ldots,L\}$, we can rewrite the previous identity as
  \begin{align*}
    1 + X_0(t,x)^L & \leq (1\vee e^{-Lkt})(1+x^L) \\
    & \quad + t  e^{(-k)^+t} (1\vee e^{-Lkt})(1+x^L)  \sum_{j=0}^L {L\choose j} ( e^{(-k)^+T}(1\vee T)(a-\sigma^2/4))^{j} \\
                   & \leq (1\vee e^{-Lkt})(1+\tilde{C}_{X_0}t)(1+x^L),
  \end{align*}
  where $\tilde{C}_{X_0}=e^{(-k)^+T}(1+e^{(-k)^+T}(1\vee T)(a-\sigma^2/4))^L$.

  We are now in position to prove the claim. For $i\le m$, we have:
  \begin{align*}
    |\partial_x^i f(X_0(t,x))| =|e^{-ikt}f^{(i)}(X_0(t,x))| & \leq e^{-ikt}\|f\|_{m,L}(1+X_0(t,x)^L)                                              \\
                                                            & \leq \|f\|_{m,L} (1 \vee e^{-mkt})  (1 \vee e^{- Lkt})  (1+\tilde{C}_{X_0}t)(1+x^L) \\
                                                            & \leq \|f\|_{m,L} e^{[\tilde{C}_{X_0}+(L+m)(-k)^+] t}(1+x^L).
  \end{align*}
  This gives $\|f(X_0(t,\cdot))\|_{m,L}\le \|f\|_{m,L} e^{[\tilde{C}_{X_0}+(L+m)(-k)^+] t}$.
\end{proof}

\subsection{Proof of~\eqref{H1_bar}} In this subsection, we prove the following result which is a direct consequence of Propositions~\ref{LCIR_Expansion} (with $\nu=2$) and~\ref{prop_H1sch} that are stated below.
\begin{prop}\label{H1_bar_mL}
  Let $Y$ satisfy~$(\mathcal{H}_Y)$, $\sigma^2\le 4a$ and $\hat{X}^x_t=\varphi(x,t,\sqrt{t}Y)$ be the scheme~\eqref{Alfonsi_scheme}. Let $m,L\in \N$ such that $L+3\ge m$ and $f \in \CpolKL{2(m+3)}{L}$. Then, there exists a constant $C\in \R_+^*$ such that for $t\in[0,T]$,
  $$ \|\E[f(\hat{X}^\cdot_t)]- \E[f({X}^\cdot_t)]\| _{m,L+3}\le C t^3 \|f\|_{2(m+3),L}.$$
\end{prop}

To prove this result, we compare each term with the expansion $f(x)+t\cL f(x)+\frac{t^2}2\cL^2 f(x)$ of order two. The next proposition analyzes the difference between such expansion and the semigroup of the CIR process.
\begin{prop}\label{LCIR_Expansion}
  Let $m,\nu,L\in \N$ such that $L+\nu+1\ge m$, $T>0$ and  $f\in\CpolKL{m+2(\nu+1)}{L}$. Let $X^x$ be the CIR process and $\cL$ its infinitesimal generator. Then, for  $t\in [0,T]$,  we have
  \begin{equation}\label{dev_funct_CIR}
    \E[f(X^x_t)] =\sum_{i=0}^\nu \frac{t^i}{i!}\cL^if(x) + t^{\nu+1}\int_0^1 \frac{(1-s)^\nu}{\nu!}\E[\cL^{\nu+1}f(X^x_{ts})]ds
  \end{equation}
  where the function $x\mapsto \int_0^1 \frac{(1-s)^\nu}{\nu!}\E[\cL^{\nu+1}f(X^x_s)]ds$ belongs to $\CpolKL{m}{L}$ and we have the following estimate for all $t\in [0,T]$,
  \begin{equation}\label{rem_CIR_estimate}
    \left\|\int_0^1 \frac{(1-s)^\nu}{\nu!}\E[\cL^{\nu+1}f(X^{\cdot}_{ts})]ds\right\|_{m,L+\nu+1} \le C\|f\|_{m+2(\nu+1),L},
  \end{equation}
  for some constant $C\in \R_+$ depending on $(a,k,\sigma,\nu,m,L,T)$.
\end{prop}

\begin{proof}
  Let $f\in\CpolKL{m+2(\nu+1)}{L}$. Since the coefficients of the CIR SDE have sublinear growth, we have bounds on the moments of $X^x_s$: for any $q\in\N^*$, there exists $C_q>0$ such that $\E[|X^x_s|^q]\leq C_q(1+x^q)$ for $s\in[0,T]$. Using iterations of Itô’s formula and a change of variable (in time), we then easily get \eqref{dev_funct_CIR} for $t \in [0, T]$.  To get the estimate \eqref{rem_CIR_estimate}, we first use Lemma~\ref{lem_estimnorm} and obtain $$\| \cL^{\nu+1} f \|_{m,L+\nu+1}\le K_{ cir}(m,\nu)^{\nu+1} \|f\|_{m+2(\nu+1),L},$$ with $K_{\bf cir}(m,\nu)=2a+(2m+4\nu+3)(|k|+\sigma^2/2)$.   By the triangle inequality, we have
  $$\left\|\int_0^1 \frac{(1-s)^\nu}{\nu!}\E[\cL^{\nu+1}f(X^{\cdot}_{ts})]ds\right\|_{m,L+\nu+1}\le \int_0^1 \frac{(1-s)^\nu}{\nu!} \left\|\E[\cL^{\nu+1}f(X^{\cdot}_{ts})]\right\|_{m,L+\nu+1}ds. $$
  Since $t\le T$, we have $\left\|\E[\cL^{\nu+1}f(X^{\cdot}_{ts})]\right\|_{m,L+\nu+1}\le C_{cir}(m,L+\nu+1,T)\left\|\cL^{\nu+1}f\right\|_{m,L+\nu+1}$ by  Proposition~\ref{deriv_CIR_functionals} using that $L+\nu+1\ge m$. This gives by Lemma~\ref{lem_estimnorm}
  $$\left\|\int_0^1 \frac{(1-s)^\nu}{\nu!}\E[\cL^{\nu+1}f(X^{\cdot}_{ts})]ds\right\|_{m,L+\nu+1}\le \frac {C_{cir}(m,L+\nu+1,T)} {(\nu+1)!} K_{ cir}(m,\nu)^{\nu+1} \|f\|_{m+2(\nu+1),L}. $$
\end{proof}

We now focus on the approximation scheme. The main difficulty comes from the  differentiation of the square-root that may lead to derivatives that blow up at the origin. Here, we exploit the fact that $Y$ is a symmetric random variable to cancel these blowing terms. More precisely, we will then need to differentiate in $x$ the following quantity
$$g(X_1(s \sqrt{t},x))+g(X_1(-s\sqrt{t},x))=g(x+\sigma s \sqrt{t} \sqrt{x} + \frac{\sigma^2}4 ts^2)+g(x-\sigma s \sqrt{t} \sqrt{x} + \frac{\sigma^2}4 t s^2),$$
and the next lemma enables us to have a sharp estimate of the derivatives.

\begin{lemma}\label{regular_rep}
  Let $g:\R_+\to \R$ be a $\mathcal{C}^{2n}$ function, $\beta\in \R_+$ and $\gamma \ge \beta^2/4$. Then,
  the function  $\psi_g(x):=g(x+\beta \sqrt{x}+\gamma)+g(x-\beta \sqrt{x}+\gamma)$, $x\ge 0$   is $\mathcal{C}^n$ with derivatives
  \begin{equation}\label{formula_psign}
    \psi_g^{(n)}(x)=\psi_{g^{(n)}}(x)+\sum_{j=1}^n\binom{n}{j}\beta^{2j}\int_0^1g^{(n+j)}(x+\beta(2u-1)\sqrt{x}+\gamma) \frac{(u-u^2)^{j-1}}{(j-1)!}du
  \end{equation}
\end{lemma}
The proof of this lemma and of the next corollary are postponed to Appendix~\ref{App_proof_sec_5}.

\begin{corollary}\label{cor_psign} 
	Let $m,L\in \N$, $\beta\ge 0$ and $g \in \CpolKL{2m}{L}$. Then, $\psi_g(x)=g((\sqrt{x}+\beta/2)^2)+ g((\sqrt{x}-\beta/2)^2)$ belongs to $\CpolKL{m}{L}$, and for all $n\in\{0,\ldots,m\}$ we have the following estimates
  \begin{equation}
    \|\psi_g\|_ {n,L} \leq C_{\beta,m,L} \|g\|_{2n,L},
  \end{equation}
  with $C_{\beta,m,L}= \big((1+\beta/2)^{2L}+(1-\beta/2)^{2L}+2(1+\beta^2/2)^{L}(1+\beta^2/2)^{m})$.
\end{corollary}

\begin{lemma}\label{lem_expan_V}
  Let $m,\nu,L\in \N, \ T>0,\ t \in [0,T]$ and $N\sim \mathcal{N}(0,1)$. Let $Y$ be a symmetric random variable such that $\E[Y^k]=\E[N^k]$ for $k\le 2\nu$ and $\E[Y^{2k}]<\infty$ for all $k\in \N$.  We have, for $f\in \CpolKL{m+\nu+1}{L}$,
  \begin{equation}
    f(X_0(t,x)) = \sum_{i=0}^\nu \frac{t^i}{i!} V^i_0f(x) +t^{\nu+1} \int_0^{1} \frac{(1-u)^{\nu}}{\nu!}V^{\nu+1}_0 f(X_0(u t,x))du,  \label{expan_X0}
  \end{equation}
  with $\|\int_0^{1} \frac{(1-u)^{\nu}}{\nu !}V^{\nu+1}_0 f(X_0(u t,\cdot))du\|_{m,L+\nu+1}\le C_0 \|f\|_{m+\nu+1,L}$; and for $f\in\CpolKL{2(m+\nu+1)}{L} $,
  \begin{align}\E[f(X_1(\sqrt{t}Y,x))] =& \sum_{i=0}^\nu \frac{t^i}{i!} \left(\frac{1}{2}V^2_1\right)^if(x) \label{expan_X1} \\&+ t^{\nu+1}\E\left[Y^{2\nu +2}\int_0^{1} \frac{(1-u)^{2\nu+1}}{(2\nu+1)!}V^{2\nu+2}_1 f(X_1(u \sqrt{t} Y,x))du \right],  \notag
  \end{align}
  with   $\left\|\E\left[Y^{2\nu +2}\int_0^{1} \frac{(1-u)^{2\nu+1}}{(2\nu+1)!}V^{2\nu+2}_1 f(X_1(u \sqrt{t} Y,\cdot))du \right]\right\|_{m,L+\nu+1}\le  C_1 \|f\|_{2(m+\nu+1),L}$, for some constants $C_0,C_1\in \R^+$ depending on $(a,k,\sigma)$, $T$, $m$, $M$ and $\nu$.
\end{lemma}
\begin{proof}
  Equation~\eqref{expan_X0} holds by using Taylor formula since $\frac{d}{dt}f(X_0(t,x))=V_0f(X_0(t,x))$. We have by Property~(6) of Lemma~\ref{lem_estimnorm} $\|V_0f\|_{m,L+1}\le |a-\frac {\sigma^2} 4|\|f'\|_{m,L+1}+|k|(2m+3) \|f'\|_{m,L}\le (2|a-\frac {\sigma^2} 4|+|k|(2m+3))\|f\|_{m+1,L}$ and thus $\|V_0^{\nu+1}f\|_{m,L+\nu+1}\le C \|f\|_{m+\nu+1,L}$ for some constant $C$ depending on $a,\sigma,k,\nu,m$.  Using the triangular inequality and Lemma~\ref{X0_inequalities}, we get the result.

  We now prove the second part of the claim. We first show Equation~\eqref{expan_X1}. Since $\frac{d}{dt}f(X_1(t,x))=V_1f(X_1(t,x))$, we get by Taylor formula
    \begin{align*}
      f(X_1(t,x))&=\sum_{i=0}^{2\nu+1} \frac{t^i}{i!} V_1^if(x) +  \int_0^t \frac{(t-s)^{2\nu+1}}{(2\nu +1)!} V_1^{2\nu+2}f(X_1(s,x))du \\
      &=\sum_{i=0}^{2\nu+1} \frac{t^i}{i!} V_1^if(x) + t^{2\nu+2} \int_0^1 \frac{(1-u)^{2\nu+1}}{(2\nu +1)!} V_1^{2\nu+2}f(X_1(ut,x))du, \ t\in \R.
    \end{align*}
We apply this formula at $\sqrt{t}Y$ and take the expectation. Since $\E[Y^{2i+1}]=0$ by symmetry and $\E[Y^{2i}]=\E[N^{2i}]=\frac{(2i)!}{i!2^i}$ for $i\le \nu$, we get~\eqref{expan_X1}.  We now analyze the norm of the remainder. We have $\|\frac 12 V_1^2f\|_{m,L+1}\le \sigma^2 (m+2)\|f\|_{m+2,L}$ by using Lemma~\ref{lem_estimnorm}~(6). Then, we observe that by symmetry of~$Y$,
  \begin{align*}
     & \E\left[Y^{2\nu +2}\int_0^{1} \frac{(1-u)^{2\nu+1}}{(2\nu+1)!}V^{2\nu+2}_1 f(X_1(u \sqrt{t} Y,x))du \right]
    \\& = \frac 12 \E\left[Y^{2\nu +2}\int_0^{1} \frac{(1-u)^{2\nu+1}}{(2\nu+1)!}[V^{2\nu+2}_1 f(X_1(u \sqrt{t} Y,x)) + V^{2\nu+2}_1 f(X_1(-u \sqrt{t} Y,x))]du  \right]
  \end{align*}
  By Corollary~\ref{cor_psign}, we have
  \begin{align*}
    \|V^{2\nu+2}_1 f(X_1(u \sqrt{t} Y,\cdot))+V^{2\nu+2}_1 f(X_1(-u \sqrt{t} Y,\cdot))\|_{m,L+\nu+1} & \le C_{\sigma u \sqrt{t} Y,m,L} \|V^{2\nu+2}_1 f \|_{2m,L+\nu+1} \\&\le C' C_{\sigma u \sqrt{t} Y,m,L} \| f \|_{2(m+\nu+1),L},
  \end{align*}
  with $C'=(4\sigma^2 (m+\nu+1))^{\nu+1}$.
  The conclusion follows by using the triangle inequality, the polynomial growth of the constant $C_{\sigma u \sqrt{t} Y,m,L}$ given by~Corollary~\ref{cor_psign} and the finite moments $\E[Y^{2k}]$ for $k$ sufficiently large.
\end{proof}

We are now in position to prove the estimate for the approximation scheme~\eqref{Alfonsi_scheme}. Since this scheme is obtained as the composition of the schemes $X_0$ and $X_1$, the proof consists  is using iteratively the estimates of Lemma~\ref{lem_expan_V}.
\begin{prop}\label{prop_H1sch}
  Let $Y$ be a symmetric random variable such that $\E[Y^k]=\E[N^k]$ for $k\le 4$ and $\E[Y^{2k}]<\infty$ for all $k\in \N$. Let $\sigma^2\le 4a$ and $\hat{X}^x_t$ be the scheme~\eqref{Alfonsi_scheme}.  Let $m\in \N,L \in \N^*$ and $f \in \CpolKL{2(m+3)}{L}$. Then, we have for $t\in[0,T]$,
  $$ \E[f(\hat{X}^x_t)]=f(x) +t \cL f(x)+\frac{t^2}{2} \cL^2f(x) +\bar{R}f(t,x),$$
  with $\|\bar{R}f(t,\cdot)\|_{m,L+3}\le C t^3 \|f\|_{2(m+3),L}$.
\end{prop}
\begin{proof}
  We use $\hat{X}^x_t=X_0(t/2,X_1(\sqrt{t}Y,X_0(t/2,x)))$
  and apply first~\eqref{expan_X0}:
  \begin{align*}
     & \E[f(X_0(t/2,X_1(\sqrt{t}Y,X_0(t/2,x))))]=\E\left[(f+\frac{t}{2} V_0 f +\frac{t^2}{8} V_0^2 f )(X_1(\sqrt{t}Y,X_0(t/2,x)))\right] + R_If(t,x), \\
     & \text{ with } R_If(t,x) = \left(\frac t 2\right)^3\int_0^{1}\frac{(1-u)^2}{2} \E[V_0^3f(X_0(u t/2,X_1(\sqrt{t}Y,X_0(t/2,x))))]ds.
  \end{align*}
  We get by using Lemma~\ref{X0_inequalities}, Corollary~\ref{cor_psign} (using the symmetry and the finite moments of~$Y$), again Lemma~\ref{X0_inequalities} and then Lemma~\ref{lem_estimnorm}~(6):
  \begin{align*}&\|\E[V_0^3f(X_0(u t/2,X_1(\sqrt{t}Y,X_0(t/2,\cdot))))]\|_{m,L+3}\le C\|\E[V_0^3f(X_0(u t/2,X_1(\sqrt{t}Y,\cdot)))]\|_{m,L+3} \\&\le C\|V_0^3f(X_0(u t/2,\cdot))\|_{2m,L+3}\le   C\|V_0^3f\|_{2m,L+3} \le C  \|f\|_{2m+3,L}
  \end{align*}
This gives $\| R_If(t,x)\|_{m,L+3}\le  C t^3 \|f\|_{2m+3,L}$, for $t\in [0,T]$.

  We now expand again and get from~\eqref{expan_X1}
  \begin{align*} & \E\left[(f+\frac{t}{2} V_0 f +\frac{t^2}{8} V_0^2 f )(X_1(\sqrt{t}Y,X_0(t/2,x)))\right]                               \\
     & =f(X_0(t/2,x))+\frac t2 V_1^2 f(X_0(t/2,x))+ \frac {t^2}{2} (V_1^2/2)^2 f(X_0(t/2,x)) + \frac{t}{2} V_0 f(X_0(t/2,x)) \\
     & \phantom{=} + \frac{t^2}{2} (V_1^2/2)V_0 f(X_0(t/2,x)) + \frac{t^2}{8} V_0^2 f(X_0(t/2,x)) +R_{II}f(t,x),
  \end{align*}
  with
  \begin{align*}
    R_{II}f(t,x)=  t^3 \E\Bigg[&Y^6 \int_0^{u} \frac{(1-u)^{5}}{5!}V^{6}_1 f(X_1(u \sqrt{t} Y,X_0(T/2,x)))du \\
    &+  \frac {Y^4} 2 \int_0^{u} \frac{(1-u)^{3}}{ 3!}V^{4}_1V_0 f(X_1(u \sqrt{t} Y,X_0(T/2,x)))du \\
                  & + \frac{Y^2}{8} \int_0^{u}  (1-u)V^{2}_1V_0^2 f(X_1(u \sqrt{t} Y,X_0(T/2,x)))du \Bigg].
  \end{align*}
  We use Lemmas~\ref{lem_expan_V},~\ref{X0_inequalities} and~\ref{lem_estimnorm} to get, for $t\in [0,T]$,
  $$\| R_{II}f(t,\cdot)\|_{m,L+3}\le Ct^3(\|f\|_{2(m+3),L}+\|V_0f\|_{2(m+2),L+1}+\|V_0^2f\|_{2(m+1),L+2})\le Ct^3 \|f\|_{2(m+3),L}.$$
  Last, we use again~\eqref{expan_X0} to get 
  \begin{align*}
     & f(X_0(t/2,x))+\frac t2 [V_1^2 +V_0] f(X_0(t/2,x))+ \frac {t^2}{2} [(V_1^2/2)^2+(V_1^2/2)V_0+V_0^2/4] f(X_0(t/2,x))     \\
     & = f(x) +\frac t2 V_0f(x)+\frac{t^2}{8}V_0^2 f(x) + \frac t2 [V_1^2 +V_0] f(x) + \frac {t^2}{4} [V_0 V_1^2 +V_0^2] f(x) \\&\phantom{=}+ \frac {t^2}{2} [(V_1^2/2)^2+(V_1^2/2)V_0+V_0^2/4] f(x)+ R_{III}f(t,x),
  \end{align*}
  where again by Lemma~\ref{lem_expan_V}  and~\ref{lem_estimnorm}, we have
  \begin{align*}
    \|R_{III}f(t,\cdot)\|_{m,L+3} & \le Ct^3( \|f\|_{m+3,L}+\|[V_1^2 +V_0]f\|_{m+2,L+1}\\& \quad +\|[(V_1^2/2)^2+(V_1^2/2)V_0+V_0^2/4]f\|_{m+1,L+2}) \\&\le Ct^3 \|f\|_{m+5,L}.
  \end{align*}
  Finally, we get $ \|\bar{R}f(t,\cdot)\|_{m,L+3}\le Ct^3 \|f\|_{2(m+3),L}$ with $\bar{R}f:=R_{I}f+R_{II}f+R_{III}f$ and
  \begin{align*}
    \E[f(X_0(t/2,X_1(\sqrt{t}Y,X_0(t/2,x))))] & = f(x)+t[V_0+V_1^2/2]f(x)                                \\&+\frac{t^2}{2}[V_0^2+V_0V_1^2/2+(V_1^2/2)V_0+ (V_1^2/2)^2]f(x) +\bar{R}f(t,x)\\
                                              & =f(x)+t \cL f(x)+\frac{t^2}{2} \cL^2f(x) +\bar{R}f(t,x).\qedhere
  \end{align*}
\end{proof}

\subsection{Proof of~\eqref{H2_bar}}

In this section, we mainly prove the following result.
\begin{prop}\label{prop_H2_NV}
  Let $\sigma^2\le 4a$ and $\hat{X}^x_t=\varphi(x,t,\sqrt{t}N)$ be the scheme~\eqref{NV_scheme} with $N\sim \mathcal{N}(0,1)$. Let $T>0$ and  $m,L\in\N$ such that $L\ge m$. We define for $n\ge 1$ and $l\in \N$, $Q_lf(x)=\E[f(\hat{X}^x_{h_l})]$with $h_l=\frac{T}{n^l}$. Then, there exists a constant $C\in \R_+^*$ such that for any $f \in \CpolKL{m}{L}$, $l\in \N$ and $t \in [0,T]$, 
  \begin{equation}\label{H2_NV}   \left\|\E[f(X^{\cdot}_t)]\right\|_{m,L} + \max_{0\le k\le n^l} \left\|Q_l^{[k]}f \right\|_{m,L} \le C \|f\|_{m,L}.
  \end{equation}
\end{prop}
We split the proof in two parts. The first one deals with the semigroup of the CIR process, for which the assumption~$\sigma^2\le 4a$ is not needed. This is stated in Proposition~\ref{deriv_CIR_functionals}, whose proof exploits the particular form of the density of~$X^x_t$. The second part that deals with the approximation scheme is quite technical. We prove in fact in  Proposition~\ref{prop_H2_sch} a slightly more general result for the scheme  $\hat{X}^x_t=\varphi(x,t,\sqrt{t}Y)$, when $Y$ is a symmetric random variable with a smooth density. However, the conditions needed on the density are quite restrictive. These conditions are satisfied by the standard normal variable by Lemma~\ref{gaussian_eta_positivity}. If we want besides to have~\eqref{H2_NV} for any $m$ and in addition to match the moments $\E[Y^2]=\E[N^2]$ and $\E[Y^4]=\E[N^4]$ -- which is required to have a second-order scheme --, then Theorem~\ref{thm_carac_gauss} shows that we necessarily have $Y\sim \mathcal{N}(0,1)$. This is why we directly state here, for sake of simplicity, Proposition~\ref{prop_H2_NV} with $Y=N\sim \mathcal{N}(0,1)$.

\subsubsection{Upper bound for the semigroup}
We first prove the estimate~\eqref{H2_bar} for the semigroup of the CIR process. To do so, we take back the arguments of~\cite[Proposition 4.1]{AA_MCMA} that gives polynomial estimates for~$(t,x)\mapsto P_tf(x)$. First we remove the polynomial Taylor expansion of the function~$f$ at~$0$, which enables then to do an integration by parts and to get the remarkable formula in Eq.~\eqref{formula_derivative} below for the iterated derivatives of $P_tf$ that gives then the desired estimate. The polynomial part is analyzed separately in Lemma~\ref{Moments_Formula_CIR2} below with standard arguments.

\begin{prop}\label{deriv_CIR_functionals}
  Let $f\in\CpolKL{m}{L}$, $L\ge m$, $T>0$ and $t\in (0,T]$. Let  $X^x$ be the CIR process starting from $x\ge 0$. Then, $\E[f(X^{\cdot}_t)]\in\CpolKL{m}{L}$ and  we have the following estimate for some constant $C_\text{cir}(m,L,T)\in \R_+$:
  \begin{equation}\label{H0_estimate_CIR}
    \left\|\E[f(X^{\cdot}_t)]\right\|_{m,L} \le C_\text{cir}(m,L,T)\|f\|_{m,L}.
  \end{equation}
\end{prop}

\begin{proof}
  Let $f\in\CpolKL{m}{L}$ and $T_m(f)(x)=\sum_{j=0}^m \frac{f^{(j)}(0)}{j!} x^j$ its Taylor polynomial expansion at~$0$ of order $m$. We define $\hat{f}_m = f-T_m(f)\in\CpolKL{m}{L}$, so we have $f=\hat{f}_m +T_m(f)$.
  By Lemma~\ref{Moments_Formula_CIR2} below, one gets $\|\E[T_m(f)(X^{\cdot}_t)]\|_{m,L} \leq C_{cir}(m,T) \|T_m(f)\|_{m,L}$ and then
  \begin{equation}\label{estim_taylor_trunc}\|\E[T_m(f)(X^{\cdot}_t)]\|_{m,L}\leq e C_{cir}(m,T) \|f\|_{m,L},
  \end{equation}
  since for all $i\in\{0,\ldots,m\}$ and $x\geq0$
  $$ \bigg|\frac{(T_m(f))^{(i)}(x)}{1+x^L}\bigg| = \bigg|\sum_{j=0}^{m-i} \frac{f^{(i+j)}(0)}{j!} \frac{x^{j}}{1+x^L}\bigg| \leq \sum_{j=0}^{m-i} \frac{1}{j!} \|f\|_{m,L} \leq  e \|f\|_{m,L}. $$

  We now focus on~$\E[\hat{f}_m(X^{\cdot}_t)]$. We recall the density of $X^x_t$~(see e.g. \cite[Proposition 1.2.11]{AA_book})\footnote{In the case $a=0$, $X^x_t$ is distributed according to the probability measure $e^{-d_t x/2}\delta_{0}(dx)+\sum_{i=1}^\infty \frac{e^{-d_t x/2}(d_t x/2)^i}{i!} \frac{c_t/2}{\Gamma(i)}\left(\frac{c_t z}{2}\right)^{i-1}e^{-c_t z/2}$. The proof works the same since $\hat{f}_m(0)=0$, so that $\E[\hat{f}_m(X^x_t)]$ only involves the absolutely continuous part of the distribution. }
  \begin{equation}\label{CIR_density}
    p(t,x,z)=\sum_{i=0}^\infty \frac{e^{-d_t x/2}(d_t x/2)^i}{i!} \frac{c_t/2}{\Gamma(i+v)}\left(\frac{c_t z}{2}\right)^{i-1+v}e^{-c_t z/2}
  \end{equation}
  where $c_t=\frac{4k}{\sigma^2(1-e^{-kt})}$, $v=2a/\sigma^2$ and $d_t=c_te^{-kt}$. Let us remark that
  \begin{equation*}
    c_t\geq c_\text{min}:=\begin{cases} \,\,\quad\frac{4k}{\sigma^2},      & k>0  \\
      \quad\frac{4}{\sigma^2T},          & k=0  \\
      \frac{4|k|}{\sigma^2(e^{|k|T}-1)}, & k<0.
    \end{cases}
  \end{equation*}
  We have
  \begin{equation*}
    \E[\hat{f}_m(X^x_t)] = \sum_{i=0}^\infty \frac{e^{-d_t x/2}(d_t x/2)^i}{i!} I_i(\hat{f}_m,c_t),\ t>0,
  \end{equation*}
  where
  \begin{equation*}
    I_i(\hat{f}_m,c_t)=\int_0^\infty \hat{f}_m(z) \frac{c_t/2}{\Gamma(i+v)}\left(\frac{c_t z}{2}\right)^{i-1+v}e^{-c_t z/2}dz.
  \end{equation*}
  Differentiating successively, we get that for $j\le m $, $t\in(0,T]$ and $x\in\R_+$
  \begin{equation*}
    \partial^x_j\E[\hat{f}(X^x_t)]= \sum_{i=0}^\infty \frac{e^{-d_t x/2}(d_t x/2)^i}{i!} \Delta^j_t(I_i(\hat{f}_m,c_t)),
  \end{equation*}
  where $\Delta_t$ : $\R^\N \rightarrow \R^\N$ is an operator defined on sequences $(I_i)_{i\geq0}\in\R^\N$ by $\Delta_t(I_i) = \frac{d_t}{2} (I_{i+1}-I_i)= \frac{e^{-kt}}{2}c_t(I_{i+1}-I_i)$. An integration by parts gives for $i\geq 1$
  \begin{align*}
    I_i(\hat{f}^{(j)}_m,c_t) & = \int_0^\infty \hat{f}^{(j-1)}_m(z) \frac{(c_t/2)^2}{\Gamma(i+v)}\left(\frac{c_t z}{2}\right)^{i-1+v}e^{-c_t z/2}dz             \\
                             & \quad- \int_0^\infty \hat{f}_m^{(j-1)}(z) \frac{(c_t/2)^2(i-1+v)}{\Gamma(i+v)}\left(\frac{c_t z}{2}\right)^{i-2+v}e^{-c_t z/2}dz \\
                             & =\frac{c_t}{2}(I_i(\hat{f}_m^{(j-1)},c_t)-I_{i-1}(\hat{f}_m^{(j-1)},c_t))=e^{kt}\Delta_t (I_{i-1}(\hat{f}_m^{(j-1)},c_t)),
  \end{align*}
  since $\hat{f}^{(j)}_m(0)=0$ for all $1\leq j\leq m$ and $\hat{f}^{(j)}_m$has a polynomial growth. By iterating, we get for all $t\in(0,T]$ and $x\in\R_+$,
  \begin{equation}\label{formula_derivative}
    \partial^x_j\E[\hat{f}_m(X^x_t)]= \sum_{i=0}^\infty \frac{e^{-d_t x/2}(d_t x/2)^i}{i!} I_{i+j}(\hat{f}_m^{(j)},c_t) e^{-kjt}.
  \end{equation}
  Note that, since for $j\leq m $, $ |\hat{f}_m^{(j)}(z)| \leq \|\hat{f}_m\|_{m,L}(1 + z^L)$ and using the well known formula for the $L$-th raw moment of gamma distribution we have for all $i\in\N$
  \begin{equation}\label{estimate_I_f_m}
    |I_i(\hat{f}_m^{(j)},c_t)| \leq \|\hat{f}_m\|_{m,L} \left(1+ \bigg(\frac{2}{c_t}\bigg)^L \frac{\Gamma(i+L+v)}{\Gamma(i+v)}\right).
  \end{equation}
  Thus, the derivation of the series~\eqref{formula_derivative} is valid, and we get that
  $$|\partial^x_j\E[\hat{f}_m(X^x_t)]|\leq \|\hat{f}_m\|_{m,L} e^{-k jt}\left(1+ \big(\frac{2}{c_t}\big)^L \sum_{i=0}^\infty \frac{e^{-d_t x/2}(d_t x/2)^i}{i!} \frac{\Gamma(i+j+L+v)}{\Gamma(i+j+v)}\right). $$
  The quotient $\frac{\Gamma(i+j+L+v)}{\Gamma(i+j+v)}$ is a polynomial function of degree $L$ with respect to $i$, and we denote $\beta^j_0,\ldots,\beta^j_{L}$ its coefficients in the basis $\{ 1, i, i(i-1),\ldots, i(i-1)\cdots(i-L+1)\}$. Thus, we get that
  \begin{align*}
    |\partial^x_j\E[\hat{f}_m(X^x_t)]| & \leq \|\hat{f}_m\|_{m,L} e^{(-k)^+jT}\left(1+ \bigg(\frac{2}{c_t}\bigg)^L \sum_{i=0}^L |\beta^j_i| \bigg(\frac{d_t}{2}\bigg)^i x^i \right) \\
                                       & \leq \|\hat{f}_m\|_{m,L} e^{(-k)^+ mT}\left(1+ \sum_{i=0}^L |\beta^j_i| \bigg(\frac{2}{c_t}\bigg)^{L-i} e^{-kit} (1+x^L) \right)           \\
                                       & \leq \|\hat{f}_m\|_{m,L} e^{(-k)^+(m+L)T}\left(1+ \sum_{i=0}^L |\beta^j_i| \bigg(\frac{2}{c_{\text{min}}}\bigg)^{L-i}   \right)(1+x^L).
  \end{align*}
  By the triangular inequality and~\eqref{estim_taylor_trunc}, we get $\|\hat{f}_m\|_{m,L}\leq(1+e)\|f\|_{m,L}$, so one has for all $t\in (0,T]$, $j\le m$
  \begin{equation}
    |\partial^x_j\E[\hat{f}_m(X^x_t)]|  \leq (1+e) e^{(-k)^+(m+L)T} \left(1+ \sum_{i=0}^L |\beta^j_i| \bigg(\frac{2}{c_{\text{min}}}\bigg)^{L-i}   \right)\|f\|_{m,L}(1+x^L),
  \end{equation}
  and thus for all $t\in(0,T]$:
  \begin{equation}
    \|\E[\hat{f}_m(X^{\cdot}_t)]\|_{m,L} \leq \hat{C}\|f\|_{m,L},
  \end{equation}
  where $\hat{C}:=  (1+e) e^{(-k)^+(m+L)T} \max_{0\le j\le m}\left(1+ \sum_{i=0}^L |\beta^j_i| \bigg(\frac{2}{c_{\text{min}}}\bigg)^{L-i}   \right)$.
  Finally, we get the desired estimate by the triangular inequality,~\eqref{estim_taylor_trunc} and Lemma~\ref{Moments_Formula_CIR2}:
  \begin{equation*}
    \left\|\E[f(X^{\cdot}_t)]\right\|_{m,L} \leq  \|\E[\hat{f}_m (X^{\cdot}_t)]\|_{m,L} + \|\E[T_m(f) (X^{\cdot}_t)]\|_{m,L} \leq (\hat{C}+C_{cir}(m,T)) \| f\|_{m,L}. \qedhere
  \end{equation*}
\end{proof}

\begin{lemma}\label{Moments_Formula_CIR2}
  Let $P\in \PLRp{m}$ be  a polynomial function of degree $m\in\N^*$ and $L\in\N^*$ such that $L\geq m$. Then, for $t\in [0,T]$ we have the following estimate
  \begin{equation}\label{Pol_Estimate_CIR_2}
    \|\E[P(X^{\cdot}_t)]\|_{m,L} \leq C_{\text{cir}}(m,T) \|P\|_{m,L},
  \end{equation}
  where $C_{\text{cir}}(m,T) = \max_{t\in[0,T]} \sum_{j=0}^m \sum_{i=0}^j |\tilde{u}_{i,j}(t)|$ with $\tilde{u}_{i,j}(t)$ defined as in Lemma~\ref{Moments_Formula_CIR} by $\E[(X^x_t)^j]=\sum_{i=0}^j\tilde{u}_{i,j}(t)x^i$.
\end{lemma}
\begin{proof}
  We consider a polynomial function $P(y)=\sum_{i=0}^m b_i y^i$ of degree $m$ and $L\geq m$. For all $l\in\{0,\ldots,m\}$ one has from Lemma~\ref{Moments_Formula_CIR}
  \begin{align*}
    \frac{|\partial_x^l\E[P(X_t^x)]|}{1+x^L} & =\bigg| \sum_{j=0}^m b_j \frac{\partial_x^l\tilde{u}_j(t,x)}{1+x^L} \bigg|  \leq \bigg| \sum_{j=0}^m b_j \sum_{i=l}^j\tilde{u}_{i,j}(t)\frac{i!}{(i-l)!}\frac{x^{i-l}}{1+x^L} \bigg| \\
                                             & \leq  \sum_{j=0}^m |b_j| j! \sum_{i=l}^j |\tilde{u}_{i,j}(t)|
    \leq \max_{t\in[0,T]} \sum_{j=0}^m \sum_{i=l}^j |\tilde{u}_{i,j}(t)| \max_{j\in\{0,\ldots,L\}}|b_j| j!,
  \end{align*}
  passing to supremum over $x\geq 0$, $l\in\{0,\ldots,m\}$ we get \eqref{Pol_Estimate_CIR_2} observing that $|b_j| j!=|P^{(j)}(0)|\le \|P\|_{m,L}$.
\end{proof}

\subsubsection{Upper bound for the approximation scheme}
We now prove the estimate~\eqref{H2_bar} for the approximation of the CIR process. The main result of this paragraph is the following.
\begin{prop}\label{prop_H2_sch}
  Let $T>0, \sigma^2\le 4a$, $m,M \in \N$,  $Y$ be a symmetric random variable with density $\eta\in \mathcal{C}^M(\R)$ such that for all $i\in\{0,\ldots,M\}$, $|\eta^{(i)}(y)|=o(|y|^{-(2L+i)})$ for $|y|\rightarrow \infty$,  and $\eta^*_m\ge 0$ for all $1 \le m\le M$ (see Lemma~\ref{regular_density} below for the definition of $\eta^*_m$).
  Let $Q_lf(x)=\E[f(\hat{X}^x_{h_l})]$ with $\hat{X}^x_t=\varphi(t,x,\sqrt{t}Y)$, $n\ge 1$, $l\in \N$ and $h_l=T/n^l$. Then, for any $L \in \N$, there exists $C\in \R_+$ such that: 
  $$  \max_{0\le j \le n^l}\|Q_l^{[j]}f\|_{m,L} \le C \|f\|_{m,L}, \ f\in \CpolKL{m}{L}, l \in \N.$$
\end{prop}
Note that by Lemma~\ref{gaussian_eta_positivity} below, the assumptions of Proposition~\ref{prop_H2_sch} are satisfied by~$Y\sim \mathcal{N}(0,1)$. Therefore,~\eqref{H2_bar} holds for  the scheme of Ninomiya and Victoir~\eqref{NV_scheme}.
\begin{proof}
  We have $\hat{X}^x_t=\varphi(t,x,\sqrt{t}Y)=X_0(t/2,X_1(\sqrt{t}Y,X_0(t/2,x)))$. Let $f \in \CpolKL{m}{L}$. We apply Lemma~\ref{X0_inequalities} and Lemma~\ref{regular_density} below to get:
  \begin{align*}  \|\E[f(X_0(t/2,X_1(\sqrt{t}Y,X_0(t/2,\cdot ))))]\|_{m,L} & \le e^{Kt/2} \|\E[f(X_1(\sqrt{t}Y,X_0(t/2,\cdot )))]\|_{m,L} \\&\le e^{Kt/2+Ct} \|f(X_0(t/2,\cdot ))\|_{m,L}\le e^{(C+K)t} \|f\|_{m,L}
  \end{align*}
  This gives $\max_{0\le j \le n^l}\|Q_l^{[j]}f\|_{m,L} \le e^{(C+K)T} \|f\|_{m,L}$.
\end{proof}

\begin{lemma}\label{regular_density}
  Let $M,L\in \N$. Let $Y$ be a symmetric random variable with density $\eta\in \mathcal{C}^M(\R)$ such that for all $i\in\{0,\ldots,M\}$, $|\eta^{(i)}(y)|=o(|y|^{-(2L+i)})$ for $|y|\rightarrow \infty$. Then, for all function $f\in\CpolKL{M}{L}$, $m\in\{1, \ldots,  M\}$ and $t\in [0,T]$ one has the following representation
  \begin{equation}\label{repres_X1_functional}
    \partial^m_x\E[f(X_1(\sqrt{t}Y,x))] = \int_{-\infty}^\infty \int_0^1 (u-u^2)^{m-1} f^{(m)}(w(u,x,y)) \eta^*_m(y) dudy
  \end{equation}
  where  $w(u,x,y)=x+(2u-1)\sigma\sqrt{t}y\sqrt{x}+ \sigma^2ty^2/4$, $\eta^*_m(y)=(-1)^{m-1}  \left(\sum_{j=1}^m c_{j,m}  y^j\eta^{(j)}(y)\right)$, and the coefficients $c_{j,m}$  are defined by induction, starting from $c_{1,1}=-1$, through the following formula
  \begin{equation}\label{recursive_coeff_formula}
    c_{j,m}   = \bigg(\frac{2j}{m-1}-4\bigg)  c_{j,m-1}\mathds{1}_{j<m} + \frac{2}{m-1}  c_{j-1,m-1}\mathds{1}_{j>1},     \  j\in\{1, \ldots, m\}, \, m\in\{2, \ldots,M\}.  \\
  \end{equation}
 In particular, $c_{m,m}=-\frac{2^{m-1}}{(m-1)!}<0$. Furthermore, if the density $\eta$ is such that  $\eta^*_{m}(y)\geq0$ for all $y\in\R$, and all $ m\in\{1, \ldots,  M\}$, then there exists $C\in \R_+$ such that
  \begin{equation}\label{estimate_X1_functional}
    	\|\E[f(X_1(\sqrt{t}Y,\cdot))]\|_{m,L} \leq (1+Ct) \|f\|_{m,L},\ t\in [0,T].
  \end{equation}
\end{lemma}
Let us stress here two things that are crucial in~\eqref{estimate_X1_functional}: the same norm is used in both sides, and the sharp time dependence of the multiplicative constant $(1+Ct)$. These properties are used in the proof of Proposition~\ref{prop_H2_sch} to get~\eqref{H2_bar}.

\begin{proof}
  We first consider $m=1$ and $f\in\CpolKL{M}{L}$. From the symmetry of $Y$, we have the equality $\E[f(X_1(\sqrt{t}Y,x))] = \E[f(X_1(\sqrt{t}Y,x))+ f(X_1(-\sqrt{t}Y,x))]/2$ and using the notation $\psi^{\pm}_f(x,y) = f(x+\sigma\sqrt{t}y\sqrt{x}+ \sigma^2ty^2/4) \pm f(x-\sigma\sqrt{t}y\sqrt{x}+ \sigma^2ty^2/4)$ we can write,
  \begin{equation*}
    \partial_x\E[f(X_1(\sqrt{t}Y,x))] = \frac{1}{2} \int_{-\infty}^\infty \partial_x\psi^+_f(x,y) \eta(y)dy.
  \end{equation*}
  One derivation and a little of algebra show that
  \begin{align*}
    \partial_x \psi^+_f(x,y) & = (1+ \frac{\sigma\sqrt{t}y}{2\sqrt{x}}) f'(x+\sigma\sqrt{t}y\sqrt{x}+ \sigma^2ty^2/4)+ (1- \frac{\sigma\sqrt{t}y}{2\sqrt{x}}) f'(x-\sigma\sqrt{t}y\sqrt{x}+ \sigma^2ty^2/4) \\
                             & = \frac{1}{\sigma\sqrt{t}\sqrt{x}}
    \begin{multlined}[t]
      \Big((\sigma^2 t y/2+ \sigma\sqrt{t}\sqrt{x}) f'(x+\sigma\sqrt{t}y\sqrt{x}+ \sigma^2ty^2/4) \\
      - (\sigma^2 t y/2- \sigma\sqrt{t}\sqrt{x}) f'(x-\sigma\sqrt{t}y\sqrt{x}+ \sigma^2ty^2/4)\Big)
    \end{multlined}                                                                                                                                                                              \\
                             & = \frac{1}{\sigma\sqrt{t}\sqrt{x}} \Big(\partial_y [f(x+\sigma\sqrt{t}y\sqrt{x}+ \sigma^2ty^2/4)] -  \partial_y [f(x-\sigma\sqrt{t}y\sqrt{x}+ \sigma^2ty^2/4)] \Big)         \\
                             & = \frac{\partial_y \psi^-_f(x,y)}{\sigma\sqrt{t}\sqrt{x}}.
  \end{align*}
  Integrating by parts in the variable $y$, observing that the boundary term vanishes since $|\eta(y)|=_{|y|\to \infty}o(|y|^{-2L})$ and $f(z)=_{z\to \infty}O(z^L)$, one has
  \begin{align*}
    \partial_x\E[f(X_1(\sqrt{t}Y,x))] & = - \frac{1}{2} \int_{-\infty}^\infty \frac{ \psi^-_f(x,y) \eta'(y)}{\sigma\sqrt{t}\sqrt{x}} dy            \\
                                      & = -  \int_{-\infty}^\infty \int_0^1 f'(x+(2u-1)\sigma\sqrt{t}y\sqrt{x}+ \sigma^2ty^2/4)  \eta'(y)y \,dudy  \\
                                      & =  \int_{-\infty}^\infty \int_0^1 f'(x+(2u-1)\sigma\sqrt{t}y\sqrt{x}+ \sigma^2ty^2/4)  (-\eta'(y)y) \,dudy
  \end{align*}
  since $\partial_u f(x+(2u-1)\sigma\sqrt{t}y\sqrt{x}+ \sigma^2ty^2/4) = 2\sigma\sqrt{t}y\sqrt{x} f'(x+(2u-1)\sigma\sqrt{t}y\sqrt{x}+ \sigma^2ty^2/4) $.
  In order to simplify the notation, we define $w(u,x,y) := x+(2u-1)\sigma\sqrt{t}y\sqrt{x}+ \sigma^2ty^2/4$, and we write explicitly the partial derivatives of $w$
  \begin{equation}\label{deriv_w}
    \begin{cases}
      \partial_u w(u,x,y) = 2\sigma\sqrt{t}y\sqrt{x},                  \\
      \partial_x w(u,x,y) = 1+\frac{(2u-1)\sigma\sqrt{t}y}{2\sqrt{x}}, \\
      \partial_y w(u,x,y) = (2u-1)\sigma\sqrt{t}\sqrt{x}+\frac{\sigma^2 t y}{2},
    \end{cases}
  \end{equation}
  and we define  for $s:[0,1]\times \R \rightarrow \R$
  \begin{equation}
    I^{(l)}_{m,n}(s)= \int_{-\infty}^\infty \int_0^1 s(u,y) (u^2-u)^{m-1}f^{(l)}(w(u,x,y))  \bigg(\sum_{j=1}^n c_{j,n}  y^j\eta^{(j)}(y)\bigg) \,dudy,
  \end{equation}
  so we can rewrite \eqref{repres_X1_functional} as $\partial^m_x\E[f(X_1(\sqrt{t}Y,x))] = I^{(m)}_{m,m}(1)$ where the 1 in the argument has to be intended as the constant map identically equal to 1. So far, we have shown that formula \eqref{repres_X1_functional} is true for $m=1$, we take now $m\geq2$, and we prove it by induction over $m$ assuming that the result holds for $m-1$. We differentiate Eq.~\eqref{repres_X1_functional}
  for $m-1$ and use the second equality of \eqref{deriv_w} to get
  \begin{equation}
    \partial^m_x\E[f(X_1(\sqrt{t}Y,x))]  =I^{(m)}_{m-1,m-1}(1) + I^{(m)}_{m-1,m-1}\left(\frac{(2u-1)\sigma\sqrt{t}y}{2\sqrt{x}}\right).
  \end{equation}
  Then, from the third equality of \eqref{deriv_w}, one has $\frac{(2u-1)\sigma\sqrt{t}y}{2\sqrt{x}}= \frac{(2u-1)}{\sigma\sqrt{t}\sqrt{x}}\partial_yw - (2u-1)^2$ and so
  \begin{align*}
    \partial^m_x\E[f(X_1(\sqrt{t}Y,x))] & =I^{(m)}_{m-1,m-1}(1- (2u-1)^2) + I^{(m)}_{m-1,m-1}\left(\frac{(2u-1)}{\sigma\sqrt{t}\sqrt{x}} \partial_y w(u,x,y)\right) \\
                                        & =-4 I^{(m)}_{m-1,m-1}(u^2-u) + I^{(m)}_{m-1,m-1}\left(\frac{(2u-1)}{\sigma\sqrt{t}\sqrt{x}} \partial_y w(u,x,y)\right)    \\
                                        & =-4 I^{(m)}_{m,m-1}(1) + I^{(m)}_{m-1,m-1}\left(\frac{(2u-1)}{\sigma\sqrt{t}\sqrt{x}} \partial_y w(u,x,y)\right).
  \end{align*}
  We work on the term $I^{(m)}_{m-1,m-1}(\frac{(2u-1)}{\sigma\sqrt{t}\sqrt{x}} \partial_y w(u,x,y))$. We use first an integration by parts in the variable $y$ and subsequently one in the variable $u$.  The boundary terms vanishes by using the hypothesis on $\eta$ since $|f^{(m)}(w(u,x,y))|\le \|f\|_{m,L}(1+w(u,x,y)^L)\underset{|y|\to \infty }=O(y^{2L})$ and to the fact that the function $u^2-u$ vanishes in $0$ and $1$. One gets
    {\small	\begin{multline}
        \int_0^1\int_{-\infty}^\infty   \frac{(2u-1)(u^2-u)^{m-2}}{\sigma\sqrt{t}\sqrt{x}}f^{(m)}(w(u,x,y))\partial_y w(u,x,y)\bigg(  \sum_{j=1}^{m-1} c_{j,m-1}  y^j\eta^{(j)}(y) \bigg)dydu \\
        \begin{aligned}
           & =-\int_{-\infty}^\infty \int_0^1  \frac{(2u-1)(u^2-u)^{m-2}}{\sigma\sqrt{t}\sqrt{x}}f^{(m-1)}(w(u,x,y))  \bigg(\sum_{j=1}^{m-1} c_{j,m-1}  (jy^{j-1}\eta^{(j)}(y)+y^{j}\eta^{(j+1)}(y)) \bigg)dudy  \\
           & =\int_{-\infty}^\infty \int_0^1 (u^2-u)^{m-1} f^{(m)}(w(u,x,y)) \bigg( \sum_{j=1}^{m-1} \frac{2}{m-1}c_{j,m-1}  (jy^{j}\eta^{(j)}(y)+y^{j+1}\eta^{(j+1)}(y)\bigg) \,dudy                            \\
           & =\int_{-\infty}^\infty \int_0^1 (u^2-u)^{m-1} f^{(m)}(w(u,x,y))  \frac{2}{m-1}\bigg(\sum_{j=1}^{m-1} jc_{j,m-1}  y^{j}\eta^{(j)}(y) +\sum_{j=2}^{m} c_{j-1,m-1}y^{j}\eta^{(j)}(y)\bigg) \bigg)dudy.
        \end{aligned}
      \end{multline}}
  Rewriting the last equality for $\partial^m_x\E[f(X_1(\sqrt{t}Y,x))]$, one has
    {\small \begin{multline}
        \partial^m_x\E[f(X_1(\sqrt{t}Y,x))] = \int_{-\infty}^\infty \int_0^1 (u^2-u)^{m-1} f^{(m)}(w(u,x,y))  \bigg(-4\sum_{j=1}^{m-1} c_{j,m-1}  y^{j}\eta^{(j)}(y)\bigg)\,dudy +\\
        + \int_{-\infty}^\infty \int_0^1 (u^2-u)^{m-1} f^{(m)}(w(u,x,y))  \frac{2}{m-1}\bigg(\sum_{j=1}^{m-1} jc_{j,m-1}  y^{j}\eta^{(j)}(y) +\sum_{j=2}^{m} c_{j-1,m-1}y^{j}\eta^{(j)}(y)\bigg) \,dudy\\
        = \begin{aligned}[t]
           & \int_{-\infty}^\infty \int_0^1 (u^2-u)^{m-1} f^{(m)}(w(u,x,y))\bigg( \big(\frac{2}{m-1} -4\big)c_{1,m-1} y\eta^{(1)}(y)+                                                 \\
           & \sum_{j=2}^{m-1}\Big(\big(\frac{2j}{m-1} -4\big)c_{j,m-1} +\frac{2}{m-1}c_{j-1,m-1}\Big) y^{j}\eta^{(j)}(y) + \frac{2}{m-1} c_{m-1,m-1} y^{m}\eta^{(m)}(y)\bigg) \,dudy,
        \end{aligned}
      \end{multline}}
  which proves the representation \eqref{repres_X1_functional}. Since  $c_{1,1}=-1$ and $c_{m,m}=-\frac{2}{m-1} c_{m-1,m-1}$ for $m\ge 2$, we get $c_{m,m}=-\frac{2^{m-1}}{(m-1)!}$ for $m\ge 1$.

  We are now able to prove the estimate using this representation. Defining $\eta^*_m(y) = (-1)^{m-1}  \\\sum_{j=0}^mc_{j,m}y^j\eta{(j)}(y)$, that is nonnegative for all $y$ by hypothesis, one has
  \begin{align*}
    |\partial^m_x\E[f(X_1(\sqrt{t}Y,x))]| & \leq \int_0^1 (u-u^2)^{m-1}  \int_{-\infty}^\infty   |f^{(m)}(w(u,x,y))| \eta^*_m(y)dydu                     \\
                                          & \leq \|f\|_{m,L}\int_0^1 (u-u^2)^{m-1}  \int_{-\infty}^\infty (1+w(u,x,y)^L) \eta^*_m(y)dydu                 \\
                                          & =  \|f\|_{m,L}\underbrace{\int_0^1 (u-u^2)^{m-1}  \int_{-\infty}^\infty  \eta^*_m(y)dydu}_{A}                \\
                                          & \quad+ \|f\|_{m,L}\underbrace{\int_0^1 (u-u^2)^{m-1}  \int_{-\infty}^\infty w(u,x,y)^L \eta^*_m(y)dydu}_{B}.
  \end{align*}
  The double integral $A$ can be seen by means of representation \eqref{repres_X1_functional} with $f(x)=\frac{x^m}{m!}$ ($f^{(m)}\equiv 1$) as
  \begin{align*}
    A = \partial^m_x\E\left[\frac{X_1(\sqrt{t}Y,x)^m}{m!}\right] & =\frac{1}{m!} \partial^m_x\sum_{j=0}^{m}{2m\choose 2j} x^{m-j}\left(\frac{\sigma\sqrt{t}}{2}\right)^j \E[Y^{2j}]=1,
  \end{align*}
  by using  the symmetry of the density $\eta$.	  In the same way, $B$ can be seen by means of the representation as
  \begin{align*}
    B & = \partial^m_x\E\left[\frac{L!}{(L+m)!}X_1(\sqrt{t}Y,x)^{L+m}\right]                                                               \\
      & = \partial^m_x\sum_{j=0}^{L+m}{2(L+m)\choose 2j} \frac{L!x^{L+m-j}}{(L+m)!}\left(\frac{\sigma\sqrt{t}}{2}\right)^{2j} \E[Y^{2j}]   \\
      & = \sum_{j=0}^{L}{2(L+m)\choose 2j} \frac{L!(L+m-j)!}{(L-j)!(L+m)!}x^{L-j}\left(\frac{\sigma\sqrt{t}}{2}\right)^{2j} \E[Y^{2j}]     \\
      & =x^L +t\sum_{j=1}^{L}{2(L+m)\choose 2j} \frac{L!(L+m-j)!}{(L-j)!(L+m)!}x^{L-j}\left(\frac{\sigma}{2}\right)^{2j}t^{j-1} \E[Y^{2j}] \\
      & \leq x^L +t(1+x^L)(1+\E[Y^{2L}])\sum_{j=1}^{L}{2(L+m)\choose 2j} \left(\frac{\sigma c_T}{2}\right)^{2j}                            \\
      & \leq x^L +Ct(1+x^L)
  \end{align*}
  where $c_T= \max(1,T)$ and $C = \frac{1}{2}\left((1+\frac{\sigma c_T}{2})^{2(L+k)}+(1-\frac{\sigma c_T}{2})^{2(L+k)}\right)(1+\E[Y^{2L}])$. Putting parts $A$ and $B$ back together one has
  \begin{equation}
    \partial^m_x\E[f(X_1(\sqrt{t}Y,x))] \leq \|f\|_{m,L} (1 +x^L + Ct(1+x^L)) = (1 +x^L)(1 +Ct)\|f\|_{m,L}  ,
  \end{equation}
  and this proves the desired norm inequality.
\end{proof}

\begin{lemma}\label{gaussian_eta_positivity}
  Let $\eta(y)=\frac 1 {\sqrt{2\pi}} e^{-y^2/2}$ be the density of a standard normal variable. Then, we have for $m\ge 1$:
  \begin{equation}\label{gauss_identity}
    \eta^*_m(y):= (-1)^{m-1}\sum_{j=1}^m  c_{j,m} y^j\eta^{(j)}(y) = -c_{m,m}y^{2m}\eta(y),
  \end{equation}
  so, in particular $\eta^*_m(y)\ge 0$ for all $y\in\R$.
\end{lemma}

\begin{proof}
  For $m=1$, \eqref{gauss_identity} is clearly true since $\eta'(y)=-y\eta(y)$.
  We now take $m\geq2$, $M\geq m$ and we suppose \eqref{gauss_identity} true for $m-1$:  for all $f\in \CpolKL{M}{L}$ and $x\in\R_+$, we have
  \begin{equation*}
    \int_{-\infty}^\infty \int_0^1 (u-u^2)^{m-2} f^{(m-1)}(w(u,x,y)) \big(\eta^*_{m-1}(y)+c_{m-1,m-1}y^{2m-2}\eta(y)\big)dudy=0.
  \end{equation*}
  Doing one differentiation step with respect to~$x$ like in the proof of Lemma~\ref{regular_density} and using that $\eta'(y)=-y\eta(y)$, we obtain
  \begin{equation*}
    \int_{-\infty}^\infty \int_0^1 (u-u^2)^{m-1} f^{(m)}(w(u,x,y)) \big(\eta^*_m(y)+c_{m,m}y^{2m}\eta(y)\big)dudy=0.
  \end{equation*}
  By choosing $f_L(x):= \frac{L!}{(L+m)!}x^{L+m}$ for $L\in \N$, we get for all $L\in \N$, $x \in \R_+$,
  \begin{equation*}
      \int_{-\infty}^\infty \int_0^1 (u-u^2)^{m-1} w(u,x,y)^L \big(\eta^*_m(y)+c_{m,m}y^{2m}\eta(y)\big)dudy=0.
  \end{equation*}
  We now take $x=0$ so that $w(u,0,y)=\frac{\sigma^2t}4 y^2$ and then
  $$ \int_{-\infty}^\infty y^{2L} \big(\eta^*_m(y)+c_{m,m}y^{2m}\eta(y)\big)dy=0, \ L\in \N.$$ We remark also that
  $\eta^*_m(y)=\big(\sum_{j=1}^m (-1)^{m+j-1}c_{j,m}y^jH_j(y)\big)\eta(y)$, where $H_j$ is the $j^{\text{th}}$ Hermite polynomial function (defined by $\eta^{(j)}(y)=(-1)^j H_j(y)\eta(y)$). Thus, $\eta^*_m(y)+c_{m,m}y^{2m}\eta(y)=P_{2m}(y)\eta(y)$ where $P_{2m}$ is  an even polynomial function of degree $2m$. We therefore obtain $\int_{-\infty}^\infty  y^{l} P_{2m}(y) \eta(y) dy=0$ for all $l\in \N$, which gives $P_{2m}=0$ and thus the claim.
\end{proof}

\begin{remark}
  Lemma~\ref{gaussian_eta_positivity} gives a remarkable formula of the monomial of order $2m$ $m\in\N^*$ in terms of the first $m$ Hermite polynomials multiplied respectively by the first $m$ monomials
  \begin{equation}\label{Hermite_inv_formula}
    y^{2m} = \sum_{j=1}^m (-1)^{m+j}\frac{c_{j,m}}{c_{m,m}}y^jH_j(y).
  \end{equation}
\end{remark}

The next result gives a kind of reciprocal result to Lemma~\ref{gaussian_eta_positivity}. It explains why we consider a normal random variable for~$Y$ in Theorem~\ref{thm_main}, since we use Proposition~\ref{prop_H2_sch} for any $M\in \N$.
\begin{theorem}\label{thm_carac_gauss}
  Let $Y$ be a symmetric random variable with a $\mathcal{C}^\infty$ probability density function~$\eta$ such that $\E[Y^2]=1$, $\E[Y^4]=3$ and $\eta_m^*\ge 0$ for all $m\ge 1$. Then, $Y\sim \mathcal{N}(0,1)$.
\end{theorem}
\begin{proof}
  By Corollary~\ref{cor_density_etam}, there exists a positive Borel measure~$\mu$ such that
  $\eta(x)=\int_0^\infty e^{-tx^2}\mu(dt)$. Since $\int_{\R}\eta=1$, we get $\int_0^\infty \sqrt{\pi/t}\mu(dt)=1$ and then   $\eta(x)=\int_0^\infty \frac{e^{-tx^2}}{\sqrt{\pi/t}} \tilde{\mu}(dt)$ with $\tilde{\mu}(dt)=\sqrt{\pi/t} \mu(dt)$ being a probability measure on~$\R_+$. We have $\E[Y^2]=\int_0^\infty \int_{\R} x^2\frac{e^{-tx^2}}{\sqrt{\pi/t}} dx \tilde{\mu}(dt)=\int_0^\infty \frac{1}{2t} \tilde{\mu}(dt)$ and $\E[Y^4]=\int_0^\infty 3\left(\frac{1}{2t}\right)^2 \tilde{\mu}(dt)$. Therefore, we have
  $$\int_0^\infty \frac{1}{2t} \tilde{\mu}(dt)=\int_0^\infty \left(\frac{1}{2t}\right)^2 \tilde{\mu}(dt) =1.$$
  The equality condition in the Cauchy-Schwarz inequality implies that $ \tilde{\mu}(dt) =\delta_{1/2}(dt)$, i.e. $Y$ is a standard normal variable.
\end{proof}

\subsection{Proof of Theorem~\ref{thm_main}}\label{Subsec_thm_main}
We prove the result for $\cPh^{2,n}$. By assumption, $f\in  \CpolKL{18}{L}$, for $L\ge 18$ sufficiently large. From~\eqref{devt_erreur}, we have
\begin{align*}
  P_Tf-\cPh^{2,n}f & =  \sum_{k=0}^{n-1}Q_1^{[n-(k+1)]}[P_{h_1}-Q_2^{[n]}]Q_1^{[k]}f \\ &\quad+\sum_{k=0}^{n-1}\sum_{k'=0}^{n-(k+2)}P_{(n-(k+k'+2))h_1}[P_{h_1}-Q_1] Q_1^{[k']}[P_{h_1}-Q_1] Q_1^{[k]},
\end{align*}
with $h_l=T/n^l$.
Using Proposition~\ref{prop_H2_NV} three times and Proposition~\ref{H1_bar_mL} twice, we get for $k\in \{0,\dots,n-1\},k' \in \{0,\dots,n-(k+2)\}$:
\begin{align*}
  \|P_{(n-(k+k'+2))h_1}[P_{h_1}-Q_1] Q_1^{[k']}[P_{h_1}-Q_1] Q_1^{[k]} f\|_{0,L+6} & \le C\|[P_{h_1}-Q_1] Q_1^{[k']}[P_{h_1}-Q_1]  Q_1^{[k]}f\|_{0,L+6} \\
                                                                                   & \le C h_1^3\| Q_1^{[k']}[P_{h_1}-Q_1] Q_1^{[k]}f\|_{6,L+3}         \\
                                                                                   & \le  C h_1^3\|[P_{h_1}-Q_1] Q_1^{[k]}f \|_{6,L+3}                  \\
                                                                                   & \le Ch_1^6\|Q_1^{[k]}f \|_{18,L} \le Ch_1^6\|f \|_{18,L}.
\end{align*}
For the other term, we write $P_{h_1}-Q_2^{[n]}=\sum_{k'=0}^{n-1}P_{(n-(k'+1))h_2}[P_{h_2}-Q_2]Q_2^{[k']}$ and get for $k,k'\in \{0,\dots,n-1\}$ by using  Proposition~\ref{prop_H2_NV}, Proposition~\ref{H1_bar_mL} and Lemma~\ref{lem_estimnorm}:
\begin{align*}
  \|Q_1^{[n-(k+1)]}P_{(n-(k'+1))h_2}[P_{h_2}-Q_2]Q_2^{[k']}Q_1^{[k]}f\|_{0,L+6} & \le C\|[P_{h_2}-Q_2]Q_2^{[k']}Q_1^{[k]}f\|_{0,L+6}   \\
                                                                                & \le C h_2^3 \|Q_2^{[k']}Q_1^{[k]}f\|_{6,L+3}         \\
                                                                                & \le  C h_2^3 \|f\|_{6,L+3}\le  C h_2^3 \|f\|_{18,L}.
\end{align*}
This gives
$$\|P_Tf-\cPh^{2,n}f\|_{0,L+6}\le C \|f\|_{18,L} n^2(h_1^6+h_2^3) \le C \|f\|_{18,L} n^{-4},$$
and in particular $P_Tf(x)-\cPh^{2,n}f(x)=O(n^{-4})$ for any $x\ge 0$.

We now consider $f\in \mathcal{C}^\infty$ with derivatives of polynomial growth. Therefore, for any $m\in \N$, it exists $L \ge m$ sufficiently large, such that $f\in  \CpolKL{m}{L}$. We can then apply \cite[Theorem~3.10]{AB} to get that for some functions ${\bf m},\ell:\N^*\to \N^*$, we have $\|P_Tf-\cPh^{\nu,n}f\|_{0,L+\ell(\nu)}\le C\|f\|_{{\bf m}(\nu),L}n^{-2\nu}$ for $L\ge {\bf m}(\nu)$, which gives the claim.


\section{Simulations results}\label{Simulations}
In order to present some numerical test, we  first explain how to implement the approximations $\cPh^{2,n}$ and $\cPh^{3,n}$ defined respectively by~\eqref{def_P_2} and~\eqref{def_P_3} (let us recall here that $\cPh^{1,n}$ is the approximation obtained on the regular time grid $\Pi^0 = \{kT/n, 0 \leq k \leq n\}$). We consider a general case of a scheme that can be written as a function of the starting point, the time step, the Brownian increment and an independent random variable, i.e.
$$Q_lf(x)=\E[\varphi(x,h_l,W_{h_l},V)].$$
The second order scheme for the CIR~\eqref{NV_scheme}  falls into this framework as well as the second order scheme for the Heston model~\eqref{H2S} that we introduce below. As illustrated in \cite{AB} the approximation~$\cPh^{2,n}$ is the simplest case for the implementation. It consists in the simulation of two starting schemes on the uniform time grid $\Pi^0$ and on the random grid : $\Pi^1 = \Pi^0 \cup  \{ \kappa T/n + k'T/n^2 , 1 \leq k' \leq n-1 \}$, where $\kappa$ is an independent uniform random variable on $\{0,\ldots,n-1\}$. We denote by $\hat{X}^{n,0}$ the scheme on $\Pi^0$
\begin{align}
  \hat{X}^{n,0}_0          & = x, \notag                                                                                \\
  \hat{X}^{n,0}_{(k+1)h_1} & = \varphi(\hat{X}^{n,0}_{kh_1}, h_1, W_{(k+1)h_1} - W_{kh_1}, V_k ),\quad 0\leq k\leq n-1, \label{Scheme0}
\end{align}
and by  $\hat{X}^{n,1}$ the scheme on $\Pi^1$:
\begin{align*}
  \hat{X}^{n,1}_{kh_1}                 & = \hat{X}^{n,0}_{kh_1},                                                                            & \quad   0\leq k \leq \kappa,   \\
  \hat{X}^{n,1}_{\kappa h_1+(k'+1)h_2} & = \varphi(\hat{X}^{n,1}_{\kappa h_1+k'h_2}, h_2, W_{\kappa h_1+(k'+1)h_2} - W_{\kappa h_1+k'h_2},V_{n+k'}), & \quad 0 \leq k' \leq n-1,      \\
  \hat{X}^{n,1}_{(k+1)h_1}             & = \varphi(\hat{X}^{n,1}_{kh_1}, h_1, W_{(k+1)h_1} - W_{kh_1},V_k),                                     & \quad \kappa+1 \leq k\leq n-1.
\end{align*}
Here, $(V_k)_{k\ge 0}$ is an i.i.d. sequence with the same law as $V$. Finally, we can give the following probabilistic representation
\begin{align}
  \cPh^{2,n}f & = Q^{[n]}_1f+n\E[Q_1^{[n-(\kappa+1)]}[Q_2^{[n]}-Q_1]Q_1^{[\kappa]}f] \nonumber            \\
           & =\E[f(\hat{X}^{n,0}_T)] + n\E[f(\hat{X}^{n,1}_T) - f(\hat{X}^{n,0}_T)]. \label{P2n_indep} 
\end{align}
Let us stress here that it is crucial for the Monte-Carlo method to use the same underlying Brownian motion for $\hat{X}^{n,0}$ and $\hat{X}^{n,1}$. Thus, the variance of $n\left(f(\hat{X}^{n,1}_T) - f(\hat{X}^{n,0}_T)\right)$ is quite moderate. It is shown in~\cite[Appendix A]{AB} that this variance is bounded when using the Euler scheme for an SDE with smooth coefficients. The theoretical analysis of the variance in our framework is beyond the scope of the paper. We only check numerically how it evolves with respect to~$n$ on our experiments, see Table~\ref{Table_VAR_CIR1} below.

The approximation $\cPh^{3,n}$ is more involved. Let $\kappa'$ be an independent uniform random variable on $\{0,\ldots,n-1\}$. We 
define the scheme  $\hat{X}^{n,2}$:
\begin{align*}
  \hat{X}^{n,2}_{kh_1}                 & = \hat{X}^{n,1}_{kh_1}, \quad \hat{X}^{n,2}_{\kappa h_1+k' h_2} = \hat{X}^{n,1}_{\kappa h_1+k'h_2},                                                                                0\leq k \leq \kappa,\ 0\leq k' \leq \kappa',   \\
  \hat{X}^{n,2}_{\kappa h_1+\kappa' h_2+(k''+1) h_3} & = \varphi(\hat{X}^{n,2}_{\kappa h_1+\kappa' h_2+k'' h_3}, h_3, W_{\kappa h_1+\kappa'h_2 +(k''+1)h_3} - W_{\kappa h_1+\kappa'h_2 +k''h_3} ,V_{2n+k''}),   \\
  &\phantom{= \varphi(\hat{X}^{n,2}_{\kappa h_1+\kappa' h_2+k'' h_3}, h_3, W_{\kappa h_1+\kappa'h_2 +(k''+1)h_3}) ********}0 \leq k'' \leq n-1,      \\
  \hat{X}^{n,2}_{\kappa h_1 +(k'+1)h_2}             & = \varphi(\hat{X}^{n,2}_{\kappa h_1 +k' h_2}, h_2, W_{\kappa h_1+(k'+1)h_2} - W_{\kappa h_1+k'h_2},V_{n+k'}),                                    \,  \kappa+1 \leq k'\leq n-1.\\
  \hat{X}^{n,2}_{(k+1)h_1}             & = \varphi(\hat{X}^{n,2}_{kh_1}, h_1, W_{(k+1)h_1} - W_{kh_1},V_k),                                       \kappa+1 \leq k\leq n-1.
\end{align*}
This is the scheme obtained on the time grid $\Pi^1 \cup  \{ \kappa T/n + \kappa' T/n^2 + k''T/n^3 , 1 \leq k'' \leq n-1 \}$.
We have $$\sum_{k=0}^{n-1}Q_1^{[n-(k+1)]}\left[\sum_{k'=0}^{n-1}Q_2^{[n-(k'+1)]}[Q_3^{[n]}-Q_2]Q_2^{[k']} \right] Q_1^{[k]}f =n^2\E[f(\hat{X}^{n,2}_T) - f(\hat{X}^{n,1}_T)].$$
We now explain how to calculate the second term in~\eqref{def_P_3}. Let $(\kappa_1,\kappa_2)$ be an independent random variable uniformly distributed on the set $\{(k_1,k_2):0\le k_1<k_2<n\}$. We define: 
\begin{align*}
  \hat{X}^{n,3}_{kh_1}                 & = \hat{X}^{n,0}_{kh_1},                                                                            & \quad   0\leq k \leq \kappa_1,   \\
  \hat{X}^{n,3}_{\kappa_1 h_1+(k'+1)h_2} & = \varphi(\hat{X}^{n,3}_{\kappa_1 h_1+k'h_2}, h_2, W_{\kappa h_1+(k'+1)h_2} - W_{\kappa h_1+k'h_2},V_{3n+k'}), & \quad 0 \leq k' \leq n-1,      \\
  \hat{X}^{n,3}_{(k+1)h_1}             & = \varphi(\hat{X}^{n,3}_{kh_1}, h_1, W_{(k+1)h_1} - W_{kh_1},V_k),                                     & \quad \kappa_1+1 \leq k\leq n-1,
\end{align*}
\begin{align*}
  \hat{X}^{n,4}_{kh_1}                 & = \hat{X}^{n,0}_{kh_1},                                                                            & \quad   0\leq k \leq \kappa_2,   \\
  \hat{X}^{n,4}_{\kappa_2 h_1+(k'+1)h_2} & = \varphi(\hat{X}^{n,4}_{\kappa_2 h_1+k'h_2}, h_2, W_{\kappa_2 h_1+(k'+1)h_2} - W_{\kappa_2 h_1+k'h_2},V_{4n+k'}), & \quad 0 \leq k' \leq n-1,      \\
  \hat{X}^{n,4}_{(k+1)h_1}             & = \varphi(\hat{X}^{n,4}_{kh_1}, h_1, W_{(k+1)h_1} - W_{kh_1},V_k),                                     & \quad \kappa_2+1 \leq k\leq n-1,
\end{align*}
and
\begin{align*}
  \hat{X}^{n,5}_{kh_1}                 & = \hat{X}^{n,3}_{kh_1},                                                                            & \quad   0\leq k \leq \kappa_2,   \\
  \hat{X}^{n,5}_{\kappa_2 h_1+(k'+1)h_2} & = \varphi(\hat{X}^{n,5}_{\kappa_2 h_1+k'h_2}, h_2, W_{\kappa_2 h_1+(k'+1)h_2} - W_{\kappa_2 h_1+k'h_2},V_{4n+k'}), & \quad 0 \leq k' \leq n-1,      \\
  \hat{X}^{n,5}_{(k+1)h_1}             & = \varphi(\hat{X}^{n,5}_{kh_1}, h_1, W_{(k+1)h_1} - W_{kh_1},V_k),                                     & \quad \kappa_2+1 \leq k\leq n-1,
\end{align*}
These schemes correspond respectively to the time grids $\Pi^0 \cup  \{ \kappa_1 T/n + k'T/n^2 , 1 \leq k' \leq n-1 \}$, $\Pi^0 \cup  \{ \kappa_2 T/n + k'T/n^2 , 1 \leq k' \leq n-1 \}$ and $\Pi^0 \cup  \{ \kappa_1 T/n + k'T/n^2 , 1 \leq k' \leq n-1 \}\cup  \{ \kappa_2 T/n + k'T/n^2 , 1 \leq k' \leq n-1 \}$. We then get \begin{align}
  \cPh^{3,n}f  =& \E[f(\hat{X}^{n,0}_T)] + n\E[f(\hat{X}^{n,1}_T) - f(\hat{X}^{n,0}_T)] +n^2 \E[f(\hat{X}^{n,2}_T) - f(\hat{X}^{n,1}_T)] \label{P3n} \\
  &+\frac{n(n-1)}{2}\E[f(\hat{X}^{n,5}_T) - f(\hat{X}^{n,4}_T)- f(\hat{X}^{n,3}_T)+ f(\hat{X}^{n,0}_T)]. \notag 
\end{align}

\subsection{Simulations result for the CIR process}\label{Sim_CIR}
In this subsection, we want to illustrate the convergence of the approximations $\cPh^{2,n}$ and $\cPh^{3,n}$, which together with the use of the second order scheme \eqref{NV_scheme} guarantee respectively approximations of order four and six by Theorem~\ref{thm_main}. In order to calculate these approximations, we use Monte-Carlo estimators of~\eqref{P2n_indep} and~\eqref{P3n}, using independent samples for each expectation. The number of samples (up to $10^{11}$)  is such that we can neglect the statistical error. In Figures \ref{CIR_Plot1}, \ref{CIR_Plot2} and \ref{CIR_Plot3} we plot the convergence in function of the time step for different parameters choices, taking advantage of the closed formula for the Laplace transform of the CIR process, see e.g.~\cite[Proposition 1.2.4]{AA_book}. The three numerical experiments test different levels of the ratio $\sigma^2 / 4a$ in decreasing order. We observe that the slopes estimated on the log-log plots are close to 2, 4 and 6 respectively, so that they are in accordance with Theorem~\ref{thm_main}. Note however that Theorem~\ref{thm_main} gives an asymptotic result for $n\to \infty$, while we are restricted here to rather small values of $n$ since we are using a large number of samples to kill the statistical error.  
In all the cases shown, the approximations of higher order outperform the one built with the simple second order scheme~\eqref{NV_scheme}. Talking about accuracies, the fourth order approximation for $n=3$ shows an absolute relative error of about $0.17\%$ in the tests in Figures \ref{CIR_Plot1}, and \ref{CIR_Plot2} and $0.02\%$ in the one in Figure \ref{CIR_Plot3}; the sixth order approximation  already for $n=3$ exhibits a relative error of $0.002\%$ in each case studied.

\begin{figure}[h]
  \centering
  \begin{subfigure}[h]{0.49\textwidth}
    \centering
    \includegraphics[width=\textwidth]{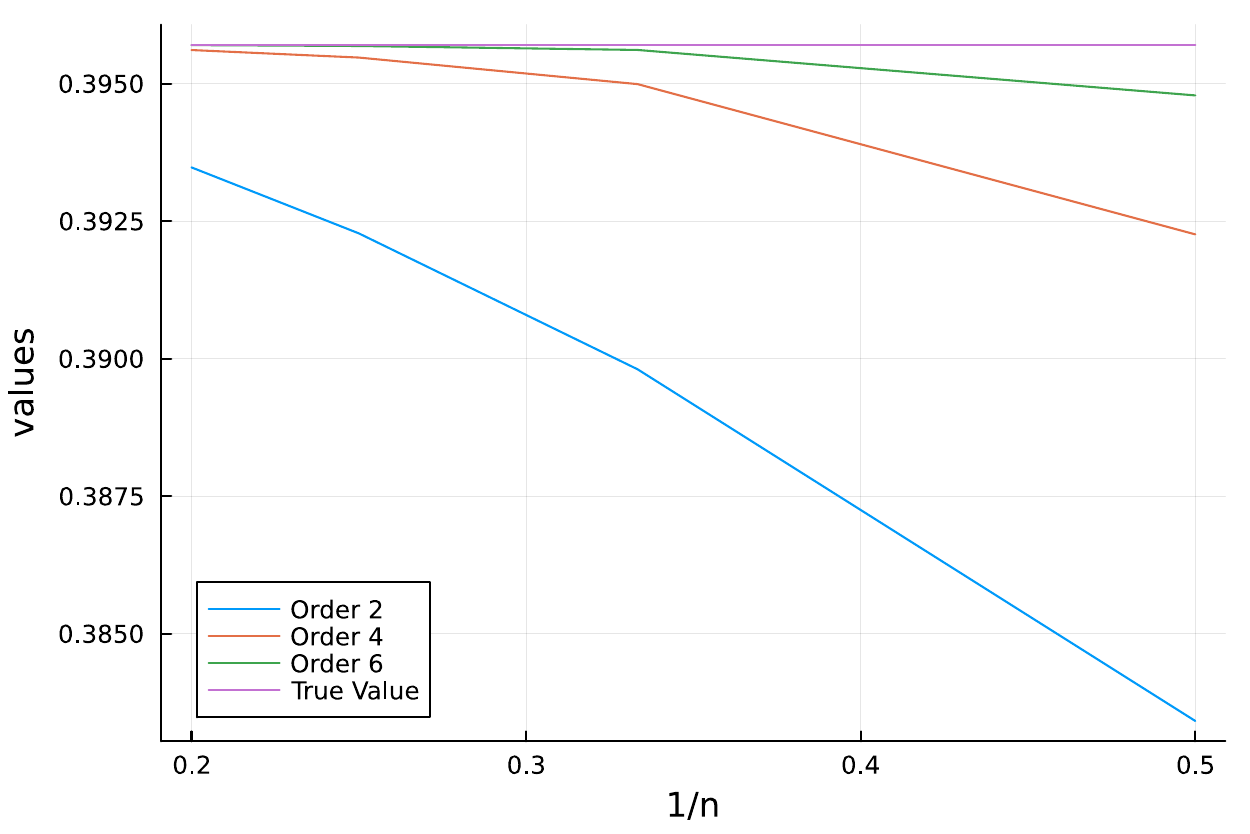}
    \caption{Values plot}
    \label{fig:values_plot13}
  \end{subfigure}
  \hfill
  \begin{subfigure}[h]{0.49\textwidth}
    \centering
    \includegraphics[width=\textwidth]{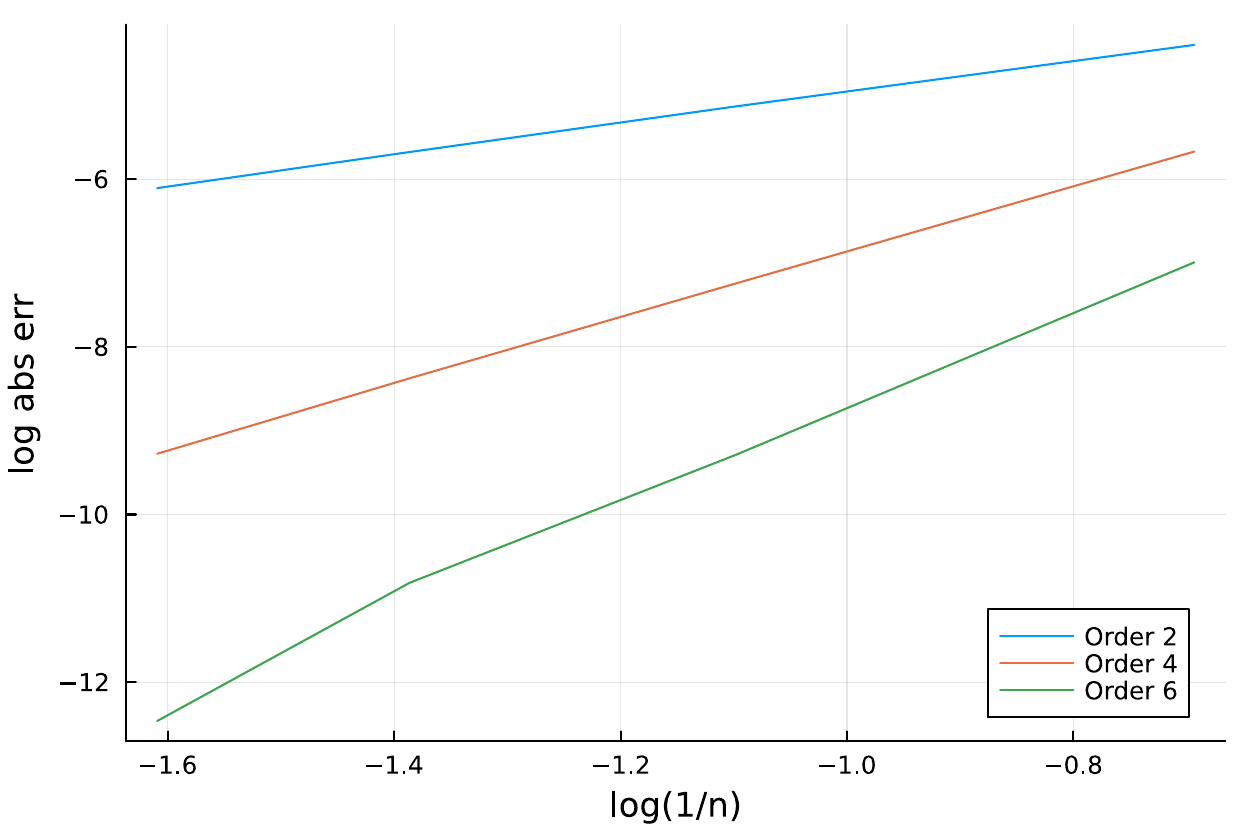}
    \caption{Log-log plot}
    \label{fig:log-log_plot13}
  \end{subfigure}
  \caption{Parameters: $x=0.0$, $a=0.2$, $k=0.5$, $\sigma=0.65$, $f(z)=\exp(-10 z)$ and $T=1$ ($\frac{\sigma^2}{2a}\approx 1.06$). Graphic~({\sc a}) shows the values of $\cPh^{1,n}f$, $\cPh^{2,n}f$, $\cPh^{3,n}f$ as a function of the time step $1/n$  and the exact value. Graphic~({\sc b}) draws $\log(|\hat{P}^{i,n}f-P_Tf|)$ in function of $\log(1/n)$: the regressed slopes are 1.86, 3.93 and 5.87 for the second, fourth and sixth order respectively.}\label{CIR_Plot1}
\end{figure}

\begin{figure}[h!]
  \centering
  \begin{subfigure}[h]{0.49\textwidth}
    \centering
    \includegraphics[width=\textwidth]{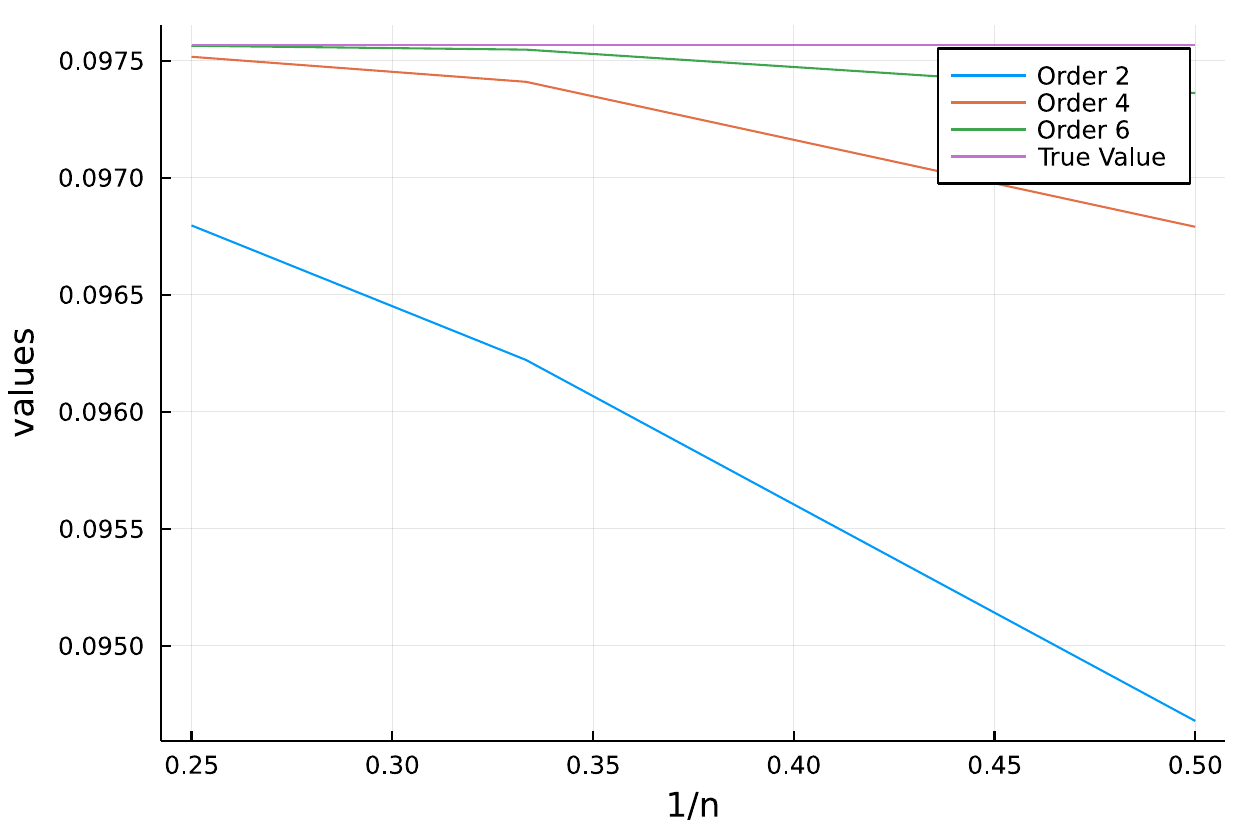}
    \caption{Values plot}
    \label{fig:values_plot9}
  \end{subfigure}
  \hfill
  \begin{subfigure}[h]{0.49\textwidth}
    \centering
    \includegraphics[width=\textwidth]{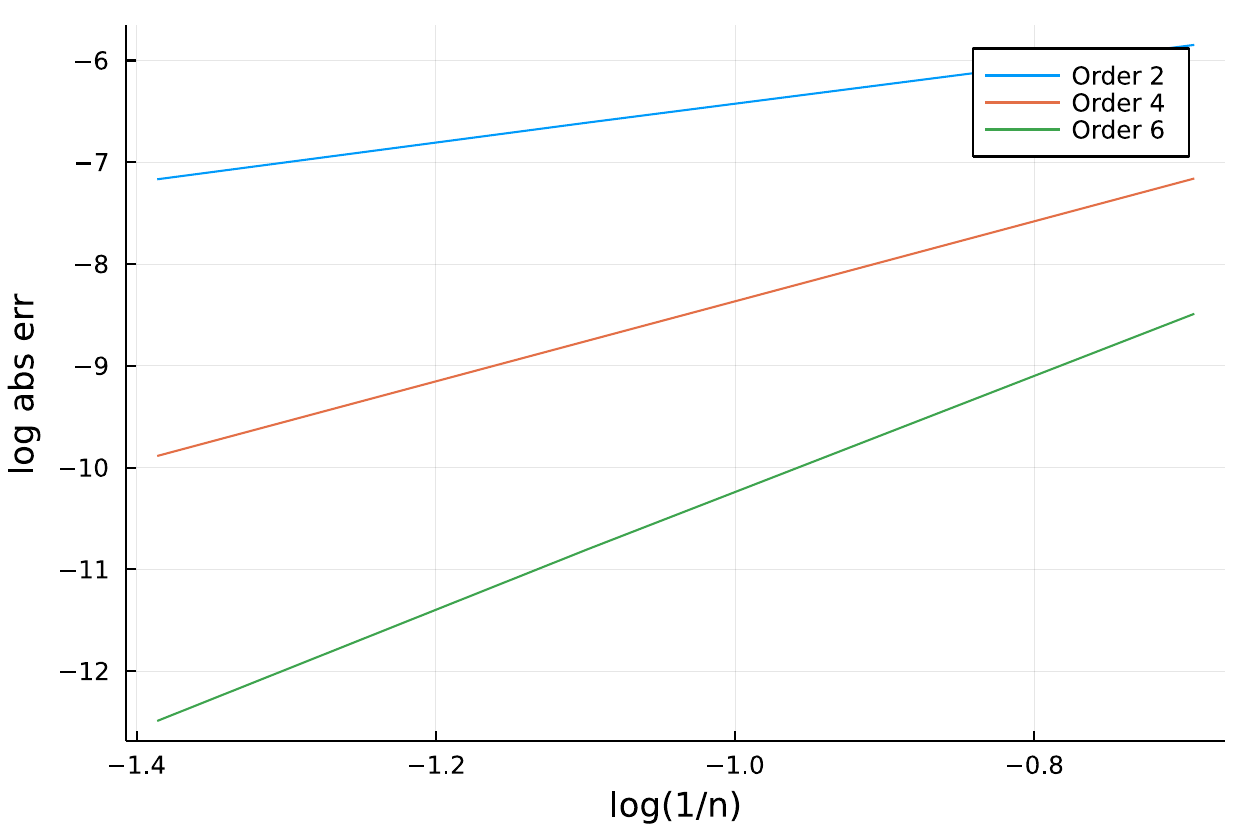}
    \caption{Log-log plot}
    \label{fig:log-log_plot9}
  \end{subfigure}
  \caption{Parameters: $x=0.3$, $a=0.4$, $k=1$, $\sigma=0.4$, $f(z)=\exp(-8 z)$ and $T=1$ ($\frac{\sigma^2}{2a}= 0.2$). Graphic~({\sc a}) shows the values of $\cPh^{1,n}f$, $\cPh^{2,n}f$, $\cPh^{3,n}f$ as a function of the time step $1/n$  and the exact value. Graphic~({\sc b}) draws $\log(|\hat{P}^{i,n}f-P_Tf|)$ in function of $\log(1/n)$:  the regressed slopes are 1.90, 3.93 and 5.77 for the second, fourth and sixth order respectively.}\label{CIR_Plot2}
\end{figure}

\begin{figure}[h]
  \centering
  \begin{subfigure}[h]{0.49\textwidth}
    \centering
    \includegraphics[width=\textwidth]{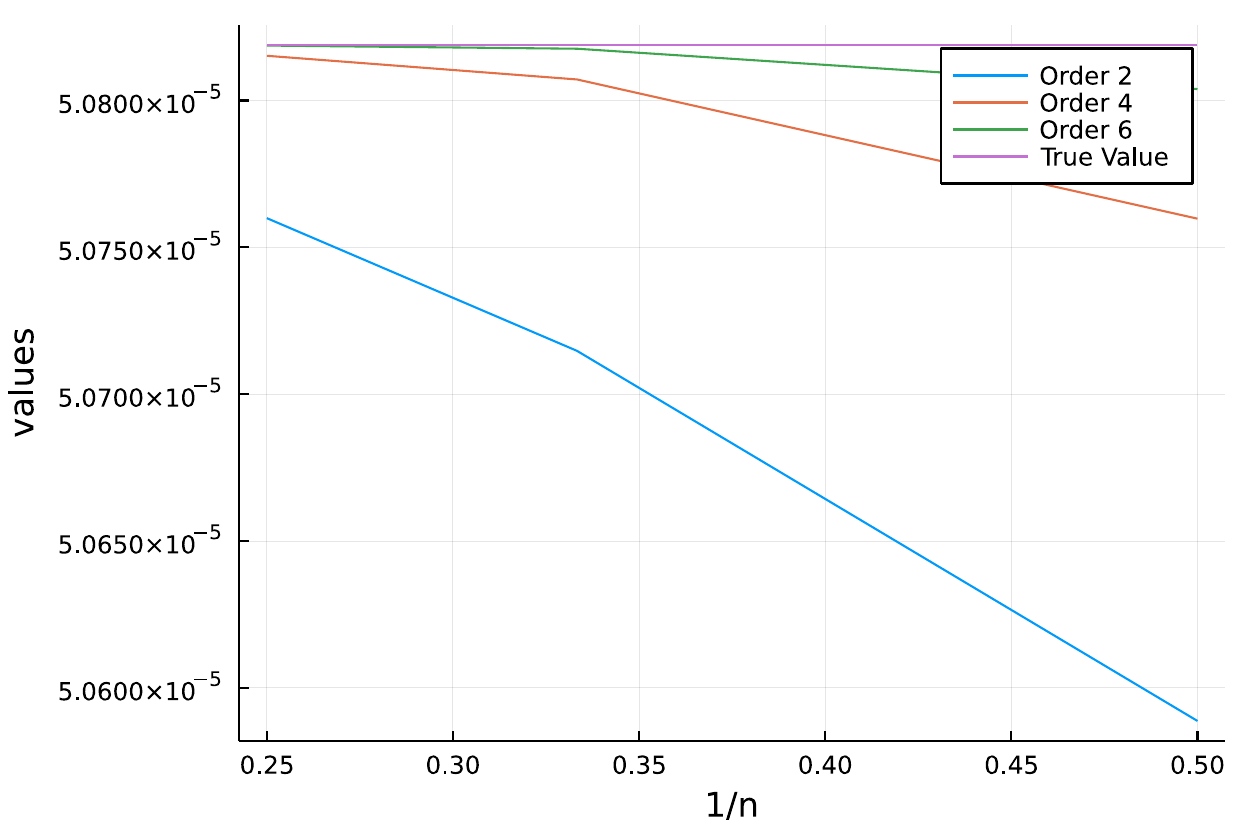}
    \caption{Values plot}
    \label{fig:values_plot15}
  \end{subfigure}
  \hfill
  \begin{subfigure}[h]{0.49\textwidth}
    \centering
    \includegraphics[width=\textwidth]{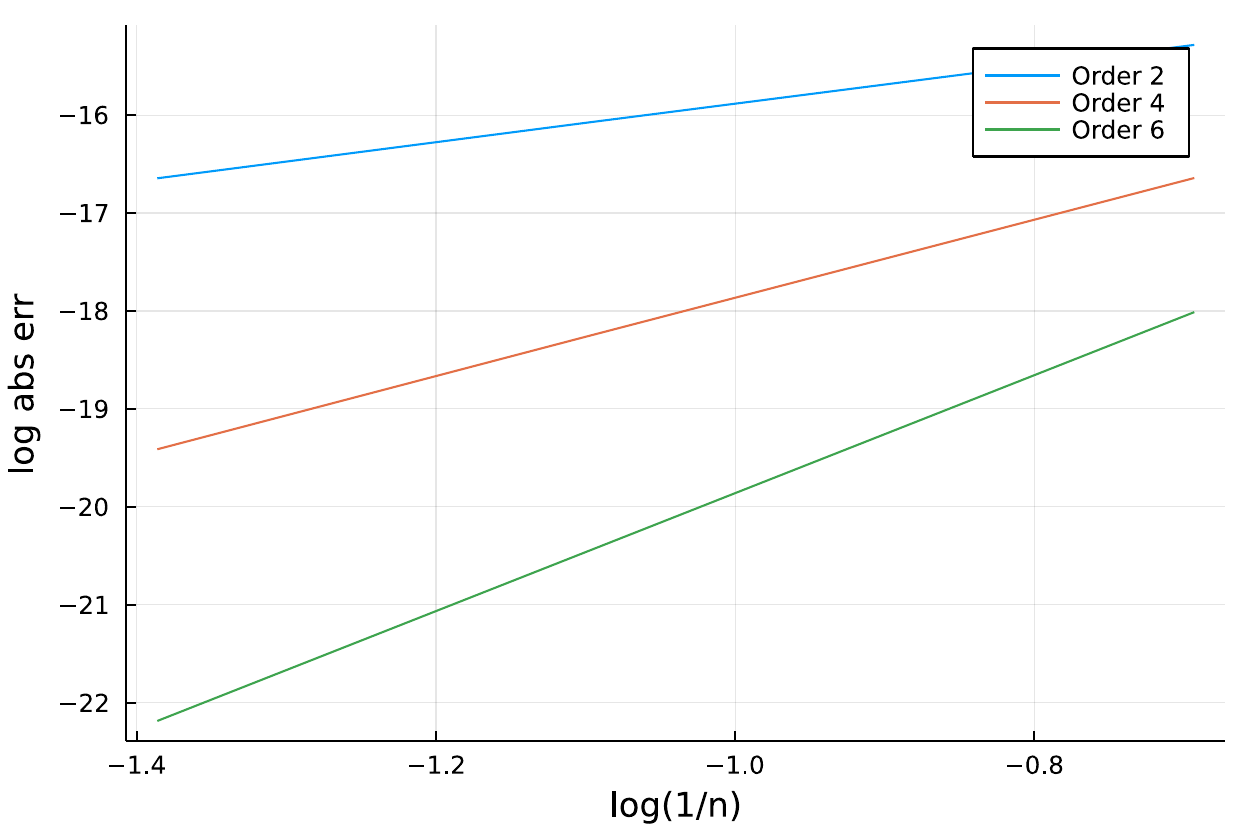}
    \caption{Log-log plot}
    \label{fig:log-log_plot15}
  \end{subfigure}
  \caption{Parameters: $x=10$, $a=10$, $k=1$, $\sigma=0.23$, $f(z)=\exp(- z)$ and $T=1$ ($\frac{\sigma^2}{2a}\approx 0.0026$).
  Graphic~({\sc a}) shows the values of $\cPh^{1,n}f$, $\cPh^{2,n}f$, $\cPh^{3,n}f$ as a function of the time step $1/n$  and the exact value. Graphic~({\sc b}) draws $\log(|\hat{P}^{i,n}f-P_Tf|)$ in function of $\log(1/n)$:  the regressed slopes are 1.96, 4.00 and 6.02 for the second, fourth and sixth order respectively.}\label{CIR_Plot3}
\end{figure}

\subsection{Simulations result for the Heston model}\label{Sim_Heston}

In this subsection, we want to test the second order scheme for the Heston model proposed by Alfonsi in~\cite{AA_MCOM} along with the approximations of order 4 and 6 obtained with combination of random grids. 
First, we recall the couple of stochastic differential equations describing this model
\begin{equation}\label{Heston SDEs}
  \begin{cases}
    dS^{(x,s)}_t = rS^{(x,s)}_t dt + \sqrt{X_t}S^{(x,s)}_t (\rho dW_t + \sqrt{1-\rho^2} dZ_t), \ S^{(x,s)}_0=s,\\
    dX^x_t = (a-kX^x_t) dt +\sigma \sqrt{X^x_t} dW_t, \ X^x_0=x,
  \end{cases}
\end{equation}
where $W$ and $Z$ are two independent Brownian motions. We define the two following random variables
\begin{align*}
  S_1\big((x,s),h,Z_h \big) & = \left(x, s\exp\Big(\sqrt{x(1-\rho^2)}Z_h\Big)                            \right)                                                                               \\
  S_2\big((x,s),h,W_h \big) & =\bigg( \varphi(x,h,W_h),\\& s\exp\left((r-\frac{\rho}{\sigma}a)h + (\frac{\rho}{\sigma}k-\frac{1}{2})\frac{x+\varphi(x,h,W_h)}{2}h + \frac{\rho}{\sigma}(\varphi(x,h,W_h)-x)\right)\bigg)
\end{align*}
where $\varphi$ is defined by~\eqref{def_varphi} anf corresponds to the second order scheme for the CIR process.
We define as in \cite{AA_MCOM} the second order scheme for~\eqref{Heston SDEs} as follows
\begin{equation}\label{H2S}
  \Phi\big((x,s),h,(W_h,Z_h),B\big) = \begin{cases} S_2\left(S_1\big((x,s),h,Z_h \big),h,W_h\right), \text{ if } B=1, \\
    S_1\left(S_2\big((x,s),h,W_h \big),h,Z_h\right), \text{ if } B=0,\end{cases} 
\end{equation}
where $B$ is an independent Bernoulli random variable of parameter 1/2.

To test the order of the approximations $\cPh^{2,n}$ and $\cPh^{3,n}$ boosting the second order scheme~\eqref{H2S}, we have calculated European put prices taking advantage of the existence of a semi closed formula for this option, see~\cite{Heston}. In Figure~\ref{Heston_orders} we draw the convergence in function of the time step. Again, we noticed that the slopes obtained on the log-log plot are in line with the expected order of convergence. 
More importantly, we see that the correction terms of the approximations $\cPh^{2,n}$ and $\cPh^{3,n}$ really improves the precision. They respectively give relative errors of a 0.035\% and 0.0023\%, already for $n=3$.

\begin{figure}[h!]
  \centering
  \begin{subfigure}[h]{0.49\textwidth}
    \centering
    \includegraphics[width=\textwidth]{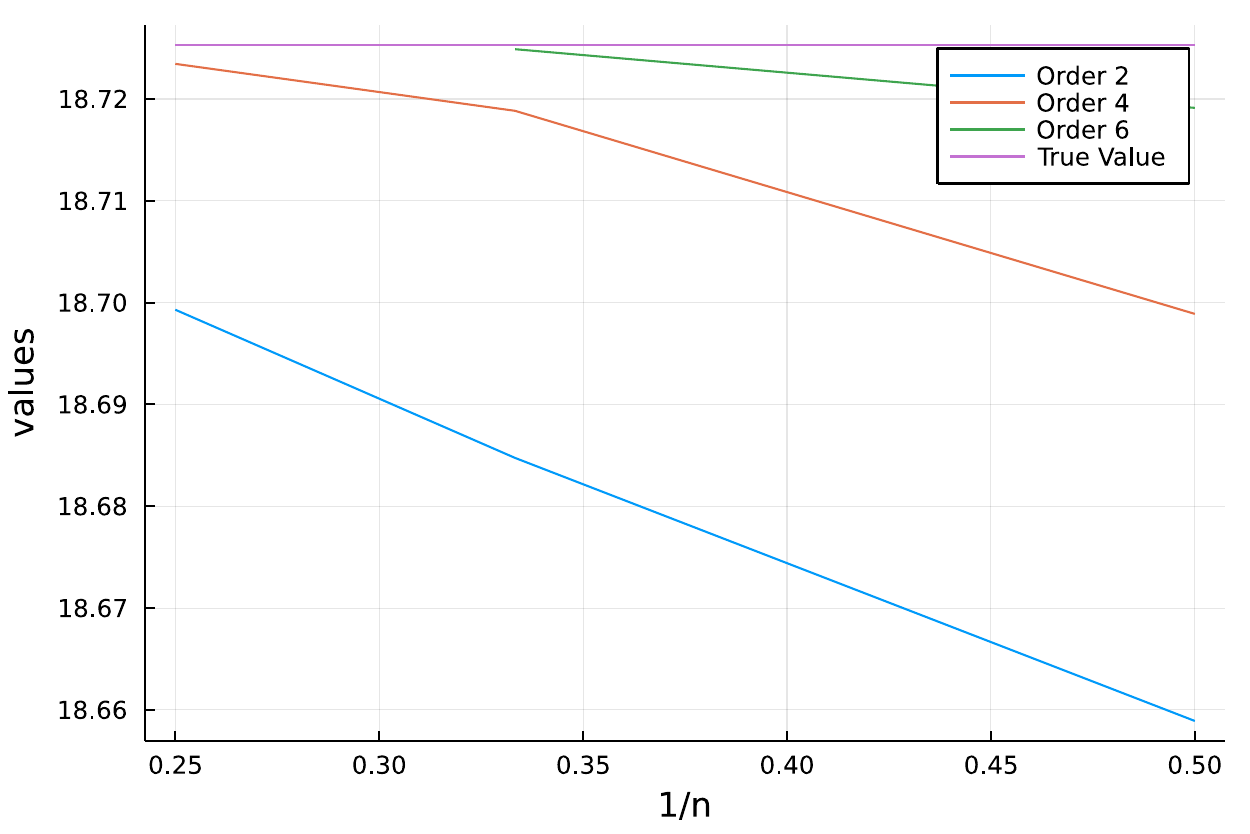}
    \caption{Values plot}
    \label{fig:values_plot_heston2}
  \end{subfigure}
  \hfill
  \begin{subfigure}[h]{0.49\textwidth}
    \centering
    \includegraphics[width=\textwidth]{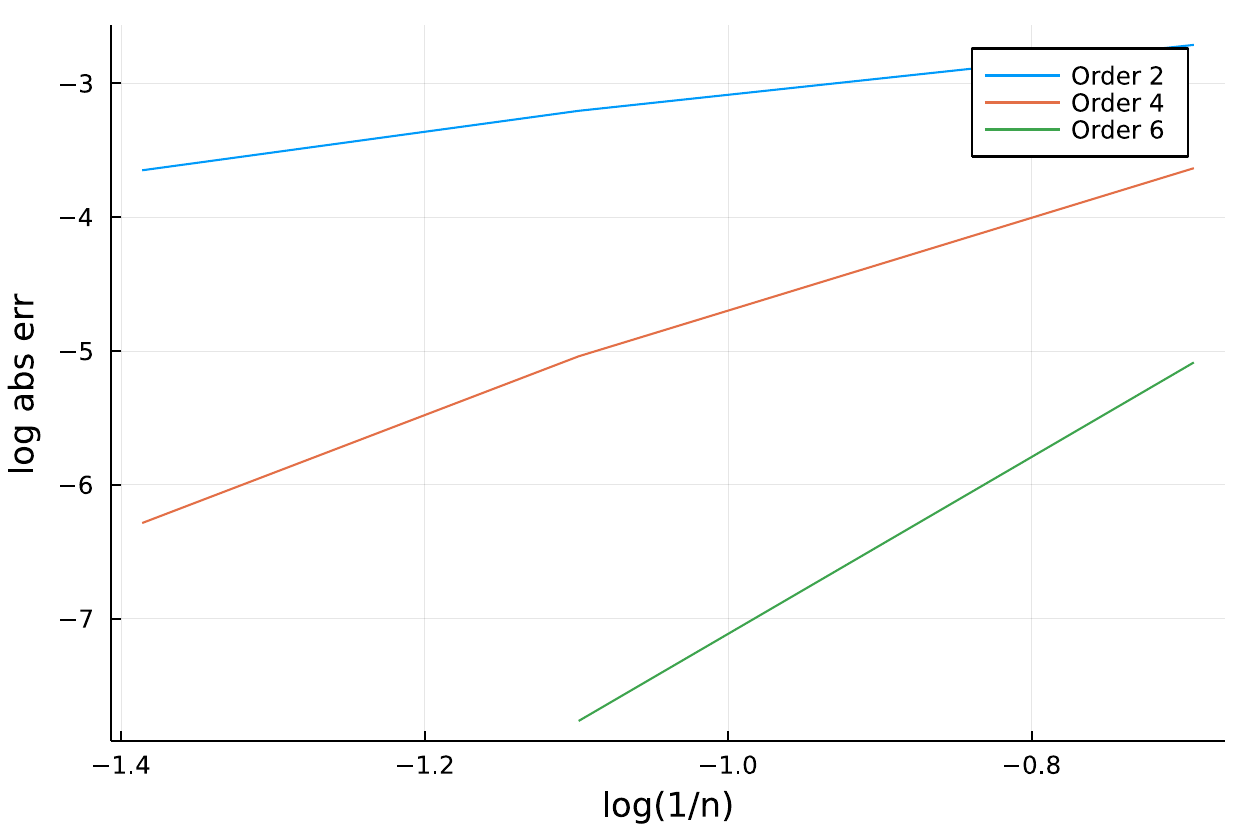}
    \caption{Log-log plot}
    \label{fig:log-log_plot_heston2}
  \end{subfigure}
  \caption{Test function: $f(x,s)=(K-s)^+$. Parameters: $S_0=100$, $r=0$, $x=0.25$, $a=0.25$, $k=1$, $\sigma=0.65$, $\rho=-0.3$, $T=1$, $K=100$ ($\frac{\sigma^2}{2a}= 0.845$).
  Graphic~({\sc a}) shows the values of $\cPh^{1,n}f$, $\cPh^{2,n}f$, $\cPh^{3,n}f$ as a function of the time step $1/n$  and the exact value. Graphic~({\sc b}) draws $\log(|\hat{P}^{i,n}f-P_Tf|)$ in function of $\log(1/n)$:  the regressed slopes are 1.34, 4.00 and 6.02 for the second, fourth and sixth order respectively.}\label{Heston_orders}
\end{figure}


\subsection{Optimized implementation of~$\cPh^{2,n}$}\label{Optim_time/var}

The approximations $\cPh^{2,n}$ and $\cPh^{3,n}$ defined respectively by~\eqref{P2n_indep} and~\eqref{P3n} involve respectively two and four expectations. The larger is $\nu$ the more  expectations are involved in $\cPh^{\nu,n}$. Thus,  for simplicity, independent samples were used  by Alfonsi and Bally~\cite{AB} to compute each term. However,  it may be interesting to reuse some samples in order to spare computation time. This is what we investigate in this subsection. 

Namely, Equation \eqref{P2n_indep} leads naturally to the  two following estimators  of $P^{2,n}f$:
\begin{align}
  \Theta_{\mathrm{I}}(M_1,M_2,n) &= \frac{1}{M_1} \sum_{j=1}^{M_1} f\big((\hat{X}^{n,0}_T)^{(j)} \big) +  \frac{1}{M_2} \sum_{i=M_1+1}^{M_1+M_2} n\left(f\big((\hat{X}^{n,1}_T)^{(i)}\big) -   f\big((\hat{X}^{n,0}_T)^{(i)}\big) \right), \label{Theta_i}\\
  \Theta_\mathrm{D}(M_1,M_2,n) &= \frac{1}{M_1} \sum_{j=1}^{M_1} f\big((\hat{X}^{n,0}_T)^{(j)} \big) +  \frac{1}{M_2} \sum_{i=1}^{M_2} n\left(f\big((\hat{X}^{n,1}_T)^{(i)}\big) -   f\big((\hat{X}^{n,0}_T)^{(i)}\big) \right).\label{Theta_d}
\end{align}
The first one takes  independent samples, and we call this estimator $\Theta_\mathrm{I}$. This approach is the one used in~\cite{AB}. In the second case, we reuse the first $M_1\wedge M_2$ simulations of $f\big((\hat{X}^{n,0}_T)^{(i)}\big)$ in both sums. We call this estimator $\Theta_\mathrm{D}$ to indicate the dependence between samples. In terms of variance, we have
\begin{align}
   & \Var\left(\Theta_\mathrm{I}(M_1,M_2,n)\right) = \frac{\Var\big(f(\hat{X}^{n,0}_T)\big)}{M_1} + \frac{\Var\big(n(f(\hat{X}^{n,1}_T) - f(\hat{X}^{n,0}_T))\big)}{M_2}, \\
   & \begin{multlined}
    \Var\left(\Theta_\mathrm{D}(M_1,M_2,n)\right) = \frac{\Var\big(f(\hat{X}^{n,0}_T)\big) }{M_1}+ 2 \frac{\Cov\big(f(\hat{X}^{n,0}_T),n(f(\hat{X}^{n,1}_T) - f(\hat{X}^{n,0}_T))\big) }{M_1\vee M_2} \\
    + \frac{\Var\big(n(f(\hat{X}^{n,1}_T) - f(\hat{X}^{n,0}_T))\big)}{M_2}.
  \end{multlined}\label{var_thetad}
\end{align}
Let us define $t_1$ as the time to generate one sample $f(\hat{X}^{n,0}_T)$ and $t_2$ as the one needed for one sample of the correction $n(f(\hat{X}^{n,1}_T) - f(\hat{X}^{n,0}_T))$. The computation time needed to compute $\Theta_\mathrm{I}$ is given by $g_i(M_1,M_2)=M_1t_1+M_2t_2$, while the one needed to compute $\Theta_\mathrm{D}$ is $g_d(M_1,M_2)=\mathbf{1}_{M_1\ge M_2} [(M_1-M_2)t_1+M_2t_2] +\mathbf{1}_{M_1< M_2}M_2t_2$. We note $\zeta=\frac{t_2}{t_1}$. From the definition of schemes $\hat{X}^{n,0}$ and $\hat{X}^{n,1}$ in~\eqref{Scheme0}, we observe that  $2\le \zeta \le 3 $ and that $\zeta \approx 2.5$ in average since these schemes are equal up to $\kappa h_1$.  
The advantage of $\Theta_d$ is not necessarily in reducing the variance, but in decreasing the number of simulations needed, making it more efficient from a computational time point of view.\\

We want to find the optimal numbers of simulations $M_1$ and $M_2$ for our estimators in order to minimize the execution time for a given variance $\varepsilon^2$. Let us define 
$\sigma^2_2(n) = \Var\big(f(\hat{X}^{n,0}_T)\big)$, $\sigma^2_{4}(n) = \Var\big(n(f(\hat{X}^{n,1}_T) - f(\hat{X}^{n,0}_T))\big)$, $\Gamma(n)=\Cov\big(f(\hat{X}^{n,0}_T), n(f(\hat{X}^{n,1}_T) - f(\hat{X}^{n,0}_T))\big)$. For~$\Theta_\mathrm{I}$, the minimization of $g_i$ given that $\sigma^2_2(n)/M_1+\sigma^4_2(n)/M_2=\varepsilon^2$ leads to $M_1=\sqrt{\zeta}\frac{\sigma_2(n)}{\sigma_4(n)}M_2$ and then to:
\begin{equation}\label{optim_i}
M_{1,\mathrm{I}}=\left\lceil \frac{1}{\varepsilon^2}\left( \sigma^2_2(n) +\sqrt{\zeta} \sigma_2(n)\sigma_4(n)   \right)  \right\rceil, \ M_{2,\mathrm{I}}= \left\lceil \frac{1}{\varepsilon^2}\left( \sigma^2_4(n) +\frac{\sigma_2(n)\sigma_4(n)}{\sqrt{\zeta}}   \right)  \right\rceil.
\end{equation}
To minimize the execution time $g_d$, one has first to decide whether we take $M_1\ge M_2$ or $M_1<M_2$. From~\eqref{var_thetad}, this amounts to compare $\frac{\sigma^2_2(n)+2\Gamma(n)}{m+\tilde{m}\zeta}$ with $\frac{\sigma^2_4(n)+2\Gamma(n)}{m+\tilde{m}}$ where $m=M_1\wedge M_2$ and $\tilde{m}\ge 0$ ($\tilde{m}$ simulations of the correction term takes the same time as $\zeta \tilde{m}$ simulations of~$f(\hat{X}^{n,0}_T)$). Taking the derivative at~$\tilde{m}=0$, we get that $M_1\ge M_2$ if $\zeta\frac{\sigma^2_2(n)+2\Gamma(n)}{\sigma^2_4(n)+2\Gamma(n)}\ge 1$, and $M_1<M_2$ otherwise. When $M_1\ge M_2$, the minimization of $g_d$ given $ \Var\left(\Theta_d(M_1,M_2,n)\right) =\varepsilon^2$ leads to
\begin{align}
 \begin{cases} 
	M_{1,\mathrm{D}}  &=\left\lceil \frac{1}{\varepsilon^2}\left( \sigma^2_2(n) + 2\Gamma(n) +\sqrt{\big(\sigma^2_2(n)+2\Gamma(n)\big)\sigma^2_4(n)(\zeta-1) }  \right)  \right\rceil, \\ 
	M_{2,\mathrm{D}} & =\left\lceil \frac{1}{\varepsilon^2}\left( \sigma^2_4(n) +\sqrt{\frac{\big(\sigma^2_2(n)+2\Gamma(n)\big)\sigma^2_4(n)}{\zeta-1} }  \right)  \right\rceil.\end{cases} \label{optim_d}
\end{align}
We have similar formulas when $M_1<M_2$. In all our numerical experiments below, we are in the case where  $\zeta\frac{\sigma^2_2(n)+2\Gamma(n)}{\sigma^2_4(n)+2\Gamma(n)}\ge 1$ and thus taking $M_1\ge M_2$ is optimal.

Now, we show the performance of the two estimators~\eqref{Theta_i} and \eqref{Theta_d}. To do this, we calculate the empirical variances $\sigma_2^2(n)$, $\sigma_4^2(n)$ and the  empirical covariance $\Gamma(n)$ on a small sampling, fix a desired precision $\varepsilon=1.96\sqrt{\Var(\Theta(M_1,M_2,n))}$ for both the estimators, so that all the terms have roughly the same statistical error with a 95\% confidence interval half-width equal to $\varepsilon$. We show two tables in which we set the precision $\varepsilon$ to $10^{-3}$. In Table~\ref{Table_T/V_I_VS_D_1}, we have   $\sigma_2^2(n) \gg \sigma_4^2(n) $, while in Table~\ref{Table_T/V_I_VS_D_2}, $\sigma_2^2(n)$ is still larger than $\sigma_4^2(n)$, but of the same order of magnitude.

\begin{table}[h!]
	\centering
  \begin{tabular}{ c|c|c|c|c|  }
    \cline{2-5}
        & $n=2$     & $n=3$     & $n=4$      & $n=5$      \\
    \hline
    \multicolumn{1}{|c|}{$\Theta_\mathrm{I}$}  & 63.04 & 96.15 & 131.84 & 165.80 \\
    \multicolumn{1}{|c|}{$\Theta_\mathrm{D}$}  & 51.61 & 87.24 & 122.76 & 152.32 \\
    \hline
  \end{tabular}
  \caption{ Computation time (in seconds) needed by the Estimators $\Theta_i$ and $\Theta_d$ for a precision $\varepsilon=10^{-3}$.
  Test function: $f(x,s)=(K-s)^+$. Parameters: $S_0=100$, $r=0$, $x=0.4$, $a=0.4$, $k=1$, $\sigma=0.2$, $\rho=-0.3$, $T=1$, $K=100$ ($\frac{\sigma^2}{2a}=0.05$)}.\label{Table_T/V_I_VS_D_1}
\end{table}

\begin{table}[h!]
	\centering
  \begin{tabular}{ c|c|c|c|c|  }
    \cline{2-5}
        & $n=2$     & $n=3$     & $n=4$      & $n=5$      \\
    \hline
    \multicolumn{1}{|c|}{$\Theta_\mathrm{I}$}  & 59.50 & 102.13 & 148.45 & 193.41 \\
    \multicolumn{1}{|c|}{$\Theta_\mathrm{D}$}  & 37.59 & 70.43  & 100.14 & 136.16 \\	
    \hline
  \end{tabular}
  \caption{Computation time (in seconds) needed by the Estimators $\Theta_i$ and $\Theta_d$ for a precision $\varepsilon=10^{-3}$.  Test function: $f(x,s)=(K-s)^+$. Parameters: $S_0=100$, $r=0$, $x=0.1$, $a=0.1$, $k=1$, $\sigma=0.63$, $\rho=-0.3$, $T=1$, $K=100$ ($\frac{\sigma^2}{2a}\approx 1.98$)}.\label{Table_T/V_I_VS_D_2}
\end{table}

We observe that we do not have a great gain in using $\Theta_i$ when $\sigma_2^2(n)   \gg \sigma_4^2(n) $ (Table \ref{Table_T/V_I_VS_D_1}), while we save up to $30\%$ of execution time, using $\Theta_d$ instead of $\Theta_i$, when $\sigma_2^2 (n) $ is of the same order of magnitude $\sigma_4^2 (n)$ (Table \ref{Table_T/V_I_VS_D_2}).  Heuristically, this can be understood as follows: when $\sigma^2_2(n)$ is of the same magnitude as~$\sigma^2_4(n)$, so are $M_{1,\i}$ and $M_{2,\i}$, which gives an important gain in reusing the simulation of the correction term.  
In any case, $\Theta_d$ turns out to be faster for each choice of parameters, and therefore we recommend it at the expense of $ \Theta_i$.


\subsection{Comparison between the second and the fourth order approximation}
Subsections~\ref{Sim_CIR} and~\ref{Sim_Heston} have confirmed  numerically the theoretical results obtained in this paper. However, they do not compare directly the computation time required by the different methods. We now present numerical tests that allow us to prove  the real advantage of using the fourth order approximation $\cPh^{2,n}$ instead of the simple second order scheme. Namely, we compare the squared $L^2$ distance of the estimator $\Theta_d$ from the true value with the same distance between the estimator of $\hat{P}^{1, n^2}$ with the true value. We plot these quantities in function of the computation time needed. Note that $\hat{P}^{2, n}$ and $\hat{P}^{1, n^2}$ converges at a  rate of $O(n^{-4})$ so that their bias have the same order of magnitude.

\begin{figure}[h!]
  \centering
  \begin{subfigure}[h]{0.49\textwidth}
    \centering
    \includegraphics[width=\textwidth]{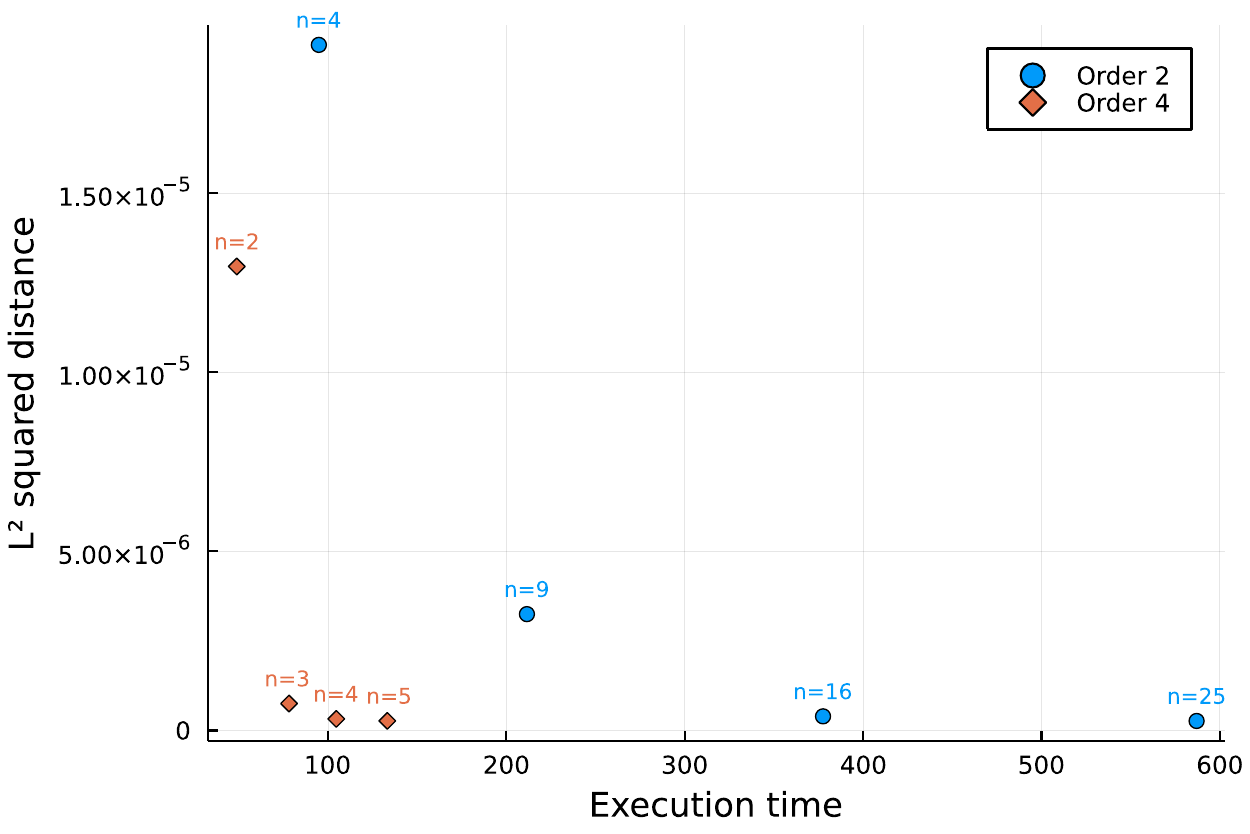}
    \caption{($\frac{\sigma^2}{2a}=0.0125$)} 
    \label{fig: L2distance_H4S_2M_Dep_VS_H2S6}
  \end{subfigure}
  \hfill
  \begin{subfigure}[h]{0.49\textwidth}
    \centering
    \includegraphics[width=\textwidth]{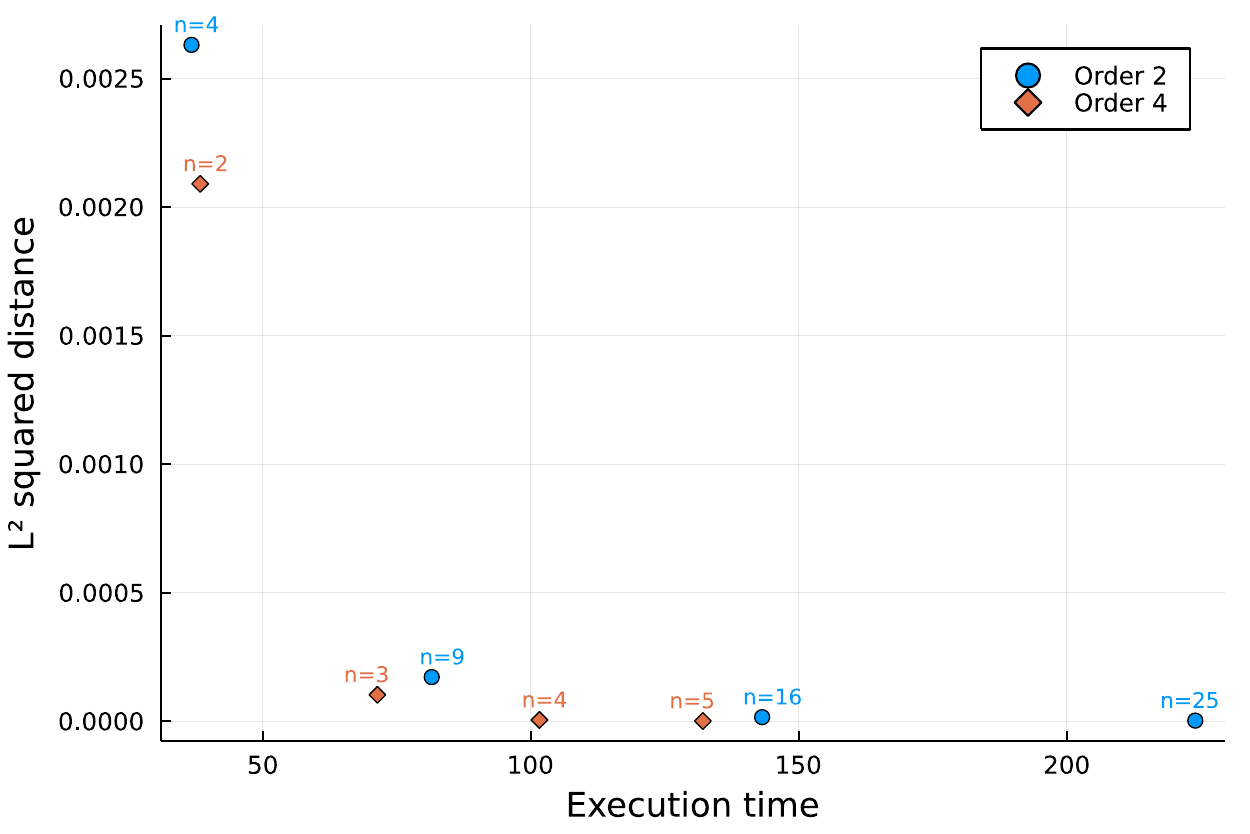}
    \caption{($\frac{\sigma^2}{2a}\approx 1.98$)} 
    \label{fig: L2distance_H4S_2M_Dep_VS_H2S13}
  \end{subfigure}
  \caption{
    $L^2$-square error in function of the execution time in seconds. Test function: $f(x,s)=(K-s)^+$. Parameters in graphic~({\sc a}) : $S_0=100$, $r=0$, $x=0.4$, $a=0.4$, $k=1$, $\sigma=0.1$, $\rho=-0.3$, $T=1$, $K=100$.\\Parameters in graphic~({\sc b}) : 
   $S_0=100$, $r=0$, $x=0.1$, $a=0.1$, $k=1$, $\sigma=0.63$, $\rho=-0.3$, $T=1$, $K=100$.}\label{H4S_dep_VS_H2S} 
\end{figure}

Figure~\ref{H4S_dep_VS_H2S} shows the results for the calculation of the price of a European put option in the Heston model with two different sets of parameters. In this numerical experience we set a precision $\varepsilon$ equal to $10^{-3}$. The empirical evidences show that the fourth order estimator $\Theta_d$ is the best choice, especially when the ratio $\frac{\sigma^2}{2a}\ll 1$ (Figure \ref{H4S_dep_VS_H2S} ({\sc a})) where the performance of the fourth order estimator is unparalleled. For example, $\hat{P}^{2, 3}$ is twice more accurate and more than twice faster than $\hat{P}^{1,9}$.
Even in Figure \ref{H4S_dep_VS_H2S} ({\sc b}), where  the ratio $\frac{\sigma^2}{2a}$ is larger and close to~2,  the fourth order estimator~$\Theta_d$ is more precise than the second order estimator and is faster from $n=3$ onward. These experiments illustrate the outperformance of the boosted estimator $\hat{P}^{2, n}$ with respect to $\hat{P}^{1, n}$.


\subsection{Numerical experiments for $\sigma^2>4a$}
In the previous subsections, we have presented analyzes to confirm numerically the theoretical rates of convergence of our approximations, and to assess their computational time. This is why we have only considered parameters such that $\sigma^2\le 4a$, since this condition is required in Theorem~\ref{thm_main}. However, it is possible to test numerically the relevance of the boosting technique using random grids when $\sigma^2> 4a$. This is the purpose of this subsection. We first present the different schemes and then analyze numerically the variance of the correcting term. Then, we present the numerical bias of the approximation $\hat{P}^{2, n}$ for the CIR and Heston models.

\subsubsection{The approximation schemes}
In order to perform the numerical tests for $\sigma^2>4a$, we consider two different second order schemes for the CIR process. The first one is the second order scheme~\eqref{Alfonsi_scheme} presented in \cite{AA_MCOM}. More precisely, we define
\begin{equation}\label{Alfonsi_scheme_detailed}
  \varphi_A(x,t,\sqrt{t}N) = \varphi_A^u(x,t,\sqrt{t}N)\mathds{1}_{x \ge K^A_2(t)}  + \varphi_A^d(x,t,\sqrt{t}N)\mathds{1}_{x<K^A_2(t)},
\end{equation}
with 
\begin{align*}
  \varphi_A^u(x,t,\sqrt{t}N) &=\varphi(x,t,-\sqrt{3t})\mathds{1}_{N<\mcN^{-1}(1/6)}  + \varphi(x,t,0)\mathds{1}_{\mcN^{-1}(1/6)\le N<\Phi_N^{-1}(5/6)} \nonumber \\
  &+\varphi(x,t,\sqrt{3t})\mathds{1}_{N\ge\mcN^{-1}(5/6)}, \\
  \varphi_A^d(x,t,\sqrt{t}N) &=\frac{\E[X^x_t]}{2(1-\pi(t,x))}\mathds{1}_{N<\mcN^{-1}(1-\pi(t,x))}  + \frac{\E[X^x_t]}{2\pi(t,x)}\mathds{1}_{ N\ge\mcN^{-1}(1-\pi(t,x))} ,
\end{align*}
where $\mcN$ is the  cumulative distribution function of the standard normal distribution, $\pi(t,x)=\frac{1-\sqrt{1-\frac{\E[X^x_t]^2}{\E[(X^x_t)^2]}}}{2}$ and $K^A_2(t)$ is the function given by~\eqref{threshold_gen} with $A_Y=\sqrt{3}$. Here, we have written the scheme $\varphi_A$ as a function of the starting point $x$, the time step $t$ and the Brownian increment $\sqrt{t} N$. When computing $n\E[\left(f(\hat{X}^{n,1}_T) - f(\hat{X}^{n,0}_T)\right)]$ by Monte-Carlo, we use the same Brownian path to sample $\hat{X}^{n,0}_T$ and $\hat{X}^{n,1}_T$, as explained at the beginning of Section~\ref{Simulations}. Thus, there is a strong dependence between these schemes.

We present also another scheme that corresponds to other choices of $Y$ and $\hat{X}^{x,d}$ in~\eqref{Alfonsi_scheme}. We use a distribution that is pretty similar to a Gaussian distribution over the threshold, and a scaled beta distribution below. Thus, we define
\begin{equation}\label{def_sch_B}
  \varphi_B(x,t,\sqrt{t}N) = \varphi_B^u(x,t,\sqrt{t}N)\mathds{1}_{x\ge K^B_2(t)}  + \varphi_B^d(x,t,\sqrt{t}N)\mathds{1}_{x<K^B_2(t)},
\end{equation}
with 
\begin{align*}
  \varphi_B^u(x,t,\sqrt{t}N) & =\varphi(x,t,-z_2)\mathds{1}_{N\le-c_2} + \varphi(x,t,-z_1)\mathds{1}_{-c_2<N\le-c_1}  + \varphi(x,t,N)\mathds{1}_{-c_1\le N<c_1}  \nonumber \\
  &+\varphi(x,t,z_1)\mathds{1}_{c_1<N\le c_2} + \varphi(x,t,z_2)\mathds{1}_{N>c_2}, \\
  \varphi_B^d(x,t,\sqrt{t}N) & = \frac{\E[X^x_t]}{2\pi(t,x)} (\mcN(N))^{\frac{1}{2\pi(t,x)}-1},
\end{align*}
where $z_1=2.7523451704710586$, $z_2 = 3.5$, $c_1=2.58$, $c_2= 3.106520327375868$, and  $K^B_2(t)$ is the function given by~\eqref{threshold_gen} with $A_Y=3.5$. Here, we have fixed the values of $c_1$ and $z_2$, and we have numerically calculated $c_2$ and $z_1$ to have $\E[Y^2]=\E[N^2]$ and $\E[Y^4]=\E[N^4]$ with 
$$Y= -z_2 \mathds{1}_{N\le-c_2} -z_1 \mathds{1}_{-c_2<N\le -c_1} +N \mathds{1}_{-c_1<N\le c_1}+z_1 \mathds{1}_{c_1<N\le c_2}+ z_2\mathds{1}_{c_2<N}. $$
The random variable $\varphi_B^d(x,t,\sqrt{t}N)$ has the same two first moments as $X^x_t$, and we can prove following the same arguments as~\cite[Theorem 2.8]{AA_MCOM} that $\varphi_B(x,t,\sqrt{t}N)$ is a second order scheme for the weak error.

\subsubsection{Numerical study of the variance of the correcting term $n\left(f(\hat{X}^{n,1}_T) - f(\hat{X}^{n,0}_T)\right)$}
We now analyze the variance of the corrections terms of the correcting term $n\left(f(\hat{X}^{n,1}_T) - f(\hat{X}^{n,0}_T)\right)$ in function of the number $n$ of discretization steps, when we use the different schemes~\eqref{Alfonsi_scheme_detailed} and~\eqref{def_sch_B}. We start with an example with $\sigma^2<4a$ for which $\varphi$ is still defined and $\varphi_A$ (resp. $\varphi_B$) does not use the auxiliary scheme $\varphi_A^d$ (resp. $\varphi_B^d$) since $K_2^{A}(t)=K_2^{B}(t)=0$ in this case. We observe in Table \ref{Table_VAR_CIR1} that the scheme~$\varphi_A$ leads to a value of $\Var(n(f(\hat{X}^{n,1}_T) - f(\hat{X}^{n,0}_T)))$ that is more than 20 times as large as that the one obtained using $\varphi$. Besides, the variance given by the scheme  $\varphi_A$ increases quite linearly with $n$, while the one obtained with $\varphi$ seems to be bounded and to decrease with $n$. One heuristic explanation for this is that  $\varphi_A$ is discrete scheme, which increases the strong error between the scheme on the fine grid~$\Pi^1$ and the scheme on the coarse grid~$\Pi^0$. Considering the scheme $\varphi_B$ that mixes Gaussian and discrete distributions leads to a much smaller variance that is rather close to the one of the scheme $\varphi$. However, as $n$ gets large, we see that the variance does not decrease in contrast to the scheme $\varphi$.
\begin{table}[h!]
	\centering
  \begin{tabular}{ c c|c|c|c|c|c|  }
    \cline{3-7}
        & & $n=2$     & $n=4$      & $n=8$       & $n=16$    & $n=32$    \\
    \hline
    \multicolumn{1}{|c|}{\multirow{2}{*}{$\varphi$} } & $\sigma^2_{4}(n)$    & 23.86e-4 & 17.43e-4 & 9.35e-4 & 4.85e-4  & 2.49e-4 \\
    \multicolumn{1}{|c|}{}& 95\% prec.& 3.2e-6 & 3.7e-6 & 2.8e-6 & 2.1e-6 & 1.5e-6 \\ 
    \hline
    \multicolumn{1}{|c|}{\multirow{2}{*}{$\varphi_A$}} & $\sigma^2_{4}(n)$     & 4.807e-2 & 10.870e-2 & 22.493e-2 & 45.437e-2 & 91.219e-2 \\ 
    \multicolumn{1}{|c|}{}& 95\% prec.     & 2.4-5 & 5.2e-5 & 11.1e-5 & 22.9e-5 & 46.3e-5 \\
    \hline
    \multicolumn{1}{|c|}{\multirow{2}{*}{$\varphi_B$} } & $\sigma^2_{4}(n)$    & 24.17e-4 & 18.37e-4 & 11.78e-4 & 10.27e-4  & 13.85e-4 \\
    \multicolumn{1}{|c|}{}& 95\% prec.& 3.2e-6 & 3.7e-6 & 2.9e-6 & 3.0e-6 & 4.5e-6 \\ 
    \hline
  \end{tabular}
  \caption{$\sigma^2_{4}(n) = \Var\big(n(f(\hat{X}^{n,1}_T) - f(\hat{X}^{n,0}_T))\big)$ for the different schemes, with $10^8$ samples and 95\% confidence interval precision.
  Test function: $f(x)=\exp(-10x)$. Parameters: $x=0.2$, $a=0.2$, $k=0.5$, $\sigma=0.5$, $T=1$ ($\frac{\sigma^2}{2a}=0.625$).}\label{Table_VAR_CIR1}
\end{table}

We now consider a case with $\sigma^2>4a$ so that the schemes $\varphi_A$ and $\varphi_B$ switch around their threshold. The scheme $\varphi$ is no longer defined. 
In Table~\ref{Table_VAR_CIR2}, we observe a huge increase of the variance in time steps with respect to Table~\ref{Table_VAR_CIR1}. We now observe that the variances grow almost linearly with respect to~$n$. Again, this can be explained heuristically by the switching that increases the strong error between the schemes on the fine grid~$\Pi^1$ and the coarse grid~$\Pi^0$. The rather high values of the variance obtained with the scheme $\varphi_A$ makes the boosting technique using random grids less interesting in practice from a computational point of view. In contrast, the scheme $\varphi_B$ produces much lower variances and the Monte-Carlo estimator of $\cPh^{2,n}f$ is more competitive. 
\begin{table}[h!]
	\centering
  \begin{tabular}{ c c|c|c|c|c|c|  }
    \cline{3-7}
      & & $n=2$     & $n=4$      & $n=8$       & $n=16$    & $n=32$    \\
      \hline
      \multicolumn{1}{|c|}{\multirow{2}{*}{$\varphi_A$} } 
      & $\sigma^2_{4}(n)$   & 0.0927 & 0.8742 & 2.7966 & 7.9095  & 21.6793 \\ 
      \multicolumn{1}{|c|}{}& 95\% prec.& 5.3e-5 & 3.3e-4 & 1.6e-3 & 6.1e-3 & 2.1e-2 \\ 
      \hline
      \multicolumn{1}{|c|}{\multirow{2}{*}{$\varphi_B$} } 
      & $\sigma^2_{4}(n)$   & 0.0757 & 0.2184 & 0.5145 & 1.1892  & 2.6600 \\ 
      \multicolumn{1}{|c|}{}& 95\% prec.& 6.4e-5 & 1.8e-4 & 5.5e-4 & 1.9e-3 & 6.2e-3 \\ 
      \hline
  \end{tabular}
  \caption{$\sigma^2_{4}(n) = \Var\big(n(f(\hat{X}^{n,1}_T) - f(\hat{X}^{n,0}_T))\big)$ with $10^8$ samples and 95\% confidence interval precision.
  Test function: $f(x)=\exp(-10x)$. Parameters: $x=0.2$, $a=0.2$, $k=0.5$, $\sigma=1.5$, $T=1$ ($\frac{\sigma^2}{2a}=5.625$).}\label{Table_VAR_CIR2}
\end{table}

\subsubsection{Numerical Convergence for the CIR}
\begin{figure}[h!]
  \centering
  \begin{subfigure}[h]{0.49\textwidth}
    \centering
    \includegraphics[width=\textwidth]{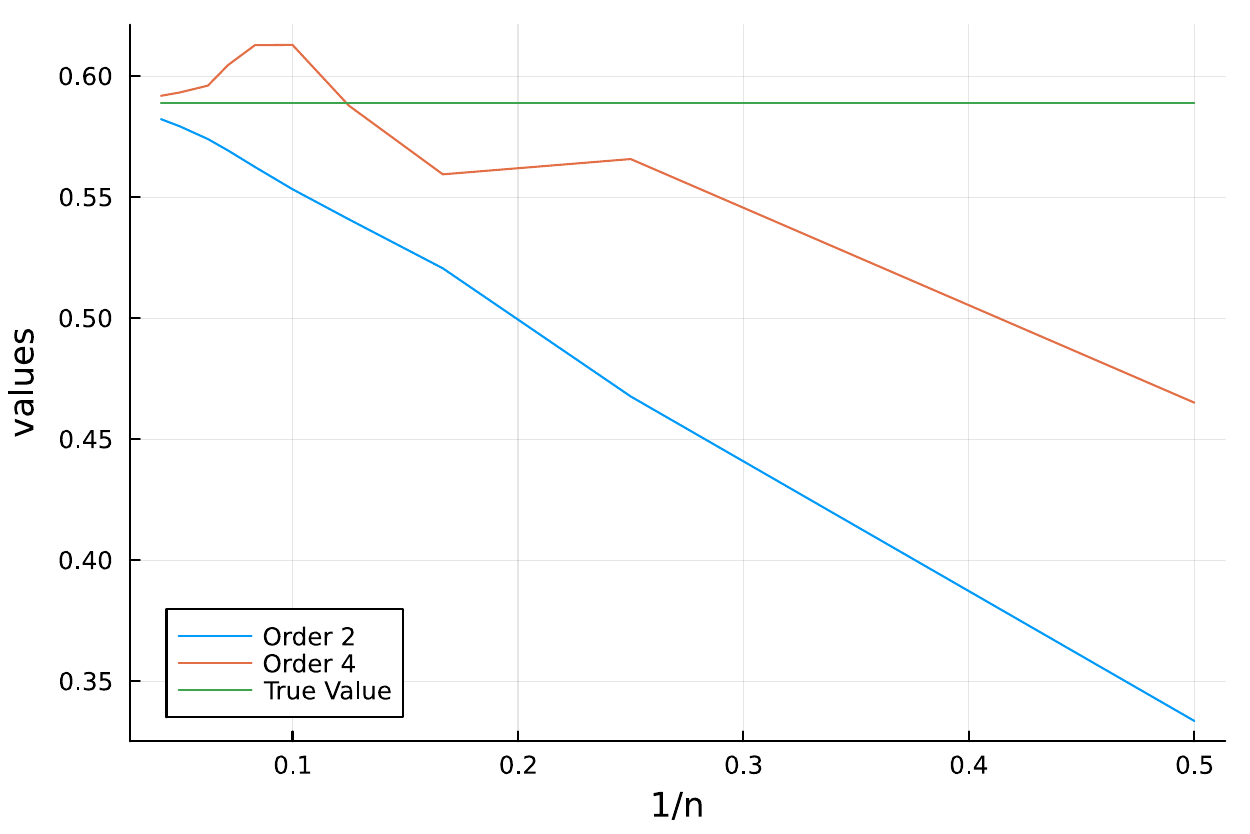}
    \caption{Values plot, scheme $\varphi_A$}
    \label{fig:values_plot_cirAA1}
  \end{subfigure}
  \hfill
  \begin{subfigure}[h]{0.49\textwidth}
    \centering
    \includegraphics[width=\textwidth]{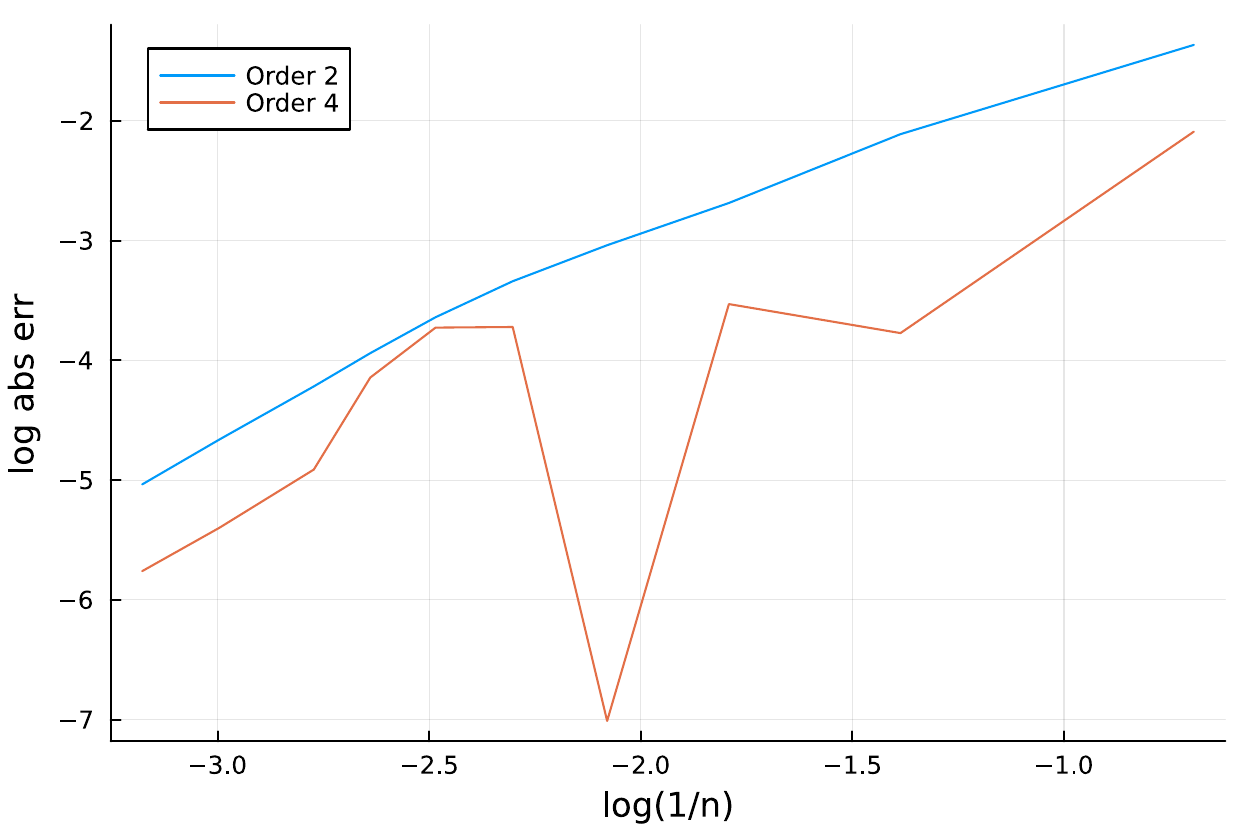}
    \caption{Log-log plot, scheme $\varphi_A$}
    \label{fig:log-log_plot_cirAA1}
  \end{subfigure}
  \begin{subfigure}[h]{0.49\textwidth}
    \centering
    \includegraphics[width=\textwidth]{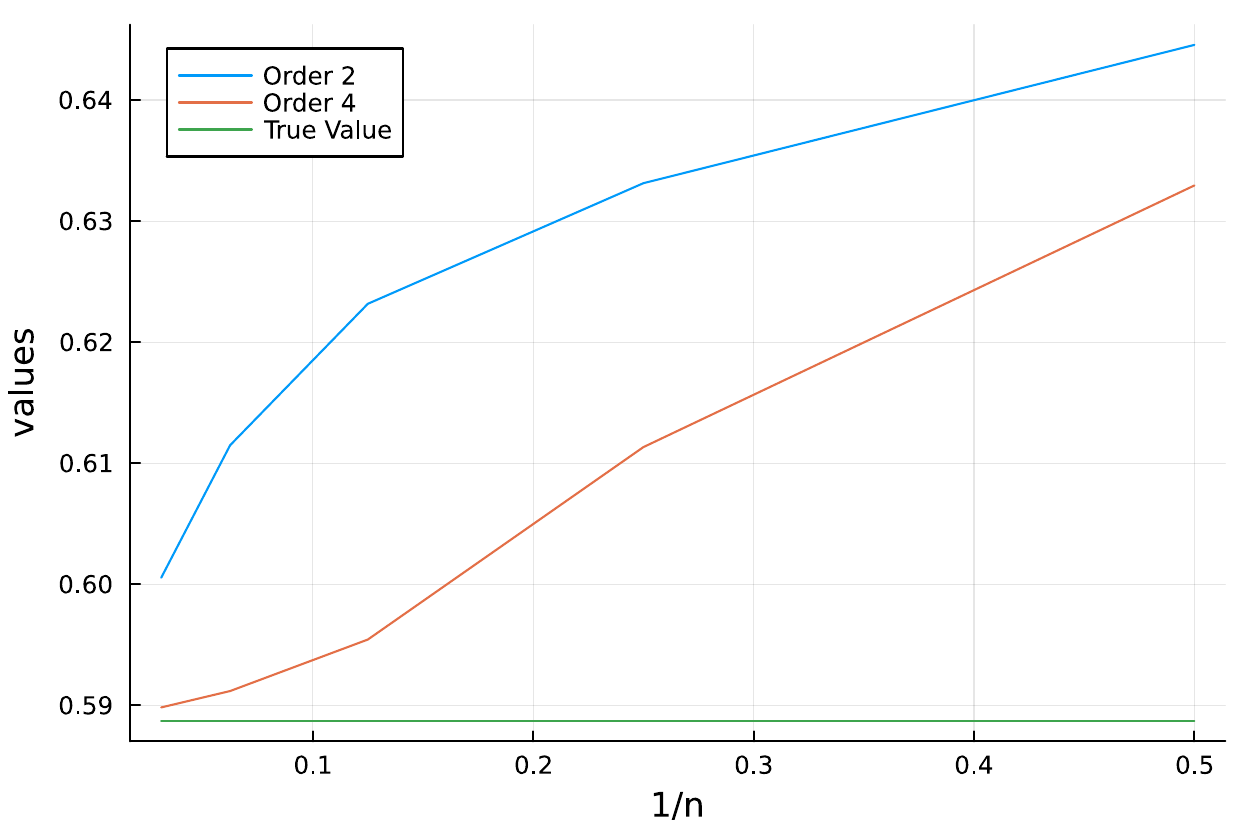}
    \caption{Values plot, scheme $\varphi_B$}
    \label{fig:values_plot_cirAE1}
  \end{subfigure}
  \hfill
  \begin{subfigure}[h]{0.49\textwidth}
    \centering
    \includegraphics[width=\textwidth]{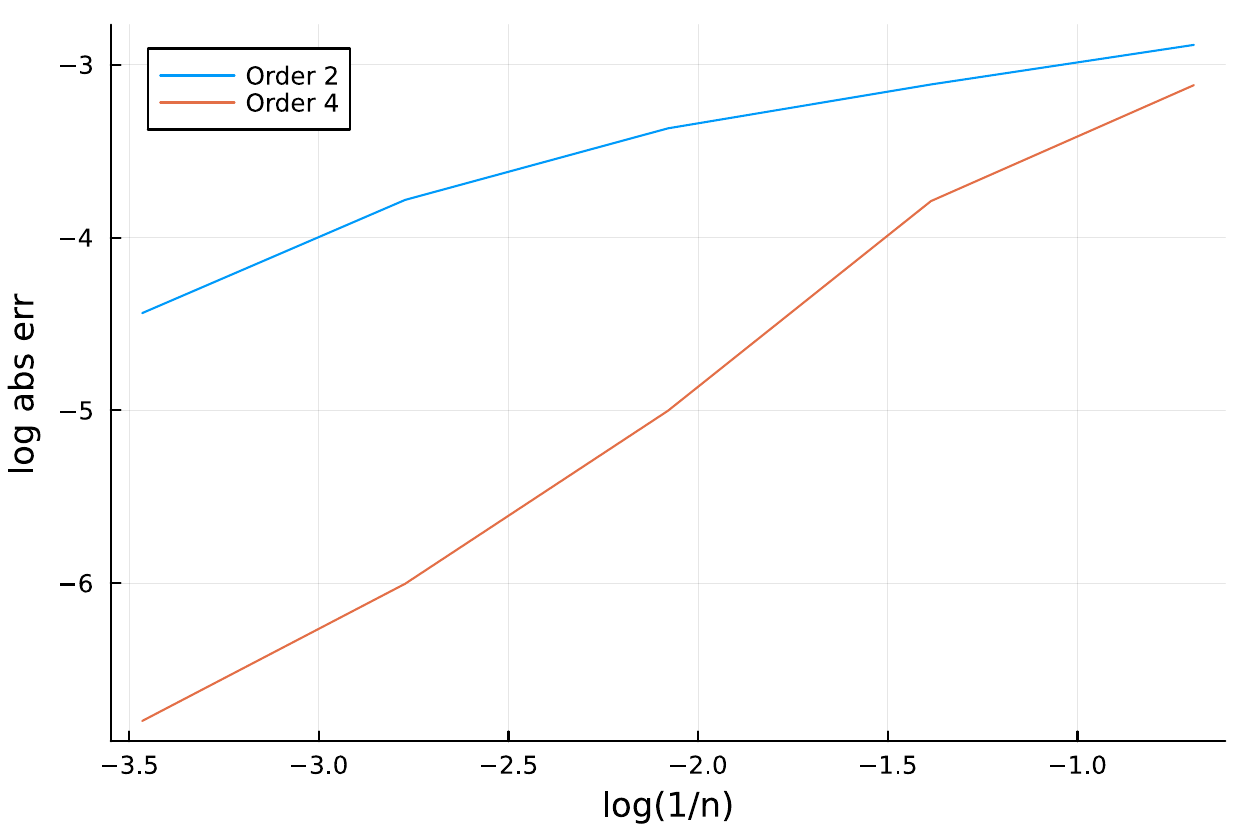}
    \caption{Log-log plot, scheme $\varphi_B$}
    \label{fig:log-log_plot_cirAE1}
  \end{subfigure}
  \caption{Test function: $f(x)=\exp(-10x)$. Parameters: $x=0.2$, $a=0.2$, $k=0.5$, $\sigma=1.5$, $T=1$ ($\frac{\sigma^2}{2a}=5.625$). Statistical precision $\varepsilon=5$e-5.
  Left graphics show the values of $\cPh^{1,n}f$, $\cPh^{2,n}f$ as a function of the time step $1/n$  and the exact value. Right graphics draw $\log(|\hat{P}^{i,n}f-P_Tf|)$ in function of $\log(1/n)$:  for the scheme $\varphi_A$ (resp. $\varphi_B$) the regressed slopes are 1.47 (resp. 0.54) and 1.14 (resp. 1.38) for the second and fourth order respectively.}\label{CIRAA_orders}
\end{figure}

We have plotted in Figure \ref{CIRAA_orders} the convergence of the estimators of the Monte-Carlo estimators $\cPh^{1,n}f$ and $\cPh^{2,n}f$ for the schemes $\varphi_A$ and $\varphi_B$.  We note that in all our experiments, $\cPh^{2,n}f$ gives a better approximation than $\cPh^{1,n}f$, though there is no theoretical guarantee of that. However, the improvement is not as good as for $\sigma^2\le 4a$. 
We know that $\cPh^{1,n}f$ leads to an asymptotic weak error of order~2: the estimated rate of convergence obtained by regression are below since we consider rather small values of~$n$ and are not in the asymptotic regime. We have instead no theoretical guarantee that $\cPh^{2,n}f$ gives an asymptotic weak error of order~4. The estimated rates are quite far from this value, indicating that a fourth order of convergence may not hold. To sum up,   even if $\cPh^{2,n}f$ is still more accurate than $\cPh^{1,n}f$ for $\sigma^2>4a$, it does not lead to obvious computational gains.

\subsubsection{Simulations in the Heston model}
We present now some numerical tests for Heston model and consider three different schemes that are well-defined for any $\sigma\ge 0$:
\begin{itemize}
\item $\Phi_A$ is the scheme~\eqref{H2S} where $\varphi_A(x,h,W_h)$ is used instead of 
$\varphi(x,h,W_h)$,
\item $\Phi_B$ is the scheme~\eqref{H2S} where $\varphi_B(x,h,W_h)$ is used instead of $\varphi(x,h,W_h)$,
\item $\Phi_E$ is the scheme~\eqref{H2S} where the exact scheme $X^x_h$ (see, e.g.~\cite[Proposition 3.1.1]{AA_book}) is used instead of $\varphi(x,h,W_h)$.
\end{itemize}
We start by comparing the variance of the correcting terms with the different schemes.
 In Table~\ref{Table_VAR_Heston1}, we consider a case with $\sigma^2<4a$ and also include the variance for the scheme~$\Phi$ given by~\eqref{H2S}. We remark that the variances of the correction term for the standard scheme $\Phi$ and for the scheme $\Phi_E$ appear to be bounded. In contrast, the variance for the schemes~$\Phi_A$ and $\Phi_B$ tends to increase with~$n$: the variance is very high for $\Phi_A$ while the one produced by $\Phi_B$ remains close to the one of $\Phi$ and $\Phi_E$.  Table~\ref{Table_VAR_Heston2} deals with a case with $\sigma^2>4a$ for which variances are much higher. We observe an approximately linear growth of the variance of the correction term for the schemes $\Phi_A$ and $\Phi_B$. The variance produced by the scheme $\Phi_E$ also increases, but in much moderate way.  
\begin{table}[h!]
	\centering
  \begin{tabular}{ c c|c|c|c|c|c|  }
    \cline{3-7}
        & & $n=2$     & $n=4$      & $n=8$       & $n=16$    & $n=32$    \\
        \hline
        \multicolumn{1}{|c|}{\multirow{2}{*}{$\Phi$} } & $\sigma^2_{4}(n)$   & 33.252 & 41.962 & 46.159 & 48.273  & 49.385 \\ 
        \multicolumn{1}{|c|}{}& 95\% prec.& 0.024 & 0.029 & 0.033 & 0.035 & 0.037 \\ 
    \hline
    \multicolumn{1}{|c|}{\multirow{2}{*}{$\Phi_A$}} & $\sigma^2_{4}(n)$     & 450.95 & 973.82 & 1976.53 & 3984.64 & 8014.19 \\ 
    \multicolumn{1}{|c|}{}& 95\% prec.     & 0.20 & 0.40 & 0.83 & 1.70 & 3.47 \\
    \hline
        \multicolumn{1}{|c|}{\multirow{2}{*}{$\Phi_B$} } & $\sigma^2_{4}(n)$   & 33.702 & 43.116 & 48.606 & 53.373  & 59.760 \\ 
        \multicolumn{1}{|c|}{}& 95\% prec.& 0.025 & 0.031 & 0.037 & 0.044 & 0.059 \\ \hline
        \multicolumn{1}{|c|}{\multirow{2}{*}{$\Phi_E$}} & $\sigma^2_{4}(n)$     & 51.99 & 53.93 & 52.46 & 51.47 &  50.99 \\     
        \multicolumn{1}{|c|}{}& 95\% prec.     & 0.032 & 0.034 & 0.036 & 0.037 & 0.037 \\
        \hline  
      \end{tabular}
  \caption{$\sigma^2_{4}(n) = \Var\big(n(f(\hat{X}^{n,1}_T,\hat{S}^{n,1}_T) - f(\hat{X}^{n,0}_T,\hat{S}^{n,0}_T))\big)$ with $10^8$ samples and 95\% confidence interval precision.
  Test function: $f(x,s)=(K-s)^+$. Parameters: $S_0=100$, $r=0$, $x=0.2$, $a=0.2$, $k=1.0$, $\sigma=0.5$, $\rho=-0.7$, $T=1$, $K=105$ ($\frac{\sigma^2}{2a}=0.625$).}\label{Table_VAR_Heston1}
\end{table}
\begin{table}[h!]
	\centering
  \begin{tabular}{ c c|c|c|c|c|c|  }
    \cline{3-7}
        & & $n=2$     & $n=4$      & $n=8$       & $n=16$    & $n=32$    \\
    \hline
    \multicolumn{1}{|c|}{\multirow{2}{*}{$\Phi_A$}} & $\sigma^2_{4}(n)$     & 799.93 & 2568.43 & 6384.48 & 14588.23 & 29798.4266 \\ 
    \multicolumn{1}{|c|}{}& 95\% prec.     & 0.58 & 1.93 & 5.88 & 16.63 & 42.38 \\
    \hline
    \multicolumn{1}{|c|}{\multirow{2}{*}{$\Phi_B$}} & $\sigma^2_{4}(n)$     & 306.87 & 581.70 & 958.06 & 1729.18 & 3185.83 \\ 
    \multicolumn{1}{|c|}{}& 95\% prec.     & 0.18 & 0.38 & 0.90 & 2.65 & 8.25 \\
    \hline
    \multicolumn{1}{|c|}{\multirow{2}{*}{$\Phi_{E}$} } & $\sigma^2_{4}(n)$    & 233.89 & 287.50 & 314.03 & 331.31 & 344.20 \\ 
    \multicolumn{1}{|c|}{}& 95\% prec.& 0.14 & 0.20 & 0.24 & 0.27 & 0.29 \\ 
\hline
  \end{tabular}
  \caption{$\sigma^2_{4}(n) = \Var\big(n(f(\hat{X}^{n,1}_T,\hat{S}^{n,1}_T) - f(\hat{X}^{n,0}_T,\hat{S}^{n,0}_T))\big)$ with $10^8$ samples and 95\% confidence interval precision.
  Test function: $f(x,s)=(K-s)^+$. Parameters: $S_0=100$, $r=0$, $x=0.2$, $a=0.2$, $k=1.0$, $\sigma=1.5$, $\rho=-0.7$, $T=1$, $K=105$ ($\frac{\sigma^2}{2a}=5.625$).}\label{Table_VAR_Heston2}
\end{table}

We now turn to the convergence of the Monte-Carlo estimators. We have plotted in Figure~\ref{HestonAA_orders2}, for the same set of parameters as in Table~\ref{Table_VAR_Heston2}, the behavior of $\cPh^{1,n}f$ and $\cPh^{2,n}f$ for the schemes $\Phi_B$ and $\Phi_E$. We have discarded the scheme $\Phi_A$ that produces a too large variance for the correcting term. As for the CIR diffusion, we  note that $\cPh^{2,n}f$ gives a better approximation than $\cPh^{1,n}f$, but the bias does not seem to be of order~$4$. For the scheme $\Phi_B$, the improvement is moderate, and do not really compensate the computational effort of calculating the correcting term. Instead, for the scheme $\Phi_E$, the improvement is rather significant, making the approximation $\cPh^{2,n}f$ interesting from a computational point of view with respect to~$\cPh^{1,n}f$. Also, the estimated rate of convergence is much higher and not so far from~$4$. A dedicated theoretical study of $\cPh^{2,n}f$ with the scheme $\Phi_E$ is left for further research. 
\begin{figure}[h!]
  \centering
  \begin{subfigure}[h]{0.49\textwidth}
    \centering
    \includegraphics[width=\textwidth]{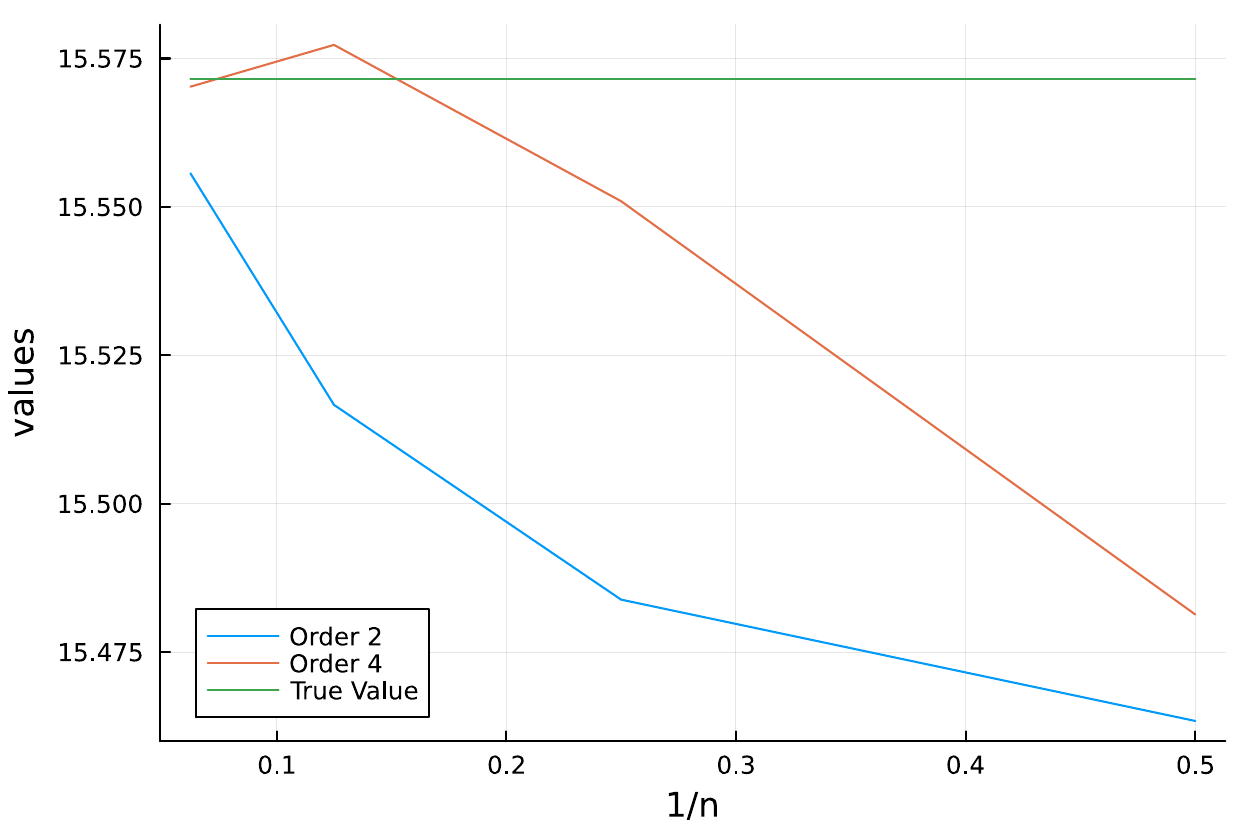}
    \caption{Values plot, scheme $\Phi_B$}
    \label{fig:values_plot_hestonAA2}
  \end{subfigure}
  \hfill
  \begin{subfigure}[h]{0.49\textwidth}
    \centering
    \includegraphics[width=\textwidth]{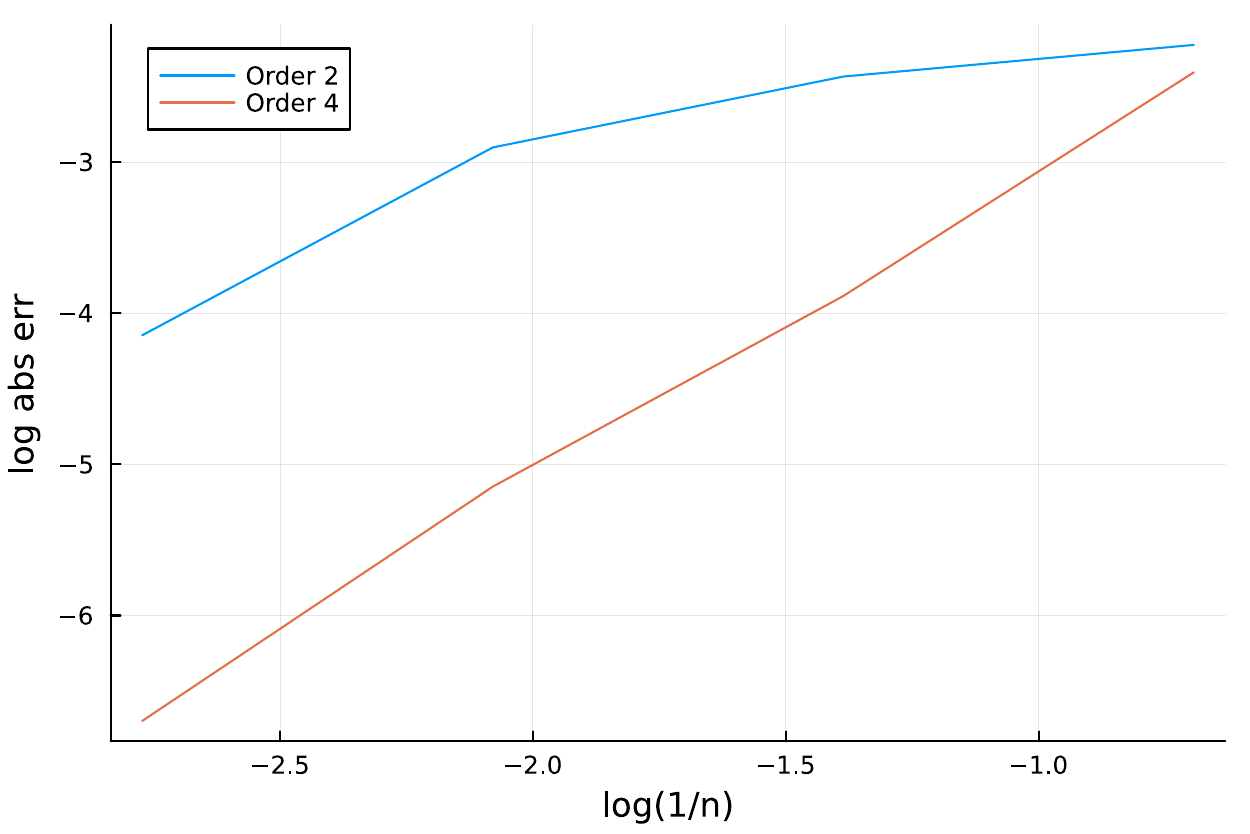}
    \caption{Log-log plot, scheme $\Phi_B$}
    \label{fig:log-log_plot_hestonAA2}
  \end{subfigure}
  \begin{subfigure}[h]{0.49\textwidth}
    \centering
    \includegraphics[width=\textwidth]{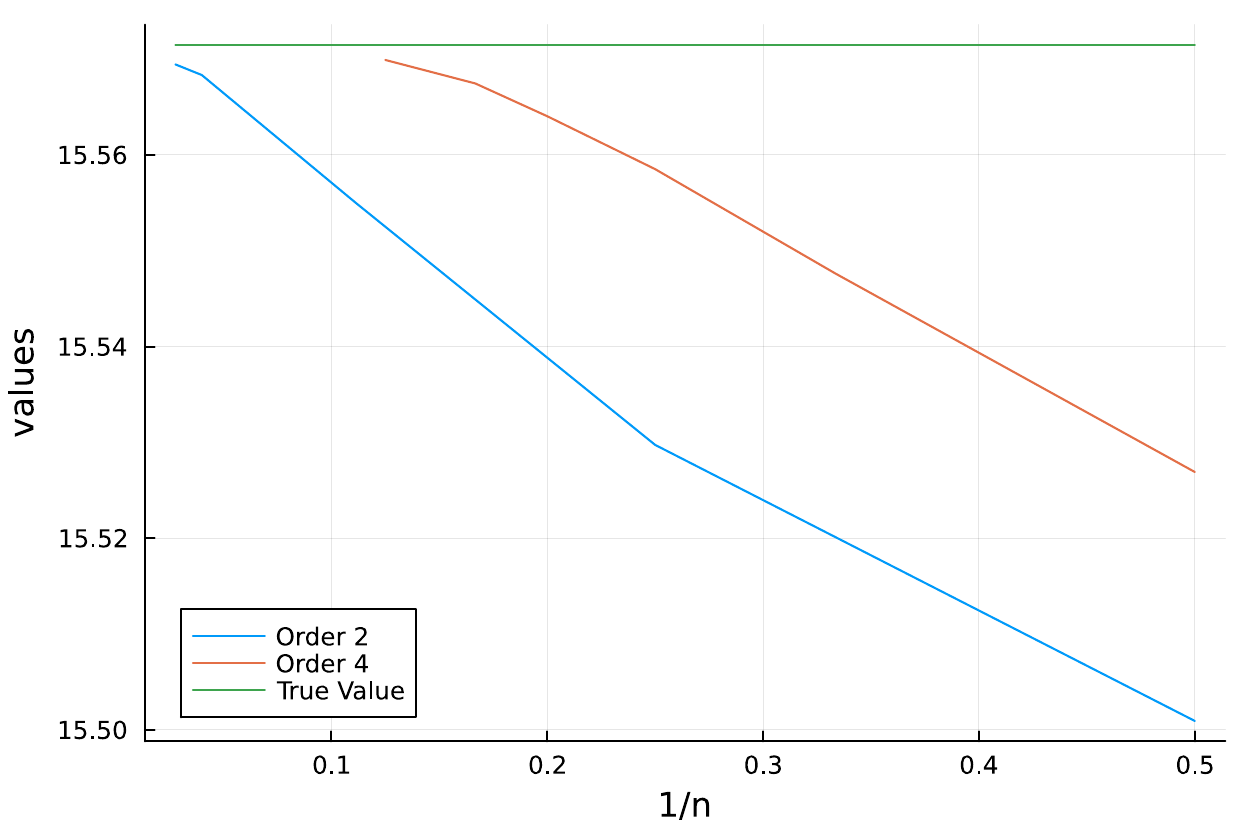}
    \caption{Values plot, scheme $\Phi_E$}
    \label{fig:values_plot_hestonCF1}
  \end{subfigure}
  \hfill
  \begin{subfigure}[h]{0.49\textwidth}
    \centering
    \includegraphics[width=\textwidth]{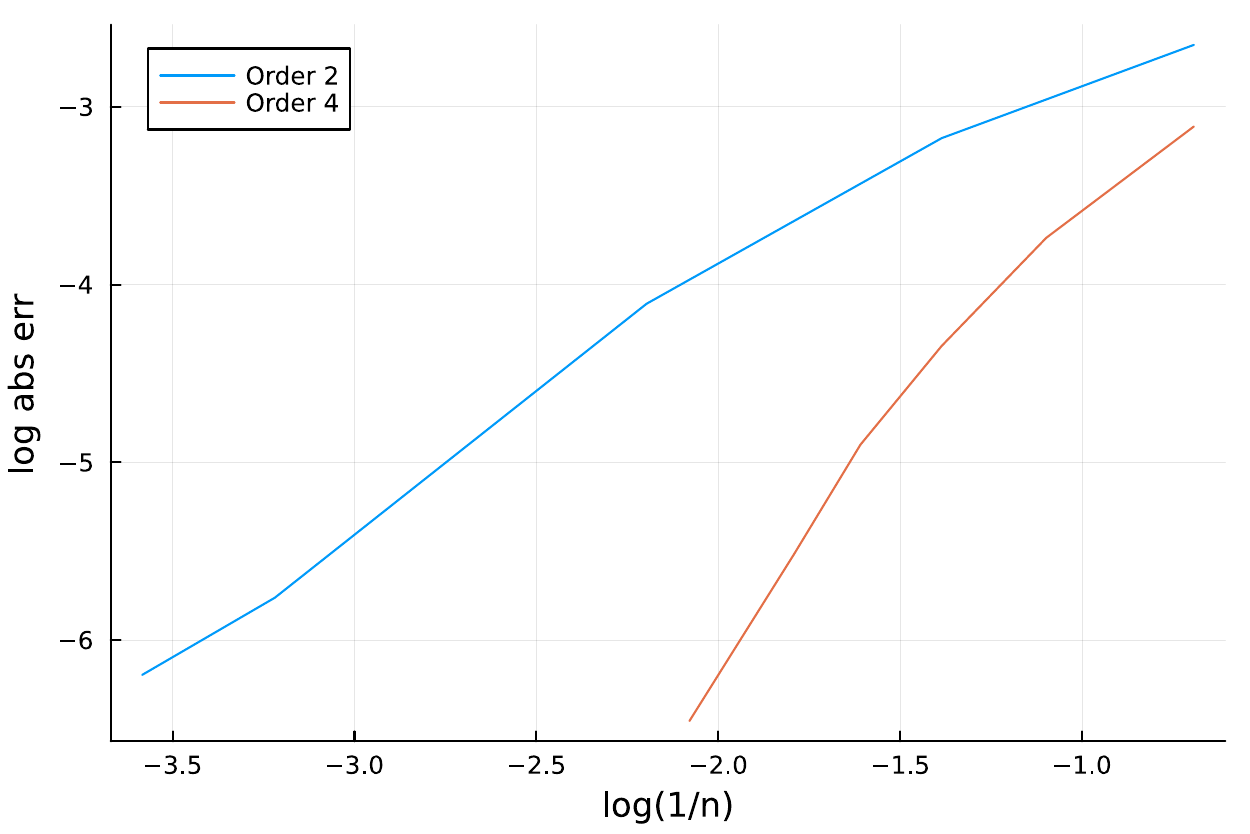}
    \caption{Log-log plot, scheme $\Phi_E$}
    \label{fig:log-log_plot_hestonCF1}
  \end{subfigure}
  \caption{Test function: $f(x,s)=(K-s)^+$. Parameters: $S_0=100$, $r=0$, $x=0.2$, $a=0.2$, $k=1$, $\sigma=1.5$, $\rho=-0.7$, $T=1$, $K=105$ ($\frac{\sigma^2}{2a}=5.625$). Statistical precision $\varepsilon=5$e-4.
  Left graphics show the values of $\cPh^{1,n}f$, $\cPh^{2,n}f$ as a function of the time step $1/n$  and the exact value. Right graphics draw $\log(|\hat{P}^{i,n}f-P_Tf|)$ in function of $\log(1/n)$: for the scheme $\Phi_B$ (resp. $\Phi_E$) the regressed slopes  are 0.90 (resp. 1.28) and 2.04 (resp. 2.40) for the second and fourth order respectively.}\label{HestonAA_orders2}
\end{figure}



\newcommand{\cP}{{\mathcal P}}
\newcommand{\cI}{{\mathcal I}}

\def\a{\alpha}
\def\b{\beta}


\chapter{High order approximations of the log-Heston process semigroup} 

\label{Chapter_Heston} 
This chapter is based on the paper \cite{AL24}, which has been accepted for publication on SIAM Journal on Financial Mathematics.
\chapabstract{\\We present weak approximations schemes of any order for the Heston model that are obtained by using the method developed by Alfonsi and Bally (2021). This method consists in combining approximation schemes calculated on different random grids to increase the order of convergence. We apply this method with either the Ninomiya-Victoir scheme (2008) or a second-order scheme that samples exactly the volatility component, and we show rigorously that we can achieve then any order of convergence. We give numerical illustrations on financial examples that validate the theoretical order of convergence, and present also promising numerical results for the multifactor/rough Heston model.}
\section*{Introduction}
The Heston model~\cite{Heston} is one of the most popular model in mathematical finance. It describes the dynamics of an asset and its instantaneous volatility by the following stochastic differential equations:
\begin{equation}\label{Heston_SDEs}
	\begin{cases}
	dS^{s,y}_t &= rS^{s,y}_t dt + \sqrt{Y^y_t}S^{s,y}_t (\rho dW_t + \sqrt{1-\rho^2} dB_t), \ S^{s,y}_0=s >0,\\
	dY^y_t &= (a-bY^y_t) dt +\sigma \sqrt{Y^y_t} dW_t, \ Y^y_0=y\ge 0,
	\end{cases}
\end{equation}
where $W$ and $B$ are two independent Brownian motions, $a\ge 0$, $b\in \R$, $\sigma>0$ and $\rho \in [-1,1]$. For the financial application, it is typically assumed in addition that $b>0$ so that the volatility mean reverts towards $a/b$, but this is not needed in the present paper. 

The goal of the paper is to propose high order weak approximations for this model and to prove their convergence. Let us recall first that exact simulation methods have been proposed for the Heston model by Broadie and Kaya~\cite{BrKa} and then by Glasserman and Kim~\cite{GlKi}. However, these methods usually require more computation time than simulation schemes. Besides, when considering variants or extensions of the Heston model, it is not clear how to simulate them exactly while approximation schemes can more simply be reused or adapted. There exists in the literature many approximation schemes of the Heston model, we mention here the works of  Andersen~\cite{Andersen}, Lord et al.~\cite{LKVD}, Ninomiya and Victoir~\cite{NV} and Alfonsi~\cite{AA_MCOM}. Few of them study from a theoretical point of view the weak convergence of these schemes. While~\cite{AA_MCOM} focuses on the volatility component, Altmayer and Neuenkirch~\cite{AlNe} proposes up to our knowledge the first analysis of the weak error for the whole Heston model. They essentially obtain for a given Euler/Milstein scheme a weak convergence rate of~$1$ under the restriction $\sigma^2<a$ on the parameter.

Like~\cite{AlNe}, we will work with the log-Heston model that solves the following SDE
\begin{equation}\label{log-Heston_SDEs}
	\begin{cases}
	dX^{x,y}_t &= (r-\frac{1}{2}Y^y_t) dt + \sqrt{Y^y_t} (\rho dW_t + \sqrt{1-\rho^2} dB_t), \ X^{x,y}_0=x=\log(s) \in \R,\\
	dY^y_t &= (a-bY^y_t) dt +\sigma \sqrt{Y^y_t} dW_t, \ Y^y_0=y.
	\end{cases}
\end{equation}
This log transformation of the asset price is standard to carry mathematical analyses: it allows to get an SDE with bounded moments since its coefficients have at most a linear growth. Our goal is to propose approximations of any order of the semigroup $P_Tf(x,y)=\E[f(X^{x,y}_T,Y^y_T)]$, where $f:\R \times \R_+ \to \R$ is a sufficiently smooth function such that $|f(x,y)|\le C(1+|x|^L+y^L)$ for some $L \in \N$. More precisely, we want to apply the recent method proposed by Alfonsi and Bally~\cite{AB} that allows to boost the convergence of an approximation scheme by using random time grids. We consider here either the Ninomiya-Victoir scheme for $\sigma^2\le 4a$ or a scheme that simulate exactly $Y$ for any $\sigma>0$.
In a previous work~\cite{AL}, we have applied the method of~\cite{AB} to the only Cox-Ingersoll-Ross component~$Y$ and we want to extend our result to the full log-Heston model. 
One difficulty with respect to the general framework developed in~\cite{AB} is to deal with the singularity of the diffusion coefficient. In particular, we need some analytical results on the Cauchy problem associated to the log-Heston model that have been obtained recently by Briani et al.~\cite{BCT}. Our main theoretical result (Theorem~\ref{main_theorem}) states that we obtain, for any $\nu\ge 1$, semigroup approximations of order $2\nu$ by using the mentioned scheme with the boosting method of~\cite{AB}.

The paper is structured as follows. Section~\ref{Sec_main_heston} presents the precise framework and in particular the functional spaces that we consider carrying our analysis. It introduces the approximation schemes and briefly presents the boosting method using random grids proposed in~\cite{AB}. The main result of the paper is then stated precisely.  Section~\ref{Sec_proof} is dedicated to the proof of the main theorem. Last, Section~\ref{Sec_num} explains how to implement our approximations and illustrates their convergence on numerical experiments motivated by the financial application. As an opening for future research, we show that our approximations can be used for the multifactor Heston model\footnote{We recall that the multifactor Heston model proposed by Abi Jaber and El Euch~\cite{AEE19} is an extension of the Heston model that is a good proxy of the rough Heston model introduced by El Euch and Rosenbaum~\cite{EER19}.} under some parameter restrictions and give very promising convergence results.

\section{Main results}\label{Sec_main_heston}

We start by introducing some functional spaces that are used through the paper.
For $k\in\N$, we denote by $\mcC^{k}(\R\times\R_+)$ the space of continuous functions $f:\R \times \R_+ \to \R$ such that the partial derivatives $\partial^\a_x \partial^\b_y  f(x,y)$ exist and are continuous with respect to $(x,y)$ for all $(\a,\b)\in\N^2$ such that $\a+2\b\le k$. We then define for $L\in \N$,
\begin{equation}\label{def_fL}
   {\bf f}_L(x,y)=  (1+x^{2L}+y^{2L}), \ x\in \R,y\in \R_+,
\end{equation}
and introduce 
\begin{multline}\label{def_CpolkL}
  \CPL{k}{L}=\{ f\in \mcC^{k}(\R\times\R_+) \mid \exists C>0 \text{ such that } \forall(\a,\b)\in\N^2, \a+2\b\le k,\\
  |\partial^\a_x \partial^\b_y  f(x,y)| \le C {\bf f}_L(x,y)   \},
\end{multline}
the space of continuously differentiable functions up to order~$k$ with derivatives with polynomial growth of order $2L$. Note that we assume twice less differentiability on the $y$ component.  
Furthermore, we set 
$$\CP{k}= \cup_{L\in \N} \CPL{k}{L} \text{ and } \CP{\infty}= \cap_{k\in \N} \CP{k}.$$
Last, we endow $\CPL{k}{L}$ with the following norm:
\begin{equation}\label{def_NormCpolKL}
  \|f\|_{k,L}=\sum_{\alpha+2\beta \le k} \sup_{(x,y) \in \R \times \R_+} \frac{|\partial^\alpha_x\partial^\beta_y f(x,y)|}{\bff_L(x,y)}.
\end{equation}

\subsection{Second order schemes for the log-Heston process}\label{subsec_2nd}

Having in mind~\cite[Theorem 2.3.8]{AA_book}, there are three properties to check to get a second-order scheme for the weak error:
\begin{enumerate}[(a)]
  \item polynomial estimates for the derivatives of the solution of the Cauchy problem,
  \item uniformly bounded moments of the approximation scheme, 
  \item a potential second order scheme, which roughly means that we have a family of random variables $(\hat{X}^{x,y}_t,\hat{Y}^y_t)_{t\ge 0}$ such that $|\E[f(\hat{X}^{x,y}_t,\hat{Y}^y_t)]-f(x,y)-t \mathcal{L} f(x,y)- \frac{t^2}2\mathcal{L}^2 f(x,y)|=_{t\to 0} O(t^3)$. 
\end{enumerate}

Let us precise this in our context. We consider a time horizon $T>0$ and a time step $h=T/n$, with $n\in \N^*$. We note $(\hat{X}^{x,y}_h,\hat{Y}^y_h)$ an approximation scheme for the SDE~\eqref{log-Heston_SDEs} starting from $(x,y)$ with time-step $h$, and $$\hat{P}_h f(x,y)=\E[f(\hat{X}^{x,y}_h,\hat{Y}^y_h)]$$ the associated semigroup approximation. The weak error analysis proposed by Talay and Tubaro~\cite{TT} consists in writing
\begin{equation}\label{eq_TT}
  \hat{P}^{[n]}_h-P_T=\hat{P}^{[n]}_h-P_h^{[n]}=\sum_{i=0}^{n-1}\hat{P}_h^{[n-(i+1)]}(\hat{P}_{h}-P_h)P_h^{[i]}=\sum_{i=0}^{n-1}\hat{P}_h^{[n-(i+1)]}(\hat{P}_{h}-P_h)P_{ih},
\end{equation}
where $\hat{P}_h^{[0]}=Id$ and $\hat{P}_h^{[i]}=\hat{P}_h^{[i-1]}\hat{P}_h$ for $i\ge 1$, and $P_h^{[i]}=P_{ih}$ by the semigroup property.   
Let us assume that the three properties hold
\begin{enumerate}[(a)]
\item $\forall k,L \in \N, \ \exists C \in \R_+, \ \forall i\in \{0,\dots,n\}, \  \|P_{ih} f\|_{k,L} \le C \|f\|_{k,L}$,
\item $\forall L \in \N, \exists C_L \in \R_+, \ \max_{0\le i\le n} \hat{P}_h^{[i]} \bff_L (x,y) \le C_L \bff_L(x,y)$,
\item $\|\hat{P}_h f-P_hf\|_{0,L+3}\le Ch^3 \|f\|_{12,L}.$
\end{enumerate}
Then, for $f \in \CPL{12}{L}$, we have for each $i \in \{0,\dots,N-1\},$
$$\|(\hat{P}_{h}-P_h)P_{ih}f \|_{0,L+3}\le Ch^3\|P_{ih}f \|_{12,L}\le C^2 h^3\|f \|_{12,L}, $$
by using the first and third properties. Then, we use that $$|(\hat{P}_{h}-P_h)P_{ih}f(x,y)|\le  C^2\|f \|_{12,L} h^3 \bff_{L+3}(x,y),$$
together with the second property to get that $|\hat{P}^{[n-(i+1)]}(\hat{P}_{h}-P_h)P_{ih}f(x,y)|\le C_LC^2 h^3\bff_L(x,y)$. This bound is uniform with respect to~$i$, and we get
\begin{equation}
  |\hat{P}^{[n]}_hf(x,y)-P_Tf(x,y)|\le C_LC^2 T \bff_{L+3}(x,y) \times \left(\frac{T}{n}\right)^2,
\end{equation}
since $h=T/n$.

Before concluding this paragraph, we comment briefly how to get the three properties (a--c). For the log-Heston SDE, the Cauchy problem has been studied by Briani et al.~\cite{BCT}, and their analysis allows us to obtain~(a). Their result is reported in Proposition~\ref{prop-rep-logHeston-estim}. Property (b) can generally be checked by simple but sometimes tedious calculation. Property (c) is the crucial one and can be obtained by using splitting technique, see~\cite[Paragraph 2.3.2]{AA_book}. We check this property in Corollary~\ref{cor_H1} for the schemes presented in this paper.

\subsection{From the second order scheme to higher orders by using random grids}

We continue the analysis and present, in our framework, the method developed by Alfonsi and Bally~\cite{AB} to get approximations of any orders by using random grids. For $l \in \N^*$, let us define the time step
$h_l=\frac{T}{n^l}$. We set $Q_l=\hat{P}_{h_l}$ the operator obtained by using the approximation scheme with the time step $h_l$. The principle is to iterate the identity~\eqref{eq_TT}. Namely, we get from~\eqref{eq_TT}
$$ \hat{P}_{h_1}^{[i]}-P_{ih_1}=\sum_{i_1=0}^{i-1} \hat{P}^{[i-(i_1+1)]}_{h_1}(\hat{P}_{h_1}-P_{h_1})P^{[i_1]}_{h_1} $$
and 
$$ \hat{P}_{h_2}^{[n]}-P_{h_1}=\hat{P}_{h_2}^{[n]}-P_{h_2}^{[n]}=\sum_{j=0}^{n-1}\hat{P}^{[n-(j+1)]}_{h_2}(\hat{P}_{h_2}-P_{h_2})P^{[j]}_{h_2}.   $$
Plugging these two identities successively in~\eqref{eq_TT}, we obtain
\begin{equation}\label{eq_boost2}\hat{P}^{[n]}_{h_1}-P_T=\sum_{i=0}^{n-1}\hat{P}_{h_1}^{[n-(i+1)]}(\hat{P}_{h_1}-\hat{P}_{h_2}^{[n]})\hat{P}^{[i]}_{h_1}+R,\end{equation}
with the remainder given by
\begin{align*}
  R=&\sum_{i=0}^{n-1}\hat{P}_{h_1}^{[n-(i+1)]} \left(\sum_{j=0}^{n-1}\hat{P}^{[n-(j+1)]}_{h_2}(\hat{P}_{h_2}-P_{h_2})P^{[j]}_{h_2} \right)\hat{P}^{[i]}_{h_1} \\&- \sum_{i=0}^{n-1}\hat{P}_{h_1}^{[n-(i+1)]} (\hat{P}_{h_1}-P_{h_1}) \sum_{i_1=0}^{i-1} \hat{P}^{[i-(i_1+1)]}_{h_1}(\hat{P}_{h_1}-P_{h_1})P^{[i_1]}_{h_1} .\end{align*}
Let us assume that we have the two following properties\footnote{We directly specify the method to our framework, and refer to~\cite{AB} or~\cite[Section 2]{AL} for a general presentation.}
\begin{align}\tag{$\bbar{H_1}$}\label{H1_bar_heston}
   &\forall l,k,L\in\N,\exists C>0,\forall f\in \CPL{k+12}{L},\, \|(P_{h_l}-Q_l)f\|_{k,L+3} \leq C\|f\|_{k+12,L} h_l^{3},\\
&\tag{$\bbar{H_2}$} \label{H2_bar_heston}
  \forall l,k,L\in \N,\exists C>0,\forall f\in \CPL{k}{L}, \, \max_{0\leq j\leq n^l}\|Q^{[j]}_l f\|_{k,L} + \sup_{t\leq T}\|P_t f\|_{k,L} \leq C\|f\|_{k,L}.
\end{align}
Then, we can upper bound the remainder for $f\in  \CPL{k+24}{L}$ by
$$ \|R f\|_{k,L+6}\le C^3n^2 \|f\|_{k+12,L+3} h_2^3+C^5 \frac{n(n-1)}2 \|f\|_{k+24,L} (h_1^3)^2 \le \tilde{C}\|f\|_{k+24,L} \left(\frac{T}{n}\right)^4,$$
where we have used twice~\eqref{H2_bar_heston} and once~\eqref{H1_bar_heston} for the first sum, and three times~\eqref{H2_bar_heston} and twice~\eqref{H1_bar_heston} for the second one. Therefore, we get from~\eqref{eq_boost2} that
\begin{equation}\label{def_boost2}
  \cPh^{2,n}:= \hat{P}^{[n]}_{h_1}+\sum_{i=0}^{n-1}\hat{P}_{h_1}^{[n-(i+1)]}(\hat{P}_{h_2}^{[n]}-\hat{P}_{h_1})\hat{P}^{[i]}_{h_1}
\end{equation}
is an approximation of order~$4$. Namely, we get
\begin{equation}\label{boost2_estimate} \forall f \in \CP{24}, \  \exists C>0,L\in \N, \   \|\cPh^{2,n} f -P_Tf\|_{0,L+6}\le C \|f\|_{24,L} \left(\frac{T}{n} \right)^4. \end{equation}
Let us note that $\hat{P}_{h_1}^{[n-(i+1)]}\hat{P}_{h_2}^{[n]}\hat{P}_{h_1}^{[i]}$ corresponds to the scheme on a time grid that is uniform, but uniformly refined on the $(i+1)$-th time step. This time grid has thus $2n$ time steps, and if one should calculate all the terms in the sum of~\eqref{def_boost2}, this would require a computational time in $O(n^2)$. Thus, the method would not be more efficient that using the second-order scheme with a time step $n^2$. This is why we use random grids and use a random variable $\kappa$ that is uniformly distributed on $\{0,\dots,n-1\}$. We have 
\begin{equation}\label{boost2_RG} \cPh^{2,n} = \hat{P}^{[n]}_{h_1}+n\E[\hat{P}_{h_1}^{[n-(\kappa+1)]}(\hat{P}_{h_2}^{[n]}-\hat{P}_{h_1})\hat{P}^{[\kappa]}_{h_1}].\end{equation}
Thus, for the correcting term, we consider a random time grid that is the obtained from the uniform time grid with time step $T/n$ by refining uniformly the $(\kappa+1)$-th time step with a time step $h_2=T/n^2$. 

We have presented here how $\cPh^{2,n}$ improves the convergence of $\cPh^{1,n}=\hat{P}^{[n]}_{h_1}$. Then, for $\nu \ge 2$, it is possible to define by induction approximations $\cPh^{\nu,n}$, such that 
\begin{equation}\label{def_Pnu}\forall f \in \CP{12 \nu}, \  \exists C>0,L\in \N, \   \|\cPh^{\nu,n} f -P_Tf\|_{0,L+3\nu}\le C \|f\|_{12\nu,L} \left(\frac{T}{n} \right)^{2 \nu}. \end{equation}
Unfortunately, the induction cannot be easily described and involves a tree structure. We refer to~\cite{AB} for the details and to~\cite[Eq. (2.8)]{AL} for an explicit expression of $\cPh^{3,n}$.

\subsection{A second-order scheme for the log-Heston model}

Before presenting the scheme, it is interesting to point similarities and  difference between the weak error analysis  of Subsection~\ref{subsec_2nd} and the present one to get higher order approximations. Property~\eqref{H1_bar_heston} is a generalization of~(c), while~\eqref{H2_bar_heston} is stronger than properties (a) and (b)\footnote{Note that $\bff_L\in \CPL{\infty}{L}$. We have, for $i \le n$, $\|P^{[i]}_{T/n} \bff_L\|_{0,L} =\|Q_1^{[i]}\bff_L\|_{0,L}\le C \|\bff_L\|_{0,L}$ by \eqref{H2_bar_heston}, which gives (b).}. We now point an important difference between the two error analysis. In Equation~\eqref{eq_TT}, the difference between the semigroup and its approximation appears only once and there is no need to have regularity properties for the function $(\hat{P}_{h}-P_h)P_{ih}f$: only a polynomial bound is needed. In contrast, for the approximation $\cPh^{2,n}$ we need some regularity to iterate and upper bound the remainder. This difference has an important incidence in the case of the log-Heston process. It is proposed in~\cite{AA_MCOM} a second-order scheme for the log-Heston process for any $\sigma \ge 0$. When $\sigma^2\ge 4a$, this scheme relies for the Cox-Ingersoll-Ross (CIR) part on bounded random variables that match the first moments of the standard normal distribution. Unfortunately, these moment-matching variables prevent us to get~\eqref{H2_bar_heston}: in a recent work on high order approximations for the CIR process, we have shown in~\cite[Theorem 5.16]{AL} that it was not possible to use these moment-matching variables together with our analysis in order to get~\eqref{H2_bar_heston}. We do not repeat here the analysis that would be quite similar for the log-Heston model, and consider either the Ninomiya-Victoir scheme for $\sigma^2\le 4a$ or the exact CIR simulation for any $\sigma>0$. We now present this in detail.

We present in this subsection the approximations scheme that we will study in this paper. It is constructed by using the splitting technique. Let  $\brho=\sqrt{1-\rho^2}$. Then, the infinitesimal generator associated to the log-Heston SDE~\eqref{log-Heston_SDEs} is given by 
\begin{equation}\label{log_Heston2_diff-op}
  \mcL = \frac{y}{2}(\partial^2_x + 2\rho\sigma\partial_x\partial_y+\sigma^2\partial_y)+ (r-\frac{y}{2})\partial_x + (a-by)\partial_y.
\end{equation}
We split this operator as the sum $\mcL = \mcL_B+\mcL_W$ where \begin{equation}\label{def_LB}\mcL_B = \big((r-\frac{\rho a}{\sigma})+(\frac{\rho b}{\sigma}-\frac 12)y\big) \partial_x +\frac{y}{2}\brho^2\partial_x^2\end{equation}  is the infinitesimal generator of the SDE
\begin{equation}\label{LB_SDEs}
  \begin{cases}
    dX_t &= \big((r-\frac{\rho a}{\sigma})+(\frac{\rho b}{\sigma}-\frac 12)Y_t\big) dt+\brho\sqrt{Y_t} dB_t,\\
    dY_t &= 0,
  \end{cases}
\end{equation}
and  \begin{equation}\label{def_LW}\mcL_W =\frac{y}{2}(\rho^2\partial^2_x + 2\rho\sigma\partial_x\partial_y+\sigma^2\partial^2_y) + (a-by)(\frac{\rho}{\sigma}\partial_x+\partial_y)\end{equation} is the infinitesimal generator of 
\begin{equation}\label{LW_SDEs}
  \begin{cases}
    dX_t &= (\frac{\rho a}{\sigma}-\frac{\rho b}{\sigma}Y_t) dt +\rho \sqrt{Y_t} dW_t,\\
    dY_t &= (a-bY_t) dt +\sigma \sqrt{Y_t} dW_t.
  \end{cases}
\end{equation}
This splitting is slightly different from the one considered in~\cite[Subsection 4.2]{AA_MCOM}: one should remark that it is chosen in order to have in~\eqref{LW_SDEs} $dX_t = \frac{\rho}{\sigma}dY_t$. This is not particularly useful to get a second order scheme. However, it avoids introducing a third coordinate corresponding to the integrated CIR process, which is more parsimonious for the mathematical analysis.

We now present two different second order schemes for the log-Heston process, for which we will be able to prove the effectiveness of the higher order approximations. The first one simply consists in sampling exactly each SDE and then using the scheme composition introduced by Strang, see e.g.~\cite[Corollary 2.3.14]{AA_book}. More precisely, let $P^B$ (respectively $P^W$) denote the semigroup associated to the SDE~\eqref{LB_SDEs} (resp.~\eqref{LW_SDEs}). We define
\begin{equation}\label{def_PEx}
  \hat{P}_t^{Ex}=P^B_{t/2}P^W_t P^B_{t/2}.
\end{equation}
Let us note that the exact scheme for~\eqref{LB_SDEs} is explicit and given by
\begin{equation}\label{varphiB}
  \varphi_B(t,x,y,N)=(x+(r-\rho a/\sigma)t -(1/2-\rho b/\sigma)yt+\brho\sqrt{ty}N,~y), \ \text{ with } N\sim\mcN(0,1).
\end{equation}
We  indeed have $P^B_tf(x,y)=\E[f(\varphi_B(t,x,y,N))]$ for all $f \in \CP{0}$.
For the SDE~\eqref{LW_SDEs}, we have $P^W_tf(x,y)=\E[f(x+\frac{\rho}{\sigma}(Y^y_t-y),Y^y_t)]$, where $Y^y$ is the solution of~\eqref{log-Heston_SDEs}. Thus, it amounts to simulate exactly the $Y^y_t$, and we refer to~\cite[Section 3.1]{AA_book} for a presentation of different exact simulation methods.

However, in practice, the exact simulation of the Cox-Ingersoll-Ross process is longer than a simple Gaussian random variable, and it can be interesting to use an approximation scheme. We use here the one introduced by Ninomiya and Victoir~\cite{NV}. Following Theorem 1.18 in~\cite{AA_MCOM}, we rewrite $\mcL_W=\mcL_0+\mcL_1$ where 
\begin{equation}\label{def_L0L1}
  \mcL_0 =(a-\frac{\sigma^2}{4}-by)(\frac{\rho}{\sigma}\partial_x + \partial_y ),\qquad \mcL_1 =\frac{y}{2}(\rho^2\partial^2_x + 2\rho\sigma\partial_x\partial_y+\sigma^2\partial^2_y) + \frac{\rho \sigma}{4}\partial_x+ \frac{\sigma^2}4 \partial_y,
\end{equation}
are the infinitesimal generator respectively associated to
$$\begin{cases} dX_t &= (\frac{\rho }{\sigma}(a-\sigma^2/4)-\frac{\rho b}{\sigma}Y_t) dt\\
dY_t &= (a-\sigma^2/4-bY_t) dt  \end{cases} \text{ and }
\begin{cases} dX_t &= \frac{\rho \sigma }{4} dt +\rho \sqrt{Y_t} dW_t\\
  dY_t &= \frac{\sigma^2}4 dt +\sigma \sqrt{Y_t} dW_t. \end{cases}$$
Let $\psi_b(t)=\frac{1-e^{-bt}}b$ (convention $\psi_b(t)=t$ for $b=0$) and define
\begin{align}
  \varphi_0(t,x,y)&=\Big(x-\frac{\rho b}{\sigma}\psi_b(t)y +\frac{\rho}{\sigma}\psi_b(t)(a-\frac{\sigma^2}{4}),~e^{-bt}y+\psi_b(t)(a-\frac{\sigma^2}{4})\Big), \label{varphi0} \\
  \varphi_1(t,x,y)&=\Big(x+\frac{\rho }{\sigma}\big((\sqrt{y}+\frac{\sigma t}{2})^2-y\big),~(\sqrt{y}+\frac{\sigma t}{2})^2\Big). \label{varphi1}
\end{align}
We have for $t \ge 0$ and $f \in \CP{0}$, 
\begin{equation}\label{def_P0P1}
  P^0_t f(x,y)= f(\varphi_0(t,x,y)) \text{ and } P^1_t f(x,y)= \E[f(\varphi_1(\sqrt{t}G,x,y))], \text{ with } G\sim \mathcal{N}(0,1).  
\end{equation}
Indeed, $\varphi_0$ is the exact solution of the ODE associated to~$\mathcal{L}_0$, starting from $(x,y)$ at time~$0$, and we can show by Itô calculus that $\varphi_1(W_t,x,y)$ is an exact solution of the SDE associated to~$\mathcal{L}_1$, starting from $(x,y)$ at time~$0$, and with the Brownian motion $d\tilde{W}_t=\text{sgn}\left(\sqrt{y}+\frac \sigma 2 W_t \right)dW_t$. The Ninomiya-Victoir scheme~\cite{NV} for $\mathcal{L}_W$ is then $P^0_{t/2}P^1_tP^0_{t/2}$, and we define \begin{equation}\label{def_PNV}
  \hat{P}_t^{NV}=P^B_{t/2}P^0_{t/2}P^1_tP^0_{t/2} P^B_{t/2}, \text{ when } \sigma^2\le 4a.
\end{equation}
This scheme is well-defined only for $\sigma^2\le 4a$, otherwise $\varphi_0$ may send the $y$ component to negative values, and the composition is not well-defined. This problem was pointed in~\cite{AA_MCOM} that introduces a second order scheme for any $\sigma\ge 0$. For this scheme,  the normal variable~$G$ in~\eqref{def_P0P1} is replaced by a bounded random variable that matches the five first moments of~$G$ and besides, an ad hoc discrete scheme is used in the neighborhood of~$0$. However, as indicated in the introduction of this subsection, this prevents us with our analysis to get~\eqref{H2_bar_heston} and thus approximations of higher order. This is why we only consider the Ninomiya-Victoir scheme here.

We now state  the main theorem of the paper.
\begin{theorem}\label{main_theorem}
 Let $\hat{P}_t$ be either $\hat{P}_t^{Ex}$ defined by~\eqref{def_PEx} or $\hat{P}_t^{NV}$ by~\eqref{def_PNV}. Let $T>0$, $n\in \N^*$ and $h_l=T/n^l$. Let $\cPh^{1,n}=\hat{P}_{h_1}^{[n]}$, $\cPh^{2,n}$ be defined by~\eqref{def_boost2} and  $\cPh^{\nu,n}$ the further approximations developed in~\cite{AB}. Let $\nu \ge 1$.  For any $f\in \CP{12\nu}$ $x\in \R$ and $y\ge 0$, we have
  $$\cPh^{\nu,n}f(x,y)-P_T f(x,y) = O(1/n^{2\nu}).$$ 
\end{theorem}
\begin{proof}
  Property~\eqref{H1_bar_heston} is proved in Corollary~\ref{cor_H1} and~\eqref{H2_bar_heston} in Lemma~\ref{lem_H2_estimate}. For $f\in \CP{12\nu}$, there exists $L$ such that $f\in \CPL{12\nu}{L}$. Let $\nu=1$. We get from~\eqref{eq_TT}, 
  $$\|\cPh^{1,n}f-P_Tf\|_{0,L}= \|\sum_{i=0}^{n-1}\hat{P}_{h_1}^{[n-(i+1)]}(\hat{P}_{h_1}-P_{h_1})P_{i{h_1}}f\|_{0,L}\le C^3T \|f\|_{12,L+3} \left(\frac{T}{n} \right)^{2} , $$
  since $\| \hat{P}_{h_1}^{[n-(i+1)]}(\hat{P}_{h_1}-P_{h_1})P_{i{h_1}}f \|_{0,L} \le C\|(\hat{P}_{h_1}-P_{h_1})P_{i{h_1}}f\|_{0,L} \le C^2\|P_{i{h_1}}f\|_{12,L+3} h_1^3 \le C^3\|f\|_{12,L+3} h_1^3$ by using~\eqref{H2_bar_heston}, then~\eqref{H1_bar_heston} and again~\eqref{H2_bar_heston}. This shows the claim for $\nu=1$. For $\nu=2$ (resp. $\nu \ge 3$), the claim is a consequence of~\eqref{boost2_estimate} (resp.~\eqref{def_Pnu}).
\end{proof}

\section{Proof of the main result}\label{Sec_proof}

\subsection{Preliminary results on the norm}

The next lemma gathers basic properties of the family of norms $\|\cdot\|_{k,L}$ defined in Equation~\eqref{def_NormCpolKL}.

\begin{lemma}\label{basic_lemma_CPL}
  Let $k,L\in\N$. We have the following basic properties:
  \begin{enumerate}
    \item $\CPL{k+1}{L}\subset \CPL{k}{L}$. For $f\in \CPL{k+1}{L}$, we have $\|f\|_{k,L}\le \|f\|_{k+1,L}$.
    \item Let $k,\a',\b'\in\N$. For $f\in\CPL{k+\a'+2\b'}{L}$ one has $\|\partial_x^{\a'} \partial_y^{\b'}  f\|_{k,L}\le \|f\|_{k+\alpha'+2\beta',L}$.
    \item $\CPL{k}{L}\subset \CPL{k}{L+1}$ and $\|f\|_{k,L+1}\le 3\|f\|_{k,L}$ for $f\in \CPL{k}{L}$.
    \item Let $\mcM_1$ be the operator defined by $\mcM_1 f(x,y)=yf(x,y)$. For $f\in \CPL{k}{L}$, we have  $\mcM_1f\in \CPL{k}{L+1}$  and $\| \mcM_1f\|_{k,L+1}\le \frac{3}{2}(k+1) \|f\|_{k,L}$.
    \item Let $\mcL$, $\mcL_B$, $\mcL_W$, $\mcL_0$ and $\mcL_1$ the infinitesimal generators defined in Equations~\eqref{log_Heston2_diff-op}, \eqref{def_LB}, \eqref{def_LW} and~\eqref{def_L0L1}. Then, for all $k\in \N$, there exists a constant $C(k)$ such that 
    \begin{align*} &\forall L\in \N, f\in\CPL{k+4}{L}, \ \|\mcL f\|_{k,L+1} + \|\mcL_W f\|_{k,L+1} + \|\mcL_1 f\|_{k,L+1}  \le C(k) \|f\|_{k+4 ,L}, \\
      &\forall L\in \N, f\in\CPL{k+2}{L}, \ \|\mcL_B f\|_{k,L+1}+\|\mcL_0 f\|_{k,L+1} \le C(k)  \|f\|_{k+2 ,L}.\end{align*}
  \end{enumerate}
\end{lemma}
\begin{proof}
  Property (1)-(2) are straightforward. We prove only (3), (4) and (5).

  \noindent(3) Let $a>0$ than $a^{2L}\le 1+a^{2L+2}$. So, we get  $\bff_L(x,y)=1+x^{2L}+y^{2L}\le 3(1+x^{2(L+1)}+y^{2(L+1)})=3\bff_{L+1}(x,y)$ and then $\frac{1}{\bff_{L+1}(x,y)}\le 3 \frac{1}{\bff_{L}(x,y)}$. This gives immediately the claim.

  \noindent(4) Let $f\in \CPL{k}{L}$.  Applying Leibniz rule, one obtains $\partial^\a_x\partial_y^\b \mcM_1f = \b\partial^\a_x\partial_y^{\b-1}f +\mcM_1\partial^\a_x\partial_y^\b f$ for $\a,\b \in \N$. Now, we write 
  \begin{align*}
    \frac{|\partial^\a_x\partial_y^\b[yf(x,y)]|}{\bff_{L+1}(x,y)}&\le \frac{\b|\partial^\a_x\partial_y^{\b-1} f(x,y)|}{\bff_{L+1}(x,y)} + \frac{y|\partial^\a_x\partial_y^\b f(x,y)|}{\bff_{L+1}(x,y)} \\
    &\le  \frac{3\b|\partial^\a_x\partial_y^{\b-1} f(x,y)|}{\bff_{L}(x,y)} +\frac{3|\partial^\a_x\partial_y^\b f(x,y)|}{2 \bff_{L}(x,y)},
  \end{align*}
  where we used the comparison above between $\bff_L$ and $\bff_{L+1}$  for the first term and, for the second term,  
  $y \bff_L(x,y) =y+y^{2L+1}+yx^{2L}+\le 1+y^{2L+2}+\frac{1+y^2}{2}x^{2L}\le \frac 32  \bff_{L+1}(x,y)$ by using $y+y^{2L+1}\le 1+y^{2L+2}$ and then Young's inequality. Then, we obtain 
  \begin{align*}
    \|\mcM_1 f\|_{k,L+1} &\le \sum_{\a+2\b\le k} \left(3\beta \sup_{(x,y)\in \R \times \R_+}\frac{3\b|\partial^\a_x\partial_y^{\b-1} f(x,y)|}{\bff_{L}(x,y)}  + \frac 3 2 \sup_{(x,y)\in \R \times \R_+}\frac{3\b|\partial^\a_x\partial_y^{\b} f(x,y)|}{\bff_{L}(x,y)} \right) \\
    &\le 3\left( \lfloor k/2 \rfloor +1/2 \right) \| f\|_{k,L}.
  \end{align*}
  \noindent(5) We prove only the estimate for $\mcL$, the others are obtained using the same arguments. We have $\|\mcL f\|_{k,L+1}\le \frac{1}{2}\|\mcM_1(\partial_x^2+2\rho\sigma\partial_x\partial_y+\sigma^2\partial_y^2-\partial_x-2b\partial_y) f\|_{k,L+1} + \|(r\partial_x+a\partial_y) f\|_{k,L+1}$, by using linearity of derivatives and the triangular inequality.  
  We conclude using property (4) and (2) for the first term,  (2) and (3) for the second and finally property number (1) to upper bound $\|\mcL f\|_{k,L+1}$ by $C(k)\|f\|_{k+4,L+1}$, where $C(k)$ is  a constant  depending on $k$ and on the parameters ($r,\rho,a,b$ and $\sigma$).
\end{proof}

\subsection{The Cauchy problem of the Log-Heston SDE}

In this subsection, we aim to prove the estimates on the derivatives of the Cauchy problem.  The representation formula presented below is a result of Briani, Caramellino and Terenzi~\cite{BCT}.

\begin{prop}\label{prop-rep-logHeston-estim}
  Let $k,L \in \mathbb{N}$ and suppose that $f \in \CPL{k}{L}$. Let $\lambda \ge 0$, $c,d \in \R$. Let $(X^{t, x, y}, Y^{t, y})$ be the solution to the SDE, for $s\ge t$, 
  \begin{equation}\label{log-Heston-ext_SDEs}
    \begin{cases}
      X^{t,x,y}_s &= x +\int_t^s (c+dY^y_r) dr + \int_t^s  \lambda \sqrt{Y^y_r} (\rho dW_r + \sqrt{1-\rho^2} dB_r)\\
      Y^{t,y}_s &= y+ \int_t^s (a-bY^y_r) dr +\sigma \int_t^s  \sqrt{Y^y_r} dW_r,
    \end{cases}
  \end{equation}
 and set
  $$ u(t, x, y)=\E[f(X_{T}^{t, x, y}, Y_{T}^{t, y})]=P_{T-t}f(x,y).$$
  Then, $u(t,\cdot,\cdot)\in \CPL{k}{L}$ and the following stochastic representation holds for $\a+2\beta \le k$,
  \small
  \begin{multline}\label{stoc_repr-new}
  \partial_{x}^{\a} \partial_{y}^{\b} u(t, x, y)=\E\bigg[e^{-\b b(T-t)} \partial_{x}^{\a} \partial_{y}^{\b} f\big(X_{T}^{\b, t, x, y}, Y_{T}^{\b,t, y}\big) \\
   \quad+\b \int_{t}^{T}e^{-\b b(s-t)}\Big(\frac{\lambda^2}{2} \partial_{x}^{\a+2} \partial_{y}^{\b-1} u + d \partial_{x}^{\a+1} \partial_{y}^{\b-1} u \Big)\big(s, X_{s}^{\b, t, x, y}, Y_{s}^{\b,t, y}\big) d s\bigg],
  \end{multline}
  where $\partial_{x}^{\a} \partial_{y}^{\b-1}  u:=0$ when $\beta=0$ and $(X^{\b, t, x, y}, Y^{\b, t, y}), \beta \geq 0$, denotes the solution starting from $(x, y)$ at time $t$ to the SDE \eqref{log-Heston-ext_SDEs}  with parameters
  \begin{equation}\label{parameters-new}
    \rho_{\b}=\rho,  \quad a_{\b}=a+\b \frac{\sigma^{2}}{2}, \quad b_{\beta}= b, \quad c_{\b}=r+\b \rho \sigma \lambda, \quad d_{\b}=d, \quad \sigma_{\b}=\sigma.
  \end{equation}
  Moreover, one has the following norm estimation for the semigroup   \begin{equation}
   \forall k, L \in \N, T>0, \  \exists C, \forall f \in  \CPL{k}{L}, t\in[0,T], \   \|P_tf\|_{k,L}\le  \|f\|_{k,L} e^{Ct}.
  \end{equation}
\end{prop}

\begin{proof}
Proposition~\ref{prop-rep-logHeston-estim} basically restates~\cite[Proposition 5.3]{BCT} in our framework (note that our space $\CPL{k}{L}$ already includes twice more derivatives in~$x$ than in $y$ and that we have added the scaling factor $\lambda$ for convenience). The only additional result is the norm estimate, which we prove now. 

Let $f \in  \CPL{k}{L}$. We will prove that for all $(\a,\b)$ such that $\alpha+2 \b \le k$, there exists a constant $C\in \R_+$ such that 
\begin{equation}\label{inter_estim}
  \sup_{x\in \R y \in \R_+} \frac{|\partial_x^\alpha \partial_y^\beta u(t,x,y)|}{\bff_L(x,y)} \le \sup_{x\in \R y \in \R_+} \frac{|\partial_x^\alpha \partial_y^\beta f(x,y)|}{\bff_L(x,y)}e^{C(T-t)}  + C \|f\|_{k,L}(T-t).
\end{equation}
Let us note that this implies $\sup_{x\in \R y \in \R_+} \frac{|\partial_x^\alpha \partial_y^\beta u(t,x,y)|}{\bff_L(x,y)} \le \tilde{C} \|f\|_{k,L}$, with $\tilde{C}=e^{CT}+CT$.

For $\beta=0$, the estimate is straightforward : from~\eqref{stoc_repr-new} and $f\in\CPL{k}{L}$, one has
$$
|\partial_{x}^{\a} u(t,x,y)| \le  \left(\sup_{x'\in \R y' \in \R_+} \frac{|\partial_x^\alpha f(x',y')|}{\bff_L(x',y')}  \right)\E \big[\bff_L(X^{t,x,y}_T,Y_T^{t,y})\big],
$$
and we get easily~\eqref{inter_estim} by using the estimate on moments~\eqref{eq_moments} that we prove below. 

Suppose now that the estimate \eqref{inter_estim} is true up to $\beta-1$, and let us prove it for $\beta$.
Using \eqref{stoc_repr-new}  and the induction hypothesis for the integral term, we get
\begin{align*}
  |\partial_{x}^{\a} \partial_{y}^{\b} u(t, x, y)|\le & 
  e^{-\beta b (T-t)} \left(\sup_{x'\in \R y' \in \R_+} \frac{|\partial_x^\alpha \partial_y^\beta f(x',y')|}{\bff_L(x',y')} \right)
  \E\left[ \bff_L(X^{\beta,t,x,y}_T,Y^{\beta,t,y}_T)  \right]\\
  &+\beta\frac{\lambda^2+|d|}{2}e^{\beta |b| T}  \tilde{C} \|f\|_{k,L} \int_t^T \E[\bff_L(X^{\beta,t,x,y}_s,Y^{\beta,t,y}_s)] ds.
\end{align*}
This gives~\eqref{inter_estim} by using again the estimate on the moments~\eqref{eq_moments}. This shows~\eqref{inter_estim} by induction, and we finally sum~\eqref{inter_estim} over $\alpha+2\beta\le k$ to get
$$\|P_{T-t}f\|_{k,L}\le \|f\|_{k,L} e^{C(T-t)}+C(k+1)^2 \|f\|_{k,L} (T-t)\le \|f\|_{k,L} e^{C(1+(k+1)^2)(T-t)},$$
proving the claim.
\end{proof}

\begin{lemma}\label{lem_moments}
  Let $(X^{x,y},Y^y)$ be the solution of~\eqref{log-Heston-ext_SDEs} starting from $(x,y)$ at time~$0$. Let $T>0$. For any $L\in \N$, there is a constant $C\in \R_+$ depending on $T$, $L$ and the SDE parameters, such that 
\begin{equation}\label{eq_moments}\E[\bff_L(X^{x,y}_t,Y^y_t)]\le e^{Ct}\bff_L (x,y), \ t\in [0,T].
\end{equation}
\end{lemma}

\begin{proof}
  We use the affine (and thus polynomial) property of the extended log-Heston process~\eqref{log-Heston-ext_SDEs}, see~\cite[Example 3.1]{CKRT}. By~\cite[Theorem 2.7]{CKRT}, we get that the log-Heston semigroup acts as a matrix exponential on the polynomial functions of degree lower than $2L$. This gives $\E[(X^{x,y}_t)^{2L}+(Y^y_t)^{2L}]=x^{2L}+y^{2L}+\sum_{i+j\le 2L}\varphi_{i,j}(t) x^i y^j$, with $\varphi_{i,j}\in \mathcal{C}^1([0,T],\R)$ such that $\varphi_{i,j}(0)=0$. Using that $|x|^i y^j\le |x|^{i+j}+ y^{i+j}\le 2 \bff_L(x,y)$ for $i+j\le 2L$ and using that $\varphi'_{i,j}$ is bounded on $[0,T]$, we get 
  $$\E[\bff_L(X^{x,y}_t,Y^y_t)]\le \bff_L(x,y)+ Ct \bff_L(x,y)\le \bff_L(x,y)e^{Ct}, $$
  with $C=2 \sum_{i,j\le 2L } \max_{[0,T]}|\varphi'_{i,j}|$. 
\end{proof}

\subsection{Proof of~\eqref{H1_bar_heston} and ~\eqref{H2_bar_heston}}

We start by proving the property~\eqref{H2_bar_heston} in the next lemma.
\begin{lemma}\label{lem_H2_estimate}
  Let $t\in[0,T]$,  $k,L\in\N$ and $f\in\CPL{k}{L}$. Let $\varphi_0 $ be the function defined in Equation~\eqref{varphi0}. Then,  there exists $C$ such that, for $I\in \{0,1,B,W\}$, $\|P^{I}_tf\|_{k,L}\le e^{Ct}\|f\|_{k,L}$, for $t\in [0,T]$. The semigroup approximations $\hat{P}^{Ex}_t$ and $\hat{P}^{NV}_t$ enjoy the same property and satisfy~\eqref{H2_bar_heston}.
\end{lemma}
\begin{proof}
We apply four times Proposition~\ref{prop-rep-logHeston-estim} with 
\begin{itemize}
  \item $\tilde{a}=a-\frac{\sigma^2}4$, $\tilde{b}
=b$, $\tilde{c}=\frac{\rho}{\sigma}\left(a-\frac{\sigma^2}4\right) $, $\tilde{d}=-b\frac{\rho}{\sigma}$, $\tilde{\lambda}=0$, $\tilde{\sigma}=0$ for $P^0$,
\item  $\tilde{a}=\frac{\sigma^2}4$, $\tilde{b}
=0$, $\tilde{c}=\frac{\rho \sigma}4 $, $\tilde{d}=0$, $\tilde{\lambda}=\rho$, $\tilde{\sigma}=\sigma$, $\tilde{\rho}=1$ for $P^1$,
\item $\tilde{a}=0$, $\tilde{b}
=0$, $\tilde{c}=r-\frac{\rho a}{\sigma} $, $\tilde{d}=\frac{\rho b}{\sigma} -\frac 12$, $\tilde{\lambda}=\bar{\rho}$,  $\tilde{\sigma}=0$, $\tilde{\rho}=0$ for $P^B$,
\item $\tilde{a}=a$, $\tilde{b}
=b$, $\tilde{c}=\frac{\rho a}\sigma $, $\tilde{d}=-\frac{\rho b}\sigma$, $\tilde{\lambda}=\rho$, $\tilde{\sigma}=\sigma$, $\tilde{\rho}=1$ for $P^W$,
\end{itemize}
where the tilde parameters are the ones used in Equation~\eqref{log-Heston-ext_SDEs}. This gives the first claim. Then, we deduce easily that $\|\hat{P}^{Ex}_tf\|_{k,L}\le e^{Ct/2} \|P^W_{t} P^B_{t/2}f\|_{k,L}\le e^{2Ct} \|f\|_{k,L}$ by using twice the estimate for $P^B$ and once for $P^W$. Similarly, we obtain $\|\hat{P}^{NV}_tf\|_{k,L}\le e^{3Ct} \|f\|_{k,L}$, by using the estimates for $P^B$, $P^0$ and $P^1$.

Now, the property~\eqref{H2_bar_heston} follows easily: consider $l\ge 1$ and $Q_l=\hat{P}^{NV}_{\frac{T}{n^l}}$, we have for $f \in\CPL{k}{L}$, $\|Q_l f\|_{k,L}\le e^{3C \frac{T}{n^l}}$ and thus for any $j\le n^l$, $\|Q^{[j]}_l f\|_{k,L}\le e^{3C \frac{jT}{n^l}}\le e^{3CT}$, which gives~\eqref{H2_bar_heston}. The same result holds for $\hat{P}^{Ex}$.  
\end{proof}

We now turn to the proof of the property~\eqref{H1_bar_heston}. We first state a general result on the composition of approximation schemes that fits our framework with the norm family $\|\cdot\|_{k,L}$.  It can be seen as a variant of~\cite[Proposition 2.3.12]{AA_book} and says, heuristically, that the composition of schemes works as a composition of operators. 

\begin{lemma}\label{lem_compo}(Scheme composition). Let $\nu \in \N$ and $T>0$. Let $V_i$, $i\in \{1,\dots,I\}$,  be infinitesimal generators such that  there exists $k_i,L_i \in \N$ such that
\begin{equation}\label{estim_V}
  \forall k \in \N, \exists C\in \R_+,  \forall f \in \CPL{k+k_i}{L}, V_i f \in  \CPL{k}{L+L_i} \text{ and } \|V_if\|_{k,L+L_i}\le C \|f\|_{k+k_i,L}.
\end{equation}
Let $k^\star=\max_{1\le i\le I} k_i$ and $L^\star=\max_{1\le i\le I} L_i$.
We assume that for any $i$, $\hat{P}^i_t:\CPL{0}{L} \to \CPL{0}{L} $ is such that
\begin{align} \forall k, L \in \N, 0\le  \bar{q}\le \nu+1, &\ \exists C,\ \forall f \in \CPL{k+\bar{q}k_i}{L}, \forall t\in[0,T], \notag \\& \|\hat{P}^i_t f -\sum_{q=0}^{\bar{q}-1}\frac{t^q}{q!} V_i^qf\|_{k, L+ \bar{q} L_i} \le C t^{\bar{q}} \|f\|_{k+\bar{q} k_i, L}. \label{assump_Phat}\end{align}
Then, we have for $\lambda_1,\dots,\lambda_I\in [0,1]$, 
\begin{align}
 & \forall k, L \in \N, 0\le  \bar{q}\le \nu+1, \ \exists C, \forall f \in \CPL{k+\bar{q}k^\star}{L}, \forall t \in [0,T]  \notag \\
  & \left\|\hat{P}^I_{\lambda_I t} \dots \hat{P}^1_{\lambda_1 t} f -\sum_{q_1+\dots+q_I\le \bar{q}-1}\frac{\lambda_1^{q_1}\dots \lambda_I^{q_I} t^{q_1+\dots+q_I} }{q_1!\dots q_I!} V_I^{q_I}\dots V_1^{q_1}f\right\|_{k, L+ \bar{q} L^\star} \le C t^{\bar{q}} \|f\|_{k+\bar{q} k^\star, L}. \label{compo_oper}
\end{align}
\end{lemma}
\begin{proof} For readability, we make the proof with $I=2$ operators. 
  Let $\bar{q}\le \nu+1$ and $f\in \CPL{k+\bar{q}k^\star}{L}$. We define $R^1f=\hat{P}^1_{\lambda_1 t} f -\sum_{q_1=0}^{\bar{q}-1}\frac{\lambda_1^{q_1} t^{q_1}}{q_1!} V_1^{q_1}f$. For $t\in [0,T]$, we have $\lambda_1 t \in [0,T]$  since $\lambda_1 \in [0,1]$ and by assumption~\eqref{assump_Phat}, we have $R^1 f \in \CPL{k+\bar{q} k^\star - \bar{q} k_1 }{L+\bar{q}L_1}$ and 
  $$\|R^1 f\|_{k+\bar{q}k^\star - \bar{q} k_1,L+\bar{q}L_1}\le C t^{\bar{q}} \|f\|_{k+\bar{q}k^\star,L}.$$  
  We now write 
  $$ \hat{P}^2_{\lambda_2 t} \hat{P}^1_{\lambda_1 t} f= \sum_{q_1=0}^{\bar{q}-1}\frac{\lambda_1^{q_1} t^{q_1}}{q_1!} \hat{P}^2_{\lambda_2 t} V_1^{q_1}f + \hat{P}^2_{\lambda_2 t} R^1 f.$$
 Since $V_1^{q_1}f \in \CPL{k+\bar{q}k^\star-q_1k_1}{L+q_1L_1}$, we  apply~\eqref{assump_Phat} to get $$\hat{P}^2_{\lambda_2 t} V_1^{q_1}f = \sum_{q_2=0}^{\bar{q}-q_1-1} \frac{\lambda_2^{q_2} t^{q_2}}{q_2!} V_2^{q_2}V_1^{q_1}f + R^2_{q_1}f,$$ with
 $\|R^2_{q_1} f \|_{k+\bar{q}k^\star-q_1k_1-(\bar{q}-q_1)k_2,L+q_1L_1+(\bar{q}-q_1)L_2}\le C t^{\bar{q}-q_1}\|f\|_{k+\bar{q}k^\star,L}$ by~\eqref{estim_V} and~\eqref{assump_Phat}. We also have $\|\hat{P}^2_{\lambda_2 t} R^1 f\|_{k+\bar{q}k^\star - \bar{q} k_1,L+\bar{q}L_1}\le C t^{\bar{q}} \|f\|_{k+\bar{q}k^\star,L}$ by~\eqref{assump_Phat}. Since  
 $$k+\bar{q}k^\star-q_1k_1-(\bar{q}-q_1)k_2\ge k, \ L+q_1L_1+(\bar{q}-q_1)L_2\le L+\bar{q} L^\star, $$
 for all $0\le q_1\le \bar{q}-1$, and using Lemma~\ref{basic_lemma_CPL} (1 and 3), we get
 $$ \left\|\hat{P}^2_{\lambda_2 t} \hat{P}^1_{\lambda_1 t} f -\sum_{q_1+q_2\le \bar{q}-1}\lambda_1^{q_1}\lambda_2^{q_2}\frac{t^{q_1}t^{q_2}}{q_1!q_2!} V_2^{q_2} V_1^{q_1}f \right\|_{k, L+ \bar{q} L^\star} \le C t^{\bar{q}} \|f\|_{k+\bar{q} k^\star, L}. \qedhere$$
\end{proof}

\begin{lemma}\label{lem_expan_CPL_scheme} 
   Let $L_0=L_1=L_B=L_W=L_H=1$, $k_0=k_B=2$, $k_1=k_W=k_H=4$. Let denote $\mcL_H=\mcL$ and $P^H_t=P_t$ the log-Heston semigroup.  Let $i\in \{0,1,B,W,H\}$. We have 
  $$ \forall k,L \in \N, \exists C\in \R_+, \forall f \in \CPL{k+k_i}{L}, \  \|\mathcal{L}_if\|_{k,L+L_i}\le C \|f\|_{k+k_i,L}. $$ 
  Besides, for any $\bar{q}\in \N$, we have 
  \begin{align*} \forall k, L \in \N,  \ \exists C,\ \forall f \in \CPL{k+\bar{q}k_i}{L}, \notag  \|P^i_t f -\sum_{q=0}^{\bar{q}-1}\frac{t^q}{q!} \mathcal{L}_i^qf\|_{k, L+ \bar{q} L_i} \le C t^{\bar{q}} \|f\|_{k+\bar{q} k_i, L}. \end{align*}
\end{lemma}
\begin{proof}
  The first part of the statement is proved in Lemma~\ref{basic_lemma_CPL}. For $\bar{q}=0$, the estimate is simply the one given by Lemma~\ref{lem_H2_estimate} (or Proposition~\ref{prop-rep-logHeston-estim} for $P^H_t$).

  We now consider $\bar{q}\ge 1$. As already pointed in the proof of Lemma~\ref{lem_H2_estimate}, each operator is the infinitesimal generator of~\eqref{log-Heston-ext_SDEs} with a suitable choice of parameter. Then, by applying Itô's formula and taking the expectation, we get $P^{i}_tf(x,y)=f(x,y)+\int_0^t P^i_s \mathcal{L}_if(x,y)ds$. By iterating,  we obtain for $f \in \CPL{k+\bar{q}k_i}{L}$,
  \begin{equation}\label{expan_P_general}
    P^{i}_tf(x,y) = \sum_{j=0}^{\bar{q}-1} \frac{t^j}{j!} \mcL_{i}^jf(x,y) + \int_0^{t} \frac{(t-s)^{\bar{q}-1}}{(\bar{q}-1)!} P^{i}_{s} \mcL^{\bar{q}}_{i}f(x,y)ds.  
  \end{equation}
We have $\|\mcL^{\bar{q}}_{i}f\|_{k, L+\bar{q} L_i}\le C^{\bar{q}} \|f\|_{k+\bar{q} k_i, L}$ by Lemma~\ref{basic_lemma_CPL}-(5) and thus $\| P^i_s \mcL^{\bar{q}}_{i}f \|_{k, L+\bar{q} L_i}\le C^{q+1} \| f \|_{k+\bar{q} k_i, L}$ for $s\in [0,T]$, by using the result for $\bar{q}=0$.  Therefore, we get by the triangle inequality
\begin{align*}
  \left\|P^i_tf-\sum_{j=0}^{\bar{q}-1} \frac{t^j}{j!} \mcL_{i}^jf\right\|_{k,L+\bar{q}L_i}\le  \int_0^t \frac{(t-s)^{\bar{q}-1}}{(\bar{q}-1)!}C^{\bar{q}+1} \| f \|_{k+\bar{q} k_i, L} ds= \frac{t^{\bar{q}}}{\bar{q}!}C^{\bar{q}+1} \| f \|_{k+\bar{q} k_i, L}.\quad \qedhere
\end{align*}
\end{proof}

\begin{corollary}\label{cor_H1}
  Let $T>0$. Let $\hat{P}_t$ denote either $\hat{P}_t^{Ex}$ or $\hat{P}_t^{NV}$. We have, for $\bar{q}\le 3$,  
  \begin{align*} \forall k, L \in \N,  \ \exists C,\ \forall f \in \CPL{k+4\bar{q}}{L}, \forall t \in [0,T],  \|\hat{P}_t f -\sum_{q=0}^{\bar{q}-1}\frac{t^q}{q!} \mathcal{L}^qf\|_{k, L+ \bar{q} } \le C t^{\bar{q}} \|f\|_{k+4\bar{q}, L}, \end{align*}
  and~\eqref{H1_bar_heston} holds.
\end{corollary}
\begin{proof}
We prove the result for $\hat{P}_t^{Ex}$, the argument is similar for $\hat{P}^{NV}_t$.   
We use Lemma~\ref{lem_expan_CPL_scheme} for $P^W_t$ and $P^B_t$ and apply then Lemma~\ref{lem_compo}. For $\bar{q}=0,1,2$, we get easily the claim. For $\bar{q}=3$, we get since $\hat{P}^{Ex}_t=P^B_{t/2}P^W_{t/2}P^B_{t/2}$, 
$$\left\|\hat{P}^{Ex}_t f -\sum_{q_1+q_2+q_3\le 2} \frac{(1/2)^{q_1+q_3} t^{q_1+q_2+q_3} }{q_1!q_2 q_3!} \mathcal{L}_B^{q_3}\mathcal{L}_W^{q_2}\mathcal{L}_B^{q_1}f\right\|_{k, L+ 3} \le C t^{3} \|f\|_{k+12, L}.$$
The term of order two is 
$$ \frac{1}8 \mathcal{L}_B^{2}f+ \frac{1}4 \mathcal{L}_B^{2}f+\frac{1}8 \mathcal{L}_B^{2}f+  \frac 12 \mathcal{L}_B\mathcal{L}_W+ \frac 12 \mathcal{L}_W\mathcal{L}_B +\frac{1}2 \mathcal{L}_W^{2}f=\frac 12 (\mathcal{L}_B+\mathcal{L}_W)^2f, $$
and thus $\sum_{q_1+q_2+q_3\le 2} \frac{(1/2)^{q_1+q_3} t^{q_1+q_2+q_3} }{q_1!q_2 q_3!} \mathcal{L}_B^{q_3}\mathcal{L}_W^{q_2}\mathcal{L}_B^{q_1}f=\sum_{q=0}^2 \frac {t^q}{q!}(\mathcal{L}_B+\mathcal{L}_W)^qf=\sum_{q=0}^2 \frac {t^q}{q!}\mathcal{L}^qf$.

Now, we use Lemma~\ref{lem_expan_CPL_scheme} to get $\left\|P_t f - \sum_{q=0}^2 \frac {t^q}{q!}\mathcal{L}^qf \right\|_{k, L+ 3} \le C t^{3} \|f\|_{k+12, L}$. The triangular inequality then gives 
$$ \left\|P_t f - \hat{P}^{Ex}_t \right\|_{k, L+ 3} \le C t^{3} \|f\|_{k+12, L},$$
which is precisely~\eqref{H1_bar_heston}.
\end{proof}

\section{Numerical results}\label{Sec_num}

\subsection{Implementation}\label{Subsec_implementation}
We explain in this subsection the implementation of the schemes associated to~$\hat{P}_t^{Ex}$ and $\hat{P}_t^{NV}$, and then of the Monte-Carlo estimator of $\cPh^{\nu,n}$, $\nu \in \{1,2\}$. We will note either $\cPh^{Ex,\nu,n}$ or $\cPh^{NV,\nu,n}$ to emphasize what semigroup approximation is used.

On a single time step, the scheme associated to $\hat{P}_t^{NV}$ is given by
\begin{align*} 
  \hX^{x,y}_t &= x +(r-\frac{\rho}{\sigma}a)t +\frac{\rho}{\sigma}(\hY^y_t-y)  +(\frac{\rho}{\sigma}b-\frac{1}{2})\frac{y+\hY^y_t}{2}t +\sqrt{(1-\rho^2)\frac{t}{2}}\bigg(\sqrt{y}N_1+\sqrt{\hY^y_t}N_2\bigg),\\
  \hY^y_t &= (a-\frac{\sigma^2}{4})\psi_b(\frac{t}{2})+e^{-b\frac{t}{2}}\bigg( \sqrt{(a-\frac{\sigma^2}{4})\psi_b(\frac{t}{2})+e^{-b\frac{t}{2}}y} +\frac{\sigma\sqrt{t}}{2}G \bigg)^2,
\end{align*}
where $N_1,N_2,G$ are three independent random variables with the standard normal distribution. It is obtained from the composition~\eqref{def_PNV} and by using accordingly the maps $\varphi_0$, $\varphi_1$ and $\varphi_B$ that represent the semigroups, see Equations~\eqref{varphiB} and~\eqref{def_P0P1}.

One should remark however that the conditional law of $\hX^{(x,y)}_t$ given  $\hY^y_t$ is normal with mean $x+\frac{\rho}{\sigma}(\hY^y_t-y)+(r-\frac{\rho}{\sigma}b)t +(\frac{\rho}{\sigma}b-\frac{1}{2})\frac{y+\hY^y_t}{2}t$ and variance $t(1-\rho^2)(y+\hY^y_t)/2 $. Therefore, we rather consider the following probabilistic representation, that has the same law and requires to simulate one standard Gaussian random variable $N$ instead of the couple $(N_1,N_2)$ for the first component:
\begin{align*}
  \hX^{x,y}_t &= x +(r-\frac{\rho}{\sigma}a)t +\frac{\rho}{\sigma}(\hY^y_t-y) +(\frac{\rho}{\sigma}b-\frac{1}{2})\frac{y+\hY^y_t}{2}t +\sqrt{(1-\rho^2)\frac{y+\hY^y_t}{2}t}N, \\
  \hY^y_t &= (a-\frac{\sigma^2}{4})\psi_b(\frac{t}{2})+e^{-b\frac{t}{2}}\bigg( \sqrt{(a-\frac{\sigma^2}{4})\psi_b(\frac{t}{2})+e^{-b\frac{t}{2}}y} +\frac{\sigma\sqrt{t}}{2}G \bigg)^2.
\end{align*}
We note $\varphi^{NV}(t,x,y,N,G):=(\hX^{x,y}_t,\hY^{y}_t)$ this map. 
The same trick can be used for~$\hat{P}^{Ex}_t$ when the exact simulation is used for the CIR component, and we define 
$$\varphi_X^{Ex}(t,x,y,N,Y^y_t)=x +(r-\frac{\rho}{\sigma}a)t +\frac{\rho}{\sigma}(Y^y_t-y) +(\frac{\rho}{\sigma}b-\frac{1}{2})\frac{y+Y^y_t}{2}t +\sqrt{(1-\rho^2)\frac{y+Y^y_t}{2}t}N,$$
the map that gives the log-stock component. 

We now explain how to get the Monte-Carlo estimator for $\cPh^{1,n}$ and then $\cPh^{2,n}$. We start with the simulation scheme for $\hat{P}^{Ex}_t$. Let us consider $T>0$, $h_1=T/n$ and the regular time grid $\Pi^0=\{kh_1, \ 0\le k\le n\}$. We simulate exactly $Y_{kh_1}$, $1\le k\le n$, the CIR component starting from $Y_0=y$, and we set
$$\hat{X}^{Ex,0}_{k h_1}=\varphi_X^{Ex,0}(h_1,\hat{X}^{Ex,0}_{(k-1) h_1},Y_{(k-1)h_1},N_k,Y_{kh_1}), \ 1\le k\le n,$$
where $(N_k)_{1\le k\le N}$ are standard normal random variable such that $N_k$ is independent of $(N_{k'})_{k'<k}$ and the process $Y$. 
The Monte-Carlo estimator of $\cPh^{Ex,1,n}$ is then 
$$\frac{1}{M_1}\sum_{m=1}^{M_1} f(\hat{X}^{Ex,0,(m)}_{T},Y^{(m)}_T),$$
where $M_1$ is the number of independent samples. We now present how to calculate the correcting term in $\cPh^{2,n}$. To do so, we draw an independent random variable~$\kappa$ that is uniformly distributed on $\{0,\dots,n-1\}$ and selects the time-step to refine. We note $\Pi^1=\Pi^0 \cup  \{ \kappa h_1 + k' h_2 , 1 \leq k' \leq n-1 \}$ the refined (random) grid, where $h_2=T/n^2$. We simulate exactly~$Y$ on this time grid and define the scheme $\hat{X}^{Ex,1}$ as follows:
\begin{align*}
  &\hat{X}^{Ex,1}_{k h_1}=\hat{X}^{Ex,0}_{k h_1} \text{ for } k\le \kappa,\\
  &\hat{X}^{Ex,1}_{\kappa h_1 +k'h_2}=\varphi_X^{Ex}(h_2,\hat{X}^{Ex,1}_{\kappa h_1 +(k'-1)h_2},Y_{\kappa h_1 +(k'-1)h_2},\tilde{N}_{k'},Y_{\kappa h_1 +k'h_2}), \ 1\le k'\le n, \\
  &\hat{X}^{Ex,1}_{k h_1}=\varphi_X^{Ex,1}(h_1,\hat{X}^{Ex,1}_{(k-1) h_1},Y_{(k-1)h_1},N_k,Y_{kh_1}), \ \kappa+2< k\le n,
\end{align*}
where $(\tilde{N}_{k'})_{1\le k'\le N}$ are i.i.d. random normal variable, independent of $\kappa$ and $(N_k,Y_{kh_1})_{k\le \kappa}$. We then define the Monte-Carlo estimator of $\cPh^{Ex,2,n}$ (see Eq.~\eqref{boost2_RG}) by
$$\frac{1}{M_1}\sum_{m=1}^{M_1} f(\hat{X}^{Ex,0,(m)}_{T},Y^{(m)}_T)+\frac{1}{M_2}\sum_{m=1}^{M_2} n\left(f(\hat{X}^{Ex,1,(m)}_{T},Y^{(m)}_T)-f(\hat{X}^{Ex,0,(m)}_{T},Y^{(m)}_T)\right).$$
Note that we reuse the same Monte-Carlo samples in the two sums as it has been observed in~\cite[Subsection 6.3]{AL} that it is more efficient. The tuning of the parameters $M_1$ and $M_2$ is made to minimize the computational cost to achieve a given precision, see~\cite[Eq. (6.11)]{AL} for the details. Let us stress here that it is important for the variance of the estimator to use the same noise for the simulations of $\hat{X}^{Ex,1}$ and $\hat{X}^{Ex,0}$. In particular, the normal random variable $N_{\kappa+1}$ should depend on $(\tilde{N}_k)_{1\le k\le N}$. A natural choice is to take
$$N_{\kappa+1}=N^{\textup{st}}_{\kappa+1} \text{ where }N^{\textup{st}}=\frac 1{\sqrt{n}}\sum_{k=1}^n\tilde{N}_k,$$
if we think of Brownian increments. We will discuss this choice later on in Subsection~\ref{Subsec_coupling}.

Let us now present the scheme for~$\hat{P}^{NV}_t$, that is well-defined for $\sigma^2\le 4a$. The scheme on the coarse grid~$\Pi^0$ is defined by
$$(\hat{X}^{NV,0}_{kh_1},\hat{Y}^{NV,0}_{kh_1})=\varphi^{NV}(h_1,\hat{X}^{NV,0}_{(k-1)h_1},\hat{Y}^{NV,0}_{(k-1)h_1},N_k,G_k), \ 1\le k\le n,$$
where $N_k,G_k$, $1\le k\le n$, are two independent standard normal variables independent of $(N_{k'},G_{k'})_{k'<k}$. The Monte-Carlo estimator of~$\cPh^{NV,1,n}$ is then $\frac{1}{M_1}\sum_{m=1}^{M_1} f(\hat{X}^{NV,0,(m)}_{T},\hat{Y}^{NV,0,(m)}_T)$. The scheme on the refined random grid~$\Pi^1$ is defined by 
\begin{align*}
  &(\hat{X}^{NV,1}_{k h_1},\hat{Y}^{NV,1}_{k h_1})=(\hat{X}^{NV,0}_{k h_1},\hat{Y}^{NV,0}_{k h_1}) \text{ for } k\le \kappa,\\
  &(\hat{X}^{NV,1}_{\kappa h_1 +k'h_2},\hat{Y}^{NV,1}_{\kappa h_1 +k'h_2})=  \varphi^{NV}(h_2,\hat{X}^{NV,1}_{\kappa h_1 +(k'-1)h_2},\hat{Y}^{NV,1}_{\kappa h_1 +(k'-1)h_2},\tilde{N}_{k'},\tilde{G}_{k'}), \ 1\le k'\le n, \\
  &(\hat{X}^{NV,1}_{k h_1},\hat{Y}^{NV,1}_{k h_1})=\varphi^{NV}(h_1,\hat{X}^{NV,1}_{(k-1)h_1},\hat{Y}^{NV,1}_{(k-1)h_1},N_k,G_k), \ \kappa+2< k\le n,
\end{align*}
where $(\tilde{N}_{k'},\tilde{G}_{k'})_{1\le k'\le N}$, are independent standard normal variables that are also independent of $\kappa$ and $(N_k,G_{k})_{k\le \kappa}$. The Monte-Carlo estimator of $\cPh^{NV,2,n}$  is then defined by
$$\frac{1}{M_1}\sum_{m=1}^{M_1} f(\hat{X}^{NV,0,(m)}_{T},Y^{NV,0,(m)}_T)+\frac{1}{M_2}\sum_{m=1}^{M_2} n\left(f(\hat{X}^{NV,1,(m)}_{T},\hat{Y}^{NV,1,(m)}_{T})-f(\hat{X}^{NV,0,(m)}_{T},Y^{NV,0,(m)}_T)\right).$$
Again, to reduce the variance of the estimator, it is important to use the same noise for the coarse and the refined grids. In particular, we take for the scheme $(\hat{X}^{NV,0}_{kh_1},\hat{Y}^{NV,0}_{kh_1})$ on the coarse grid 
$$ N_{\kappa+1}= N^{\textup{st}} \text{ and }G_{\kappa+1}= G^{\textup{st}}:= \frac{1}{\sqrt{n}} \sum_{k=1}^n \tilde{G}_k.$$
Another choice will be considered for $N_{\kappa+1}$ in Subsection~\ref{Subsec_coupling}, but we will always use $G_{\kappa+1}= G^{\textup{st}}_{\kappa+1}$ in our experiments.

\begin{figure}[h!]
  \centering
  \begin{subfigure}[h]{0.49\textwidth}
    \centering
    \includegraphics[width=\textwidth]{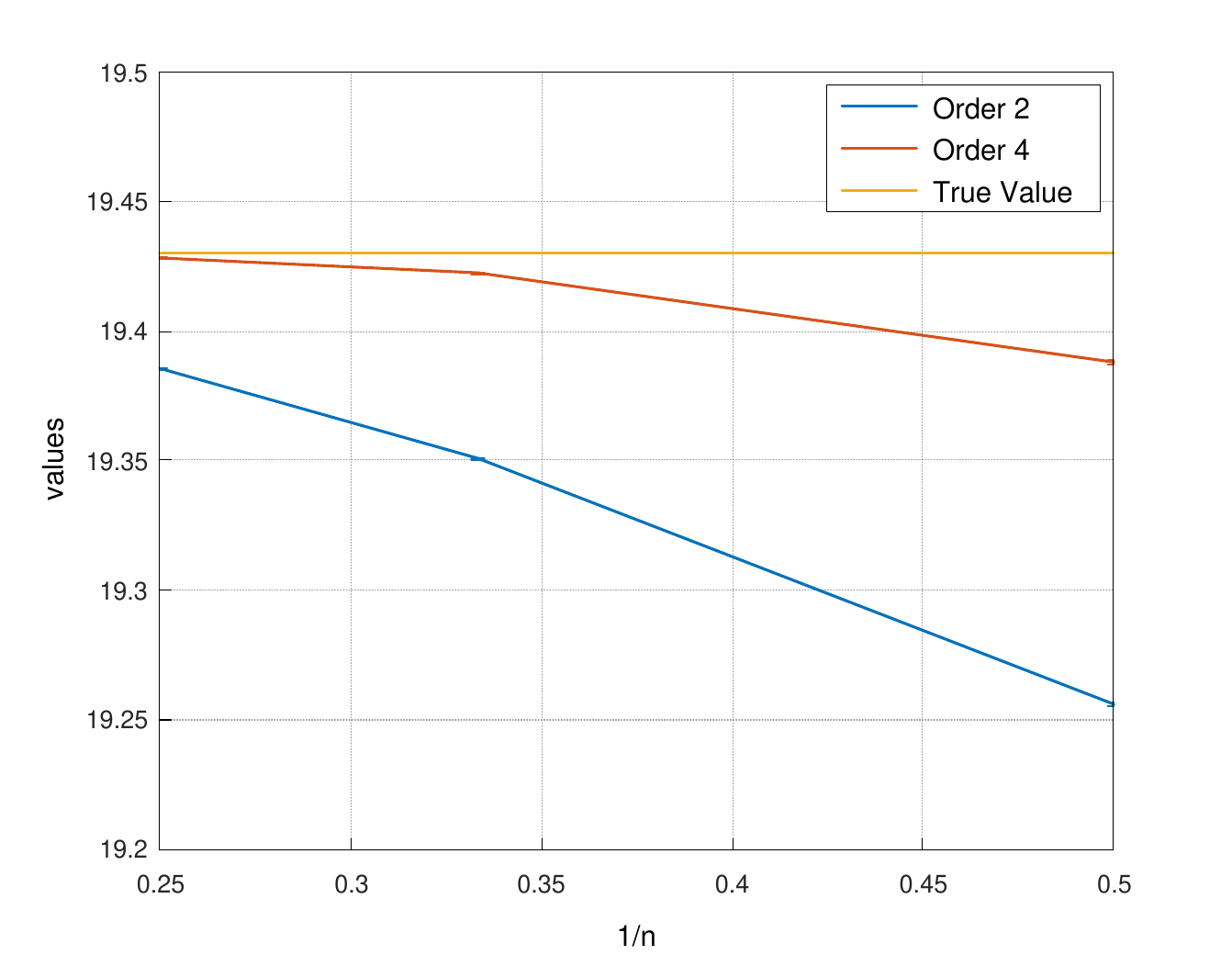}
    \caption{Values plot}
    \label{fig:values_plot_heston1}
  \end{subfigure}
  \hfill
  \begin{subfigure}[h]{0.49\textwidth}
    \centering
    \includegraphics[width=\textwidth]{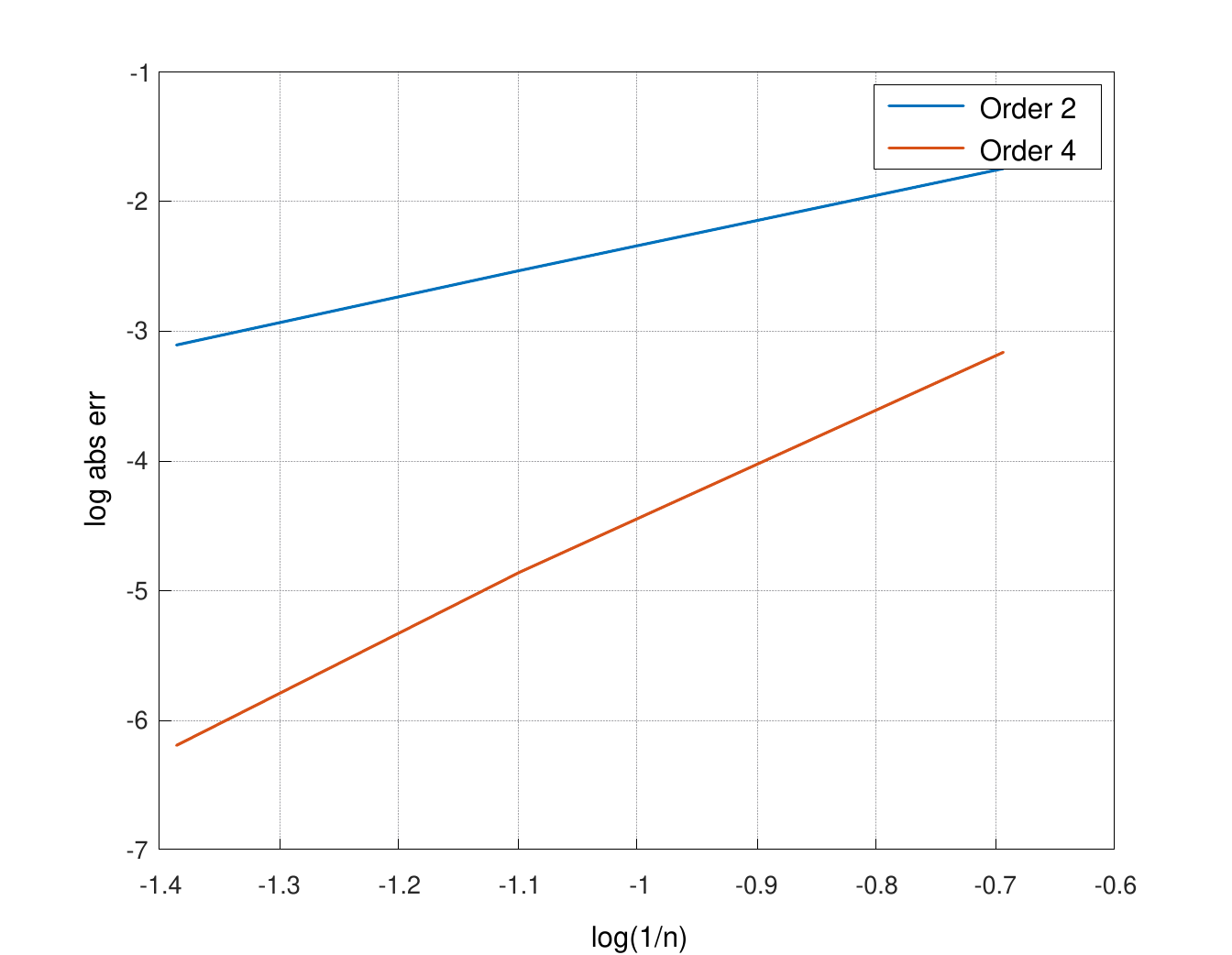}
    \caption{Log-log plot}
    \label{fig:log-log_plot_heston1}
  \end{subfigure}
  \caption{Test function: $f(x,y)=(K-e^x)^+$. Parameters: $S_0=e^{x}=100$, $r=0$, $y=0.2$, $a=0.2$, $b=1$, $\sigma=0.5$, $\rho=-0.7$, $T=1$, $K=105$. Statistical precision $\varepsilon=5$e-4.
  Graphic~({\sc a}) shows the Monte Carlo estimated values of $\cPh^{NV,1,n}f$, $\cPh^{NV,2,n}f$ as a function of the time step $1/n$  and the exact value. Graphic~({\sc b}) draws $\log(|\cPh^{NV,\nu,n}f-P_Tf|)$ in function of $\log(1/n)$: the regressed slopes are 1.89 and 4.27 for the second and fourth order respectively.}\label{Heston_orders}
\end{figure}

\begin{figure}[h!]
  \centering
  \begin{subfigure}[h]{0.49\textwidth}
    \centering
    \includegraphics[width=\textwidth]{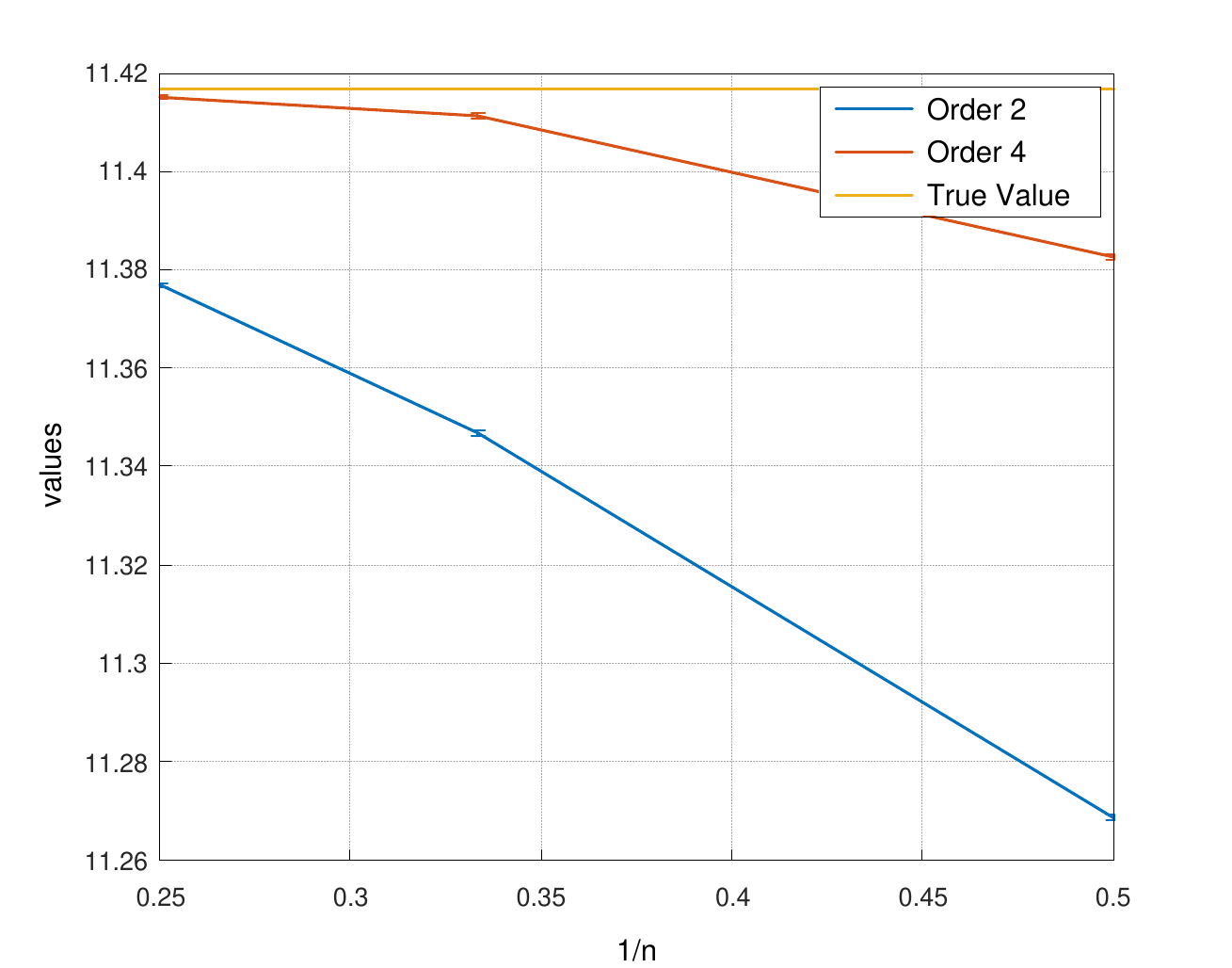}
    \caption{Values plot}
    \label{fig:values_plot_hestonCF1}
  \end{subfigure}
  \hfill
  \begin{subfigure}[h]{0.49\textwidth}
    \centering
    \includegraphics[width=\textwidth]{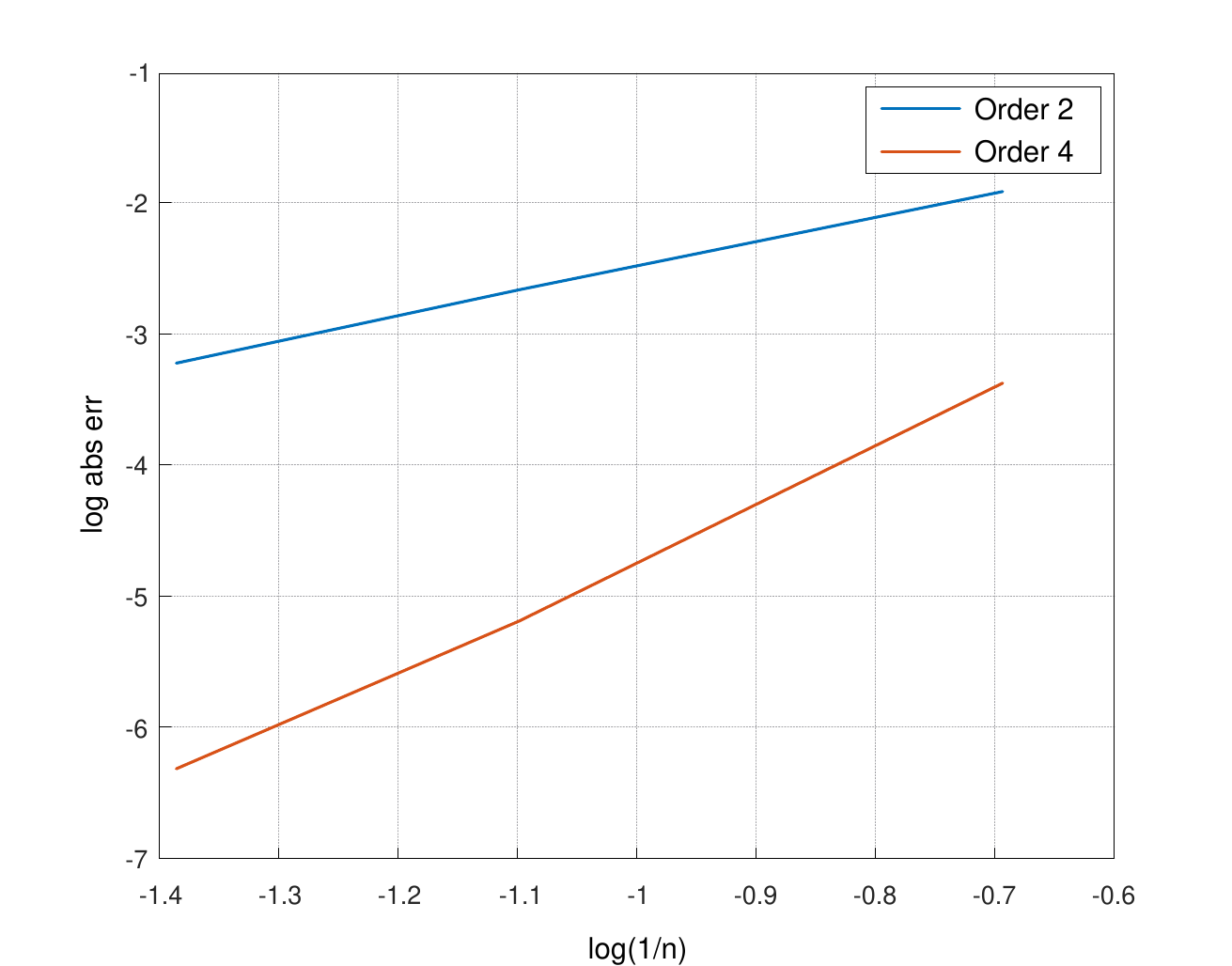}
    \caption{Log-log plot}
    \label{fig:log-log_plot_hestonCF1}
  \end{subfigure}
  \caption{Test function: $f(x,y)=(K-e^x)^+$. Parameters: $S_0=e^x=100$, $r=0$, $y=0.1$, $a=0.1$, $b=1$, $\sigma=1.0$, $\rho=-0.9$, $T=1$, $K=105$. Statistical precision $\varepsilon=5$e-4.
  Graphic~({\sc a}) shows the Monte Carlo estimated values of $\cPh^{Ex,1,n}f$, $\cPh^{Ex,2,n}f$ as a function of the time step $1/n$  and the exact value. Graphic~({\sc b}) draws $\log(|\hat{P}^{Ex,\nu,n}f-P_Tf|)$ in function of $\log(1/n)$: the regressed slopes are 1.89 and 4.26 for the second and fourth order respectively.}\label{HestonCF_orders}
\end{figure}

\subsection{Pricing of European and Asian options}\label{Subsec_Pricing}
We present in Figure~\ref{Heston_orders} the convergence of the approximations $\cPh^{NV,1,n}$ and $\cPh^{NV,2,n}$ for the price of a European option in a case where $\sigma^2\le 4a$. On the left graphic, we draw the values in function of the time step and the exact value of the option price $P_Tf$, that can be calculated with Fourier transform techniques. On the right graphic is plotted the log error in function of the log time step: the estimated slopes are in line with the theoretical order of convergence (2 and 4), even though the test function $f(x)=(K-e^x)_+$ is not as regular as required by Theorem~\ref{main_theorem}.  In Figure~\ref{HestonCF_orders}, we illustrate similarly the convergence of the approximations $\cPh^{Ex,1,n}$ and $\cPh^{Ex,2,n}$ for the price of a European option in a case where $\sigma^2\gg4a$. Again, we observe the theoretical rates of convergence given by Theorem~\ref{main_theorem}.

We now consider the case of Asian options, for which we need to simulate a third coordinate: $\mathcal{I}_t=\int_0^t S^{s,y}_u du=\int_0^t e^{X^{x,y}_u} du$. We explain how to simulate this coordinate for $\hat{P}^{Ex}$, and we do exactly the same then for $\hat{P}^{NV}$. We approximate the integral $\mathcal{I}_t$ by the trapezoidal rule. This gives 
\begin{align*}
  &\hat{\cI}^{Ex,0}_{k h_1} =  \hat{\cI}^{Ex,0}_{(k-1) h_1} + \frac{e^{\hat{X}^{Ex,0}_{(k-1) h_1}}+e^{\hat{X}^{Ex,0}_{k h_1}}}2 h_1 , \ 1\le k \le n, \\
  &\hat{\cI}^{Ex,1}_{k h_1}=\hat{\cI}^{Ex,0}_{k h_1}, \ 0\le k \le \kappa, \\
  &\hat{\cI}^{Ex,1}_{\kappa h_1 +k'h_2}= \hat{\cI}^{Ex,1}_{\kappa h_1 +(k'-1) h_2} +\frac{e^{\hat{X}^{Ex,1}_{\kappa h_1+(k'-1)h_2}}+e^{\hat{X}^{Ex,1}_{\kappa h_1 +k'h_2}} }2 h_2 , \ 1\le k'\le n, \\
  &\hat{\cI}^{Ex,1}_{k h_1}= \hat{\cI}^{Ex,1}_{(k-1) h_1} +\frac{e^{\hat{X}^{Ex,1}_{(k-1) h_1 }}+e^{\hat{X}^{Ex,1}_{kh_1}} }2 h_1, \ \kappa+2< k\le n,
\end{align*}
with $\hat{\cI}^{Ex,0}_{0}=\hat{\cI}^{Ex,1}_{0}=0$.
Let us mention here that the trapezoidal rule corresponds to the Strang splitting for the generator $\mathcal{L}+e^x\partial_{\mathcal{I}}$. Our formalism would allow to analyse the convergence rate for the Strang splitting for $\mathcal{L}+h(x)\partial_{\mathcal{I}}$, when $h$ is smooth with derivatives of polynomial growth. The exponential function does not fit this condition, and we analyse here the convergence on numerical experiments. 

\begin{figure}[h!]
  \centering
  \begin{subfigure}[h]{0.49\textwidth}
    \centering
    \includegraphics[width=\textwidth]{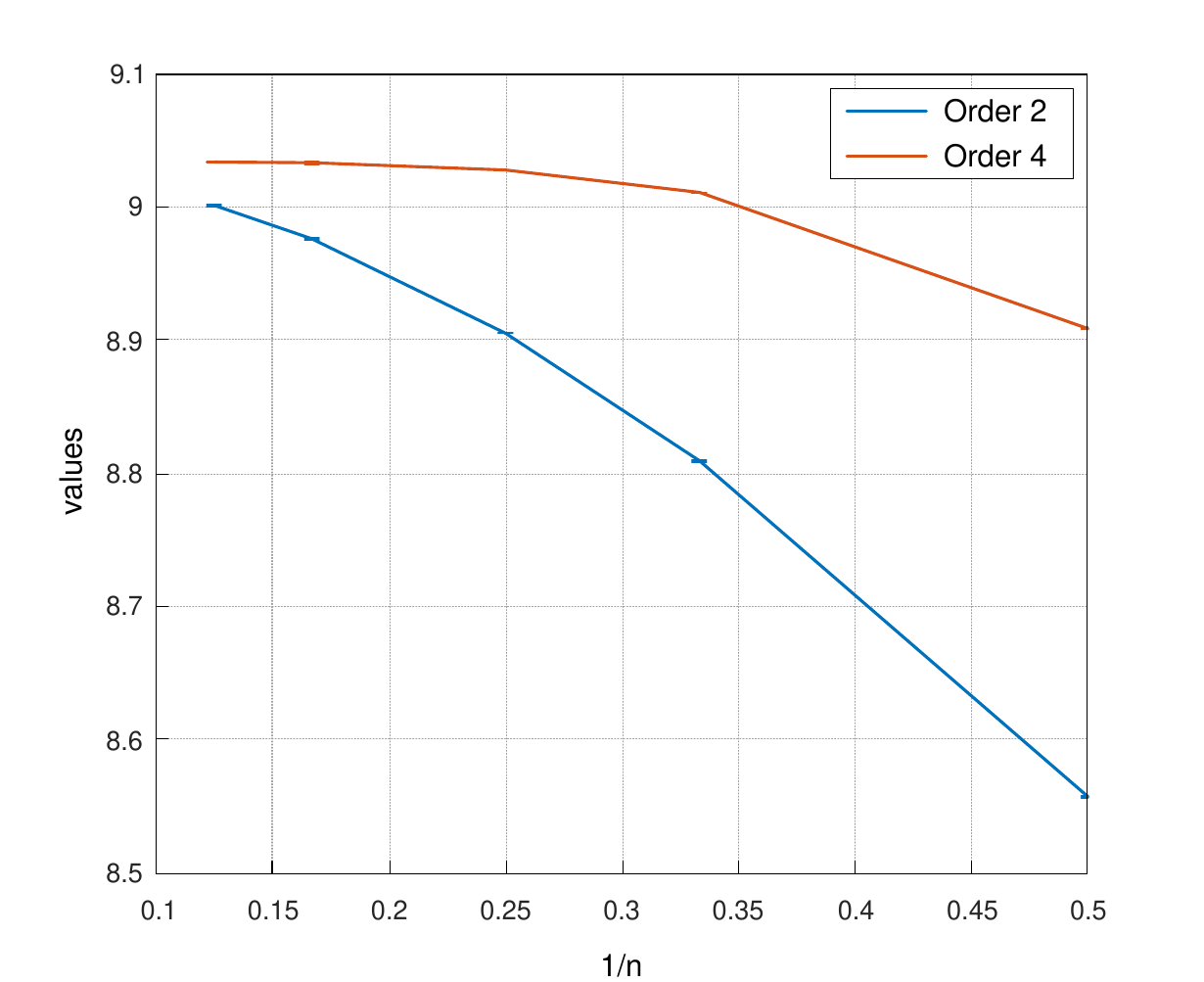}
    \caption{Values plot}
    \label{fig:values_plot_Aheston1}
  \end{subfigure}
  \hfill
  \begin{subfigure}[h]{0.49\textwidth}
    \centering
    \includegraphics[width=\textwidth]{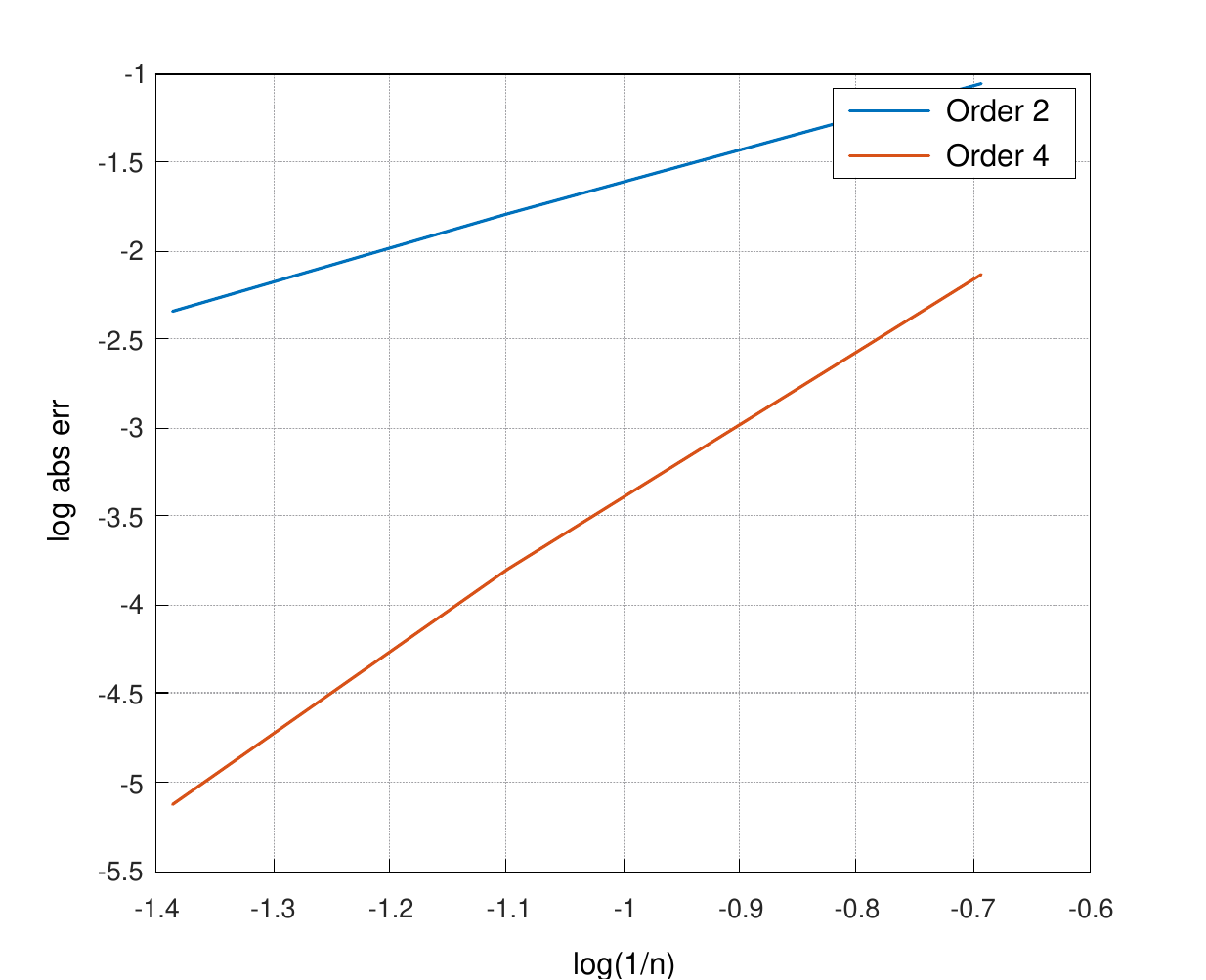}
    \caption{Log-log plot}
    \label{fig:log-log_plot_Aheston1}
  \end{subfigure}
  \caption{Test function: $f(x,y,i)=(K-i/T)^+$. Parameters: $e^{x}=100$, $r=0$, $y=0.2$, $a=0.2$, $b=2$, $\sigma=0.5$, $\rho=-0.7$, $T=1$, $K=100$. Statistical precision $\varepsilon=5$e-4.
  Graphic~({\sc a}) shows the Monte Carlo estimated values of $\cPh^{NV,1,n}f$, $\cPh^{NV,2,n}f$ as a function of the time step $1/n$. Graphic~({\sc b}) draws $\log(|\cPh^{NV,\nu,2n}f-\cPh^{NV,\nu,n}f|)$ in function of $\log(1/n)$: the regressed slopes are 1.85 and 4.30 for the second and fourth order respectively.}\label{Heston_orders_asian}
\end{figure}

\begin{figure}[h!]
  \centering
  \begin{subfigure}[h]{0.49\textwidth}
    \centering
    \includegraphics[width=\textwidth]{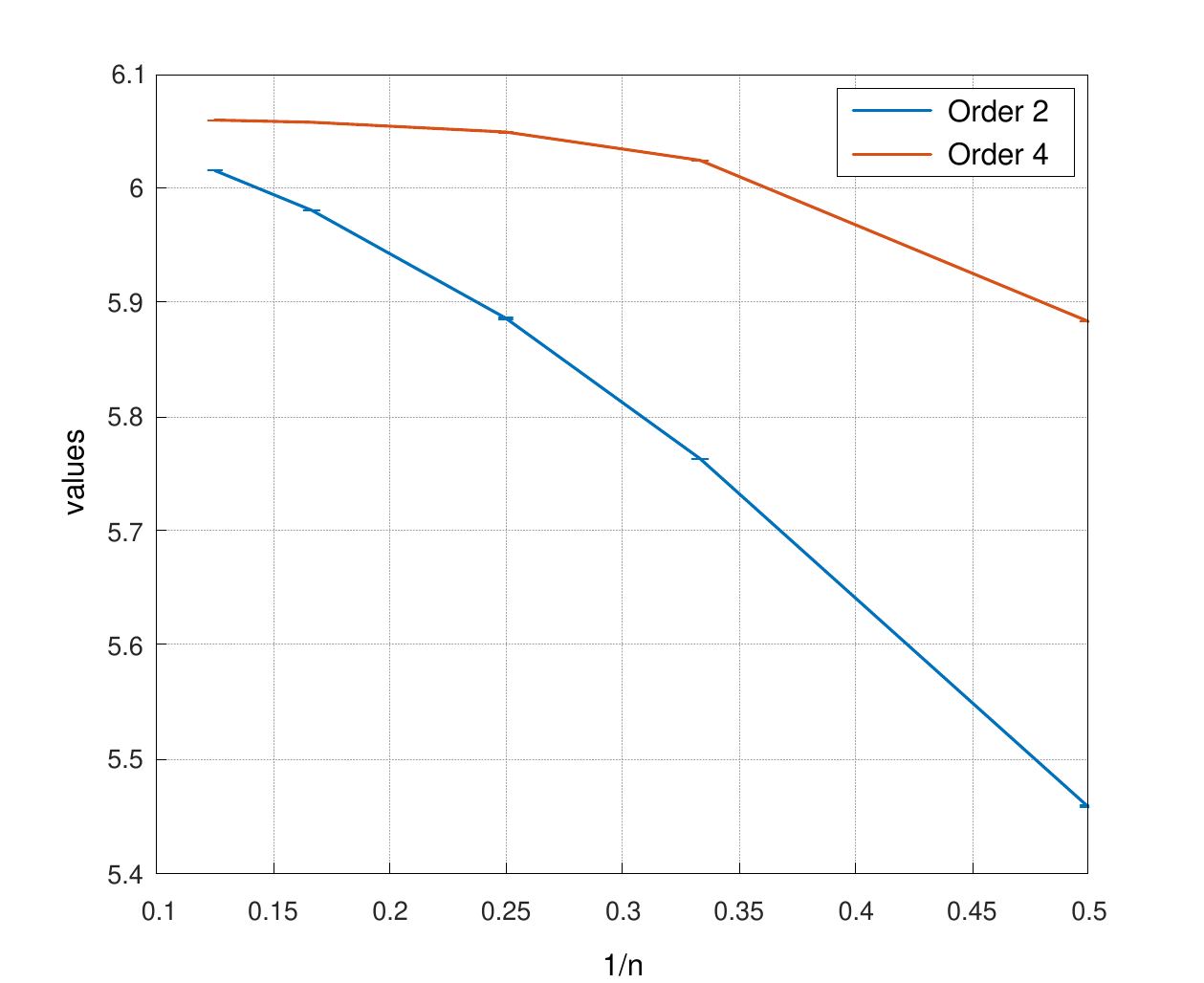}
    \caption{Values plot}
    \label{fig:values_plot_AhestonCF1}
  \end{subfigure}
  \hfill
  \begin{subfigure}[h]{0.49\textwidth}
    \centering
    \includegraphics[width=\textwidth]{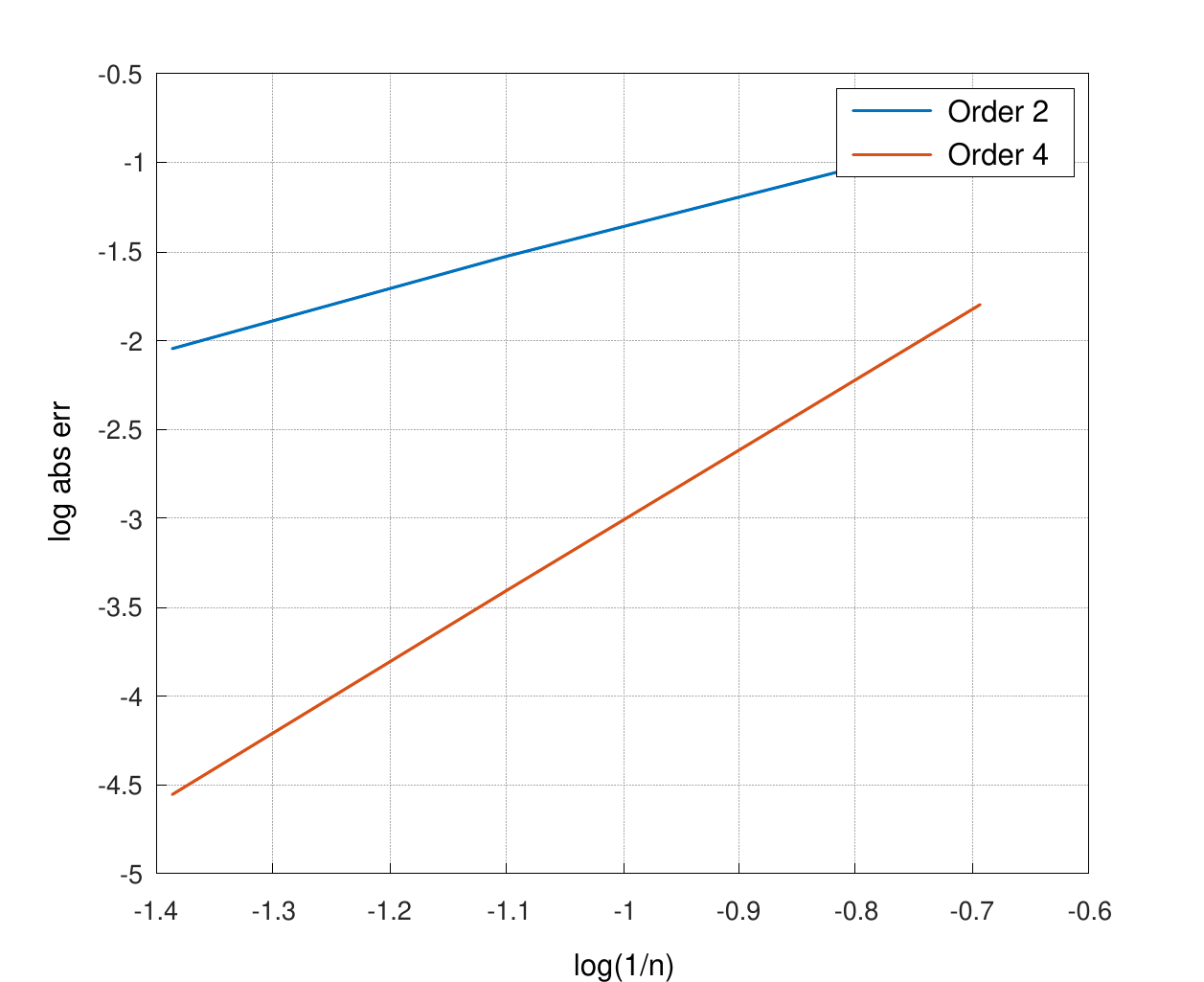}
    \caption{Log-log plot}
    \label{fig:log-log_plot_AhestonCF1}
  \end{subfigure}
  \caption{Test function: $f(x,y,i)=(K-i/T)^+$. Parameters: $e^x=100$, $r=0$, $y=0.1$, $a=0.1$, $b=1$, $\sigma=1.0$, $\rho=-0.9$, $T=1$, $K=100$. Statistical precision $\varepsilon=5$e-4.
  Graphic~({\sc a}) shows the Monte Carlo estimated values of $\cPh^{Ex,1,n}f$, $\cPh^{Ex,2,n}f$ as a function of the time step $1/n$. Graphic~({\sc b}) draws $\log(|\cPh^{Ex,\nu,2n}f-\cPh^{Ex,\nu,n}f|)$ in function of $\log(1/n)$: the regressed slopes are 1.72 and 3.98 for the second and fourth order respectively.}\label{HestonCF_orders_asian}
\end{figure}

Figure~\ref{Heston_orders_asian} shows the convergence of the approximations~$\cPh^{NV,1,n}$ and~$\cPh^{NV,2,n}$ to calculate the Asian option price $P_Tf=\E[(K-\cI_T/T)^+]$, with $f(x,y,i)=(K-i/T)^+$.  The left graphic draws the obtained value in function of the time step. This time, we do not have an exact value, and we draw in the log-log plot the logarithm of the difference between $\cPh^{NV,\nu,2n}$ and $\cPh^{NV,\nu,n}$. If  $\cPh^{NV,\nu,n}=P_Tf +\frac{c}{n^\eta} +o(n^{-\eta})$ for some $\eta>0$, then $\log(|\cPh^{NV,\nu,2n}-\cPh^{NV,\nu,n}|)=\log(|c|(1-2^{-\eta}))-\eta \log(n) +o(\log(n))$, and therefore the slope of the log-log plot can be seen as an estimation of the rate of convergence. Again, we find empirical rates that are close to 2 for $\nu=1$ and 4 for $\nu=2$, which is in line with the theoretical results. The same observation holds in Figure~\ref{HestonCF_orders_asian} for $\cPh^{Ex,\nu,n}$ in a case where $\sigma^2\ge 4a$.

\subsection{Estimators variance and schemes coupling}\label{Subsec_coupling}

In this paragraph,  we discuss  how to couple the refined path and the coarse one in order to minimize the variance of the correction term $$n\left(f(\hat{X}^{SCH,1}_T,\hat{Y}^{SCH,1}_T)-f(\hat{X}^{SCH,0}_T,\hat{Y}^{SCH,0}_T)\right),$$
where $SCH\in \{Ex, NV\}$ indicates the scheme used. We will note $$\mathbf{V}(n)=\text{Var}\left(n\left(f(\hat{X}^{SCH,1}_T,\hat{Y}^{SCH,1}_T)-f(\hat{X}^{SCH,0}_T,\hat{Y}^{SCH,0}_T)\right)\right).$$ While it is rather natural to take the same driving noise for the other time steps, the difficulty is to find a good coupling on $[\kappa h_1, (\kappa+1) h_1]$ between the noise used on the refined time grid and the one of the coarse grid. This issue does not exist for~$Y$ when it is simulated exactly, and for the Ninomiya-Victoir scheme we always take $G_{\kappa+1}=\frac 1 {\sqrt{n}} \sum_{k=1}^n \tilde{G}_k$. 
We therefore discuss the choice of $N_{\kappa+1}$ that is used for the simulation of~$X$. We consider the two following choices:
$$
N_{\kappa+1}=N^{\textup{st}}=\frac 1{\sqrt{n}}\sum_{k=1}^n\tilde{N}_k,\ \text{ or }  N_{\kappa+1}=N^{\textup{av}} = \frac{\sum_{k=1}^n \sqrt{\hY^{SCH,1}_{\kappa h_1 +(k-1)h_2}+\hY^{SCH,1}_{\kappa h_1 +kh_2}} \tilde{N}_k}{\sqrt{\sum_{k=1}^n \hY^{SCH,1}_{\kappa h_1 +(k-1)h_2}+\hY^{SCH,1}_{\kappa h_1 +kh_2}}}.
$$
Note that $N^{\textup{av}}\sim \mathcal{N}(0,1)$, since the normal variables $\tilde{N}_k$, $1\le k\le n$, are independent of the $Y$ component. This second choice is also rather natural since it weights each normal variable with the corresponding volatility on each fine time-step. A similar coupling has been proposed by Zheng~\cite{ZhengC} in a context of Multi-Level Monte-Carlo for the Heston model.

Besides this choice of coupling, we also consider another scheme for the Heston model. In fact, an alternative of Strang's scheme is to introduce a Bernoulli random variable of parameter~$1/2$ that selects which scheme is used first. We want to see if this additional random variable has an incidence on the variance of the correcting term.  This  scheme is given by
$$\hX^{SCH,x,y}_t = x +(r-\frac{\rho}{\sigma}a)t +\frac{\rho}{\sigma}(\hY^y_t-y) +(\frac{\rho}{\sigma}b-\frac{1}{2})\frac{y+\hY^y_t}{2}t +\sqrt{y + B (1-\rho^2)(\hY^{SCH,y}_t-y) t}N, 
$$
where $N\sim \mathcal{N}(0,1)$ and $B\sim \mathcal{B}(1/2)$ is an independent Bernoulli random variable. The random variable $\hY^{SCH,y}_t$ is either equal to $Y^y_t$ for $SCH=Ex$ or to $\hY^y_t$ for $SCH=NV$. This scheme has been used in the numerical experiments of~\cite{AL} and is indicated with "Bernoulli" in the following tables.

\begin{table}[h!]
  \centering 
  \begin{tabular}{ |c |c|c|c|c|c|c| }
    \cline{1-7}
        Scheme & Coupling & $n=2$     & $n=4$      & $n=8$       & $n=16$    & $n=32$    \\
        \hline
        \multicolumn{1}{|c|}{\multirow{2}{*}{$NV$}} &  \multirow{2}{*}{$N^{\textup{st}}$}    & 12.13 & 18.48 & 21.85 & 23.56 & 24.41 \\ 
        \multicolumn{1}{|c|}{}&     & (0.01) & (0.01) & (0.01) & (0.02) & (0.02) \\
      \hline
        \multicolumn{1}{|c|}{ \multirow{2}{*}{$NV$}} &  \multirow{2}{*}{$N^{\textup{av}}$}      & 8.31 & 9.08 & 8.91 & 8.70 & 8.57\\ 
        \multicolumn{1}{|c|}{}& & (0.01) & (0.01) & (0.01) & (0.01) & (0.01) \\ 
         \hline
        \multicolumn{1}{|c|}{\multirow{2}{*}{$NV$, Bernoulli}} &  \multirow{2}{*}{$N^{\textup{st}}$} & 33.27 & 41.96 & 46.14 & 48.27 & 49.37 \\ 
        \multicolumn{1}{|c|}{}&      & (0.02) & (0.03) & (0.03) & (0.04) & (0.04) \\
      \hline
        \multicolumn{1}{|c|}{\multirow{2}{*}{$NV$, Bernoulli}} &  \multirow{2}{*}{$N^{\textup{av}}$} & 25.11 & 28.55 & 30.74 & 32.13 & 32.85\\ 
        \multicolumn{1}{|c|}{}&   & (0.02) & (0.02) & (0.03) & (0.03) & (0.03) \\       
    \hline
    \multicolumn{1}{|c|}{\multirow{2}{*}{$Ex$}} &  \multirow{2}{*}{$N^{\textup{st}}$}     & 30.19 & 30.19 & 28.09 & 26.74 & 26.02 \\ 
    \multicolumn{1}{|c|}{}&  & (0.02) & (0.02) & (0.02) & (0.02) & (0.02) \\
    \hline
    \multicolumn{1}{|c|}{\multirow{2}{*}{$Ex$}} &  \multirow{2}{*}{$N^{\textup{av}}$}  & 26.35 & 20.80 & 15.17 & 11.88 &  10.18\\ 
    \multicolumn{1}{|c|}{}& & (0.01) & (0.01) & (0.01) & (0.01) & (0.01) \\ 
    \hline
    
  \end{tabular}
  \caption{Variance $\mathbf{V}(n)$ estimated  with $10^8$ samples, the 95\% confidence precision is indicated below in parentheses. Test function: $f(x,y)=(K-e^x)^+$. Parameters: $e^x=100$, $r=0$, $x=0.2$, $a=0.2$, $b=1.0$, $\sigma=0.5$, $\rho=-0.7$, $T=1$, $K=105$.}\label{Table_VAR_Heston}
\end{table}

\begin{table}[h!]
  \centering 
  \begin{tabular}{ |c |c|c|c|c|c|c| }
    \cline{1-7}
    Scheme & Coupling & $n=2$     & $n=4$      & $n=8$       & $n=16$    & $n=32$    \\
    \hline
    \multicolumn{1}{|c|}{\multirow{2}{*}{$Ex$}} &  \multirow{2}{*}{$N^{\textup{st}}$}      & 38.69 & 39.51 & 36.96 & 35.23 & 34.32 \\ 
    \multicolumn{1}{|c|}{}&      & (0.03) & (0.03) & (0.03) & (0.03) & (0.03) \\
  \hline
    \multicolumn{1}{|c|}{\multirow{2}{*}{$Ex$}} &  \multirow{2}{*}{$N^{\textup{av}}$}     & 32.49 & 26.01 & 19.16 & 15.20 & 13.17 \\ 
    \multicolumn{1}{|c|}{}& & (0.02) & (0.02) & (0.02) & (0.01) & (0.01) \\ 
    \hline
        \multicolumn{1}{|c|}{\multirow{2}{*}{$Ex$, Bernoulli}} &  \multirow{2}{*}{$N^{\textup{st}}$}  & 65.66 & 68.93 & 66.95 & 65.47 & 65.01 \\ 
        \multicolumn{1}{|c|}{}&      & (0.04) & (0.05) & (0.05) & (0.05) & (0.05) \\
        \hline
        \multicolumn{1}{|c|}{\multirow{2}{*}{$Ex$, Bernoulli}} &  \multirow{2}{*}{$N^{\textup{av}}$}    & 61.04 & 57.45 & 50.98 & 47.03 & 45.12 \\ 
        \multicolumn{1}{|c|}{}&   & (0.04) & (0.04) & (0.04) & (0.04) & (0.04) \\      
    \hline
    
  \end{tabular}
  \caption{Variance $\mathbf{V}(n)$ estimated  with $10^8$ samples,  the 95\% confidence precision is indicated below in parentheses. Test function: $f(x,y)=(K-e^x)^+$. Parameters: $e^x=100$, $r=0$, $x=0.1$, $a=0.1$, $b=1.0$, $\sigma=1.0$, $\rho=-0.9$, $T=1$, $K=105$.}\label{Table_VAR_HestonCF}
\end{table}

We have reported in Tables~\ref{Table_VAR_Heston} and \ref{Table_VAR_HestonCF} the variance of the correcting term for the different schemes, the two different choices for $N_{\kappa+1}$ and different values of~$n$.  Table~\ref{Table_VAR_Heston} reports a case with $\sigma^2\le 4a$ where the Ninomiya-Victoir scheme is well-defined, while Table~\ref{Table_VAR_HestonCF} reports a case with $\sigma^2>4a$. In both cases, we have taken the example of a European Put option. In both tables, we observe that the scheme using a Bernoulli random variable has a correcting term of  higher variance. Besides, it requires to simulate one more random variable. Thus, the schemes based on the Strang composition are better suited with the convergence acceleration using random grids.   

We now comment the coupling of the schemes. In all our experiments, the coupling using~$N^{\textup{av}}$ gives a lower variance than the one using~$N^{\textup{st}}$. Besides, we observe that the gain factor between the two choices is  increasing with~$n$. We have a gain factor of $\frac{24.41}{8.57}\approx 2.85$ in Table~\ref{Table_VAR_Heston} for the Ninomiya-Victoir scheme and $n=32$, and of $2.32$ in Table~\ref{Table_VAR_HestonCF} for the scheme $Ex$ with $n=16$. As a consequence, we recommend the use of $N^{\textup{av}}$ to couple the schemes on the coarse and fine grids.

\subsection{Towards higher order approximations of Rough Heston process}
In this last paragraph, we propose to investigate numerically the approximations with random grids in the case of the rough Heston model. We first recall that the rough Heston model proposed by El Euch and Rosenbaum~\cite{EER19} is given by $S_t=e^{X^{x,y}_t}$, where
\begin{align}
  X^{x,y}_t &= x + \int_0^t \left(r-\frac 12 Y^y_u \right) du + \int_0^t\sqrt{Y^y_u}  (\rho dW_u + \sqrt{1-\rho^2} dB_u),\\
  Y^y_t &= y + \int_0^t K(t-u) (a-bY^y_u) du +\int_0^t K(t-u)\sigma\sqrt{Y^y_u} dW_u, \label{rough_vol_SVE}
\end{align}
 where $K$ is the fractional kernel given by
\begin{equation}\label{fractional_kernel}
  K(t)= \frac{t^{H-1/2}}{\Gamma(H+1/2)}
\end{equation}
with Hurst parameter $H\in(0,1/2)$.
The convolution through the kernel $K$ in \eqref{rough_vol_SVE} introduces a dependence of the volatility $Y$ on the past, and the process $(X,Y)$ is not Markovian. Despite this, it is possible to find a process in larger dimension that is Markovian and approximates the rough process well. It is well known (see e.g. Alfonsi and Kebaier \cite[Proposition 2.1]{AK24}) that if we replace the rough kernel $K$ in \eqref{rough_vol_SVE}  by a discrete completely monotone kernel
\begin{equation}\label{discrete_completely_monotone_kernel}
  \tK(t) = \sum_{k=1}^d \gamma_k e^{-\rho_k t},\qquad \gamma_k,\rho_k\ge 0,\,k\in\{1,\ldots,d\},
\end{equation} 
then the solution of the Stochastic Volterra Equation
\begin{equation}\label{dcmk_vol_SVE}
  \tY_t = y + \int_0^t \tK(t-u) (a-b\tY^y_u) du +\int_0^t \tK(t-u)\sigma\sqrt{\tY^y_u} dW_u, 
\end{equation}
is given by $\tY_t = y + \sum_{k=1}^d \gamma_k\tY^k_t$, where $\tbfY=(\tY^1,\ldots,\tY^d)$ solves the SDE in $\R^d$:
\begin{equation}\label{dcmk_vol_SDE}
  \tY^k_t = -\rho_k\int_0^t \tY^k_u du + \int_0^t (a-b\tY_u) du + \int_0^t \sigma \sqrt{\tY_u}dW_u,\quad k\in\{1,\ldots,d\}, t\ge 0.
\end{equation}
We want to build a second order scheme for \eqref{dcmk_vol_SDE} along with
\begin{equation*}
  \tX_t = x + \int_0^t(r-\frac 12 \tY_u) du + \int_0^t\sqrt{\tY_u} (\rho dW_u + \sqrt{1-\rho^2} dB_u).
\end{equation*}
This multifactor model has been first developed by Abi Jaber and El Euch \cite{AEE19} and can be seen under a suitable choice of $\tK(t) = \sum_{k=1}^d \gamma_k e^{-\rho_k t}$  as an approximation of the rough Heston model.

We present here a second order approximation scheme for the couple $(\tX,\tY)$ that preserve the positivity of $\tY$ as proved by Alfonsi in \cite[Theorem 4.2 and Subsection 4.3]{AA_NPK}. The infinitesimal generator of the $d+1$ dimensional process $(\tX,\tbfY)$ is given by 
\begin{multline}\label{inf_gen_dcmk_vol}
  \mcL f(x,\bfy)  = (r-\frac{1}{2}y')\partial_{x} f(x,\bfy)+\sum_{k=1}^d (a-\rho_k y_k-b y')\partial_{y_k} f(x,\bfy)\\ 
  +\frac{1}{2} \partial_x^2f(x,\bfy)+\sum_{k=1}^d 2\rho\sigma \partial_{x}\partial_{y_k} f(x,\bfy) +\frac{1}{2} \sum_{k,l=1}^d \sigma^2 y' \partial_{y_k}\partial_{y_l} f(x,\bfy),
\end{multline}
where $\bfy = (y_1,\ldots,y_d)$ and $y'=y+\sum_{j=1}^d \gamma_j y_j$. We use the following splitting $\mcL=\mcL_1+\mcL_2$, where $\mcL_1 f = -\sum_{k=1}^d \rho_k y_k \partial_{y_k} f$ is the infinitesimal generator associated to
\begin{equation}\label{mcL1_linear_ODE}
  \begin{aligned}
    dX_t   &=0, \\
    dY^k_t &= -\rho_k Y^k_t dt,\qquad k\in\{1,\ldots,d\},
  \end{aligned}
\end{equation}
and $\mcL_2$ is associated to
\begin{equation}\label{mcL2_SDE}
  \begin{aligned}
    dX_t   &=(r-\frac{1}{2}Y_t)d_t+\sqrt{Y_t} (\rho dW_t + \sqrt{1-\rho^2} dB_t), \\
    dY^k_t &=  (a-bY_t) dt +  \sigma \sqrt{Y_t}dW_t, \text{ with }Y_t = y+\sum_{k=1}^d \gamma_k Y^k_t \quad k\in\{1,\ldots,d\}.
  \end{aligned}
\end{equation}
The linear ODE \eqref{mcL1_linear_ODE} has the exact solution
\begin{equation}
  \psi_1(t,x,{\bf y})=(x,{\bf y}_t), \text { with } {\bf y}_t=(y_1e^{-\rho_1 t}, \ldots, y_de^{-\rho_d t}).
\end{equation}
From \eqref{mcL2_SDE}, we obtain that $(X_t,Y_t)$ satisfies the following log-Heston SDEs
\begin{equation}\label{K0_logHeston_SDE}
  \begin{aligned}
    X_t   &= x + \int_0^t(r-\frac{1}{2}Y_u)d_t+\int_0^t\sqrt{Y_u} (\rho dW_u + \sqrt{1-\rho^2} dB_u), \\
    Y_t &= y' + \int_0^t K(0)(a-bY_u) du +\int_0^tK(0)\sigma \sqrt{Y_u} dW_u,
  \end{aligned}
\end{equation}
and $dY^k_t=\frac{1}{K(0)} dY_t$ (note that $K(0)=\sum_{j=1}^d \gamma_j$).
So, having a second order scheme $(\hat{X}^{x,y'}_t,\hat{Y}^{y'}_t)$ for $(X_t,Y_t)$, we can build a second order scheme for~\eqref{mcL2_SDE} by
\begin{equation}
  (\hat{X}^{x,y}_t,\hat{Y}^{1,y}_t,\dots,\hat{Y}^{d,y}_t) = (\hat{X}^{x,y'}_t,A_\bfy(\hat{Y}^{y'}_t)),
\end{equation}
where 
\begin{equation}
  A_\bfy(z) = \left(y_1 +\frac{z-y'}{K(0)}, \ldots, y_d+\frac{z-y'}{K(0)}  \right).
\end{equation}
In the end, we use again the Strang composition to get the second order scheme for~\eqref{inf_gen_dcmk_vol} starting from $(x,{\bf y})$ and time-step $t>0$:
\begin{equation}\label{scheme_multif_Heston} \psi_1 \left(t/2,\hat{X}^{x,y'_{t/2}}_t,A_{\bfy_{t/2}}(\hat{Y}^{y'_{t/2}}_t) \right),\end{equation}
where $y'_{t/2}=y+\sum_{j=1}^d \gamma_j y_j e^{-\rho_jt/2}$.

Now that we have defined the approximation scheme~\eqref{scheme_multif_Heston} for the multifactor Heston model, we want to use it to test numerically the convergence acceleration provided by the random grids. The construction of the estimators is identical to the one of $\cPh^{NV,1,n}$ and $\cPh^{NV,2,n}$ in Subsection~\ref{Subsec_implementation} and we do not reproduce it here. Also, by a slight abuse of notation, we still denote by $\cPh^{NV,1,n}$ and $\cPh^{NV,2,n}$ these estimators that are well-defined $\tK(0)\sigma^2<4a$. Unfortunately, there does not exist yet -- up to our knowledge -- efficient exact simulation method for the multifactor Cox-Ingersoll-Ross process. It it were the case, we could then define the corresponding estimators $\cPh^{Ex,1,n}$ and $\cPh^{Ex,2,n}$ exactly as in Subsection~\ref{Subsec_implementation}, for any $\sigma>0$. Here, we thus present only simulation in the case $\tK(0)\sigma^2<4a$. These simulations are intended to be a first attempt to get higher order approximations of the multifactor Heston model. We let the case $\tK(0)\sigma^2>4a$ as well as theoretical proofs of convergence in this model for future studies.

Multi exponential approximations of the rough kernel are available in literature, see e.g. Abi Jaber, El Euch \cite{AEE19} and Alfonsi, Kebaier \cite{AK24}. In our simulation we will use the algorithm BL2 suggested by Bayer and Breneis in \cite{BB23}, that optimizes the $\L^2([0,T])$-error between $K$ and $\tK$ while limiting high values of $\rho_k$. In particular, we will use the approximate BL2 Kernel with $d=3$ exponential factors, that has been proven to approximate a whole volatility surface of rough Heston call prices with approximately 1\% of maximal relative error~\cite[Table 4, third column]{BB23}.
When the Hurst parameter $H=0.1$ the nodes and weights are resumed following table
\begin{center}
  \begin{tabular}{ |r|r|r| } 
   \hline
   $\rho_1 = 0.08399474$ & $\rho_2 = 5.64850577$ & $\rho_3 = 118.00624702$ \\
   \hline
   $\gamma_1 = 0.80386099$ & $\gamma_2 = 1.60786461$ & $\gamma_3 = 8.80775525$  \\ 
   \hline
  \end{tabular}
\end{center}
We consider European put option prices and present in Figure~\ref{fig:values_plot_rheston1}  a plot of the values of $\cPh^{NV,1,n}f$ and $cPh^{NV,2,n}f$ as a function of the time step with the exact value obtained by Fourier techniques. In Figure~\ref{fig:log-log_plot_rheston1}, we draw a log-log plot to quantify the order of convergence.
First, we observe that we obtain a much larger bias than in our previous numerical experiments for the Heston process, Figure~\ref{fig:values_plot_hestonCF1}. This is mainly due to the map $\psi_1$ that has a relatively large nodes, namely $\rho_2$ and especially $\rho_3$. The contribution of these exponential factors in the dynamics of the scheme gets more important when the time step is sufficiently small. However, even if the bias is more important, the speed of convergence are still in line with the theoretical ones.  The regressed slopes for $\cPh^{NV,1,n}f$ and $cPh^{NV,2,n}f$ are respectively 1.89 and 3.98, showing that the scheme is indeed a second-order scheme and that the boosting technique with random grids works again in this case.

\begin{figure}[h!]
  \centering
  \begin{subfigure}[h]{0.49\textwidth}
    \centering
    \includegraphics[width=\textwidth]{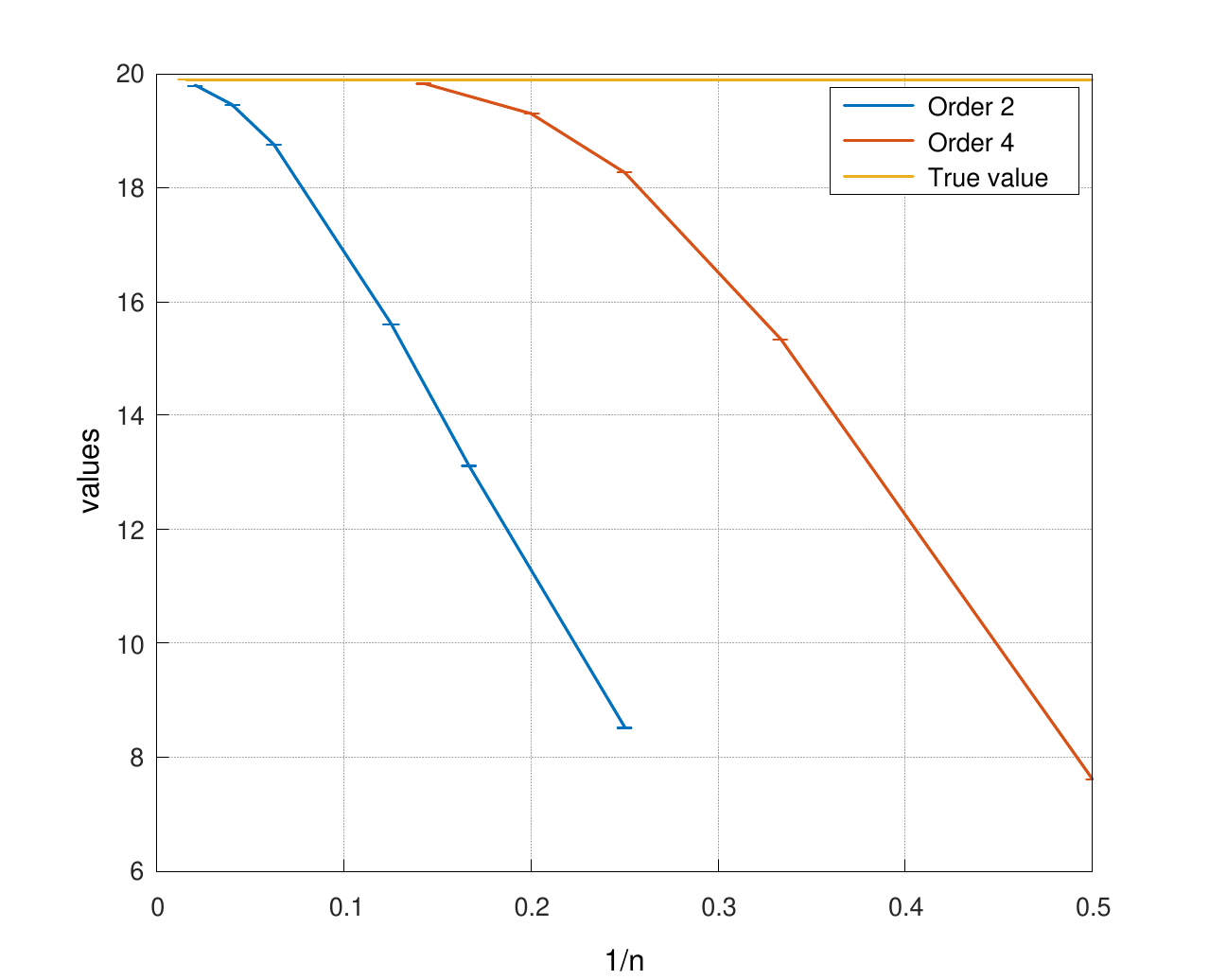}
    \caption{Values plot}
    \label{fig:values_plot_rheston1}
  \end{subfigure}
  \hfill
  \begin{subfigure}[h]{0.49\textwidth}
    \centering
    \includegraphics[width=\textwidth]{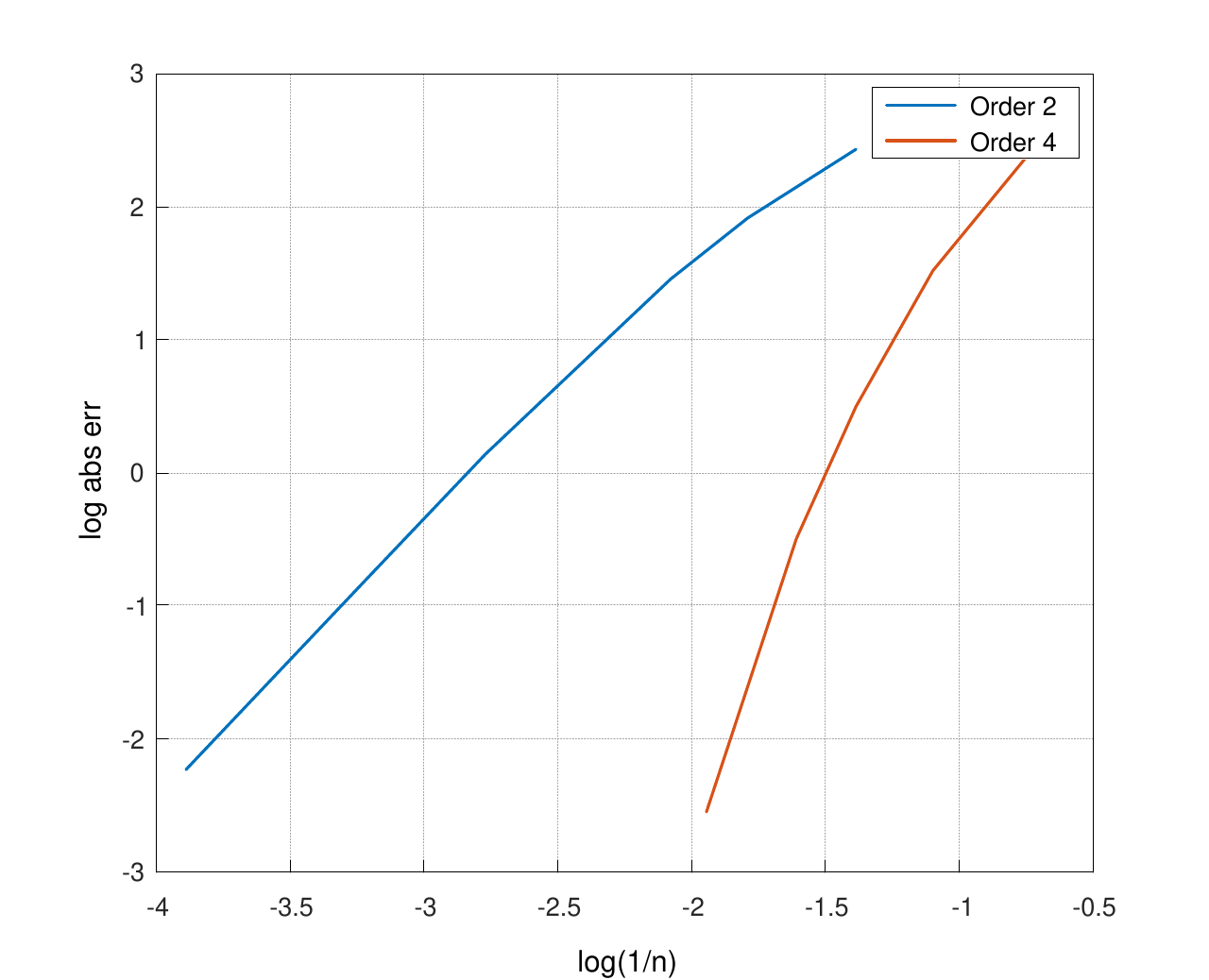}
    \caption{Log-log plot}
    \label{fig:log-log_plot_rheston1}
  \end{subfigure}
  \caption{Test function: $f(x,y)=(K-e^x)^+$. Parameters: $S_0=e^x=100$, $r=0$, $y=0.1$, $a=0.3$, $b=1$, $\sigma=0.1$, $\rho=-0.7$, $T=1$, $K=105$. Statistical precision $\varepsilon=2$e-3.
  Graphic~({\sc a}) shows the Monte Carlo estimated values of $\cPh^{NV,1,n}f$, $\cPh^{NV,2,n}f$ as a function of the time step $1/n$  and the exact option value. Graphic~({\sc b}) draws $\log(|\cPh^{NV,\nu,n}f-P_Tf|)$ in function of $\log(1/n)$: the regressed slopes are 1.89 and 3.98 for the second and fourth order respectively.}\label{rHeston_orders}
\end{figure}


\chapter{Some PDE results in Heston model with applications} 

\label{Chapter_PDE} 

This chapter is based on the paper \cite{Lpde25}.
\chapabstract{\\We present here some results for the PDE related to the logHeston model. We present different regularity results and prove a verification theorem that shows that the solution produced via the Feynman-Kac Theorem is the unique viscosity solution for a wide choice of initial data (even discontinuous) and source data. In addition, our techniques do not use Feller's condition at any time.
In the end, we prove a convergence theorem to approximate this solution by means of a hybrid (finite differences/tree scheme) approach.}

\section*{Introduction}
The stochastic volatility model proposed by Heston in \cite{Heston} is one of the most famous and used models in mathematical finance. It describes the evolution of the price of an asset $S$ and its instantaneous volatility $Y$, according to the following couple of stochastic differential equations
\begin{equation}\label{Heston_model_sde_intro}
  \begin{split}
 dS_t &= (r-\delta)S_t dt+  S_t\sqrt{Y_t} (\rho dW_t + \brho dB_t),\\
 dY_t &= (a-bY_t)dt + \sigma\sqrt{Y_t}dW_t,
  \end{split}
\end{equation}
where $r\in\R$, $\delta\ge0$, $a,b,\sigma>0$, $\rho\in(-1,1)$, $\brho=\sqrt{1-\rho^2}$, $(x,y)\in\R\times[0,\infty)$ and $(W,B)$ is a standard 2-dimensional Brownian motion.
In order to study and discretize the asset $S$, it is useful to consider the logHeston diffusion obtained by applying the transformation $(s,y)\mapsto(\log(s),y)$ to the asset and the volatility. To this purpose, we consider a slightly general model from which we can recover the logHeston by a precise choice of the parameters:
\begin{equation}\label{logHeston_model_sde_intro}
  \begin{split}
 dX_t&= (c+d Y_t)dt+\lambda\sqrt{Y_t}(\rho dW_t + \brho dB_t)\\
 dY_t&= (a-bY_t)dt+\sigma\sqrt{Y_t}dW_t,
  \end{split}
  \end{equation}
where $b,c,d\in\R$, $a,\lambda,\sigma>0$ $\rho\in(-1,1)$. Indeed, one can show that the price of financial derivatives written on $(X,Y)$, or equivalently on $(X,Y)$, satisfies a peculiar parabolic PDE that is degenerate, i.e. the matrix of the second-order derivative coefficients fails to be strictly positive definite when the boundary $\{y=0\}$ is attained. 
In this case, the classical existence and uniqueness results using the uniform ellipticity property fail, and an ad hoc method must be found to prove them.

The literature presents various existence and uniqueness results derived from analyzing the Heston and logHeston PDEs. Ekström and Tysk \cite{ET2010} examined a PDE arising from a generalized Heston model. While their model employs more general drift and volatility functions, these retain key characteristics of the original ones, such as positive drift of the volatility process when $Y_t=0$ and sufficient regularity of the squared volatility function. Within their model, they establish a verification theorem and a uniqueness result contingent upon certain mild assumptions on the payoff function $f$.
Costantini et al. groundbreaking study in \cite{CPD12} establish the existence and uniqueness of a viscosity solution for a PDE encompassing a wide array of jump-diffusion models, including Heston, applicable to both European and Asian options. Notably, when applied to the Heston model, their innovative approach necessitates a condition akin to the Feller condition (i.e. $\sigma^2\le 2a$), ensuring the volatility process remains strictly positive. Consequently, their result cannot cover the full range of parameters in the Heston case.
For numerical reasons, Briani et al. \cite{BCT} need regularity results for functional of the diffusion $(X^{t,x,y},Y^{t,y})$, representing the solution of \eqref{logHeston_model_sde_intro} starting from $(x,y)$ at time $t\in[0,T)$. This is important because the expectation of such functionals gives the price of European options.
In order to do that, they prove first a verification result (cf. \cite[Lemma 5.7]{BCT}) for the logHeston PDE \eqref{reference_PDE_generalset}: under appropriate regularity hypotheses on $f$ and $h$, the function
\begin{equation}\label{u_sol_gen_intro}
 u(t,x,y)=\E\bigg[e^{\varrho (T-t)}f(X^{t,x,y}_T,Y^{t,y}_T)-\int_t^Te^{\varrho (s-t)} h(s,X^{t,x,y}_s,Y^{t,y}_s)ds\bigg]
\end{equation}
is a solution of  
\begin{equation}\label{reference_PDE_generalset}
  \begin{cases}
    \partial_tu(t,x,y) +\mcL u(t,x,y) +\varrho u(t,x,y) = h(t,x,y),\quad   t\in[0,T), &(x,y)\in\mcO, \\
 u(T,x,y) = f(x,y), \quad  &(x,y)\in\mcO,
  \end{cases}
\end{equation}
where $\mcO=\R\times(0,\infty)$ and $\mcL$ is the infinitesimal generator of \eqref{logHeston_model_sde_intro}:
\begin{equation}\label{reference_infinit_gen}
\mcL = \frac{y}{2}(\lambda^2\partial^2_{x} + 2 \rho\lambda\sigma \partial_{x}\partial_{y} +  \sigma^2 \partial^2_{y}) +(c+dy) \partial_x + (a-by) \partial_y.
\end{equation}
Moreover, when the Feller condition $\sigma^2\le2a$ is satisfied, the uniqueness of the solution holds. \label{Page:Feller_condition_uniqueness} In fact, since the boundary $\R\times\{0\}$ is inaccessible to the process $(X,Y)$ under the Feller condition, the behaviour of $u$ if $\R\times\{y=0\}$ is irrelevant and one can achieve uniqueness of the solution. In second instance, they prove a stochastic representation for the derivatives of $u$. As a consequence of this result, and under specific conditions on final data $f$, they show that $u$ is regular enough to solve, by continuity, the problem even on $\overline{\mcO}=\R\times[0,\infty)$, so the PDE is satisfied even when volatility collapses in 0. This gives an additional boundary condition, giving an equation involving the function and its derivatives at the domain boundary. It is worth noting that this is called a Robin boundary condition.
Briani et al. did not establish uniqueness for the PDE over $\overline{\mcO}$, as their primary objective was different. 

In this paper we restart from the logHeston setting of \cite{BCT} and study minimal hypotheses over $f$ and $h$ under which we can prove a verification theorem that characterizes $u$ in \eqref{u_gen_sol_intro} as the \textit{unique} solution of \eqref{reference_PDE_generalset} even when the Feller condition is not met ($\sigma^2>2a$).
Let us stress that, due to the connection between $u$ in \eqref{u_sol_gen_intro} and option prices problems, weaker requests on $f$ and $h$ translate into the choice of more general payoffs and $h$ running costs in finance.
First, we consider classical solutions. 
To this purpose, in Section \ref{classical_section}, after reviewing a slight extension of the regularity result obtained in \cite{BCT}, we characterize $u$ as the unique classical solution of \eqref{reference_PDE_generalset} over $\overline{\mcO}=\R\times[0,\infty)$ under relaxed hypotheses on $f$ and $h$. However, this requires that $f$ and $h$ have some differentiability conditions: merely continuity properties are not enough. In order to contour this difficulty, solutions in weak or viscosity sense are typically taken into account. Here, in Section \ref{viscosity_section}, we tackle the problem from a viscosity solutions point of view: we prove an existence and uniqueness result without the restriction of the Feller condition. We stress that the initial data may have some types of discontinuities, allowing us to deal with valuable financial examples such as Digital options.

We point out that, under the Feller condition, uniqueness results for classical, viscosity or weak solutions over $\mcO$ or $\overline{\mcO}$ have already been studied in the literature (see, e.g. \cite{BCT,CPD12,CMA17}), possibly requiring strong hypotheses on $f$ and $h$.
However, when the Feller condition does not hold, to the best of our knowledge, the literature is very poor on results concerning the uniqueness of classical and viscosity solutions (see, e.g. \cite{ET2010} for classical solutions). Thus, the main original contributions of this paper delve into this direction.

Finally, we deal with the convergence of a hybrid numerical method introduced in \cite{BCZ}.
If $f$ is smooth enough, the convergence rate has already been provided in Briani et al. \cite{BCT}, independently of the validity of the Feller condition. As the authors need the regularity of the price function, they strongly use classical solution results. Thus, their approach must keep the regularity of $f$. Here, by exploiting the tools and techniques introduced to get the results concerning viscosity solutions, we can prove the convergence of the above-cited hybrid numerical scheme whenever $f$ is continuous, see Theorem \ref{convergence_theorem}. The result of this theorem is confirmed empirically by the numerical experiment carried out in \cite{BCZ}, which computes the price of a European put option in the Heston model. Other numerical experiments that use the hybrid algorithm for the Bates and for the Heston-Hull-White models have been carried out in \cite{BCTZ} and \cite{BCZ2} respectively.


\section{Existence and uniqueness of classical solutions}\label{classical_section}
This section contains a slight improvement of some results proven in \cite{BCT} regarding the log-Heston PDE in which we add a uniqueness result inspired by \cite{ET2010}.

We start by introducing some notations. We set $\R_+=[0,\infty)$, $\R^*_+=(0,\infty)$ and name $\mcC^{q}(\R\times\R_+)$ the set of all functions on $\R\times\R_+$ which are $q$-times continuously differentiable. We set $\mcC_{\pol}^{q}(\R\times\R_+)$ the set of functions $g \in \mcC^{q}(\R\times\R_+)$ such that there exist $C, L>0$ for which
$$ 
|\partial_{x}^{\alpha}\partial_{y}^{\beta} g(x,y)| \leq C(1+|x|^{L}+y^{L}), \quad (x,y) \in \R\times\R_+,\; \alpha+\beta \leq q .
$$
For $T>0$, we set $\mcC_{\pol, T}^{q}(\R\times\R_+)$ the set of functions $v=v(t, x, y)$ such that $v \in \mcC^{\lfloor q / 2\rfloor, q}([0, T) \times (\R\times\R_+))$ and there exist $C, L>0$ for which
$$ 
\sup _{t \in[0, T)}|\partial_{t}^{k}\partial_{x}^{\alpha}\partial_{y}^{\beta}  v(t, y)| \leq C(1+|x|^{L}+y^{L}), \quad (x,y) \in \R\times\R_+,\; 2 k+\alpha+\beta \leq q. 
$$
We set $\mcC(\R\times\R_+)=\mcC^{0}(\R\times\R_+)$, $\mcC_{\pol}(\R\times\R_+)=\mcC_{\pol }^{0}(\R\times\R_+)$ and $\mcC_{\pol,T}(\R\times\R_+)=\mcC_{\pol,T}^{0}(\R\times\R_+)$. We also need another functional space, that we call $\mcC_{\pol}^{p, q}(\R, \R_+), p \in[1, \infty], q \in \N, m \in \N^{*}: g=g(x, y) \in \mcC_{\pol}^{p, q}(\R,\R_+)$ if $g \in \mcC_{\pol}^{q}( \R\times\R_+)$ and there exist $C, c>0$ such that
$$ 
|\partial_{x}^{\alpha} \partial_{y}^{\beta} g(\cdot, y)|_{L^{p}(\R, dx)} \leq C(1+|y|^{c}), \quad\alpha+\beta \leq q,
$$
where $|h|_{L^p,dx}=(\int_\R h(x)^pdx)^{1/p}$ if $p>1$, and the standard sup norm if $p=\infty$.
Similarly, as above, we set $\mcC_{\pol,T}^{p, q}(\R, \R_+)$ the set of the function $v \in \mcC_{\pol,T}^{q}(\R \times \R_+)$ such that
$$ 
\sup _{t \in[0, T)}|\partial_{t}^{k} \partial_{x}^{l'} \partial_{y}^{l} v(t, \cdot, y)|_{L^{p}(\R, dx)} \leq C(1+|y|^{c}), \quad 2 k+|l'|+|l| \leq q. 
$$

We call the solution of the logHeston SDE \eqref{logHeston_model_sde_intro}
\begin{align}\label{referenceDiffusion}
 X_T^{t,x,y} &= x + \int_t^T\big(c + d Y^{t,y}_s\big) ds + \int_t^T\lambda\rho\sqrt{Y^{t,y}_s} dW_s +\int_t^T\lambda\bar{\rho}\sqrt{Y^{t,y}_s} dB_s, \nonumber\\
 Y^{t,y}_T &= y + \int_t^T(a-bY^{t,y}_s)ds + \int_t^T\sigma\sqrt{Y^{t,y}_s}dW_s.
\end{align}
We define the candidate solution
\begin{equation}\label{u_gen_sol}
 u(t,x,y)=\E\bigg[e^{\varrho (T-t)}f(X^{t,x,y}_T,Y^{t,y}_T)-\int_t^Te^{\varrho (s-t)} h(s,X^{t,x,y}_s,Y^{t,y}_s)ds\bigg],
\end{equation}
to the reference PDE
\begin{equation}\label{reference_PDE}
  \begin{cases}
    \partial_tu(t,x,y) +\mcL u(t,x,y) +\varrho u(t,x,y) = h(t,x,y),\quad   t\in[0,T), &(x,y)\in\R\times\R_+, \\
 u(T,x,y) = f(x,y), \quad  &(x,y)\in\R\times\R_+,
  \end{cases}
\end{equation}
where $\mcL$ is defined in \eqref{reference_infinit_gen}.
One can remark that when $c=r-\delta$ (interest rate minus dividend rate) and $d=-\frac 12$, then $(X,Y)$ is the standard logHeston model for the log-price and volatility. When instead $\rho=0$, $c=r-\delta-\frac \rho{\sigma}a$ and $d=\frac \rho{\sigma}b-\frac 12$ and $\lambda = \brho$, we recover a formulation that will be useful to discretize the solution $u$ in Section \ref{approximation_section}.

In order to present the main contribution in this section, we present a slight extension of a regularity result presented in \cite{BCT}, that can be summarized Lemma \ref{lemma-reg} and Proposition \ref{prop-reg-new}.

\begin{lemma}\label{lemma-reg}
 Let $u$ be defined in \eqref{u_gen_sol}, with $f$ and $h$ such that, as $j=0,1$, $\partial_x^{2j}g\in C^{1-j}_{\pol}(\R\times\R_+)$,  $\partial_x^{2j}h\in \mathcal{C}^{1-j}_{\pol,T}(\R\times\R_+)$ along with $
 h$ and $\partial_y h$ locally Hölder continuous in $[0,T)\times\R\times\R_+^*$.
 Then 
  $\partial^{2j}_xu\in \mathcal{C}^{1-j}_{\pol,T}(\R\times\R_+)$ for $j=0,1$, and one has for 
  \begin{align}
    \partial^m_x u(t,x,y) &=\E\left[e^{\varrho (T-t)} \partial^m_x g(X^{t,x,y}_T,Y^{t,x,y}_T)-\int_t^Te^{\varrho (s-t)} \partial^m_x h(s,X^{t,x,y}_s,Y^{t,x,y}_s)ds  \right],\,\, m=1,2,\label{u-x-xx}\\
    \partial_y u(t,x,y) &=\E\bigg[e^{(\varrho-b) (T-t)} \partial_yg(X^{1,t,x,y}_T,Y^{1,t,x,y}_T)\nonumber \\
  &\qquad\qquad+\int_t^Te^{(\varrho-b) (s-t)} \Big[\frac \lambda2 \partial^2_xu+d\partial_xu-\partial_yh\Big](s,X^{1,t,x,y}_s,Y^{1,t,x,y}_s)ds  \bigg],\label{u-y}
  \end{align}
 where
  $(X^{1,t,x,y}_s,Y^{1,t,x,y}_s)$ solves  \eqref{logHeston_model_sde_intro} with new parameters   $\rho_1=\rho$, $c_1=c+\rho\lambda\sigma$, $d_1=d$, $\lambda_1=\lambda$,  $b_1=b$, $a_1=a+\frac{\sigma^2}{2}$, $\sigma_1=\sigma$. Furthermore, $v=\partial_y u$ is the unique solution to the following PDE
  \begin{equation}
    \begin{cases}
 \big[(\partial_t+\mcL_1+\varrho-b)v +(\lambda^2/2 \partial_x^2+d\partial_x)u-\partial_yh\big](t,x,y)=0,\quad  t\in[0,T), \!\!\!\!&(x,y)\in\R\times\R_+^*,  \\
 v(T,x,y) = \partial_y f(x,y), \quad  \!\!\!&(x,y)\in\R\times\R_+^*,
    \end{cases}
  \end{equation}
 where $\mcL_1=\frac{y}{2}(\lambda^2\partial^2_{x} + 2 \rho\sigma \partial_{x}\partial_{y} +  \sigma^2 \partial^2_{y}) +(c +\rho\lambda\sigma+d y) \partial_x + (a+\sigma^2/2 -by) \partial_y$.
\end{lemma}

Iterating this Lemma, it is possible to prove the following result.

\begin{prop} \label{prop-reg-new}
 Let $q\in\N$. For every $j=0,1,\ldots,q$, $\partial_x ^{2j}f\in \mcC^{q-j}_{\pol}(\R\times\R_+)$, $\partial_x ^{2j}h\in \mcC^{q-j}_{\pol, T}(\R\times\R_+)$, and for all $(m,n)\in\N^2$ such that $m+2n\le 2q$ and $m\le 2(q-1)$, $\partial^m_x\partial^n_y h$ locally Hölder continuous in $[0,T)\times\R\times\R_+^*$.
 Let $u$ as in \eqref{u_gen_sol}.
  
 Then $\partial_x ^{2j}u\in \mcC^{q-j}_{\pol, T}(\R\times\R_+)$ for every $j=0,1,\ldots,q$. 
 Moreover, the following stochastic representation holds: for $m+2n\leq 2q$,
  \begin{equation}\label{stoc_repr-new_chap4}
  \begin{split}
  &\partial^{m}_x\partial^{n}_yu(t,x,y)=\E\left[e^{(\varrho-nb) (T-t)} \partial^m_x\partial^n_yf(X^{n,t,x,y}_T,Y^{n,t,x,y}_T)\right]\\
  &\quad+
 \E\left[\int_t^Te^{(\varrho-nb) (s-t)}\left[n\Big(\frac \lambda 2 \partial^{m+2}_x\partial^{n-1}_yu+d\partial^{m+1}_x\partial^{n-1}_yu\Big)-\partial^{m}_x\partial^n_y h\right](s,X^{n,t,x,y}_s,Y^{n,t,x,y}_s)ds  \right],
  \end{split}
  \end{equation}
 where $\partial^{m}_x\partial^{n-1}_y u:=0$ when $n=0$ and $(X^{n,t,x,y},Y^{n,t,x,y})$, $n\geq 0$, denotes the solution starting from $(x,y)$ at time $t$ to the SDE \eqref{logHeston_model_sde_intro} with parameters
  \begin{equation}\label{parameters-new_chap4}
  \rho_n=\rho,\quad c_n=c+n\rho\lambda\sigma,\quad d_n=d,\quad \lambda_n=\lambda,\quad a_n=a+n\frac{\sigma^2}{2}, \quad b_n=b,\quad \sigma_n=\sigma.
  \end{equation}
 In particular, if $q\geq 2$ then $u\in \mcC^{1,2}([0,T]\times \R\times\R_+)$,  solves the PDE
  \begin{equation}\label{PDE-closed}
  \begin{cases}
  \partial_t u(t,x,y)+ \mcL u(t,x,y) +\varrho u(t,x,y)= h(t,x,y),\qquad& t\in [0,T), \,(x,y)\in \R\times\R_+,\\
 u(T,x,y)= f(x,y),\qquad& (x,y)\in \R\times\R_+.
  \end{cases} 
  \end{equation}
\end{prop}

\begin{remark}
 For discretization purposes, as done in Briani et al. \cite{BCT}, one can consider an $L^p$ property for $x\mapsto u(t,x,y)$, and similarly for the derivatives. In this case, one can reformulate Proposition \ref{prop-reg-new} as follows. Let $p\in [1, \infty]$, $q\in\N$. For every $j=0,1,\ldots,q$, $\partial_x ^{2j}f\in \mcC^{p,q-j}_{\pol}(\R,\R_+)$, $\partial_x ^{2j}h\in \mcC^{p,q-j}_{\pol, T}(\R,\R_+)$, and for all $(m,n)\in\N^2$ such that $m+2n\le 2q$ and $m\le 2(q-1)$, $\partial^m_x\partial^n_y h$ locally Hölder continuous in $[0,T)\times\R\times\R^*_+$.
 Then $\partial_x ^{2j}u\in \mcC^{p,q-j}_{\pol, T}(\R,\R_+)$ for every $j=0,1,\ldots,q$. Moreover, the stochastic representation \eqref{stoc_repr-new} holds and, if $q\geq 2$, $u$ solves PDE \eqref{PDE-closed}.
\end{remark}

It is possible to prove these results with the exact proofs presented in \cite{BCT}, with little changes due to the presence of the source term $h$, so we omit the proofs here.

One could be interested to see if we can ask for less regular $f$ and $h$ and still have existence of a classical solution and in which case the solution is unique. In order to do that we first fix other notations that will be required in what follows.

Let $T>0$ and a convex domain $\mcD\subset[0,T]\times\R^{m}$ and $P_1=(t_1,z_1),P_2=(t_2,z_2)\in\mcD$ we define the ``parabolic'' distance $d_\mcP:\mcD\times\mcD\rightarrow\R_+$ as
\begin{equation}\label{parabolic_distance}
 d_\mcP(P_1,P_2) = \big(|t_1-t_2|+|z_1-z_2|^2\big)^{1/2}.
\end{equation}
Let $v:\mcD\rightarrow\R$, using the notation $|v|^\mcD_0 = \sup_{P\in \mcD} |v(P)|$, we introduce the following notation of the $\alpha$-Hölder norm. For $\alpha\in(0,1)$:
\begin{equation}\label{alfa_holder_norm}
 \overline{|v|}^\mcD_\a = |v|^\mcD_0 + \bH^{\mcD}_\a(v),\quad \text{ where } \quad \bH^{\mcD}_\a(v) = \sup_{P\neq Q\mid P,Q\in \mcD}\frac{|v(P)-v(Q)|}{d_\mcP(P,Q)^\a} 
\end{equation}
We will say that $v$ is $\alpha$-Hölder for the parabolic distance if $\bH^{\mcD}_\a(u)<\infty$ that is equivalent to say that $v=v(t,z)$ is $\alpha/2$-Hölder in $t$ and $\alpha$-Hölder in $z$.
We define the $(2+\alpha)$-Hölder norm 
\begin{equation}
 \overline{|v|}^\mcD_{2+\a} =\overline{|v|}^\mcD_\a  +\overline{|\partial_t v|}^\mcD_\a  +\sum_{1\le |l|\le 2} \overline{|\partial^l_z v|}^\mcD_\a.
\end{equation}
To define the weighted $\alpha$-Hölder norm, we must first introduce the notion of weight.
We call for $\tau>0$ and $Q_i=(\tau_i,z)$ $i=1,2$
\begin{equation}\label{weights_for_norm}
  \partial\mcD_\tau=\{(t,z)\in \partial\mcD  \mid t\in[0,\tau]\},\quad d_{Q_i} = \inf_{P\in\partial\mcD_{\tau_i}}d_\mcP(Q_i,P) \text{ and } d_{Q_1,Q_2}=\min(d_{Q_1},d_{Q_2}),
\end{equation}
where $d_{Q_i}$, $i\in\{1,2\}$, measures the parabolic distance of $Q_i$ from the boundary $\partial\mcD_{\tau_i}$. 
Similarly to \eqref{alfa_holder_norm}, for $m\in\N$, $\alpha\in(0,1)$ we define
\begin{equation}\label{weighted_alfa_holder_norm_1}
 |v|^\mcD_{\a,m} = |v|^\mcD_{0,m} + H^{\mcD}_{\a,m}(v),
\end{equation}
where
\begin{equation}\label{weighted_alfa_holder_norm_2}
 |v|^\mcD_{0,m}= \sup_{P\in \mcD} d_{P}^m |v(P)|, \quad H^{\mcD}_{\a,m}(v) = \sup_{P\neq Q\mid P,Q\in \mcD}d_{P,Q}^{m+\a}\frac{|v(P)-v(Q)|}{d_\mcP(P,Q)^\a}.
\end{equation}
With the notation $|v|^\mcD_\a=|v|^\mcD_{\alpha,0}$ weighted $(2+\alpha)$-Hölder norm is as follows
\begin{equation}
 |v|^\mcD_{2+\a} =|v|^\mcD_\a  + |\partial_t v|^\mcD_{\a,1}  +\sum_{1\le |l|\le 2} |\partial^l_z v|^\mcD_{\a,|l|},
\end{equation}
where $l=(l_1,\ldots,l_m)$, $\partial^l_z=\partial^{l_1}_{z_1}\cdots \partial^{l_m}_{z_m}$ and $|l|=\sum_{i=1}^m l_i$.
The main difference between the standard and the weighted $\alpha$-Hölder norm is that the latter allows explosions for the derivatives of $v$ and the difference $|v(P)-v(Q)|$ when we evaluate them near the boundary.

Now we can state the following result,  which clarifies the behavior of the second order spatial derivatives of the solution $u$ near the spatial boundary $\R\times\{0\}$, in which we place hypotheses slightly stronger than the ones in Lemma \ref{lemma-reg}.

\begin{prop}\label{min_hp_pde_existence_sol}
 Let $u$ as in \eqref{u_gen_sol}, $f$ and $h$ such that, for $j=0,1$, $\partial_x^{2j}f\in \mcC^{1-j}_{\pol}(\R\times\R_+)$,  $\partial_x^{2j}h\in \mcC^{1-j}_{\pol,T}(\R\times\R_+)$. Furthermore, we take $|h|^K_{\alpha,2}$, $|\partial_y h|^K_{\alpha,2}<\infty$, for all $K$ convex compact set contained in $[0,T)\times\R\times\R_+$.
Then, for all $t_0\in(0,T)$ and $x_0\in\R$ one has
\begin{equation}\label{0lims_dyyu_dxyu}
 \lim_{(t,x,y)\rightarrow (t_0,x_0,0)} y \partial^2_y u(t,x,y)=0 \quad \text{ and } \quad \lim_{(t,x,y)\rightarrow (t_0,x_0,0)} y \partial_x\partial_y u(t,x,y)=0.
\end{equation}
\end{prop}
\begin{proof}
 The proof takes inspiration from \cite{ET2010}. Under these hypotheses, as shown in Lemma \ref{lemma-reg}, we know that $v=\partial_y u$ solves 
  \begin{equation}
    \begin{cases}
 \big[(\partial_t+\mcL_1+\varrho-b)v +(\lambda^2/2 \partial_x^2+d\partial_x)u-\partial_yh\big](t,x,y)=0,\quad  t\in[0,T), \!\!\!\!&(x,y)\in\R\times\R_+^*,  \\
 v(T,x,y) = \partial_y f(x,y), \quad  \!\!\!&(x,y)\in\R\times\R_+^*,
    \end{cases}
  \end{equation}
 where $\mcL_1=\frac{y}{2}(\lambda^2\partial^2_{x} + 2 \rho\sigma \partial_{x}\partial_{y} +  \sigma^2 \partial^2_{y}) +(c +\rho\lambda\sigma+d y) \partial_x + (a+\sigma^2/2 -by) \partial_y$. 
 We consider, now, a sequence $(t_n,x_n,y_n)_{n \in{\N^*}}\subset [0,T)\times\R\times\R^*_+$ converging to $(t_0,x_0,0)$, where $t_0\in[0,T)$ and $x_0\in\R$. By the convergence $y_n\rightarrow0$, there exists $n_0$ such that for all $n\ge n_0$ $y_n\in(\frac{1}{m_n},\frac{2}{m_n})$, and $m_n\rightarrow\infty$ when $n\rightarrow\infty$. Then we define
  \begin{align*}
    \chi_n:(t,x,y)&\longmapsto(s,\eta,\zeta)=(m_n(t-t_n),m_n(x-x_n),m_ny)
  \end{align*}
 and the functions $w_n$ as
  \begin{equation*}
 w_n(s,\eta,\zeta)= v(\chi_n^{-1}(s,\eta,\zeta))=v\Big(\frac{s}{m_n}+t_n, \frac{\eta}{m_n}+x_n, \frac{\zeta}{m_n}\Big).
  \end{equation*}
 One can check that $w_n$ satisfies the following PDE
  \begin{multline}\label{w_pde}
    \partial_s w_{n}+\frac{\zeta}{2}(\lambda \partial_\eta^2 +2\rho\lambda\sigma\partial_\eta \partial_\zeta+\sigma^2\partial_\zeta^2) w_{n} +(c+\rho\lambda\sigma+d\frac{\zeta}{{m_n}})\partial_\eta w_{n}\\
 +(a+\frac{\sigma^2}{2}+b\frac{\zeta}{{m_n}})\partial_\zeta w_{n} +(\varrho-b)w_{n} + \frac{1}{{m_n}} g_{n} =0,
  \end{multline}
 where
  $$
 g_{n}(s,\eta,\zeta)=(\lambda^2/2 \partial_x^2+d\partial_x)u(\chi_n^{-1}(s,\eta,\zeta)) + \partial_y h(\chi_n^{-1}(s,\eta,\zeta)).
  $$
 We also define 

  $$
 u_n = u\circ \chi^{-1}_n,\quad h_n= h \circ\chi^{-1}_n
  $$

 Then, we consider the rectangle $R_n =[t_n-\frac{2\delta}{m_n},t_n+\frac{2\delta}{m_n}]\times[x_n-\frac{2}{m_n},x_n+\frac{2}{m_n}]\times[\frac{1}{2m_n},\frac{4}{m_n}]$, $(t_n,x_n,y_n)$, and we define $\mcR = \chi_n R_n=[-2\delta,2\delta]\times[-2,2]\times[1/2,4]$.
 By \eqref{w_pde} in $\mcR$ thanks to the Schauder interior estimates (cf. Theorem 5 in Sec.2 of Chap. 3 \cite{AF_book}), one has
  \begin{align*}
 |w_n|^{\mcR}_{2+\alpha} &\le |w_n|^{\mcR}_0 +\frac{1}{m_n}|g_n|^{\mcR}_{\a,2} \\
    &\le |w_n|^{\mcR}_0 +\frac{1}{m_n}\Big(\frac{\lambda^2}{2}|\partial^2_x u_n|^{\mcR}_{\a,2}+|d||\partial_x u_n|^{\mcR}_{\a,2}+|\partial_y h_n|^{\mcR}_{\a,2}\Big) \\
    &\le |v|^{R_n}_0 +\frac{1}{m_n}\Big(\frac{\lambda^2}{2}|\partial^2_x u|^{R_n}_{\a,2}+|d||\partial_x u|^{R_n}_{\a,2}+|\partial_y h|^{R_n}_{\a,2}\Big) \\
    &\le |v|^{R_n}_0 +\frac{C}{m_n}\Big(|u|^{R_n}_0+|h|^{R_n}_{\a,2}+|\partial_y h|^{R_n}_{\a,2}\Big)\\
    &\le |v|^{R_n}_0 +\frac{C}{m_n}\Big(|u|^{K}_0+|h|^{K}_{\a,2}+|\partial_y h|^{K}_{\a,2}\Big)<\hat{C}<\infty,
  \end{align*}
 where we pass from the third to fourth line using in succession: that exist $C_1>0$ such $|\partial_x u|^{R_n}_{\a,2}\le C_1 |\partial_x u|^{R_n}_{\a,1}$, that we can upper bound thanks to Schauder interior estimates on $u$ the two seminorms $|\partial_x u|^{R_n}_{\a,1}$, $|\partial^2_x u|^{R_n}_{\a,2}$ as follows
  $$
 |\partial_x u|^{R_n}_{\a,1} + |\partial^2_x u|^{R_n}_{\a,2} \le |u|^{R_n}_{2+\a}\le |u|^{R_n}_{0} + |h|^{R_n}_{\a,2}.
  $$
 The passage to the fifth line is because there exist $\hat{n}\in\N$ and $K$ compact set such that $R_n\subset K$ for all $n\ge\hat{n}$. The uniform bound by $\hat{C}$ follow by $|u|^{R_n}_0\rightarrow |v(t_0,x_0,0)|$.
 Now that we have the weighted holder norm estimate $|w_n|^{\mcR}_{2+\alpha}<\hat{C}$, we can consider a smaller rectangle $\mcR' = [\delta,\delta]\times[-1,1]\times[1,2]$ who has strictly positive distance from $\partial \mcR$ and get, in terms of standard holder norms, the uniform bound $\overline{|w_n|}^{\mcR '}_{2+\alpha}<\tilde{C}$. Now using the equi-boundedness and equi-continuity of a general subsequence $(w_{n_j})_j$ and of its derivatives $(\partial_\eta w_{n_j})_j$ and $(\partial_\zeta w_{n_j})_j$, by Ascoli-Arzela Theorem one can find subsequences $(w_{n_{j_k}})_k$, $(\partial_\eta w_{n_{j_k}})_k$ and $(\partial_\zeta w_{n_{j_k}})_k$ that converge uniformly on $\mcR'$ to continuous functions $w$, $\partial_\eta w$ and $\partial_\zeta w$ respectively. Being the uniform limit of the original sequence $w_n$ a constant equal to $v(t_0,x_0,0)$ then $\partial_\eta w_{n}$ and $\partial_\zeta w_{n}$ have to converge uniformly to $\partial_\eta w=0$ and $\partial_\zeta w=0$. So
  $$
  0\xleftarrow[\infty\leftarrow n]{|\cdot|^{\mcR'}_0} \zeta\partial_\eta w_{n}(s,\eta,\zeta) = \frac{\zeta}{m_n}\partial_y v\Big(\frac{s}{m_n}+t_n, \frac{\eta}{m_n}+x_n, \frac{\zeta}{m_n}\Big)
  $$
 and 
  $$
  0\xleftarrow[\infty\leftarrow n]{|\cdot|^{\mcR'}_0} \zeta\partial_\zeta w_{n}(s,\eta,\zeta) = \frac{\zeta}{m_n}\partial_x v\Big(\frac{s}{m_n}+t_n, \frac{\eta}{m_n}+x_n, \frac{\zeta}{m_n}\Big),
  $$
so, in particular, the limits hold if we restrict to the sequence $(t_n,x_n,y_n)_{n\in\N^*}$. Being $\partial_y v = \partial^2_y u$ and $\partial_x v = \partial_x \partial_y u$, then the proof is completed.
\end{proof}

\begin{remark} \label{remark_PDE_on_boundary}
 Under the hypotheses of Proposition \ref{min_hp_pde_existence_sol}, the equation $$
  \partial_t u(t,x,0) + (c \partial_x + a\partial_y +\varrho) u(t,x,0) =h(t,x,0)
  $$ 
 is satisfied for all $t\in[0,T)$, and so does not hold only as a limit. 
\end{remark}
\begin{proof}
 Let $t\in(0,T)$, since $\partial_y u$ is continuous up to the boundary, we expand $u$ in the direction $y$ around $(t,x,0)$ and $(t+\epsilon,x,0)$, and use the mean value theorem to get
  \begin{align*}
    \partial_t u(t,x,0)= \lim_{\epsilon \rightarrow 0^+} \frac{u(t+\epsilon,x,0) - u(t,x,0)}{\epsilon} &= \lim_{\epsilon\rightarrow 0^+}\frac{u(t+\epsilon,x,\epsilon^2) - u(t,x,\epsilon^2)+O(\epsilon^2)}{\epsilon} \\
    &= \lim_{\epsilon\rightarrow 0^+}  \partial_t u(t+\xi_\epsilon, x,\epsilon^2) +O(\epsilon),
  \end{align*}
 for some $\xi_\epsilon\in(0,\epsilon)$. Since for $y>0$ one has $ \partial_t u= -(\mcL +\varrho)u + h$, then by \eqref{0lims_dyyu_dxyu} and continuity of $\partial_x u$ and $\partial_y u$ up to the boundary $\{y=0\}$, one has 
  $$ 
  \partial_t u(t,x,0)=\lim_{\epsilon\rightarrow 0^+}-(\mcL+\varrho)u(t+\xi_\epsilon, x,\epsilon^2) + h(t+\xi_\epsilon, x,\epsilon^2) =  -(c \partial_x + a\partial_y +\varrho) u(t,x,0) +h(t,x,0).
  $$
Since the functions on both sides are continuous up to $t=0$, this is verified for $t=0$, too.
\end{proof}

So, we have just proven the following result.
\begin{prop}\label{hp_minimal_u_sol}
 Let $u$ defined as in \eqref{u_gen_sol}. Let $f$ and $h$ such that, as $j=0,1$, $\partial_x^{2j}f\in \mcC^{1-j}_{\pol}(\R\times\R_+)$,  $\partial_x^{2j}h\in \mcC^{1-j}_{\pol,T}(\R\times\R_+)$, $h$, $\partial_y h$ and $|h|^K_{\alpha,2}$, $|\partial_y h|^K_{\alpha,2}<\infty$ for all $K$ compact set contained in $[0,T)\times\R\times\R_+$.
 Then $\partial_x^{2j}u\in \mcC^{1-j}_{\pol,T}(\R\times\R_+)$ for $j=0,1$ and solves
  \begin{equation}
    \begin{cases}
    \partial_t u(t,x,y)+ \mcL u(t,x,y) +\varrho u(t,x,y)= h(t,x,y),\qquad& t\in [0,T), \,(x,y)\in \R\times\R_+,\\
 u(T,x,y)= f(x,y),\qquad& (x,y)\in \R\times\R_+.
    \end{cases} 
  \end{equation}
\end{prop}

We can compare the result we have just proven with the one in Proposition \ref{prop-reg-new}.
In the latter one, Briani et al. use the stochastic representation \eqref{stoc_repr-new_chap4}, with $q=2$, to prove the function $u$ belongs to $\mcC^{1,2}([0,T]\times\R\times\R_+)$ and so by continuity of all the derivatives involved in the problem,  that initially (for example taking only the final data $f$ just continuous) is solved only over $[0,T)\times\R\times\R^*_+$, is solved even over $[0,T)\times\R\times\{0\}$. 
To get all this regularity for $u$ one has to request a lot of regularity on $f$ and $h$: $\partial_x ^{2j}f\in \mcC^{2-j}_{\pol}(\R\times\R_+)$, $\partial_x ^{2j}h\in \mcC^{2-j}_{\pol, T}(\R\times\R_+)$, and for all $(m,n)\in\N^2$ such that $m+2n\le 4$ and $m\le 6$, $\partial^m_x\partial^n_y h$ locally Hölder continuous in $[0,T)\times\R\times\R_+^*$. We are capable of lowering these requests on $f$ and $h$ because even if  $u$ is less regular, it can still solve the reference problem. With the hypotheses considered in \eqref{hp_minimal_u_sol}, we do not have the continuity of $\partial_x\partial_y u$ and $\partial_y^2$, but we prove only \eqref{0lims_dyyu_dxyu} and this is enough, as shown in Remark \ref{remark_PDE_on_boundary}, to prove the PDE is solved over the boundary $[0,T)\times\R\times\{0\}$.

We now state sufficient conditions to ensure the uniqueness of the solution.
\begin{prop}
 There is at most one classical solution $u\in \mcC^{1,2}([0,T)\times(\R\times\R^*_+)) \cap \mcC^{1,1,1}([0,T)\times\R\times\R_+)\cap \mcC([0,T]\times\R\times\R_+)$ to the PDE \eqref{reference_PDE} such that the solution has polynomial growth in $(x,y)$ uniformly in $t$. In particular, under the hypothesis of Proposition \ref{hp_minimal_u_sol}, $u$ defined as in \eqref{u_gen_sol} is the unique solution.
\end{prop}
\begin{proof}
 Suppose that $u$ and $v$ are two solutions in the reference space $\mcC^{1,2}([0,T)\times(\R\times\R^*_+)) \cap \mcC^{1,1,1}([0,T)\times\R\times\R_+)\cap \mcC([0,T]\times\R\times\R_+)$, clearly the difference $w=u-v$ lies in the same space. For simplicity, we reverse the time by the change of variable $t\mapsto T-t$, so $w$ solves
  \begin{equation}
    \begin{cases}
 (\partial_t-\mcL -\varrho)w(t,x,y) = 0,\quad   t\in(0,T], &x\in\R, y\in\R_+, \\
 w(0,x,y) = 0, \quad  &x\in\R, y\in \R_+.
    \end{cases}
  \end{equation}
 From now on we consider $\varrho=0$, because $\exp(-\varrho t)w(t,x,y)$ solves the problem with the constant equal to 0, and the function $\mcM_{\varrho}:w(t,x,y)\mapsto \exp(-\varrho t)w(t,x,y)$ is a bijection from the reference space to itself.
 Let $L$ the first even integer such that $|w(t,x,y)|\le C(1+x^L+y^L)$. We call $h(x,y)=1+x^{L+2}+y^{L+2}$ and with a little algebra one can show that exists $M>0$ such that
  $$
 \mcL h(x,y)= \frac{(L+2)(L+1)}{2}y(\lambda^2 x^L+\sigma^2y^L)+(L+2)((c+d y)x^{L+1}+(a-by)y^{L+1})<Mh(x,y).
  $$
 Let $\varepsilon>0$, we define $w^\varepsilon:[0,T]\times\R\times\R_+\rightarrow \R$ by 
  $$w^\varepsilon(t,x,y) = w(t,x,y)+\varepsilon e^{Mt}h(x,y),$$ 
 then
  $$
 (\partial_t -\mcL) w^\varepsilon(t,x,y) = \varepsilon e^{Mt}(M-\mcL)h(x,y)>0
  $$
 for all the interior points. Let $\Gamma:=\{(t,x,y)\mid w^\varepsilon(t,x,y)< 0\} $, we remark that $\Gamma$ is bounded (because of the growth bound $|w(t,x,y)|\le C(1+x^L+y^L)$), and then $\bbar{\Gamma}$ is compact by continuity of $w^\varepsilon$. Assume that $\Gamma\neq \emptyset$ and define $t_0=\inf\{t\ge0\mid (t,x,y)\in\bbar{\Gamma} \text{, for some } (x,y)\in\R\times\R_+ \}$. We consider a point $(t_0,x_0,y_0)\in\bbar{\Gamma}$. By continuity of $w^\varepsilon$ and definition of $t_0$, $w$ must be equal 0 in $(t_0,x_0,y_0)$. In the meantime, being $w^\varepsilon(0,x,y)\ge 1$ and $\bbar{\Gamma}$ compact, one has $t_0>0$.
 We suppose first $y_0=0$. Then, by the fact that $t_0$ is an infimum, we have 
  $$
  \partial_t w^\varepsilon(t_0,x_0,0)\le 0, \qquad \partial_x w^\varepsilon(t_0,x_0,0)=0, \qquad\partial_y  w^\varepsilon(t_0,x_0,0)\ge 0 
  $$
 otherwise we can find a triple $(t_1,x_1,y_1)$ with $t_1<t_2$ such that $w^\varepsilon(t_1,x_1,y_1)<0$ contradicting the minimality of $t_0$. Being $a>0$, one has
  $$
  0\ge \partial_t w^\varepsilon(t_0,x_0,0) - a \partial_y w^\varepsilon(t_0,x_0,0) = \varepsilon e^{M t}M(1+x_0^{L+2})>0
  $$
 so $y_0=0$ is not possible. We consider now $y_0>0$, then $(t_0,x_0,y_0)$ is an interior point of the domain. By minimality of $t_0$ one has $\partial_t w^\varepsilon(t_0,x_0,y_0)\le 0$ and $(x_0,y_0)$ is a minimum point for the map $(x,y) \mapsto w^\varepsilon(t_0,x,y)$. Then, this map has a gradient equal to 0 and Hessian positive semi-definite, so $\mcL w^\varepsilon \ge 0$. One has
  $$
    0 \ge (\partial_t -\mcL) w^\varepsilon(t_0,x_0,y_0) = \varepsilon e^{Mt}(M-\mcL)h(x_0,y_0)>0.
  $$
 This contradiction implies that $\Gamma = \emptyset$. Since this holds for all $\varepsilon>0$, it follows that $w\ge0$. Applying the same argument to $-w$ shows that $w\le0$ and so $w=0$.
\end{proof}

\section{Existence and uniqueness of viscosity solutions}\label{viscosity_section}
Here, we want to explore the extended problem \eqref{reference_PDE} from the point of view of viscosity solutions in order to reduce the regularity on the function $f$. Now, we introduce some definitions (cf. \cite{CIL92}) that will be useful from now on.

Let $F:(0,T]\times\R^m\times\R\times\R^m\times\mcS(m)\rightarrow\R$ a continuous function where $\mcS(m)$ is the set of $m\times m$-dimensional, $\R$-valued symmetric matrices.
\begin{definition}[Degenerate ellipticity]
  $F$ is called degenerate elliptic if it is nonincreasing in its matrix argument
    $$
      F(t,x,u,p,X) \leq F(t,x,u,p,Y) \text{ for } Y\leq X,
    $$
    with the classical ordering $\leq$ over $\mcS(m)$ defined by the relation
    $$
      Y\leq X  \Leftrightarrow \langle Y\zeta,\zeta  \rangle \leq \langle X\zeta,\zeta  \rangle \text{ for all } \zeta\in\R^m.
    $$
\end{definition}

\begin{definition}[Proper]
  $F$ is called proper if it is degenerate elliptic and nondecreasing in $u$.
\end{definition}

\begin{remark}
  With the change of variable $s=T-t$ the original problem \eqref{reference_PDE} becomes
  \begin{equation}\label{reference_PDE_forward}
    \begin{cases}
      \partial_s u(s,x,y) +F(s,(x,y), u(s,x,y), D_{(x,y)}u(s,x,y),D^2_{(x,y)}u(s,x,y) )  = 0,  \\ 
      \phantom{\partial_s u(s,x,y) +F(s,(x,y), u(s,x,y), D_{(x,y)}u(s,x,y),D^2_{(x,y)})}\forall s\in(0,T],\, x\in\R,\, y\in\R_+, 
      \\
      u(0,x,y) = f(x,y), \quad \forall x\in\R, y\in \R_+,
    \end{cases}
  \end{equation}
  where
  \begin{equation}\label{F_logheston}
    F(s,(x,y),r,p,X)= -\frac{y}{2}(\lambda^2X_{1,1} +2\rho\lambda\sigma X_{1,2} +\sigma^2X_{2,2}) -\mu_X(y) p_1 -\mu_Y(y) p_2 - \varrho r+h(T-s,x,y)
  \end{equation}
  is degenerate elliptic (and proper if $\varrho\le0$).
\end{remark}

We consider a convex domain (possibly closed) $\mcO\subseteq \R^m$, $T>0$ and we name $\mcO_T = (0,T]\times\mcO$.
We study the following partial differential equation problem
\begin{equation}
  \begin{cases}
    u_t(t,x) + F(t,x,u(t,x),D_xu(t,x),D^2_xu(t,x))=0 \quad&\text{ if } t\in(0,T],\text{ } x\in\mcO,\\
    u(0,x)= f(x),   &\qquad\qquad\qquad x\in\mcO.
  \end{cases}
\end{equation}
We define 
\begin{align*}
  \LSC(\mcO_T) &= \{f:\mcO_T\rightarrow\R \mid   \text{$f$ is lower semi-continuous at every } (t,x)\in \mcO_T \}, \\
  \USC(\mcO_T) &= \{f:\mcO_T\rightarrow\R \mid   \text{$f$ is upper semi-continuous at every } (t,x)\in \mcO_T  \},
\end{align*}
and we give the following definition.
\begin{definition}\label{viscosity_property}
  Given a function $u$ and $(t,x)\in(0,T]\times\mcO$, we say that at $(t,x)$
  $$
  \partial_t u(t,x) +F(t,x,u(t,x),D_xu(t,x),D^2_x u(t,x))\ge 0\,(\text{resp.}\le 0)\text{ in viscosity sense }
  $$   if, for each smooth function $\phi$ such that $u-\phi$ has a local minimum (resp. a local maximum) at $(t,x)$ 
  $$
  \partial_t \phi(t,x) + F(t,x, u(t,x),D_x\phi(t,x),D^2_x\phi(t,x))\ge 0\,(\text{resp.}\le 0) \text{ in classical sense} .
  $$ 
\end{definition}
We introduce the notion of semijets to give an equivalent definition that will be useful in the following.
\begin{definition}\label{semijets}
  Let $u:\mcO_T\rightarrow \R$, then its upper parabolic second order semijet $\mcP_\mcO^{2,+} u$ is defined by
  \begin{align*}
    \mcP_\mcO^{2,+}u:\mcO_T&\rightarrow \mathscr{P}(\R\times\R^m\times\mcS(m))\\
    (t,x)&\mapsto \mcP_\mcO^{2,+}u(t,x)
  \end{align*}
  where $(c,p,X)$ lies in the set $\mcP_\mcO^{2,+}u(t,x)$ if 
  \begin{equation}\label{upper_semijet}
    \begin{split}
      u(s,z) \le  u(&t,x) +b(s-t)+ \langle p,z-x\rangle + \frac{1}{2}\langle X(z-x),x-z\rangle \\
      &+o(|s-t|+|z-x|^2) \text{ as } \mcO_T\ni(s,z)\rightarrow (t,x),
    \end{split}
  \end{equation}
  and we define the lower parabolic second order semijet $\mcP_\mcO^{2,-} u:= -\mcP_\mcO^{2,+}(-u)$.
  We also define the closure of these set-valued mappings as 
  \begin{multline}
    \bmcP_\mcO^{2,+}u(t,x)=\{(c,p,X)\in\R\times\R^m\times\mcS(m)\mid \exists\big((t_n,x_n),c_n,p_n,X_n\big)\in\mcO_T\times\R\times\R^m\times\mcS(m)\text{ s.t.} \\
    (c_n,p_n,X_n)\in \bmcP_\mcO^{2,+}u(t_n,x_n)\text{ and } \big((t_n,x_n),u(t_n,x_n),c_n,p_n,X_n\big)\rightarrow \big((t,x),u(t,x),c_n,p_n,X_n\big) \},
  \end{multline}
  and $\bmcP_\mcO^{2,-}u$, closure of $\mcP_\mcO^{2,-}u$, is defined in the same way.
\end{definition}

We now give the definition of viscosity super and sub-solutions. 
\begin{definition}[Viscosity super-solution (sub-solution)]\label{def_viscosity_sol}
  Let $F$ be a proper operator and $T>0$. A function $u$ that is $\LSC(\mcO_T)$  (resp. $\USC(\mcO_T)$)  is called a viscosity super-solution (resp. sub-solution) with initial value $f$ if, 
  \begin{itemize}
    \item for any $(t,x)\in\mcO_T$, $\partial_t\phi(t,x) + F(t,x,u(t,x),D_x u(t,x),D^2_x u(t,x))\geq 0 \text{ $($resp. $\le0)$}$ in viscosity sense,
    \item for any $x\in\mcO$, $u(0,x)\geq f(x) \text{ $($resp. $\le f(x)).$}$
  \end{itemize}
  A function $u$ that is both a super and a sub-solution is called a viscosity solution.
\end{definition}
 The following result gives an interesting equivalent definition.
\begin{prop}
  Let $(t,x)\in\mcO_T$. Then
  \begin{align}
    \mcP_\mcO^{2,+}u(t,x) =\{(\partial_t\phi(t,x),D_x\phi(t,x),D^2_x\phi(t,x))\mid \phi\text{ is } \mcC^{1,2} \text{ and } u-\phi  \text{ has local max. at } (t,x) \},\\
    \mcP_\mcO^{2,-}u(t,x) =\{(\partial_t\phi(t,x),D_x\phi(t,x),D^2_x\phi(t,x))\mid \phi\text{ is } \mcC^{1,2} \text{ and } u-\phi  \text{ has local min. at } (t,x) \}.
  \end{align}
\end{prop}
\begin{proof}
  The inclusion $\supseteq$ follows easily by using the local maximum (respectively minimum) property and developing $\phi$ using Taylor Theorem around $(t,x)$ up to the first order in $t$, and to the second in $x$. The nontrivial inclusion $\subseteq$ requires constructing, for every $(c,p,X)\in\mcP_\mcO^{2,+}u(t,x)$, a regular function such that the difference $u-\phi$ has a local minimum at $(t,x)$. We refer to Fleming and Soner \cite{FS93book}, V.4 Proposition 4.1.
\end{proof}
As an immediate consequence, we have the following characterization of super and sub-viscosity solutions.
\begin{corollary}
  A function $w\in\USC(\mcO_T)$ is a viscosity sub-solution with initial value $f$, if and only if
  \begin{equation}\label{subsol_2nd_def}
    \begin{cases}
      c + F(t,x,w(t,x),p,X)\leq 0 \text{ for all } (t,x)\in\mcO_T \text{ and } (c,p,X)\in\mcP_\mcO^{2,+}w(t,x), \\
      w(0,x)\leq f(x) , \text{ for any } x\in\mcO.
    \end{cases}
  \end{equation}
  A function $v\in\LSC(\mcO_T)$  is a viscosity super-solution with initial value $f$, if and only if
  \begin{equation}\label{supersol_2nd_def}
    \begin{cases}
      c + F(t,x,v(t,x),p,X)\geq 0,\text{ for all } (t,x)\in\mcO_T \text{ and } (c,p,X)\in\mcP_\mcO^{2,-}v(t,x), \\
      v(0,x)\geq f(x) , \text{ for any } x\in\mcO.
    \end{cases}
  \end{equation}
\end{corollary}

Here, we give a key Lemma to prove the verification Theorem, which says that viscosity solutions are stable under local uniform convergence. 
\begin{lemma}[Stability]\label{sequence_lemma_general}
  Let $F,(F_n)_{n\in\N}$ be continuous and degenerate elliptic such that for all $\mcK^*_T\subseteq \mcO_T\times\R\times\R^m\times\mcS(m)$
  $$
    |F-F_n|_0^{\mcK^*_T}\xrightarrow[n\rightarrow\infty]{}0,
  $$
  and $(u_n)_{n_\in\N}\subset C(\mcO_T)$ such that
  \begin{enumerate}
    \item for all $n$,  $\partial_t u_n(t,x) + F_n(t, x,u_n(t, x),D_x u_n(t, x),D^2_x u_n(t, x))\geq 0 $ $($resp. $\le0)$ for all $(t,x)\in\mcO_T$ in viscosity sense, \label{property_one _seq_lem_gen}
    \item there exists $u$ such that for each $\mcK_T\subseteq \mcO_T$ compact set one has 
    $$|u_n-u|_0^{\mcK_T}\xrightarrow[n\rightarrow\infty]{} 0.$$
  \end{enumerate}
  Then $u\in C(\mcO_T)$ satisfies $\partial_t u(t,x) + F(t, x,u(t, x),D_x u(t, x),D^2_x u(t, x))\geq 0 $ $($resp. $\le0)$  for all $(t,x)\in\mcO_T$ in viscosity sense.
\end{lemma}

\begin{proof}
  We only prove that $u$ satisfies $\partial_t u(t,x) + F(t, x,u(t, x),D_x u(t, x),D^2_x u(t, x))\geq 0$ in viscosity sense, the reverse inequality can be proven in the same way. Let $\phi\in \mcC^{1,2}(\mcO_T)$ and $(t,x)\in{\mcO_T}$ such that is a global minimum for $u-\phi$. We consider a compact neighbourhood $\mcK_T$ of $(t,x)$. Suppose the minimum is strict at $(t,x)$, then thanks to the local uniform convergence of $(u_n)_{n\in\N^*}$  exists a sequence of points $(t_n,x_n)_{n\in\N^*}$ eventually in the interior of $\mcK_T$ that are minima for the sequence $(u_n-\phi)_{n\in\N^*}$ and such that $(t_n,x_n)\rightarrow (t,x)$.
  Being $(t_n,x_n)$ minimizer for $u_n-\phi$ with $\phi$ smooth, then by \textit{(1)}
  \begin{equation*}
    0\le \partial_t \phi(t_n,x_n) + F_n(t_n,x_n,u_n(t_n,x_n),D_x\phi(t_n,x_n),D^2_x\phi(t_n,x_n)).
  \end{equation*}
  By uniform convergence of $u_n$ through $u$ over $\mcK_T$, one has $u_n(t_n,x_n)\rightarrow u(t,x)$. Then thanks to continuity of $D\phi$ and $D^2\phi$ one has that $(t_n,x_n,u_n(t_n,x_n),D_x u_n(t_n,x_n),D^2_x u_n(t_n,x_n) )_{n\in\N}\subset \mcK_T \times\mcR$ compact set of $(0,T]\times\mcO\times\R\times\R^m\times\mcS(d)$.
  So, thanks to uniformly convergence of $F_n$ through $F$  over $\mcK_T\times\mcR$  we conclude that
  \begin{multline}
    0\le \lim_{n\rightarrow\infty} \partial_t \phi(t_n,x_n) + F_n(t_n, x_n, u_n(t_n,x_n),D_x\phi(t_n,x_n),D^2_x\phi(t_n,x_n)) \\
    = \partial_t \phi(t,x) + F(t,x, u(t,x),D_x\phi(t,x),D^2_x\phi(t,x)).
  \end{multline}
 If the minimum is not strict we consider the function $\phi^*(s,y) = \phi(s,y)-(t-s)^2 -|x-y|^4$. This function $\phi^*$ is such that $u-\phi^*$ has strict min in $(t,x)$ and has same derivatives up to first order in $t$ and up to the second one in $(t,x)$, so we can apply the previous technique and conclude remarking
 \begin{multline*} 
  \lim_{n\rightarrow\infty} \partial_t \phi^*(t_n,x_n) + F_n(t_n,x_n, u_n(t_n,x_n), D_x\phi^*(t_n,x_n), D^2_x\phi^*(t_n,x_n))\\
  = \lim_{n\rightarrow\infty} \partial_t \phi(t_n,x_n) + F_n(t_n,x_n, u_n(t_n,x_n), D_x\phi(t_n,x_n), D^2_x\phi(t_n,x_n)). 
\end{multline*}
\end{proof}

We want to use the notion of viscosity solution to extend the verification results obtained in the previous section. Given 
\begin{equation}\label{u_sol_chap3}
  u(t,x,y)=\E\bigg[e^{\varrho (T-t)}f(X^{t,x,y}_T,Y^{t,y}_T)-\int_t^Te^{\varrho (s-t)} h(s,X^{t,x,y}_s,Y^{t,y}_s)ds\bigg],
\end{equation}
we will verify that $u^F$ (candidate forward solution)
\begin{equation}\label{uF_sol_chap3}
  u^F(t,x,y)=u(T-t,x,y)
\end{equation}
is a viscosity solution of the forward problem \eqref{reference_PDE_forward} with more general initial data $f$ (even with discontinuities) and source terms $h$. We emphasize that whenever we have $u$ solution of \eqref{reference_PDE}, we know that $u^F$ is a solution of \eqref{reference_PDE_forward} and vice versa.

\subsection{Continuous initial data}
In this subsection, we deal with problem \eqref{reference_PDE_forward} where the initial data $f$ is chosen to be just continuous.

In the sequel, we will need some smoothing arguments, which can be resumed as follows.
\begin{lemma}\label{Mollifier_lemma}
  Let $f\in \mcC(\R\times\R_+)$ and $h\in \mcC([0,T]\times \R\times\R_+)$.
  Then there exist 
  \begin{itemize}
    \item $(f_n)_{n\in\N}\subset \mcC^ \infty(\R^2) $ such that $|f_n-f|_0^{\mcK}\rightarrow 0$ for every compact set $\mcK \subset \R\times\R_+$,
    \item $(h_n)_{n\in\N}\subset \mcC^ \infty(\R^3) $ such that $|h_n-h|_0^{\mcK_T}\rightarrow 0$ for every compact set $\mcK_T \subset [0,T]\times\R\times\R_+$.
  \end{itemize}
  Furthermore, if $f$ and $h$ are uniformly continuous and bounded then $|f_n-f|_0^{\R\times\R_+}\rightarrow 0$ and $|h_n-h|_0^{[0,T]\times\R\times\R_+}\rightarrow 0$.
\end{lemma}
\begin{proof}
  We need only to extend $f$ and $h$ in a continous way respectively around $\R\times\R_+$ and $[0,T]\times \R\times\R_+$ and than to take convolution with a mollifier $(\varphi_n)_{n\in\N}$. We finish applying Proposition 4.21 of \cite{Brezis2010book}.  
  We start by extending $f$ and $h$ in the following continuous way
  \begin{equation*}
    \tf(x_1,x_2)= f(x_1,0\vee x_2) \quad{ and } \quad \tilde{h}(t,x_1,x_2) = h(0\vee t\wedge T,x_1,x_2).
  \end{equation*}
  If $f$ and $h$ are uniformly continuous and bounded, then the same proof in \cite{Brezis2010book} guarantees uniform convergence without any restriction over compact sets (so in $\R\times\R_+$ and $[0,T]\times\R\times\R_+)$.
\end{proof}

We recall that for every compact set $\mcK_T$ in $[0,T]\times\R\times\R_+$ and $p\in\N$, one has
  \begin{equation}
    \sup_{(t,x,y)\in\mcK_T} \E\left[\big|X_T^{t,x,y}\big|^p+\big|Y^{t,y}_T\big|^p\right] < \infty,
  \end{equation}
where $\big((X_T^{t,x,y},Y^{t,y}_T)\big)_{t\in[0,T]}$ denotes the solution to \eqref{referenceDiffusion}.

We are now ready to prove that \eqref{reference_PDE} has a viscosity solution, with quite general requests on the function $f$ giving the Cauchy condition. We underline that we do not operate any restriction on the parameters: no Feller's condition is required.

\begin{prop}\label{existence_viscsol_fcont}
  Let $f\in C(\R\times \R_+)$ and $h\in C([0,T)\times\R\times\R_+)$ be such that for all compact set $\mcK_T\in[0,T]\times\R\times\R_+$ there exists $p>1$ such that
  \begin{equation}\label{integrability_condition}
    \sup_{(t,x,y)\in\mcK_T}\|f(X_T^{t,x,y},Y^{t,y}_T)\big\|_{L^p(\Omega)}, \sup_{(t,x,y)\in\mcK_T}\int_t^T \|h(s,X^{t,x,y}_s,Y^{t,y}_s)\|_{L^p(\Omega)}ds<\infty .
  \end{equation}
  Then, $$u(t,x,y)=\E\left[e^{\varrho (T-t)}f(X_T^{t,x,y},Y^{t,y}_T)-\int_t^T e^{\varrho (s-t)} h(s,X^{t,x,y}_s,Y^{t,y}_s)ds\right]$$ is $C([0,T]\times\R\times\R_+)$ and is a viscosity solution to the PDE \eqref{reference_PDE}.
\end{prop}

\begin{proof}
  In order to simplify the proof, we write $u=v-w$ where
  $$
    v(t,x,y) = \E[e^{\varrho (T-t)}f(X_T^{t,x,y},Y^{t,y}_T)] \quad\text{ and }\quad w(t,x,y) = \E\left[\int_t^T e^{\varrho (s-t)} h(s,X^{t,x,y}_s,Y^{t,y}_s)ds\right],
  $$
  and we will show all the due convergences in the ``$w$'' part, being the $v$ part similar and simpler.
  Let $R>1/T$ and consider the smoothly truncated version $f^R$ and $h^R$ of $f$ and $h$, that is, $f^R(x,y)=f(x,y)\zeta_R(x,y)$ and $h^R(t,x,y)=h(t,x,y)\xi_R(t)\zeta_R(x,y)$, where  $\zeta_R$ and $\xi_R$ are smooth function such that 
  $$
  \ind{B(0,R)}\leq \zeta_R\leq \ind{B(0,R+1)}\quad \text{ and }\quad\ind{[0,T-\frac{1}{R}]}\leq \xi_R\leq \ind{[0,T-\frac{1}{R+1}]}.
  $$
  $f^R$ and $h^R$, being continuous and having compact support, are, in particular, uniformly continuous and bounded. Then by Lemma \ref{Mollifier_lemma} there exist two sequences $(f^R_n)_{n\in\N}$, $(h^R_n)_{n\in\N}$ that approximate in uniform norm $f^R$ over $\R\times\R_+$  and $h^R$ over $[0,T]\times\R\times\R_+$.
  We define in an intuitive way $v^R,v^R_n, w^R \text{ and } w^R_n$ as the functions obtained by replacing in $v$ and $w$ the functions $f$ and $h$ with $f^R$, $f^R_n$ and $h^R$, $h^R_n$. 
  Being the sequences $(f^R_n)_{n\in\N}\subset C^\infty_c(\R\times\R_+)\subset \mcC^{\infty}_{\pol}(\R\times\R_+)$ and $(h^R_n)_{n\in\N}\subset C^\infty_c([0,T]\times\R\times\R_+)\subset \mcC^{\infty}_{\pol}([0,T]\times\R\times\R_+)$, they satisfy the regularity conditions in Proposition \ref{prop-reg-new}, then for all $n\in\N$, $u^R_n(t,x,y)$ is a $\mcC^\infty$ classical solution to \eqref{reference_PDE} with final value $f^R_n$ and source term $h^R_n$ instead of $h$. Now, chosen a compact set $\mcK_T\subset[0,T]\times\R\times\R_+$, for all $R>1/T$ one has
  \begin{align*}
    |u^R-u^R_n|_0^{\mcK_T} \le& \, |v^R-v^R_n|_0^{\mcK_T} +|w^R-w^R_n|_0^{\mcK_T} \\
    =& \sup_{(t,x,y)\in \mcK_T} \big|\E\big[e^{\varrho (T-t)}\big(f^R(X_T^{t,x,y},Y^{t,y}_T)-f^R_n(X_T^{t,x,y},Y^{t,y}_T)\big)\big]\big| \\
    &+ \sup_{(t,x,y)\in \mcK_T} \bigg|\E\Big[\int_t^T e^{\varrho (s-t)} \big(h^R(s,X^{t,x,y}_s,Y^{t,y}_s)-h^R_n(s,X^{t,x,y}_s,Y^{t,y}_s)\big)ds\Big]\bigg| \\
    \le& \, C_1|f^R-f^R_n|_0^{\R\times\R_+}+ C_2|h^R-h^R_n|_0^{[0,T]\times\R\times\R_+}\longrightarrow 0.
  \end{align*}
  We show now that $|w^R-w|_0^{\mcK_T}\rightarrow0$ when $R\rightarrow\infty$, similarly one can do the same for $v$ and get the convergence in uniform norm for $u$. We write $z=(x,y)$ and $ Z^{t,z}_T = (X_T^{t,x,y},Y^{t,y}_T)$ and show that 
  \begin{align*}
    \bigg|\E\Big[\int_t^T e^{\varrho (s-t)} \big(h(s,X^{t,x,y}_s,Y^{t,y}_s)-&h^R(s,X^{t,x,y}_s,Y^{t,y}_s)\big)ds\Big]\bigg|\\
    &\le 2e^{(0\vee\varrho)T} \E\Big[\int_t^T |h(s,Z^{t,z}_s)|\ind{([0,T-\frac{1}{R}]\times B_R(0))^C}(s,Z^{t,z}_s)ds\Big],   \\
    &\le 2e^{(0\vee\varrho)T} \int_t^T\E\big[ |h(s,Z^{t,z}_s)| \ind{B^C_R(0)}(Z^{t,z}_s)\big] ds,   \\
    &\le 2e^{(0\vee\varrho)T} \int_t^T \|h(s,Z^{t,z}_s)\|_{L^P(\Omega)}ds\, \P(|Z^{t,z}_T|>R)^{\frac{p-1}{p}}, 
  \end{align*}
  where we used the Tonelli Theorem to exchange the order of the expected value and the integral and the Hölder inequality to get the last inequality. Now, using Markov inequality,
  \begin{equation*}
    \P(|Z^{t,z}_T|>R)\leq \frac{\E[|Z^{t,z}_T|]}{R}.
  \end{equation*}
  Using this last inequality and passing to the supremum over $\mcK_T$, we get
  {\small
   \begin{equation*}
    |w^R-w|_0^{\mcK_T}\le C \sup_{(t,x,y)\in\mcK_T}\int_t^T \|h(s,X^{t,x,y}_s,Y^{t,y}_s)\|_{L^p(\Omega)}ds \frac{\sup_{(t,x,y)\in\mcK_T}\E[|(X_T^{t,x,y},Y^{t,y}_T)|]^{\frac{p-1}{p}}}{R^\frac{(p-1)}{p}}\xrightarrow[R\rightarrow\infty]{} 0,
   \end{equation*}}
   that proves the limit. 
  Furthermore,  thanks to triangular inequality, one has 
  $$
  |u-u^R_n|_0^{\mcK_T} \le  |u-u^R|_0^{\mcK_T} + |u^R-u^R_n|_0^{\mcK_T},
  $$
  so, for all $n\in\N^*$, we can pick $R_n$ such that the second norm on the right-hand side is upper bounded by $1/(2n)$. Then, replacing $R$ with $R_n$ in the first norm of the right-hand part of the inequality, it exists $k(n)$ such that this norm is upper bounded by $1/(2n)$ too. So we define $\hat{u}_n = u^{R_n}_{k(n)}$ and
  \begin{equation}\label{convergence_prop_existence_f_cont}
    |u-\hat{u}_n|_0^{\mcK_T}\le \frac{1}{n}.
  \end{equation} 
  The functions $(u^F_n)_{n\in\N}=\big(\hat{u}_n(T-\cdot,\cdot,\cdot)\big)_{n\in\N}$ are classical solutions (so in particular viscosity solutions) to \eqref{reference_PDE_forward} where we have $F_n$ instead of $F$ by replacing in it $h$ with $h^{R_n}_{k(n)}$. $u$ (and so $u^F$) is $C([0,T]\times\R\times\R_+)$ thanks to \eqref{convergence_prop_existence_f_cont}. Furthermore is easy to prove the uniform convergence hypothesis $F_n\rightarrow F$ (because $h^{R_n}_{k(n)}\rightarrow h$ uniformly over the compact sets of $[0,T)\times\R\times\R_+$), so thanks to Lemma \ref{sequence_lemma_general}, $u^F$ satisfies $$\partial_t u^F(t,x) + F(t, x,u^F(t, x),D_x u^F(t, x),D^2_x u^F(t, x))= 0 \text{ for } (t,x,y)\in(0,T]\times\R\times\R_+$$ in viscosity sense. Furthermore, $u^F(0,x,y)=f(x,y)$ for every $(x,y)\in\R\times\R_+$, so $u^F$ is a viscosity solution of \eqref{reference_PDE_forward} with initial value $f$.
\end{proof}

\begin{remark}
  The hypothesis \eqref{integrability_condition} with $p>1$ is not restrictive. For example all the functions $f\in C_{\pol}(\R\times\R_+)$, $h\in C_{\pol,T}(\R\times\R_+)$ satisfy this hypothesis for all $p>1$. In fact
  $$
  \sup_{(t,x,y)\in\mcK_T}\|f(X_T^{t,x,y},Y^{t,y}_T)\|_{L^p(\Omega)} \le [C(1 + \sup_{(t,x,y)\in\mcK_T}\E[|(X_T^{t,x,y}|^{ap}+|Y^{t,y}_T)|^{ap}])]^{\frac{1}{p}}< \infty,
  $$
  and
  $$
  \sup_{(t,x,y)\in\mcK_T}\int_t^T \|h(s,X^{t,x,y}_s,Y^{t,y}_s)\|_{L^p(\Omega)}ds \le C(1 + \sup_{(t,x,y)\in\hat{\mcK}_T}\E[|(X_T^{t,x,y}|^{ap}+|Y^{t,y}_T)|^{ap}])< \infty
  $$
  with $\hat{\mcK}_T=\{(t,x,y)\in[0,T]\times\R\times\R_+\mid \exists t_0\in[0,T] \text{ such that } (t_0,x,y)\in\mcK_T \}$.
\end{remark}

\subsection{Comparison principle and uniqueness for continuous initial data}
In this subsection, we prove a comparison principle for our reference PDE \eqref{reference_PDE_forward}. In Subsection \ref{discont_initial_data_subsec}, we prove this result allows getting uniqueness of solutions even for initial data $f$ that present ``some discontinuities''.

We start by stating two results that will be crucial to prove a comparison argument needed to prove the uniqueness of the solution. We first recall a lemma proved in \cite{CIL92}, Proposition 3.7.

\begin{lemma}\label{lemma_doubling}
  Let $M\in\N^*$, $\mcA$ be a subset of $\R^M$, $\Phi\in \USC(\mcA)$, $\Psi\in \LSC(\mcA)$,
  \begin{equation}
    M_\alpha = \sup_{\mcA}(\Phi(x)-\alpha\Psi(x))
  \end{equation}
  for $\alpha>0$. Let $\lim_{\alpha\rightarrow\infty} M_\alpha \in\R$ and $x_\alpha\in\mcA$ be chosen so that 
  \begin{equation}
    \lim_{\alpha\rightarrow\infty} \big(M_\alpha-(\Phi(x_\alpha)-\alpha\Psi(x_\alpha))\big)=0.
  \end{equation}
  Then, the following hold
  \begin{equation}\label{implication_lemma_doubling}
    \begin{cases}
      (i)&\lim_{\alpha\rightarrow\infty} \alpha\Psi(\alpha)=0, \\
      (ii)& \Psi(\hx)=0 \text{ and } \lim_{\alpha\rightarrow\infty} =\Phi(\hx)=\sup_{\{x\in\mcA\mid\Psi(x)=0\}}\Phi(x)\\
       &\qquad\qquad\qquad\text{whenever } \hx\in\mcA \text{ is a limit point of } x_\alpha \text{ as } \alpha\rightarrow\infty.
    \end{cases}
  \end{equation}
\end{lemma}

\begin{prop}\label{semijet_estimates}
  Let $u_i\in\USC((0,T]\times\mcO_i)$ for $i=1,\ldots,k$ where $\mcO_i$ is s locally compact subset of $\R^{N_i}$. Let $\varphi$ be defined on an open neighborhood of $(0,T]\times\mcO_1\times\cdots\times\mcO_k$ and such that $\varphi:(t,x_1,\ldots,x_k)\mapsto \varphi(t,x_1,\ldots,x_k)$ is once continuously differentiable in $t$ and twice continuously differentiable in $(x_1,\ldots,x_k)\in\mcO_1\times\cdots\times\mcO_k$. Suppose that $\hat{t}\in(0,T]$, $\hx_i\in\mcO_i$ for $i=1,\ldots,k$ and
  \begin{align*}
    w(t,x_1,\ldots,x_k) &\equiv u_1(t,x_1) + \cdots + u_k(t,x_k) - \varphi(t,x_1,\ldots,x_k)\\
    &\le w(\hat{t},\hx_1,\ldots,\hx_k),
  \end{align*}
  for $0<t\le T$ and  $x_i\in\mcO_i$. Assume, moreover, that there is an $r>0$ such that for every $M>0$ there is a $C$ such that for $i=1,\ldots,k$
  \begin{multline}
    b_i\le C \text{ whenever } (b_i,q_i,X_i)\in\mcP^{2,+}_{\mcO_i}u_i(t,x_i),\\
    |x_i-\hx_i|+|t-\hat{t}|\le r \text{ and } |u_i(t,x_i)|+|q_i|+\|X_i\|\le M.
  \end{multline}
  Then for each $\varepsilon>0$ there are $X_i\in S(N_i)$ such that 
  \begin{equation}
    \begin{cases}
      (i)& (b_i,D_{x_i}\varphi(\hat{t},\hx_1,\ldots,\hx_k),X_i )\in\bmcP^{2,+}_{\mcO_i} u_i(\hat{t},\hx_i) \text{ for } i=1,\ldots,k, \\
      (ii)& -\left(\frac{1}{\varepsilon} +\|A\| \right) I \le \begin{pmatrix}
        X_{1} & \cdots  & 0 \\
        \vdots & \ddots & \vdots \\
        0& \cdots  & X_{k}  
    \end{pmatrix} \le A+\varepsilon A^2\\
      (iii)& b_1+\cdots+b_k = \partial_t\varphi(\hat{t},\hx_1,\ldots,\hx_k),
    \end{cases}
  \end{equation}
  where $A=(D^2_x\varphi)(\hat{t},\hx_1,\ldots,\hx_k)$.
\end{prop}

\begin{proof}
  We refer to the proof in \cite{CI90} with a small modification. Here, the $u_i$-s are $\USC$ up to $T$, so we must consider possible maximum points over ${T}\times\mcO_1\times\cdots\times\mcO_k$. In the proof \cite[Lemma 8]{CI90} we redifine the $v_i$-s functions equal to $-\infty$ only when $|x_i|>1$ and $t<s/2$, so  $t_{i,\delta}$ belongs to $[s/2,T]$ and the rest of the proof still the same.
\end{proof}
\
We can now state a comparison principle for semicontinuous functions of the problem \eqref{reference_PDE_forward}.

\begin{prop}\label{comparison_principle}
  \textbf{(Comparison principle)}
  Let $w\in\USC([0,T]\times\R\times\R_+)$ and  $v\in\LSC([0,T]\times\R\times\R_+)$ be respectively a subsolution and a supersolution to \eqref{reference_PDE_forward} with polynomial growth uniformly in time, where $F$ is as in \eqref{F_logheston}. Then $w\le v$.
\end{prop}
\begin{proof}
  It is simple to show (using the Definition \ref{def_viscosity_sol}) that if $w$ and $v$ are subsolution and supersolution to the general problem \eqref{reference_PDE_forward} with $\varrho\in\R$ then $e^{-\theta  t}w(t,x,y)$ and $e^{-\theta t}v(t,x,y)$ are subsolution and supersolution to \eqref{reference_PDE_forward} with $\varrho$ replaced by $\varrho-\theta$.
  So we need only to prove the result when $\varrho=0$.
  We start by remarking that $w$, $v$ and $h$ have polynomial growth uniformly in time, then
  \begin{equation}\label{pol_growth}
    \sup_{t\in[0,T]}|w(t,x,y)|,\sup_{t\in[0,T]}|v(t,x,y)|,\sup_{t\in[0,T]}|h(t,x,y)|\le C(1+x^L+y^L) \text{ for some } C>0,  L\in2\N^* 
  \end{equation} 
  
  We define the function $\phi_\varepsilon(t,x,y)=\varepsilon e^{Mt}(1+x^{L+2}+y^{L+2})$. It's easy to check that $\phi_\varepsilon$ is $C^\infty([0,T]\times\R\times\R_+)$ and, for every $\varepsilon$, $M>0$ can be chosen such that
  \begin{equation}\label{classical_strict_supersolution}
    \begin{cases}
      (\partial_t-\mcL)\phi_\varepsilon(t,x,y) \ge \varepsilon,\quad  \forall t\in(0,T],\, x\in\R,\, y\in\R_+, \\
      \phi_\varepsilon(0,x,y) \ge \varepsilon, \quad \forall x\in\R, y\in \R_+.
    \end{cases}
  \end{equation}
  Then, called $v_\varepsilon = v+\phi_\varepsilon$ and $\mcO=\R_+\times\R$, for all $(t,x,y)\in\mcO_T=(0,T]\times\mcO$ one has   
  \begin{multline*}
    \mcP^{2,-}_{\mcO_T}(v+\phi_\varepsilon)(t,x,y)= \big\{ (\alpha +\partial_t\phi_\varepsilon(t,x,y),(\beta,\gamma) +D_{(x,y)}\phi_\varepsilon(t,x,y), X + D^2_{(x,y)}\phi_\varepsilon(t,x,y))  \mid \\ (\alpha,(\beta,\gamma),X)\in\mcP^{2,-}_{\mcO_T}v(t,x,y) \big\},
  \end{multline*}
 and by the linearity of the reference PDE, one has that
 \begin{equation}\label{strict_supersolution}
  c + F(t,(x,y),v_\varepsilon(t,x,y),p,X)\geq \varepsilon,\text{ for all } (t,x,y)\in\mcO_T \text{ and } (c,p,X)\in\mcP_{\mcO_T}^{2,-}v_\varepsilon(t,x,y),
 \end{equation}
 and clearly $w-v_\varepsilon\le-\varepsilon$, that means $v_\varepsilon$ is a strict super-solution.
 Thanks to the growth hypothesis \eqref{pol_growth}
 \begin{equation*}
  w(t,x,y)-v_\varepsilon(t,x,y)\le 2C(1+x^L+y^L) -\varepsilon (1+x^{L+2}+y^{L+2}) \xrightarrow[|(x,y)|\rightarrow\infty]{} -\infty,
 \end{equation*}
 then there exists $R>0$ such that,  $w-v_\varepsilon\le0$ outside a rectangle $[0,T]\times\mcR$, where $\mcR=\times[-R,R]\times[0,R]$.
 We now suppose that there exist a point $(t_0,x_0,y_0)\in\mcR$ such that $(u-v_\varepsilon)(t_0,x_0,y_0)=\delta>0$. If such a point exists, $t_0$ must be $>0$ by the initial condition.
 In order to come up with a contradiction we use the well known technique in classical viscosity solutions framework of doubling the variables.
 We define, for all $\alpha>0$
 $$\varphi_\alpha:([0,T]\times\R^2\times\R^2)\ni (t,\eta,\zeta)\mapsto \frac{1}{2}|\eta-\zeta|^2\in\R_+.$$
 We penalize the function $w-v_\varepsilon$ subtracting the function $\alpha\varphi_\alpha$, while we double the spatial variables, 
 and study 
 $$M_\alpha=\sup_{(t,\eta,\zeta)\in[0,T]\times\mcR\times\mcR}w(t,\eta)-v_\varepsilon(t,\zeta) -\alpha\varphi_\alpha(t,\eta,\zeta). $$
 By the upper semi-continuity of the function $w(t,\eta)-v_\varepsilon(t,\zeta) -\alpha\varphi_\alpha(t,\eta,\zeta)$, and compactness of $[0,T]\times\mcR\times\mcR$, 
 $$\delta\le M_\a =  w(t_\a,\eta_\a)-v_\varepsilon(t_\a,\zeta_\a) -\alpha\varphi_\alpha(t_\a,\eta_\a,\zeta_\a)$$ for some $(t_\a,\eta_\a,\zeta_\a)\in[0,T]\times\mcR\times\mcR$. 
 We apply now Lemma \ref{lemma_doubling} with $\mcO= [0,T]\times\mcR\times\mcR$, $\Psi=\varphi$ and we chose the point $x_\alpha$ in the lemma to be the point $(t_\a,\eta_\a,\zeta_\a)$ that realizes the maximum $M_\alpha$. Then \eqref{implication_lemma_doubling} translates to
 \begin{equation}\label{implication_lemma_doubling_parab}
  \begin{cases}
    (i)&\lim_{\alpha\rightarrow\infty} \frac{\a}{2}|\eta_\a-\zeta_\a|^2=0, \\
    (ii)& \lim_{\alpha\rightarrow\infty} M_\a = u(\hat{t},\hat{\eta})-v_\varepsilon(\hat{t},\hat{\eta})=\sup_{\{(t,\eta)\in[0,T]\times\mcR\}}w(t,\eta)-v_\varepsilon(t,\eta)\\
     &\qquad\text{whenever } (\hat{t},\hat{\eta})\in[0,T]\times\mcR\ \text{ is a limit point of } (t_\a,\eta_\a) \text{ as } \alpha\rightarrow\infty.
  \end{cases}
 \end{equation}
 Now, because of the initial condition, \eqref{implication_lemma_doubling_parab} $(i)$ and $(ii)$, $(t_\a,\eta_\a,\zeta_\a)$ must lie inside $(0,T]\times\mcR\times\mcR$ for large $\a$.
 We want to show that there exists values in $\bmcP^{2,+}_{(0,T]\times\mcR}w(t_\a,\eta_\a)$ and $\bmcP^{2,-}_{(0,T]\times\mcR}v_\varepsilon(t_\a,\eta_\a)$ that are not compatible.
 We apply Proposition \ref{semijet_estimates} to the point $(t_\a,\eta_\a,\zeta_\a)$ with $u_1=w$ and $u_2=-v_\varepsilon$, $\varphi=\varphi_\a$, $\mcO_1=\mcO_2=\R\times\R_+$,  and $\varepsilon$ equal to $\alpha^{-1}$, one get
 $$
  (c_1,\alpha(\eta_\a-\zeta_\a),X_\alpha)\in\bmcP^{2,+}_\mcO w(t_\a,\eta_\a),\qquad (c_2,\alpha(\eta_\a-\zeta_\a),Y_\alpha)\in\bmcP^{2,-}_\mcO v_\varepsilon(t_\a,\zeta_\a)
 $$
 such that $c_1=c_2$ and
 \begin{equation*}
   -3\a\begin{pmatrix}
    I & 0 \\
    0 &  I
  \end{pmatrix} \le \begin{pmatrix}
    X_\a & 0 \\
    0 & -Y_\a
  \end{pmatrix} \le 3\a \begin{pmatrix}
    I  & -I \\
    -I &  I
  \end{pmatrix},
 \end{equation*}
 from which we derive
 \begin{equation}
  X_\a\le Y_\a \quad\text{ and }\quad |\langle X_\a z,z\rangle|\le 3 \alpha|z|^2.
 \end{equation}
 Using that $u$ is a sub-solution \eqref{subsol_2nd_def} and $v_\varepsilon$  is a strict super-solution \eqref{strict_supersolution}, one has
 \begin{align*}
  c_1 + F(t_\a,\eta_\a, w(t_\a,\eta_\a), \a(\eta_\a-\zeta_\a),X_\a)&\le 0, \\
  c_2 + F(t_\a,\zeta_\a, v_\varepsilon(t_\a,\zeta_\a), \a(\eta_\a-\zeta_\a),Y_\a)&\ge \varepsilon,
 \end{align*}
 using $X_\a\le Y_\alpha$, $v_\varepsilon(t_\a,\zeta_\a)\le w(t_\a,\eta_\a)$   and that $F$ is proper ($\varrho=0$) imply
 \begin{align*}
  \varepsilon &\le F(t_\a,\zeta_\a, v_\varepsilon(t_\a,\zeta_\a), \a(\eta_\a-\zeta_\a),Y_\a) -F(t_\a,\eta_\a,  w(t_\a,\eta_\a), \a(\eta_\a-\zeta_\a),X_\a) \\
  &\le F(t_\a,\zeta_\a, w(t_\a,\eta_\a), \a(\eta_\a-\zeta_\a),X_\a) - F(t_\a,\eta_\a,  w(t_\a,\eta_\a), \a(\eta_\a-\zeta_\a),X_\a).
 \end{align*}
 Naming $\eta_\a =(x^\eta_\a,y^\eta_\a)$, $\zeta_\a =(x^\zeta_\a,y^\zeta_\a)$, and using the definition of $F$ in \eqref{F_logheston} one has

 \begin{multline}
  \varepsilon\le \frac{(y^\zeta_\a-y^\eta_\a)}{2}(\lambda^2X^{1,1}_\a+2\rho\lambda\sigma X^{1,2}_\a+\sigma^2 X^{2,2}_\a)+d(y^\zeta_\a-y^\eta_\a)\alpha(x^\zeta_\a-x^\eta_\a) \\
  -b(y^\zeta_\a-y^\eta_\a)\alpha(y^\zeta_\a-y^\eta_\a) +h(T-t_\a,x^\zeta_\a,y^\zeta_\a)-h(T-t_\a,x^\eta_\a,y^\eta_\a).
 \end{multline}
 The first term on the right-hand side is upper bounded by $3/2(\lambda^2+\sigma^2)\alpha(y^\zeta_\a-y^\eta_\a)$, so the first three terms go to 0 thanks to \eqref{implication_lemma_doubling_parab} $(i)$ while $h(T-t_\a,x^\zeta_\a,y^\zeta_\a)-h(T-t_\a,x^\eta_\a,y^\eta_\a)$ goes to 0 when $\alpha\rightarrow\infty$ by continuity of $h$ (up to choose a subsequence of $\alpha$s for which $(t_\a,\eta_\a)\rightarrow(\hat{t},\hat{\eta})$). So $\varepsilon$, that is strictly positive, is upper bounded by a quantity that goes to 0 when $\alpha\rightarrow\infty$.
 This contradiction yields that there is no point $(t_0,x_0,y_0)$ in which $u-v_\varepsilon>0$ than $u\le v_\varepsilon$ for every $\varepsilon>0$ and so $u\le v$.
\end{proof}

We finish by stating and proving the two following results, which discuss the uniqueness of the reference PDE and the regularity of the solution.

\begin{corollary}
  The problem \eqref{reference_PDE} has a unique viscosity solution that is continuous over $[0,T]\times\R\times\R_+ $ and that has polynomial growth in $(x,y)$ uniformly in $t$. 
\end{corollary}

\begin{proof}
  For simplicity, we consider the equivalent forward PDE \eqref{reference_PDE_forward} and two solutions $u_1$ and $u_2$ with polynomial growth. Then, by the linearity of the PDE, $u=u_1-u_2$ is a viscosity solution of
  \begin{equation}
    \begin{cases}
      (\partial_t-\mcL)u(t,x,y) = 0,\quad   t\in(0,T], &x\in\R, y\in\R_+, \\
      u(0,x,y) = 0, \quad  &x\in\R, y\in \R_+,
    \end{cases}
  \end{equation}
  with polynomial growth, so it exists $L\in2\N$ such that $|u(t,x,y)|\le C(1+x^L+y^L)$.
  We remark that $v(t,x,y)=0$ is a solution to the problem, then by Proposition \ref{comparison_principle} $u\le v$ and $u\ge v$, so $u_1=u_2$ is the unique solution.

\end{proof}

\begin{prop}\label{uniqueness_and_regularity_f_cont}
  Let $f\in \mcC_{\pol}(\R\times\R_+)$ and $h\in \mcC_{\pol,T}(\R\times\R_+)$.
  Then $u$ as in \eqref{u_gen_sol} is the unique viscosity solution of \eqref{reference_PDE} that belongs to $\mcC([0,T]\times\R\times\R_+)$ and that has polynomial growth in $(x,y)$ uniformly in $t$. 
  
  Furthermore, if $h$ is locally Hölder in the compact sets of $[0,T)\times\R\times\R^*_+$ then $u\in \mcC^{1,2}\big([0,T)\times(\R\times\R^*_+)\big)$.
\end{prop}
\begin{proof}
  The first hypotheses over $f$ and $h$ guarantee $u$ to be continuous over $[0,T]\times\R\times\R_+ $ and to have polynomial growth, so to be the unique solution. Furthermore, if $f$ and $h$ are locally Hölder in the compact sets of $[0,T)\times\R\times\R^*_+$, one can consider the PDE locally in a compact inside $[0,T)\times\R\times\R^*_+$ to prove further regularity.
  Let $t<S\in[0,T)$, $(x,y)\in\R\times\R_+^*$ and $\mcR=(x-R,x+R) \times (y/2,2y)$, $R>0$,  $Q=[0,S)\times\mcR$ and consider the PDE problem
  \begin{equation}
  \begin{cases}\label{parabolic_problem_bounded_domain}
  \partial_tv+ \mcL v+\varrho v= h,\qquad&\mbox{ in }Q,\\
  v=u,\qquad&\mbox{ in }\partial_0Q,
  \end{cases}	
  \end{equation}
  $\partial_0Q$ denoting the parabolic boundary of $Q$. The coefficients satisfy in $Q$ all the classical assumptions (see A. Friedman  \cite{AF_book} Theorem 9 and Corollary 2 in Chapter 3, Sec. 4), so a unique (bounded) solution $v\in \mcC^{1,2}([0,S)\times\mcR)\cap \mcC([0,S]\times\bar{\mcR})$ actually exists (and have Hölder continuous derivatives $v_t$, $D_{(x,y)} v$ and $D^2_{(x,y)}v$ in any compact set contained inside $Q$). In Proposition \ref{comparison_principle} we proved a comparison principle on the unbounded spatial domain $\R\times\R_+$ it is easy (even easier) to prove with the same techniques a comparison principle for the problem above \eqref{parabolic_problem_bounded_domain}: this time the ``strictificate'' the super-solution $v$ can just add the term $\phi_\varepsilon= \varepsilon e^{Mt}$. So $u$ that is a viscosity solution and $v$ that is a classic one (hence a viscosity one, too) must be the same over $\overline{Q}$. Then we have the regularity claimed for $u$ over a generic point $(t,x,y)\in[0,T)\times(\R\times\R^*_+)$.
\end{proof}

\subsection{Discontinuous initial data}\label{discont_initial_data_subsec}
Here, we present a general result that allows us to reduce the regularity of the initial data and consider functions that are not continuous everywhere. 

Let $f$ be a locally bounded function defined over a locally compact domain $\mcO$. We define the upper-semicontinuous envelope $f^*$ and the lower-semicontinuous envelope $f_*$ on $\bar{\mcO}$ by
  $$
  f^*(x) = \limsup_{z\rightarrow x} f(z), \quad f_*(x) = \liminf_{z\rightarrow x} f(z). $$
We call $D_f$ the set of the discontinuity points of $f$, i.e.
\begin{equation}\label{discontinuity_set}
  D_f = \{x\in\bmcO\mid f^*(x)\neq f_*(x) \}.
\end{equation} 
Now, we state and prove a result that says the lower semicontinuous and upper semicontinuous functions belong to the Baire class 1  functions (i.e. they are the pointwise limit of continuous functions) and can be approximated in a monotone way.


\begin{prop}[Baire]\label{baire_approx}
  Let $X\subset\R^d$ closed. Let $f\in\LSC(X)$ then there exist a non-decreasing sequence $(f^-_n)_{n\in\N_*}\subset C(X)$ such that $f^-_n\rightarrow f$. If $f\in\USC(X)$ then there exist a non-increasing sequence $(f^+_n)_{n\in\N_*}\subset C(X)$ such that $f^+_n\rightarrow f$.
\end{prop}
\begin{proof}
  We prove only the first result; the second follows from the first, considering $-f$.
  We now sketch the proof.
  Let $f\le 0$, for $f$ not bounded below, we can find a continuous function $h\le f$, and we apply what follows to $\hf=f-h$, find approximations $\hf_n$ and define $f_n=\hf_n+h$.
  We consider the continuous function $g:x\rightarrow x/(1+x)$, prove the result for $\tf= g(f)$ taking values in $[0,1]$ (closed set) as in the proof in Proposition 11 in Section 2 of Chapter 9 of \cite{Bourbaki66}, consider the approximations $\tf_n$ given by proving the result for $\tf$ and then name $f_n=\tf_n/(1-\tf_n)$.
\end{proof}

We want to show that $(t,x,y)\mapsto \E[e^{\varrho(T-t)}f(X_T^{t,x,y},Y^{t,y}_T)]$ is a continuous mapping for $t<T$, even when the final data is discontinuous.
To this purpose, we use the fact that the distribution of the couple $(X^{t,x,y}_T,Y^{t,y}_T)$ for $t<T$ is absolutely continuous (with respect to Lebesgue measure). In fact, in \cite{PP2015}, it has been shown that $(\exp(X^{t,x,y}_T),Y^{t,y}_T)$ has a $\mcC^\infty(\R^*_+\times\R^*_+)$ density so using a change of variable argument, it easily follows that $(X^{t,x,y}_T,Y^{t,y}_T)$ has a $\mcC^\infty$ density over $\R\times\R^*_+$. 

\begin{prop}\label{continuity_prop}
  Let $f:\R\times\R_+\mapsto\R$ be a function such that the closure $\bDf$ of the set of its discontinuity points has zero Lebesgue measure. We assume that for all compact set $\mcK_T\in[0,T]\times\R\times\R_+$ there exists $p>1$ such that 
  $$
  \sup_{(t,x,y)\in\mcK_T}\|f(X_T^{t,x,y},Y^{t,y}_T)\|_{L^p(\Omega)}<\infty.
  $$ 
  Then $v(t,x,y)=\E[e^{\varrho(T-t)}f(X_T^{t,x,y},Y^{t,y}_T)]\in \mcC\big(([0,T]\times\R\times\R_+)\backslash (\{T\}\times \bDf) \big)$.
\end{prop}

\begin{proof}
  Since $(t,x,y)\mapsto e^{\varrho(T-t)}$ is continuous, we can consider the case $\varrho=0$. 
  We consider $(t,x,y)\in[0,T)\times\R\times\R_+$ and a sequence $(t_k,x_k,y_k)$ that converges towards it, we want to show that 
  \begin{equation*}
    |\E[f(X_T^{t_k,x_k,y_k},Y^{t_k,y_k}_T)] - \E[f(X_T^{t,x,y},Y^{t,y}_T)]| \xrightarrow[k\rightarrow\infty]{} 0.
  \end{equation*}
  Let $\varepsilon>0$. First, thanks to Hölder inequality
  \begin{multline*}
    |\E[f(X_T^{t_k,x_k,y_k},Y^{t_k,y_k}_T)\mathds{1}_{|(X_T^{t_k,x_k,y_k},Y^{t_k,y_k}_T)|>R}] |\le \\
     || f(X_T^{t_k,x_k,y_k},Y^{t_k,y_k}_T)||_{L^p(\Omega)} \P(|(X_T^{t_k,x_k,y_k},Y^{t_k,y_k}_T)|>R)^{\frac{p}{p-1}}.
  \end{multline*}
  One can choose $R$ such that the same inequality holds replacing $(X_T^{t_k,x_k,y_k},Y^{t_k,y_k}_T)$ with $(X_T^{t,x,y},Y^{t,y}_T)$. So, for $R$ big enough, one has
  \begin{multline*}
    |\E[f(X_T^{t_k,x_k,y_k},Y^{t_k,y_k}_T)] - \E[f(X_T^{t,x,y},Y^{t,y}_T)]| \le \\
    |\E[f(X_T^{t_k,x_k,y_k},Y^{t_k,y_k}_T)\mathds{1}_{|(X_T^{t_k,x_k,y_k},Y^{t_k,y_k}_T)|\le R}]- \E[f(X_T^{t,x,y},Y^{t,y}_T)\mathds{1}_{|(X_T^{t,x,y},Y^{t,y}_T)|\le R}]| +2\varepsilon.
  \end{multline*}
  Let $\delta>0$, we define the compact set $A^\delta_R=\{z\in\R\times\R_+ \mid d(z,\bDf)\ge \delta, |z|\le R \}$ and $C^\delta_R=\{z\in\R\times\R_+ \mid d(z,\bDf)< \delta, |z|\le R \}$. $C^\delta_R \searrow \bDf\cap \overline{B_R(0)}$ that is a null set, so thanks to the absolute continuity of $(X_T^{t,x,y},Y^{t,y}_T)$ one has 
  \begin{equation}\label{conv_prob_0}
    \P((X_T^{t,x,y},Y^{t,y}_T)\in C^\delta_R)\rightarrow 0 \text{ when } \delta\rightarrow0.
  \end{equation}
  Furthermore, the boundary $\partial C^\delta_R$ is a null set. In fact, it is contained in $\bDf \cup \partial B_R(0)\cup \{ z\in \R\times\R_+ \mid d(z,\bDf)= \delta \}$ that are three null sets, the first by hypothesis and the second two being sets whose points are exactly distant a strictly positive number from closed sets ($\{0\}$ and $\bDf$) (look here \cite{Erdos1945} for a simple proof). So, by convergence in distribution of $(X_T^{t_k,x_k,y_k},Y^{t_k,y_k}_T)$ towards $(X_T^{t,x,y},Y^{t,y}_T)$
  \begin{equation}\label{conv_distrib}
    \P((X_T^{t_k,x_k,y_k},Y^{t_k,y_k}_T)\in C^\delta_R)\xrightarrow[k\rightarrow \infty]{} \P((X_T^{t,x,y},Y^{t,y}_T)\in C^\delta_R).
  \end{equation}
  Thanks to the following inequality
  \begin{multline*}
    |\E[f(X_T^{t_k,x_k,y_k},Y^{t_k,y_k}_T)\mathds{1}_{(X_T^{t_k,x_k,y_k},Y^{t_k,y_k}_T)\in C^\delta_R}] | \le \\
     || f(X_T^{t_k,x_k,y_k},Y^{t_k,y_k}_T)||_{L^P(\Omega)} \P((X_T^{t_k,x_k,y_k},Y^{t_k,y_k}_T)\in C^\delta_R)^{\frac{p}{p-1}},
  \end{multline*}
  that still valid replacing $(t_k,x_k,y_k)$ with $(t,x,y)$, using the uniform boundedness in $L^P$ hypothesis, \eqref{conv_prob_0} and \eqref{conv_distrib}, if $\delta$ is small enough and $k\ge k_0$ one has
  \begin{equation*}
    |\E[f(X_T^{t_k,x_k,y_k},Y^{t_k,y_k}_T)\mathds{1}_{(X_T^{t_k,x_k,y_k},Y^{t_k,y_k}_T)\in C^\delta_R}] - \E[f(X_T^{t,x,y},Y^{t,y}_T)\mathds{1}_{(X_T^{t,x,y},Y^{t,y}_T)\in C^\delta_R}]|\le 2\varepsilon.
  \end{equation*}
  The set $\partial A^\delta_R$ is a null set, because it is contained in $\partial B_R(0)\cup \{ z\in \R\times\R_+ \mid d(z,\bDf)= \delta \}$, that are two null set, as explained above.
  So, the function $f\mathds{1}_{A^\delta_R}$ are continuous over the compact $A^\delta_R$, therefore, they are bounded and are discontinuous only over the null set $\partial A^\delta_R$. Thanks to the absolute continuity of $(X_T^{t,x,y},Y^{t,y}_T)$ and convergence in distribution of $(X_T^{t_k,x_k,y_k},Y^{t_k,y_k}_T)$ towards it, if $k\ge k_1$
  $$
  |\E[f(X_T^{t_k,x_k,y_k},Y^{t_k,y_k}_T)\mathds{1}_{(X_T^{t_k,x_k,y_k},Y^{t_k,y_k}_T)\in A^\delta_R}]- \E[f(X_T^{t,x,y},Y^{t,y}_T)\mathds{1}_{(X_T^{t,x,y},Y^{t,y}_T)\in A^\delta_R}]| \le \varepsilon
  $$
  Finally if $k\ge \max(k_0,k_1)$ one has 
  $$
  |\E[f(X_T^{t_k,x_k,y_k},Y^{t_k,y_k}_T)] - \E[f(X_T^{t,x,y},Y^{t,y}_T)]| \le 5\varepsilon.
  $$
  If we consider $(t,x,y)\in\{T\}\times A$, where $A=(\R\times\R_+)\backslash\bDf$ then for any sequence $(t_k, x_k, y_k)_{k\in\N}$ that convergences to $(t,x,y)$ and all the previous estimation still works because the limit law this time is a Dirac mass over $(t,x,y)$ that is not a discontinuity point of $f$.
\end{proof}

Thanks to the previous proposition, we can consider some type of discontinuities in the final data $f$ (or initial data if we consider the forward problem). We state and prove the following result.
\begin{theorem}\textbf{(Verification Theorem)}\label{verification_theorem}
  Let $f:\R\times\R_+\rightarrow\R$ be a polynomial growth function such that $\bDf$ has zero Lebesgue measure, and $h\in \mcC_{\pol,T}(\R\times\R_+)$.
  Then $u$ in \eqref{u_sol_chap3} is the unique viscosity solution to the problem \eqref{reference_PDE} that is $\mcC\big(([0,T]\times\R\times\R_+)\setminus (\{T\}\times \bDf) \big)$ and that has polynomial growth in $(x,y)$ uniformly in $t$. 
\end{theorem}

\begin{proof}
  Let $u$ be as in \eqref{u_sol_chap3} that is a function in the class considered, and let $v$ a solution in this same class. We use an approach that directly proves that $u$ is a solution and is the only one in the class considered.
  We show that exist a continuous sequence $(u^-_n)_{n\in\N}$ of sub-solution and a continuous sequence $(u^+_n)_{n\in\N}$ of super-solution such that
  \begin{align}
    &u^-_n(T,\cdot,\cdot)\le v_*(T,\cdot,\cdot)\le v^*(T,\cdot,\cdot)\le u^+_n(T,\cdot,\cdot), \label{squeeze_v} \\
    \text{for any compact } &\text{set } \mcK_T\subset[0,T)\times\R\times\R_+, \text{ one has }\lim_{n\rightarrow\infty} |u^\pm_n-u|_0^{\mcK_T}  \label{local_uniform_conv_VT},
  \end{align}
  where $v_*,v^*$ are respectively the lower and the upper semicontinuous envelope of $v$. 
  The local uniform convergence \eqref{local_uniform_conv_VT} tells us, thanks to Lemma \ref{sequence_lemma_general}, the limit $u$ is both sub and a super-solution, so it is a solution. 
  Then, if one has the inequalities \eqref{squeeze_v}, one can apply the comparison principle (reverting in time the solutions) comparing $u^-_n$ to $v_*$, and $u^+_n$ to $v^*$ getting the relations
  \begin{align*}
    u(t,x,y)&=\lim_{n\rightarrow\infty}u^+_n(t,x,y)\ge v^*(t,x,y)=v(t,x,y),\quad \text{ for all } t<T, \, (x,y)\in\R\times\R_+,\\
    u(t,x,y)&=\lim_{n\rightarrow\infty}u^-_n(t,x,y)\le v_*(t,x,y)=v(t,x,y),\quad \text{ for all } t<T, \, (x,y)\in\R\times\R_+,
  \end{align*}
  where we used the fact that $v_*(t,x,y)=v(t,x,y)=v^*(t,x,y)$ because $v$ is continuous for every point such that $t<T$. Then, $u=v$ everywhere because they have the same final data.
  We prove, so, the existence of continuous sequences of solutions that satisfy \eqref{squeeze_v}. We consider a modified version of final data $f^\pm$ defined as follows
  \begin{equation}
    \begin{cases}
      f^\pm(x,y)= \pm \chi(x,y) &\text{for }(x,y)\in N \\
      f^\pm(x,y)= f(x,y) &\text{otherwise},
      \end{cases}
  \end{equation}
  where $\chi(x,y)=C(1+|x|^L+y^L)$ $(C>0,L\in\N^*)$ is such that $|u(t,x,y)|,|v(t,x,y)|\le \chi(x,y)$, and we define the functions $u^\pm$
  $$
  u^\pm(t,x,y)=\E\left[e^{\varrho (T-t)}f^\pm(X_T^{t,x,y},Y^{t,y}_T)-\int_t^T e^{\varrho (s-t)} h(s,X^{t,x,y}_s,Y^{t,y}_s)ds\right].
  $$
  Being $f^-$ and $f^+$ respectively lower semicontinuous and upper semicontinuous, thanks to Proposition $\ref{baire_approx}$, there exists 
  $(f^-_n)_{n\in\N_*}\subset C(\R\times\R_+)$ non-decreasing sequence converging to $f^-$ and such that $f^-_n\ge-\chi$ (otherwise consider $\hat{f}^-_n=f^-_n\vee -\chi$) and $(f^+_n)_{n\in\N_*}\subset C(\R\times\R_+)$ non-increasing sequence converging to $f^+$ and such that $f^+_n\ge\chi$ (otherwise consider $\hat{f}^+_n=f^+_n\wedge \chi)$. We define in the same way as $u^\pm$
  $$
  u^\pm_n(t,x,y)=\E\left[e^{\varrho (T-t)}f^\pm_n(X_T^{t,x,y},Y^{t,y}_T)-\int_t^T e^{\varrho (s-t)} h(s,X^{t,x,y}_s,Y^{t,y}_s)ds\right].
  $$
  By Proposition \ref{existence_viscsol_fcont} $u^+_n$ are continuous solution with final data $f^+_n\ge f\ge v_*(T,\cdot,\cdot)$ and $u^-_n$ are continuous solution with final data $f^-_n\le f\le v^*(T,\cdot,\cdot)$ (proving \eqref{squeeze_v}). Furthermore, thanks to Lebesgue Theorem $u^\pm_n\rightarrow u^\pm$ when $n\rightarrow\infty$ and $u^\pm=u$ for $t<T$ because $f$ differs from $f^\pm$ only on a negligible set and $(X^{t,x,y}_T,Y^{t,y}_T)$ has density, and the convergence is monotone in $n$ for any point. Then, considering the convergence on any compact set $\mcK_T\subset[0,T)\times\R\times\R_+$ we have a monotone sequence  ($u^+_n$ or $u^-_n$) of continuous functions that converges everywhere over $\mcK_T$ to a continuous function $u$, so this convergence must be uniform for Dini's Theorem (proving \eqref{local_uniform_conv_VT}).
\end{proof}

We conclude with one remark and an example.

\begin{remark}\label{remark_regularity_f_discont}
  One should remark that if we add in Theorem \ref{verification_theorem} that $h$ is locally Hölder in the compact sets of $[0,T)\times\R\times\R^*_+$, then $u$ belongs also to $\mcC^{1,2}\big([0,T)\times(\R\times\R^*_+)\big)$.
  The proof of the regularity is the same as in Proposition \ref{uniqueness_and_regularity_f_cont}, so we do not repeat here.
\end{remark}

\begin{ex}\label{digital_option}
  Theorem \ref{verification_theorem} covers a wide range of possibilities. One of them is the case of digital options in the Heston Model. We fix the parameters $c = r-\delta,$ $d=-1/2$, $\lambda=1$ and $\varrho =r$
  \begin{equation*}
    \begin{cases}
      \partial_tu(t,x,y) +\mcL u(t,x,y) +r u(t,x,y) = 0,\quad   t\in[0,T), &x\in\R, y\in\R_+, \\
      u(T,x,y) = \ind{[c,d)}(\exp(x)), \quad  &x\in\R, y\in \R_+,
    \end{cases}
  \end{equation*}
  where $0\le c < d \le\infty$.
\end{ex}

%

\section{Application to finance: a hybrid approximation scheme for the viscosity solution}\label{approximation_section}
Consider the standard Heston model given by the following SDE
\begin{align}\label{Heston_model_sde}
  S_T^{t,s,y} &= s + \int_t^T (r-\delta)S_u du+\int_t^T\rho S_u\sqrt{Y^{t,y}_u} dW_u +\int_t^T\brho S_u\sqrt{Y^{t,y}_u} dB_u, \nonumber\\
  Y^{t,y}_T &= y + \int_t^T(a-bY^{t,y}_u)du + \int_t^T\sigma\sqrt{Y^{t,y}_u}dW_u.
\end{align}
In order to build our approximation, we  apply the transformation $(s,y)\mapsto (\log(s)-\frac{\rho}{\sigma}y,y)$ obtaining the following SDE
\begin{align}\label{logHestonDiffusion_nocorr}
  X_T^{t,x,y} &= x + \int_t^T\Big(r-\delta-\frac{\rho}{\sigma}a+\big(\frac{\rho}{\sigma}b-\frac{1}{2}\big)Y^{t,y}_s\Big) ds +\int_t^T\bar{\rho}\sqrt{Y^{t,y}_s} dB_s, \nonumber\\
  Y^{t,y}_T &= y + \int_t^T(a-bY^{t,y}_s)ds + \int_t^T\sigma\sqrt{Y^{t,y}_s}dW_s,
\end{align}
that given corresponds to a precise choice of the parameters in \eqref{referenceDiffusion}: $\rho=0$, $c=r-\delta-\frac \rho{\sigma}a$, $d=\frac \rho{\sigma}b-\frac 12$ and $\lambda = \brho$. 
The advantage of studying \eqref{logHestonDiffusion_nocorr} instead of the SDE obtained by $(s,y)\mapsto (\log(s),y)$ is that we can exploit that the noise driving the law of $X|Y$ is independent of the one driving $Y$.  
Hereafter, we fix $T>0$, $f\in\mcC_{\pol}(\R\times\R_+)$ and define
\begin{equation}\label{u_solution}
  u(t,x,y)=\E[f(X^{t,x,y}_T,Y^{t,y}_T)],\qquad(t,x,y)\in[0,T]\times\R\times\R_+.
\end{equation}
We know that
\begin{enumerate}
  \item $\P((X_s^{t,x,y},Y_s^{t,y})\in\R\times\R_+,\forall s\in[t,T])=1;$
  \item the function $u$ in \eqref{u_solution} solves the PDE
      \begin{equation}\label{logHeston_PDE_nocorr}
        \begin{cases}
          (\partial_t+\mcL)u(t,x,y) = 0,\quad   t\in[0,T), &x\in\R, y\in\R_+, \\
          u(T,x,y) = f(x,y), \quad  &x\in\R, y\in \R_+,
        \end{cases}
      \end{equation}
      where 
      \begin{equation}\label{Heston_infinit_gen_nocorr}
        \mcL = \frac{y}{2}(\brho^2 \partial^2_{x} +  \sigma^2 \partial^2_{y}) +\mu_X(y) \partial_x + \mu_Y(y) \partial_y,
      \end{equation}
      and  $\mu_X(y)= r-\delta-\rho a/\sigma +(\rho b/\sigma-1/2) y$,  $\mu_Y(y)=a-by$.
\end{enumerate}
In what follows, we prove the convergence for a large space of functions using the recent hybrid approach introduced in \cite{BCT} that we recall in what follows.

\subsection{The hybrid procedure}
Let $u$ be given in \eqref{u_solution}. We recall, briefly, the main ideas and describe the approximation of $u$. We use the Markov property to represent the solution $u(nh,x,y)$ at times $nh$, $h=T/N$, $n=0,\ldots,N$ for $(x,y)\in\R\times\R_+$
\begin{equation}\label{dynamic_programming_principle}
  \begin{cases}
    u(T,x,y)=f(x,y) \quad \text{ and } n=N-1,\ldots,0: \\
    u(nh,x,y)= \E[u((n+1)h,X^{nh,x,y}_{(n+1)h},Y^{nh,y}_{(n+1)h})].
  \end{cases}
\end{equation}

The goal is to build good approximations of the expectations in \eqref{dynamic_programming_principle}.
First, let $(\hY^h_n)_{n=0,\ldots,N}$ be a Markov chain which approximates $Y$, such that $(\hY^h_n)_{n=0,\ldots,N}$ is independent of the noise driving $X$. 
Then, at each step $n=0,1,\ldots,N-1$, for every $y\in\mcY^h_n\subset\R_+$ (the state space of $\hY^h_n$), one writes
$$
  \E[u((n+1)h,X^{nh,x,y}_{(n+1)h},Y^{nh,y}_{(n+1)h})] \approx \E[u((n+1)h,X^{nh,x,y}_{(n+1)h},\hY^h_{n+1})\mid \hY^h_n=y].
$$
As a second step, one approximates the component $X$ on $[nh,(n+1)h]$ by freezing the coefficients in \eqref{logHestonDiffusion_nocorr} at the observed position $\hY^h_n=y$, that is, for $t\in[nh,(n+1)h]$,
$$
X^{nh,x,y}_t \overset{\text{law}}{\approx} \hX^{nh,x,y}_t = x + \Big(r-\delta-\frac{\rho}{\sigma}a+\big(\frac{\rho}{\sigma}b-\frac{1}{2}\big)y\Big)(t-nh)+\brho\sqrt{y}(Z_t-Z_{nh}).
$$
Therefore, by the fact that the Markov chain and the noise driving $X$ are independent, one can write
\begin{align*}
  \E[u((n+1)h,X^{nh,x,y}_{(n+1)h},Y^{nh,y}_{(n+1)h})] &\approx \E[u((n+1)h,\hX^{nh,x,y}_{(n+1)h},\hY^{h}_{n+1})\mid \hY^h_n=y] \\
  &=\E[\phi(\hY^h_{n+1};x,y)\mid \hY^h_n=y]
\end{align*}
where
\begin{equation}
  \phi(\zeta;x,y)=\E[u((n+1)h,\hX^{nh,x,y}_{(n+1)h},\zeta)].
\end{equation}
From the Feynman-Kac formula, one gets $\phi(\zeta;x,y)=v(nh,x;y,\zeta)$, where $(t,x)\mapsto v(t,x;y,\zeta)$ is the solution at time $nh$ of the parabolic PDE Cauchy problem
\begin{equation}\label{PDE_frozencoeff}
  \begin{cases}
    \partial_t v +\mcL^{(y)}v=0, &\text{in } [nh,(n+1)h)\times\R,\\
    v((n+1)h,x;y,\zeta) = u((n+1)h,x,\zeta), &x\in\R,
  \end{cases}
\end{equation}
where $\mcL^{(y)}$ acts on function $g=g(x)$ as follows
\begin{equation}\label{mcLy_frozencoeff}
  \mcL^{(y)} g(x)= \Big(r-\delta-\frac{\rho}{\sigma}a+\big(\frac{\rho}{\sigma}b-\frac{1}{2}\big)y\Big)\partial_x g(x) + \frac{1}{2}\brho^2y\partial^2_xg(x).
\end{equation}
We remark that in \eqref{PDE_frozencoeff}-\eqref{mcLy_frozencoeff}, $y\in\R_+$ is just a parameter, so $\mcL^{(y)}$ has constant coefficients.
Consider now a numerical solution of the PDE \eqref{PDE_frozencoeff}. Let $\Dx$ denote a fixed initial spatial step, and set $\mcX$ as a grid on $\R$ given by $\mcX=\{x\in\R \mid x=X_0+i\Dx, i\in\Z \}$. For $y\in\R$, let $\PhDx(y)$ be a linear operator (acting on suitable functions on $\mcX$) which gives the approximating solution to the PDE \eqref{PDE_frozencoeff} at time $nh$. Then, as $x\in\mcX$, we get the numerical approximation
$$
\E[u((n+1)h,X^{nh,x,y}_{(n+1)h},Y^{nh,y}_{(n+1)h})] \approx \E[\PhDx(y)u((n+1)h,\cdot,\hY^{h}_{n+1})(x)\mid \hY^h_n=y].
$$
Therefore, by inserting in \eqref{dynamic_programming_principle}, one sees that the hybrid numerical procedure works as follows: the function $x\mapsto u(0,x,Y_0)$, $x\in\mcX$, is approximated by $u^h_0(x,Y_0)$ backward-defined as 
\begin{equation}\label{hybrid_scheme}
  \begin{cases}
    u^h_N(x,y) = f(x,y), \qquad (x,y)\in \mcX\times\mcY^h_N,\quad \text{and as } n=N-1,\ldots,0: \\
    u^h_n(x,y) = \E[\PhDx(y) u^h_{n+1}(\cdot,\hat{Y}^h_{n+1})(x) \mid \hat{Y}^h_{n}=y], \quad (x,y)\in\mcX\times\mcY^h_n.
    \end{cases}
\end{equation}

\subsection{Convergence in $\ell^\infty$}
We recall the finite difference scheme and the Markov chain $(\hY^h_n)_{n=0,\ldots,N}$, that under suitable hypothesis on $f$ assures the convergence of the Hybrid procedure to the solution $u$ \eqref{u_solution}.
Specifically, if $\mu_X(y)=\frac{h}{\Dx}r-\delta-\frac{\rho}{\sigma}a+\big(\frac{\rho}{\sigma}b-\frac{1}{2}\big)y\geq 0$, we approximate  $(\partial_t+\mcL^{(y)})v$ by using the scheme 
$$
\frac{v^{n+1}_i-v^n_i}{h}+ \mu_X(y)\frac{v^{n}_{i+1}-v^n_i}{\Dx} + \frac{1}{2}\brho^2y\frac{v^{n}_{i+1}-2v^n_i+v^n_{i-1}}{\Dx^2} ,
$$
while, if $\mu_X(y)\leq 0$, we use the approximation
$$
\frac{v^{n+1}_i-v^n_i}{h}+ \mu_X(y)\frac{v^{n}_{i}-v^n_{i-1}}{\Dx} + \frac{1}{2}\brho^2y\frac{v^{n}_{i+1}-2v^n_i+v^n_{i-1}}{\Dx^2}.
$$
The resulting scheme is 
\begin{equation}\label{equaz2}
A^h_{\Dx}(y)v^n=v^{n+1},
\end{equation}
where $A^h_{\Dx}(y)$ is the linear operator given by  
\begin{equation}\label{A2}
(A^h_{\Dx})_{ij}(y)=\begin{cases}
-\beta^h_{\Dx}(y)-|\alpha^h_{\Dx}(y)|\ind{\alpha^h_{\Dx}(y)<0},\qquad &\mbox{ if }i=j+1,\\
1+2\beta^h_{\Dx}(y)+|\alpha^h_{\Dx}(y)|,\qquad &\mbox{ if }i=j,\\
-\beta^h_{\Dx}(y)-|\alpha^h_{\Dx}(y)|\ind{\alpha^h_{\Dx}(y)>0},\qquad &\mbox{ if }i=j-1,\\
0, &\mbox{ if }|i-j|>1,
\end{cases}
\end{equation}  
with
$$
\alpha^h_{\Dx}(y)=\frac{h}{\Dx}r-\delta-\frac{\rho}{\sigma}a+\big(\frac{\rho}{\sigma}b-\frac{1}{2}\big)y, \qquad \beta^h_{\Dx}(y)=\frac{h}{2\Dx^2}y.
$$
We finally define $\PhDx(y)=\big(A^h_{\Dx}(y)\big)^{-1}$. 

Along with the finite different scheme, we need a Markov chain $(\hY^h_n)_{n=0,1,\ldots,N}$ approximating the CIR process $Y$ over the time grid $(nT/N)_{n=0,1,\ldots,N}$. The state space is the following, for $n=0,1,\ldots,N$ one has the lattice
\begin{equation}\label{vnk}
\mathcal{Y}_n^h=\{y^n_k\}_{k=0,1,\ldots,n}\quad\mbox{with}\quad
y^n_k=\Big(\sqrt {y}+\frac{\sigma} 2(2k-n)\sqrt{h}\Big)^2\ind{\{\sqrt {y}+\frac{\sigma} 2(2k-n)\sqrt{h}>0\}}.
\end{equation}
Note that $\mathcal{Y}_0^h=\{y\}$. For each fixed node $(n,k)\in\{0,1,\ldots,N-1\}\times\{0,1,\ldots,n\}$, the ``up'' jump $k_u(n,k)$ and the ``down'' jump $k_d(n,k)$ from $y^n_k\in\mathcal{Y}_n^h$ are defined as 
\begin{align}
\label{ku2}
&k_u(n,k) =
\min \{k^*\,:\, k+1\leq k^*\leq n+1\mbox{ and }y^n_k+\mu_Y(y^n_k)h \le y^{n+1}_{k^*}\},\\
\label{kd2}
&k_d(n,k)=
\max \{k^*\,:\, 0\leq k^*\leq k \mbox{ and }y^n_k+\mu_Y(y^n_k)h \ge y^{n+1}_{ k^*}\},
\end{align}
where $\mu_Y(y)=a-by$ and
with the understanding $k_u(n,k)=n+1$, resp. $k_d(n,k)=0$,  if the set in \eqref{ku2}, resp. \eqref{kd2}, is empty. 
Starting from the node $(n,k)$ the probability that the process jumps to $k_u(n,k)$ and $k_d(n,k)$ at time-step $n+1$ are set respectively as
\begin{equation*}
p_u(n,k)
=0\vee \frac{\mu_Y(y^n_k)h+ y^n_k-y^{n+1}_{k_d(n,k)} }{y^{n+1}_{k_u(n,k)}-y^{n+1}_{k_d(n,k)}}\wedge 1
\quad\mbox{and}\quad p_d(n,k)=1-p_u(n,k).
\end{equation*}
We call $(\hY^h_n)_{n=0,1,\ldots,N}$ the Markov chain governed by the above jump probabilities.

Let $\ell^\infty(\mcX)=\{g:\mcX\rightarrow\R\mid \sup_{x\in\mcX} g(x)<\infty\}$ with the norm $|g|_{\ell^\infty}=\sup_{x\in\mcX} g(x)$.
With the above Markov chain, Briani et al. in \cite{BCT} proved that $\PhDx(\cdot)$ satisfies the following Assumption $\mcK$ with $c=1$ and $\mcE=h+\Dx$.

\begin{definition}[\textbf{Assumption} $\mcK(\infty,c,\mcE)$]
  Let $c=c(y)\ge 0$, $y\in\R_+$, and $\mcE=\mcE(h,\Dx)\ge 0$ such that $\lim_{(h,\Dx)\rightarrow0}\mcE(h,\Dx)=0$. We say that the linear operator $\PhDx(y):\ell^\infty(\mcX)\rightarrow\ell^\infty(\mcX)$, $y\in\mcD$, satisfies this assumption if 
\begin{equation}
  \|\PhDx(y)\|_\infty:= \sup_{|f|_{\ell^\infty}=1}|\PhDx(y)f|_{\ell^\infty}\le 1+c(y)h, 
\end{equation}
and, with $u$ being defined in \eqref{u_solution}, for every $n=0,\ldots,N-1$, one has
\begin{equation}
  \E[\PhDx(\hY^h_n)u((n+1)h,\cdot,\hY^{h}_{n+1})(x)\mid Y^h_n] = u(nh,x,\hY^h_n)+\mcR^h_n(x,\hY^h_n),
\end{equation}
where the remainder $\mcR^h_n(x,\hY^h_n)$ satisfies the following property: there exists $\bar{h},C>0$ such that for every $h<\bar{h},\Dx<1$, and $n\le N=\lfloor T/h \rfloor$ one has
\begin{equation}
  \left\|e^{\sum_{l=1}^n c(\hY^h_l)h}|\mcR^h_n(\cdot,\hY^h_n)|_{\ell^\infty} \right\|_{L^1(\Omega)} \le C h\mcE(h,\Dx).
\end{equation}
\end{definition}

In \cite{BCT}, Briani et al. proved the following Theorem.
\begin{theorem}\label{BCT_regular_approx}
  Let $u$ defined in \eqref{u_solution}, $(u^h_n)_{n=0,\ldots,N}$ be given by \eqref{hybrid_scheme} with the choice
  $$
  \PhDx(y)=\big(A^h_{\Dx}(y)\big)^{-1},
  $$
  where $A^h_\Dx(y)$ is given in \eqref{A2}, and $(\hY^h_n)_{n=0,1,\ldots,N}$ defined as above.
  If $\partial_x ^{2j}f\in \mcC^{\infty,q-j}_{\pol}(\R,\R_+)$, for every $j=0,1,\ldots,4$, then, there exist $\bar{h},C>0$ such that for every $h<\bar{h}$ and $\Dx<1$ one has 
  \begin{equation}
    |u(0,\cdot,y) - u^h_0(\cdot,y)|_{\ell^\infty} \le C(h+\Dx).
  \end{equation}
\end{theorem}

We are about to show that, under less regular $f$, we still have the convergence of the hybrid procedure. We present a lemma that will help us prove this result.
\begin{lemma}\label{stability_conv_Cpol}
  Let $f\in \mcC^{\infty,0}_{\pol}(\R,\R_+)$, $(\varphi_l)_{l\in\N^*}$ a sequence of mollifiers over $\R^2$ and
  \begin{equation}
    \tf(x,y) = f(x,0\vee y) \quad \text{and} \quad f_l=\tf\ast \varphi_l.
  \end{equation}
  Then, for all $q\in\N$, $f_l\in \mcC^{\infty,q}_{\pol}(\R,\R_+)$. In particular  $\exists C_0,C^*_0 >0$ such that for all $l\in\N$, $x\in\R$ and $y\in\R_+$
  \begin{equation}\label{const_stability_conv}
    |f_l(x,y)|\le C_0(1+|x|^L+y^L), \quad \sup_{x\in\R}|f_l(x,y)|\le C^*_0(1+y^L).
  \end{equation}
\end{lemma}
\begin{proof}
  Let $q\in\N$. Let $f\in \mcC^{\infty,0}_{\pol}(\R,\R_+)$, then $\tf\in \mcC^{\infty,0}_{\pol}(\R,\R)$ and in particular $\tilde{f}\in\mathbb{L}^\infty_{\text{loc}}(\R^2)$, hence $f_l\in C^\infty(\R^2)$ and, for all multi-index $k$, $D^k f_l = \tilde{f}*D^k\varphi_l$. It remains for us to prove the polynomial growth of $D^kf_l$ and $\sup_{x\in\R}|D^k f_l(x,y)|$ for all multi-index $k$ such that $|k|\le q$. Let $x\in\R,~y\in\R_+$, using that $\tf\in \mcC^{\infty,0}_{\pol}(\R,\R)$ and $(a+b)^L\le 2^{L-1}(a^L+b^L)$ for all $a,b\ge0$
  \begin{align*}
    |D^k f_l(x,y)|&= \left|\int_{\R^2}\tilde{f}(\zeta-x,\eta-y)D^k\varphi_l(\zeta,\eta)d\zeta d\eta \right| \le \int_{\R^2}|\tilde{f}(\zeta-x,\eta-y)| |D^k\varphi_l(\zeta,\eta)|d\zeta d\eta \\
    &\le \int_{\R^2}C(1+|\zeta-x|^L+|\eta-y|^L) |D^k\varphi_l(\zeta,\eta)|d\zeta d\eta \\
    &\le C +C_L(|x|^L+y^L) +C_L \int_{\R^2}(|\zeta|^L+|\eta|^L) |D^k\varphi_l(\zeta,\eta)|d\zeta d\eta \\
    &\le C_k(l)(1 +|x|^L+y^L) \\
    &\le \Big(\max_{k \,\text{s.t.}\, |k|\le q} C_k(l)\Big) (1 +|x|^L+y^L),
  \end{align*}
  where we used the fact that $D^k\varphi_l$ is $C^\infty_c(\R^2)$.
  Furthermore, if $k=0$ than $|D^k\varphi_l(\zeta,\eta)|=\varphi_l(\zeta,\eta)$ and the integral $\int_{\R^2}(|\zeta|^L+|\eta|^L) \varphi_l(\zeta,\eta) d\zeta d\eta \rightarrow 0$ when $l$ goes to $\infty$. So the constant $C_0(l)=C_0$.
  Now, let $y\in\R_+$,
  \begin{align*}
    \sup_{x\in\R} |D^k f_l(x,y)|  &= \sup_{x\in\R} \left|  \int_{\R^2}\tilde{f}(\zeta-x,\eta-y) D^k\varphi_l(\zeta,\eta) d\zeta d\eta\right| \\
    & \le\int_{\R^2}\sup_{z\in\R}|\tilde{f}(z,\eta-y)| |D^k\varphi_l(\zeta,\eta)|d\zeta d\eta, \\
    &\le (C+C_L y^L) +C_L \int_{\R^2}|\eta|^L |D^k\varphi_l(\zeta,\eta)|d\zeta d\eta \\
    &\le C^*_k(l)(1+y^L) \\
    &\le \Big(\max_{k \,\text{s.t.}\, |k|\le q} C^*_k(l)\Big)  (1+y^L),
  \end{align*}
  where once again, we used the fact that $D^k\varphi_l$ is $C^\infty_c(\R^2)$. As in the previous estimate, if $k=0$ than $|D^k\varphi_l(\zeta,\eta)|=\varphi_l(\zeta,\eta)$ and the integral $\int_{\R^2} |\eta|^L \varphi_l(\zeta,\eta) d\zeta d\eta \rightarrow 0$ when $l$ goes to $\infty$. So the constant $C^*_0(l)=C^*_0$.
\end{proof}

We state and prove the main contribution of this section.

\begin{theorem}\label{convergence_theorem}
  Let $f\in \mcC^{\infty,0}_{\pol}(\R,\R_+)$ and suppose that $f$ is uniformly continuous over the sets $\R\times[0,M]$ for all $M>0$. Let $u$ be the viscosity solution defined in \eqref{u_solution} and $u^h_0$ the discrete solution \eqref{hybrid_scheme} produced by the backward hybrid procedure starting from $f$.
  Then
  \begin{equation}
    u^h_0(\cdot,y) \xrightarrow[(h,\Dx)\rightarrow 0]{l^\infty} u(0,\cdot,y).
  \end{equation}
\end{theorem}
\begin{proof}
  Let $(\varphi_l)_{l\in\N^*}$ a sequence of mollifiers. For all $l\in\N^*$, we define  $u_l(t,x,y) = \E[f_l(X^{t,x,y}_T,Y^{t,y}_T)]$ where we replaced $f$ with a mollified final data $f_l=f\ast\varphi_l$, and $u^{h}_{n,l}$ as the discrete solution produced by the backward hybrid procedure starting from $f_l$. Let $x\in\mcX$ and $y\in\R_+$, then
  \begin{align*}
    |u(0,\cdot,y) - u^h_0(\cdot,y)|_{\ell^\infty} \le
    &\underbrace{|u(0,\cdot,y) - \hu_l(0,\cdot,y)|_{\ell^\infty}}_{I}
    +\underbrace{|u_l(0,\cdot,y)-u^{h}_{0,l}(\cdot,y)|_{\ell^\infty}}_{II} \\
    +&\underbrace{|u^{h}_{0,l}(\cdot,y) - u^h_0(\cdot,y)|_{\ell^\infty}}_{III}.
  \end{align*}
Thanks to Theorem \ref{BCT_regular_approx}, term $II$ can be upper bounded by $C_lT(h+\Dx)$ that goes to zero when $(h,\Dx)$ goes to zero (and does not depend on $x$).
Regarding the term $I$, defined $(X^{\cdot,y}_T,Y^{y}_T)=(X^{0,\cdot,y}_T,Y^{0,y}_T)$ one has
\begin{align*}
  |u(0,\cdot,y) - u_l(0,\cdot,y)|_{\ell^\infty} &\le |\E[f(X^{\cdot,y}_T,Y^{y}_T)-f_l(X^{\cdot,y}_T,Y^{y}_T)]|_{\ell^\infty} \\
  &\le |\E[\big(f(X^{\cdot,y}_T,Y^{y}_T)-f_l(X^{\cdot,y}_T,Y^{y}_T)\big) \mathds{1}_{Y^{y}_T\le M}]|_{\ell^\infty} \\
  &\quad+| \E[\big(f(X^{\cdot,y}_T,Y^{y}_T)-f_l(X^{\cdot,y}_T,Y^{y}_T)\big) \mathds{1}_{Y^{y}_T>M}]|_{\ell^\infty} \\
  &\le \sup_{x\in\R,y\in[0,M]}|f(x,y)-f_l(x,y) | \\
  &\quad+ \E\big[\big(|f(\cdot,Y^{y}_T)|_{\ell^\infty}+|f_l(\cdot,Y^{y}_T)|_{\ell^\infty}\big) \mathds{1}_{Y^{y}_T>M}\big].
\end{align*}
Using \eqref{const_stability_conv}, $\exists C>0$ such that $|f(\cdot,Y^{y}_T)|_{\ell^\infty}+|f_l(\cdot,Y^{y}_T)|_{\ell^\infty}\le C(1+(Y^{y}_T)^L)$, then using Holder inequality and then Markov inequality one has
\begin{equation}\label{estim_I}
  I \le \sup_{x\in\R,y\in[0,M]}|f(x,y)-f_l(x,y) |+ C\|(1+(Y^{y}_T)^L)\|_{L^p(\Omega)} \left(\frac{E[Y^{y}_T]}{M}\right)^{\frac{p-1}{p}}.
\end{equation}
Regarding the term $III$, by linearity of the conditional expectation and of the linear operator $\Pi^h_\Dx(y)$ for $n=0,\ldots,N-1$
$$
u^{h}_{n,l}(\cdot,y) - u^h_n(\cdot,y) = \E\big[\Pi^h_\Dx(\hat{Y}^h_{n}) \big( u^{h}_{n,l}(\cdot,\hat{Y}^h_{n+1}) - u^h_{n+1}(\cdot,\hat{Y}^h_{n+1}) \big) \mid \hat{Y}^h_{n}=y \big].
$$
Then we can rewrite the difference  $u^{h}_{0,l}(\cdot,y) - u^h_0(\cdot,y)$ can be seen as the discrete solution constructed starting from $f_l-f$. In fact, using the tower properties, we get by induction
\begin{align*}
  u^{h}_{0,l}(\cdot,y) - u^h_0(\cdot,y) = &\E\bigg[\prod_{j=0}^{N-1}\Pi^h_\Dx(\hat{Y}^h_{j}) \big( f_l(\cdot,\hat{Y}^h_{N}) - f(\cdot,\hat{Y}^h_{N}) \big) \mid \hat{Y}^h_{0}=y \bigg] \\
  = &\E\bigg[\prod_{j=0}^{N-1}\Pi^h_\Dx(\hat{Y}^h_{j}) \big( f_l(\cdot,\hat{Y}^h_{N}) - f(\cdot,\hat{Y}^h_{N}) \big)\bigg].
\end{align*}
Furthermore, \cite[Lemma 4.7]{BCT} guarantees  $\|\Pi^h_\Dx(y)\|_{\infty} \le 1$ and so
$$
\Big\|\prod_{j=0}^{N-1}\Pi^h_\Dx(\hat{Y}^h_{j})\Big\|_{\infty} \le 1.
$$
Then
\begin{align*}
  |u^{h}_{0,l}(\cdot,y) - u^h_0(\cdot,y)|_{\ell^\infty} &\le \E\bigg[\Big\|\prod_{j=0}^{N-1}\Pi^h_\Dx(\hat{Y}^h_{j})\Big\|_{\infty}  | f_l(\cdot,\hat{Y}^h_{N}) - f(\cdot,\hat{Y}^h_{N})|_{\ell^\infty}\bigg] \\
  &\le \E[|f_l(\cdot,\hat{Y}^h_{N}) - f(\cdot,\hY^h_{N})|_{\ell^\infty}]. 
\end{align*}
Now proceeding as for term $I$, one can show
\begin{equation}\label{estim_III}
  III \le \sup_{x\in\R,y\in[0,M]}|f(x,y)-f_l(x,y) |+ C\|(1+(\hY^h_{N})^L)\|_{L^p(\Omega)} \left(\frac{E[\hY^h_{N}]}{M}\right)^{\frac{p-1}{p}}.
\end{equation}
Finally, for every $\varepsilon>0$, thanks to the boundedness of the moments of $Y^{0,y}_T$ and the uniform (in $N$) boundedness of all the moments of $\hY^h_{N}$,  we can choose $M>0$ big enough to guarantee that the second term in the right-hand sides of \eqref{estim_I} and \eqref{estim_III} are less than $\varepsilon/5$. Chosen $M$, thanks to the uniform continuity of $f$, we can take $l$ big enough to guarantee $\sup_{x\in\R,y\in[0,M]}|f-f_l|\le \varepsilon/5$  and $N_0\in\N^*$, $\delta>0$ such that for every $N\ge N_0$ and $\Dx<\delta$ one has $C_lT(h+\Dx)<\varepsilon/5$, and so 
$$
|u(0,\cdot,y) - u^h_0(\cdot,y)|_{\ell^\infty} \le \frac{2\varepsilon}{5}+\frac{\varepsilon}{5}+\frac{2\varepsilon}{5} = \varepsilon.
$$
\end{proof}


\appendix 


\chapter{Appendices of Chapter \ref{Chapter_CIR}}\label{Appendix_Chapter_CIR}

\section{Proofs of Section~\ref{Sec_pol_fct}}\label{App_proof_sec_pol}
\begin{proof}[Proof of Lemma~\ref{estimates_pol}.]
  \textit{(1)} Let $f\in \PLRp{L}$.  We have $X_0(t,x)^j=\sum_{i=0}^j {j\choose i}((a-\sigma^2/4)\psi_k(t))^{j-i}e^{-kti}x^i$ and thus
  \begin{align*}
    f(X_0(t,x)) & =\sum_{j=0}^L a_j \sum_{i=0}^j {j\choose i}((a-\sigma^2/4)\psi_k(t))^{j-i}e^{-kti} x^i. 
  \end{align*}
  Therefore, $f(X_0(t,\cdot)) \in \PLRp{L}$ and we have
  $$\|f(X_0(t,\cdot)) \|\le \sum_{j=0}^L |a_j| \sum_{i=0}^j {j\choose i}(|a-\sigma^2/4|\psi_k(t))^{j-i}e^{-kti}   =\sum_{j=0}^L |a_j| \tilde{X}_0(t)^j, $$
  with $ \tilde{X}_0(t)=e^{-kt}+  |a-\sigma^2/4|\psi_k(t)$. For $k\ge 0$, we have $0\leq\psi_k(t)\leq t$ and thus $\tilde{X}_0(t)\le (1+|a-\sigma^2/4|t)$. For $k<0$, we have $\tilde{X}_0(t)=e^{-kt}(1+  |a-\sigma^2/4|\psi_{-k}(t))\le e^{-kt}(1+  |a-\sigma^2/4|t )$. Since $(1+  |a-\sigma^2/4|t )^L\le 1+t\sum_{i= 1}^j {j\choose i}  |a-\sigma^2/4|^i  (1\vee T)^i\le 1+t(1+  |a-\sigma^2/4|(1\vee T) )^L$, we get $\tilde{X}_0(t)^j\le (1\vee e^{-k Lt})[ 1+t(1+  |a-\sigma^2/4|(1\vee T) )^L]$ for $j\in\{0,\dots,L\}$ and then
  $$\|f(X_0(t,\cdot))\|\leq (1\vee e^{-kLt})(1+(1+|a-\sigma^2/4|(1\vee T))^Lt) \|f \|,$$ which gives the claim with $C_{X_0}=1+|a-\sigma^2/4|(1\vee T)$.\\

  \textit{(2)} Since $Y$ is a symmetric random variable, we have
  \begin{align*}
    \E[f(X_1(\sqrt{t}Y,x ))] & = \sjzL a_j \E[X_1(\sqrt{t}Y,x)^j] = \sjzL a_j \sum_{i=0}^{2j} {2j\choose i}  \bigg(\frac{\sigma \sqrt{t}}{2}\bigg)^{2j-i}\E[Y^{2j-i}]x^{i/2} \\
                             & =\sjzL a_j \sum_{i=0}^j {2j\choose 2i}  \bigg(\frac{\sigma^2 t}{4}\bigg)^{j-i}\E[Y^{2(j-i)}]x^i                                               \\
                             & =\sjzL a_j x^j + t\sjzL a_j \sum_{i=0}^{j-1} {2j\choose 2i}  \bigg(\frac{\sigma^2 }{4}\bigg)^{j-i} t^{j-i-1}\E[Y^{2(j-i)}]x^i.
  \end{align*}
  This proves that $\E[f(X_1(\sqrt{t}Y,\cdot))]\in \PLRp{L}$. We note that $\E[Y^{2j}]\ge 1$ by Hölder inequality since $\E[Y^2]=1$, and thus $\E[Y^{2j}]\le \E[Y^{2L}]$ for $j\in \{0,\dots,L\}$.  We get
  \begin{align*}
    \| f(X_1(\sqrt{t}Y, \cdot ))\| & \le \|f\| +t \E[Y^{2L}]\sjzL |a_j| \sum_{i=0}^{j-1} {2j\choose 2i}  \bigg(\frac{\sigma^2 }{4}\bigg)^{j-i} (1\vee T)^{j-i} \\
                                   & \le \|f\|\left(1+t  \E[Y^{2L}]\left(1+\frac \sigma 2 \sqrt{1\vee T}\right)^{2L}\right),
  \end{align*}
  since $\sum_{i=0}^{j-1} {2j\choose 2i}  \left(\frac{\sigma^2 }{4}\right)^{j-i} (1\vee T)^{j-i}\le \left(1+\frac \sigma 2 \sqrt{1\vee T}\right)^{2j} \le \left(1+\frac \sigma 2 \sqrt{1\vee T}\right)^{2L} $. This gives the claim with $C_{X_1}=\left(1+\frac \sigma 2 \sqrt{1\vee T}\right)^{2}$.
\end{proof}

\begin{proof}[Proof of Lemma~\ref{Moments_Formula_CIR}.]
  We have $\tilde{u}_0(t,x)=1$, and in the case $m=1$, we have $\tilde{u}_1(t,x)=x +\int_0^t(a-k \tilde{u}_1(s,x))ds$ that has the solution:
  \begin{equation*}
    \tilde{u}_1(t,x) = x e^{-kt} + a \psi_k(t)
  \end{equation*}
  where $\psi_k(t)=\frac{1-e^{-kt}}{k}$ if $k\neq 0$ and $\psi_k(t)=t$ otherwise.
  This gives the claim for $m=1$ with $\tilde{u}_{0,1}=a \psi_k(t)$ and  $\tilde{u}_{1,1}=e^{-kt}$. We then prove the result by induction and consider $m\geq 2$. Using It\^o formula and taking the expected value, one has $\partial_t \tilde{u}_m(t,x) = (am+\sigma^2 m(m-1)/2) \tilde{u}_{m-1}(t,x)-km\tilde{u}_m(t,x)$. Hence, we have
  \begin{equation*}
    \tilde{u}_m(t,x) = (e^{-kt})^m\left( x^m + \int_0^t (am+\sigma^2 m(m-1)/2)(e^{ks})^m  \tilde{u}_{m-1}(s,x) ds\right),
  \end{equation*}
  and we get the following induction relations that give us the representation \eqref{Pol_Estimate_CIR}
  \begin{align*}
    \begin{cases}\tilde{u}_{j,m}(t) = (e^{-kt})^m \int_0^t (am+\sigma^2 m(m-1)/2)(e^{ks})^m  \tilde{u}_{j,m-1}(s) ds, \  0\le j\le m-1, \\
      \tilde{u}_{m,m}(t) = (e^{-kt})^m.
    \end{cases}
  \end{align*}

  Let $f\in \PLRp{L}$. We clearly get from the preceding result that $\E[f(X^\cdot_t)]\in \PLRp{L}$ and
  $$\|\E[f(X^\cdot_t)]\|\le \sum_{m=0}^L |a_m|\sum_{j=0}^m|\tilde{u}_{j,m}(t)| \le C_{\text{cir}}(L,T) \|f\|.  $$
\end{proof}

\section{Proofs of Section~\ref{Sec_main}}\label{App_proof_sec_5}

\begin{proof}[Proof of Lemma~\ref{lem_estimnorm}]
    Properties (1)--(3) are straightforward, and we prove only (4)--(6).
    \item[$(4)$] We use the fact that $1+x^L\leq 2(1+x^{L+1})$ for $x\geq 0$, hence
    $$ \max_{j\in\{0,\ldots, m\}} \sup_{x\geq 0}\frac{|f^{(j)}(x)|}{1+x^{L+1}}\leq 2 \max_{j\in\{0,\ldots, m\}} \sup_{x\geq 0}\frac{|f^{(j)}(x)|}{1+x^L}.$$
    \item[$(5)$] Let $f\in \CpolKL{m}{L}$. We will use the fact that for all $ x\geq 0$, $(1+x)(1+x^L)\leq 3(1+x^{L+1})$ so
    $$\sup_{x\geq 0}\frac{x|f^{(j)}(x)|}{1+x^{L+1}} \leq 3\sup_{x\geq 0}\frac{x}{1+x}\sup_{x\geq 0}\frac{|f^{(j)}(x)|}{1+x^L}= 3\sup_{x\geq 0}\frac{|f^{(j)}(x)|}{1+x^L}.$$  Now, we use the Leibniz rule on $\mathcal{M}_1 f$ and get $(xf(x))^{(j)}=jf^{(j-1)}(x)+xf^{(j)}(x)$, so
    $$\sup_{x\geq 0}\frac{|(xf(x))^{(j)}|}{1+x^{L+1}}\leq j\sup_{x\geq 0}\frac{|f^{(j-1)}(x)|}{1+x^{L+1}} +  \sup_{x\geq 0}\frac{x|f^{(j)}(x)|}{1+x^{L+1}}.$$
    Maximizing both sides on $j\in\{0,\ldots, m\}$ and using the previous inequality gives $\|\mathcal{M}_1 f\|_{m,L+1}\le  m\| f\|_{m-1,L+1} + 3\|f\|_{m,L}$. We get the bound by using properties $(2)$ and $(4)$.

    \item[$(6)$] We have $\|\cL f\|_{m,L+1}\le a \| f' \|_{m,L+1} + (2m+3)[|k|\| f' \|_{m,L}+\frac {\sigma^2} 2 \|f''\|_{m,L}]$ by using the property~(5). We get the estimate by using~(3),~(4) and~(2). The other estimate for $V_1^2/2$ is obtained by taking $a=\sigma^2/4$ and $k=0$, while the one for $V_0$ follows by using the same arguments.
\end{proof}

\begin{proof}[Proof of Lemma~\ref{regular_rep}]
  For $x>0$, we have
  \begin{align*}
    \psi'_g(x) & =\left(1+\frac{\beta}{2\sqrt{x}}\right) g'(x+\beta \sqrt{x}+\gamma)+ \left(1-\frac{\beta}{2\sqrt{x}}\right) g'(x-\beta \sqrt{x}+\gamma)             \\
               & =\underset{\psi_{g'}(x)}{\underbrace{g'(x+\beta \sqrt{x}+\gamma)+g'(x-\beta \sqrt{x}+\gamma)}}+\beta^2 \int_0^1g''(x+\beta(2u-1)\sqrt{x}+\gamma)du,
  \end{align*}
  since $\frac{d}{du}g'(x+\beta(2u-1)\sqrt{x}+\gamma)=2\beta\sqrt{x} g''(x+\beta(2u-1)\sqrt{x}+\gamma)$.
  Clearly, this derivative is continuous at~$0$ which shows that $\psi_g$ is $C^1$.

  We are now in position to prove~\eqref{formula_psign} by induction on~$n$. It is true for $n=0,1$.  We assume that it is true for~$n$. Then, we get by using the case $n=1$, differentiating~\eqref{formula_psign} and an integration by parts for the fourth term:
  \begin{align*}
    \psi_g^{(n+1)}(x)= & \, \psi_{g^{(n+1)}}(x)+\beta^2 \int_0^1g^{(n+2)}(x+\beta(2u-1)\sqrt{x}+\gamma)du                                                                   \\
                       & +     \sum_{j=1}^n\binom{n}{j}\beta^{2j} \left(\int_0^1g^{(n+j+1)}(x+\beta(2u-1)\sqrt{x}+\gamma) \frac{(u-u^2)^{j-1}}{(j-1)!}du \right.            \\
                       & \phantom{+   \sum_{k=1}^n\binom{n}{j}\beta^{2j}} \left.+\beta^2\int_0^1g^{(n+j+2)}(x+\beta(2u-1)\sqrt{x}+\gamma) \frac{(u-u^2)^{j}}{j!}du \right).
  \end{align*}
  We then reorganize the terms as follows
  \begin{align*}
    \psi_g^{(n+1)}(x)   = & \, \psi_{g^{(n+1)}}(x)+(n+1)\beta^2 \int_0^1g^{(n+2)}(x+\beta(2u-1)\sqrt{x}+\gamma)du                                               \\
                          & +\beta^{2n+2}\int_0^1g^{(2n+2)}(x+\beta(2u-1)\sqrt{x}+\gamma) \frac{(u-u^2)^{n}}{n!}du                                              \\
                          & + \sum_{j=2}^n\binom{n}{j}\beta^{2j} \left(\int_0^1g^{(n+j+1)}(x+\beta(2u-1)\sqrt{x}+\gamma) \frac{(u-u^2)^{j-1}}{(j-1)!}du \right) \\
                          & +\sum_{j=1}^{n-1}\binom{n}{j}\beta^{2j+2} \left(\int_0^1g^{(n+j+2)}(x+\beta(2u-1)\sqrt{x}+\gamma) \frac{(u-u^2)^{j}}{j!}du \right).
  \end{align*}
  The last sum is equal to $\sum_{j=2}^{n}\binom{n}{j-1}\beta^{2j} \left(\int_0^1g^{(n+j+1)}(x+\beta(2u-1)\sqrt{x}+\gamma) \frac{(u-u^2)^{j-1}}{(j-1)!}du \right)$ by changing $j$ to $j-1$, and we conclude by using that $\binom{n}{j}+\binom{n}{j-1}=\binom{n+1}{j}$.
\end{proof}

\begin{proof}[Proof of Corollary~\ref{cor_psign}]
  We use~\eqref{formula_psign} with $\gamma=\beta^2/4$. We first notice that
  \begin{align*}
    |\psi_{g^{(n)}}(x)| & \le \|g\|_{n,L}(2+(\sqrt{x}+\beta/2)^{2L}+(\sqrt{x}-\beta/2)^{2L})                      \\
                        & =\|g\|_{n,L}\left(2+2x^L +2 \sum_{i=1}^{L}\binom{2L}{2i}(\beta/2)^{2i}  x^{L-i}\right).
  \end{align*}
  Using that $x^{i}\le 1+x^L$ for $0\le i\le L-1$, we get
  $$ |\psi_{g^{(n)}}(x)|\le  2 \|g\|_{n,L}(1+x^L)  \sum_{i=0}^{L}\binom{2L}{2i}(\beta/2)^{2i}=\|g\|_{n,L}(1+x^L)\big((1+\beta/2)^{2L}+(1-\beta/2)^{2L}\big).$$
  For the other terms, we use that for $u\in[0,1]$, $x\ge 0$ and $j\in\{1,\dots,n\}$,
  \begin{align*}
    |g^{(n+j)}(x+\beta(2u-1)\sqrt{x}+\beta^2/4)| & \le \|g\|_{2n,L}(1+(x+\beta(2u-1)\sqrt{x}+\beta^2/4)^L) \\
                                                 & \le \|g\|_{2n,L}(1+( \sqrt{x}+\beta/2)^{2L}).
  \end{align*}
  We again expand $( \sqrt{x}+\beta/2)^L=x^L+ \sum_{i=1}^{2L}\binom{2L}{i}(\beta/2)^{i}x^{(L-i)/2}$ and use that $x^{(L-i)/2}\le 1+x^L$ to get
  $$|g^{(n+j)}(x+\beta(2u-1)\sqrt{x}+\beta^2/4)|\le \|g\|_{2n,L}(1+\beta/2)^{2L}(1+x^L).$$
  Besides, we have $u-u^2\le 1/4$ for $u\in[0,1]$ and thus $\int_0^1(u-u^2)^jdu\leq \frac{1}{4^{j}}$, which gives $\sup_{x\ge 0}\frac{|\psi_{g^{(n)}}(x)|}{1+x^L}\le \tilde{C}(\beta)$ with
  \begin{align*}
    \tilde{C}(\beta) & =\|g\|_{n,L}\big((1+\beta/2)^{2L}+(1-\beta/2)^{2L}\big)+\|g\|_{2n,L}(1+\beta/2)^{2L}\sum_{j=1}^n \binom{n}{j}\bigg(\frac{\beta^2}{4}\bigg)^j \\
                     & =\|g\|_{n,L}\big((1+\beta/2)^{2L}+(1-\beta/2)^{2L}\big)+\|g\|_{2n,L}(1+\beta/2)^{2L}(1+\beta^2/4)^{n}                                        \\
                     & \leq\|g\|_{2n,L}\big((1+\beta/2)^{2L}+(1-\beta/2)^{2L}+(1+\beta/2)^{2L}(1+\beta^2/4)^{m}\big)=C_{\beta,m,L}\|g\|_{2n,L},
  \end{align*}
  which gives the claim.
\end{proof}

\section{Assumption~\eqref{H1_bar} for symmetric random variables}

\begin{theorem}
  Let $\eta:\R\to \R_+$ be a $C^\infty$ even function. Then, $\eta^*_m\geq 0$ for all $m\in\N^*$ if and only if $\eta(\sqrt{\cdot})$ is the Laplace transform of a finite positive Borel measure $\mu$ on $[0,\infty)$, i.e. $\eta(\sqrt{x})=\int_0^{\infty}e^{-tx}\mu(dt)$ for all $x\in\R_+$.
\end{theorem}
\begin{proof}
  We start to prove that $\eta^*_m\geq 0$ for all $m\in\N^*$ implies $\eta(\sqrt{x})=\int_0^{\infty}e^{-tx}\mu(dt)$ for all $x\in\R$. To prove this, we use Bernstein's Theorem for completely monotone functions (see e.g. \cite[Theorem 12a p.~160]{Widder}) and show that  for all $m\in \N$ and $x\in\R_+^*$, $(-1)^m \partial_x^m [\eta(\sqrt{x})]\geq 0$. 
  To do so, we prove by induction on~$m$ the representation $$\partial^m_x[\eta(\sqrt{x})]=-\frac{(m-1)!}{2^{2m-1}}x^{-m}\sum_{j=1}^m c_{j,m}x^{\frac{j}{2}}\eta^{(j)}(\sqrt{x})=(-1)^m\frac{(m-1)!}{2^{2m-1}}x^{-m}\eta^*_m(\sqrt{x}).$$ For $m=1$, we have $\eta^*_1(\sqrt{x})=c_{1,1} \sqrt{x} \eta'(\sqrt{x})$ and the representation holds from $\partial_x[\eta(\sqrt{x})]=\frac{1}{2\sqrt{x}}\eta'(\sqrt{x})=-\frac{1}{2x}\eta^*_1(\sqrt{x})$ using that $c_{1,1}=-1$. Now, let $m\geq 2$ and suppose the representation is true for $m-1$, so
  \begin{equation*}
    \partial_x^{m}[\eta(\sqrt{x})]=\partial_x(\partial_x^{m-1}[\eta(\sqrt{x})])=\partial_x\Bigg(-\frac{(m-2)!}{2^{2m-3}}x^{-(m-1)}\sum_{j=1}^{m-1}c_{j,m-1}x^{\frac{j}{2}}\eta^{(j)}(\sqrt{x}) \Bigg).
  \end{equation*}
  Differentiating and using that $\partial_x \big(x^{\frac{j}{2}}\eta^{(j)}(\sqrt{x})\big)= \frac{1}{2x}\big(j x^\frac{j}{2}\eta^{(j)}(\sqrt{x})+x^{\frac{j+1}{2}}\eta^{(j+1)}(\sqrt{x})\big)$, we get
  \begin{align*}
    \partial_x^{m}[\eta(\sqrt{x})] & =
    \begin{multlined}[t] -\frac{(m-2)!}{2^{2m-3}}\Bigg(-\frac{m-1}{x^{m}} \sum_{j=1}^{m-1}c_{j,m-1}		x^{\frac{j}{2}}\eta^{(j)}(\sqrt{x}) \\
      +\frac{1}{2x^{m}}\sum_{j=1}^{m-1} c_{j,m-1}\bigg(j x^\frac{j}{2}	\eta^{(j)}(\sqrt{x})+x^{\frac{j+1}{2}}\eta^{(j+1)}(\sqrt{x})\bigg) \Bigg)\end{multlined}                                 \\
                                 & =\begin{multlined}[t] -\frac{(m-2)!}{2^{2m-3}}x^{-m}\Bigg( \Big(\frac{1}{2}- m-1\Big)  c_{1,m-1} x^{\frac{1}{2}}\eta^{(1)}(\sqrt{x})\\
      +\sum_{j=1}^{m-1}\bigg(\Big(\frac{j}{2}- m+1\Big)  c_{j,m-1} +\frac{1}{2} 	c_{j-1,m-1}\Big)	x^{\frac{j}{2}}\eta^{(j)}(\sqrt{x}) \\
      +\frac{1}{2}c_{m-1,m-1}x^\frac{m}{2}	\eta^{(m)}(\sqrt{x})\Bigg)\end{multlined} \\
                                 & =\begin{multlined}[t] -\frac{(m-1)!}{2^{2m-1}}x^{-m}\Bigg( \Big(\frac{2}{m-1}- 4\Big)  c_{1,m-1} x^{\frac{1}{2}}\eta^{(1)}(\sqrt{x})\\
      +\sum_{j=1}^{m-1}\bigg(\Big(\frac{2j}{m-1}- 4\Big)  c_{j,m-1} +\frac{2}{m-1} 	c_{j-1,m-1}\Big)	x^{\frac{j}{2}}\eta^{(j)}(\sqrt{x}) \\
      +\frac{2}{m-1}c_{m-1,m-1}x^\frac{m}{2}	\eta^{(m)}(\sqrt{x})\Bigg)\end{multlined}
  \end{align*}
  and we conclude using the recursion formula \eqref{recursive_coeff_formula} for $c_{j,m}$.

  We now assume that $\eta(\sqrt{x})=\int_0^\infty e^{-tx} \mu(dt)$ and show that $\eta^*_m\ge 0$ for all $m\ge 1$. We define $\eta_g(x)=e^{-\frac{x^2}{2}}$ and consider for all $t>0$ the function $\eta_t(x)=e^{-tx^2}$. We remark that for all $t>0$, $\eta_t(x)=\eta_g(h_t(x))$ with $h_t(x)=\sqrt{2t}x$ and so we can write by Lemma~\ref{gaussian_eta_positivity}
  $$(\eta_t)^*_m(x)=(-1)^{m-1}\sum_{j=1}^m c_{j,m}x^j\eta_t^{(j)}(x)=(-1)^{m-1}\sum_{j=1}^m c_{j,m}(\sqrt{2t}x)^j\eta_g^{(j)}(\sqrt{2t}x)=(\eta_g)^*_m(\sqrt{2t}x).
  $$
  Therefore, $(\eta_t)^*_m(x)\geq 0$ for all $t>0$ and $x\in\R$.
  We now consider an even function $\eta:\R \to \R_+$ such that $\eta(\sqrt{x})=\int_0^{\infty}e^{-tx}\mu(dt)$ for some Borel measure~$\mu$ on $[0,\infty)$. We then have for all $x\in\R$, $\eta(x)=\int_0^{\infty} e^{-tx^2}\mu(dt)=\int_0^{\infty}\eta_t(x) \mu(dt)$ and thus $\eta^{(j)}(x)=\int_0^{\infty}\eta^{(j)}_t(x) \mu(dt)$. This gives, for all $m\in\N^*$,
  \begin{align*}
    \eta^*_m(x) = (-1)^{m-1}\sum_{j=1}^m c_{j,m}x^j\eta^{(j)}(x) & =\int_0^\infty(-1)^{m-1}\sum_{j=1}^m c_{j,m}x^j\eta_t^{(j)}(x)\mu(dt) \\
                                                                 & =\int_0^\infty(\eta_t)^*_m(x)\mu(dt)\geq 0
  \end{align*}
  where the last integral is positive for all $x\in\R$ because is an integral of a positive function against a positive measure.
\end{proof}

\begin{corollary}\label{cor_density_etam}
  All the densities that satisfy the hypothesis of the representation Lemma \ref{regular_density} for all $m\in\N^*$ are such that $\eta(\sqrt{\cdot})$ is the Laplace transform of a finite positive Borel measure $\mu$ over $[0,\infty)$.
\end{corollary}

\chapter{Some other results for CIR process}\label{other_results_CIR}
In this Appendix chapter, we list some results that have been proven concerning the CIR process, but that we have not attached to any article.

\section{{Improvements in $\CpolKL{k}{L}$ theory for CIR}}\label{improvement_CIR_result}
 Here, we present an extension of the representation presented in Lemma \ref{regular_density}
    \begin{lemma}\label{regular_density_mod}
 Let $M,L,\Xi\in \N$. Let $Y$ be a symmetric random variable with density $\eta\in \mathcal{C}^M(\R)$ such that for all $i\in\{0,\ldots,M\}$, there exists $\varepsilon>0$ such that $|\eta^{(i)}(v)|=o(|v|^{-2(L+\Xi+1 +\varepsilon)-i})$ for $|v|\rightarrow \infty$. Then, for all function $f\in\CpolKL{M}{L}$, $m\in\{1, \ldots,  M\}$, $\nu\in\{0,\ldots,\Xi\}$ and $t\in [0,T]$ one has the following representation
        \begin{equation}\label{repres_X1_functional_mod}
        \partial^m_x\E[Y^{2\nu}f(X_1(\sqrt{t}Y,x))] = \int_{-\infty}^\infty \int_0^1 (u-u^2)^{m-1} f^{(m)}(w(u,x,v)) \eta^{*,\nu}_m(v) dudv
        \end{equation}
 where  $w(u,x,v)=x+(2u-1)\sigma\sqrt{t}v\sqrt{x}+ \sigma^2tv^2/4$, $\eta^{*,\nu}_{m}(v)=(-1)^{m-1}  \left(\sum_{j=1}^m c_{j,m}  v^j\hetanu^{(j)}(v)\right)$, with $\hetanu(v)=v^{2\nu}\eta(v)$ and the coefficients $c_{j,m}$ are defined by induction, starting from $c_{1,1}=-1$, through the following formula
        \begin{equation}\label{recursive_coeff_formula_mod}
 c_{l,m}   = \bigg(\frac{2l}{m-1}-4\bigg)  c_{l,m-1}\mathds{1}_{l<m} + \frac{2}{m-1}  c_{l-1,m-1}\mathds{1}_{l>2},     \qquad l\in\{1, \ldots, m\}, \, m\in\{2, \ldots,M\}.  \\
        \end{equation}
 Let now $\nu=0$ and define $\eta^*_{m}(v)= \eta^{*,0}_m(v)$. If the density $\eta$ is such that  $\eta^*_{m}(v)\geq0$ for all $v\in\R$, and all $ m\in\{1, \ldots,  M\}$, then there exists $C\in \R_+$ such that
        \begin{equation}\label{sharp_estimate_X1_functional}
        \forall m \in \{1,\dots,M\}, \forall t \in[0,T] \   \|\E[f(X_1(\sqrt{t}Y,\cdot))]\|_{m,L} \leq (1+Ct) \|f\|_{m,L}.
        \end{equation}
 Furthermore, for all $\nu\in\{1,\ldots,\Xi\}$, there exists $C\in \R_+$ such that
        \begin{equation}\label{estimate_X1_functional_mod}
        \forall m \in \{1,\dots,M\}, \forall t \in[0,T] \   \|\E[Y^{2\nu} f(X_1(\sqrt{t}Y,\cdot))]\|_{m,L} \leq C \|f\|_{m,L}.
        \end{equation}
    \end{lemma}
 Let us stress here one important fact. When $\nu\ge 1$, differently to the case $\nu=0$, we do not want to prove the sharper estimate with coefficient $1+Ct$, because we will use the estimate \eqref{estimate_X1_functional_mod} a fixed finite number of times to prove a (less demanding in derivatives) $H_1$ property.

    \begin{proof}
 We sketch the proof that is almost identical to the proof of Lemma \ref{regular_density}. The proof of \eqref{repres_X1_functional_mod} is identical to that used to prove \eqref{repres_X1_functional}  but this time we replace $\eta(v)$ with $\hetanu(v)=v^{2\nu}\eta(v)$. \eqref{sharp_estimate_X1_functional} is the same as \eqref{estimate_X1_functional}. Regarding \eqref{estimate_X1_functional_mod},we have
  \begin{align*}
 |\partial^m_x\E[Y^{2\nu}f(X_1(\sqrt{t}Y,x))]| & \leq \int_0^1 (u-u^2)^{m-1}  \int_{-\infty}^\infty |f^{(m)}(w(u,x,v))| |\eta^{*,\nu}_m(v)|dvdu                     \\
                                          & \leq \|f\|_{m,L}\int_0^1 (u-u^2)^{m-1}  \int_{-\infty}^\infty (1+w(u,x,v)^L) |\eta^{*,\nu}_m(v)|dvdu.          
  \end{align*}
 We use that $1+w(u,x,v)^L\le (1+2^L(1+\sigma^2Tv^2/4)^L)(1+x^L)$, and $(u-u^2)^{m-1}\le1$ to get

  \begin{equation*}
 |\partial^m_x\E[Y^{2\nu}f(X_1(\sqrt{t}Y,x))]|\le \\ 
    \|f\|_{m,L}\int_{-\infty}^\infty (1+2^L(1+\sigma^2Tv^2/4)^L) |\eta^{*,\nu}_m(v)|dv (1+x^L),
  \end{equation*}
 and the hypothesis on $\eta$ guarantees that the integral is finite and less of a constant $C$ depending on $\sigma,L$ and $ T$. Finally, one has
  $$
 |\partial^m_x\E[Y^{2\nu}f(X_1(\sqrt{t}Y,x))]|\le (1+x^L)C \|f\|_{m,L},
  $$
 and this proves the desired norm inequality.
    \end{proof}

 Thanks to estimate \eqref{estimate_X1_functional_mod} in Lemma \ref{regular_density_mod}, we can improve the estimate in the remainder in formula \eqref{expan_X1} in Lemma \ref{lem_expan_V}. We state the following result.

    \begin{lemma}\label{lem_expan_V_regular_density}
 Let $m,L\in \N, \ T>0,\ t \in [0,T]$ and $Y\sim \mathcal{N}(0,1)$. We have, for $f\in\CpolKL{m+2(\nu+1)}{L} $,
        \begin{multline}\label{expan_X1_regular_density}
 \E\big[f(X_1(\sqrt{t}Y,x))\big] = \sum_{i=0}^\nu \frac{t^i}{i!} \left(\frac{1}{2}V^2_1\right)^if(x)  \\
 + t^{\nu+1}\int_0^{1} \frac{(1-u)^{2\nu+1}}{(2\nu+1)!}\E\big[Y^{2\nu +2}V^{2\nu+2}_1 f(X_1(u \sqrt{t} Y,x))\big]du, 
        \end{multline}
 with 
        \begin{equation}\label{estim_expan_X1_regular_density}
 \left\|\int_0^{1} \frac{(1-u)^{2\nu+1}}{(2\nu+1)!}\E[Y^{2\nu +2}V^{2\nu+2}_1 f(X_1(u \sqrt{t} Y,\cdot))]du\right\|_{m,L+\nu+1}\le  C_1 \|f\|_{m+2(\nu+1),L}
        \end{equation}
 and $C_1\in \R^+$ depending on $(a,k,\sigma)$, $T$, $m$, $M$ and $\nu$.
    \end{lemma}
    \begin{proof}
 To get \eqref{expan_X1_regular_density}, we proceed like in the proof of \eqref{expan_X1} in Lemma \ref{lem_expan_V}, and we use Fubini Theorem in the end to exchange the integral with the expected value. Thanks to Lemma $\ref{gaussian_eta_positivity}$, being $Y\sim \mathcal{N}(0,1)$,  $\eta^*_{m}\geq0$ for all $m\in\N$, then we can get the better estimate \eqref{estim_expan_X1_regular_density} using \eqref{estimate_X1_functional_mod} instead of exploiting the estimate in Corollary \ref{cor_psign}.
    \end{proof}

 We can now prove a sharper version of Proposition \ref{prop_H1sch}.
    \begin{prop}\label{prop_H1sch_regular_density}
 Let $Y\sim\mcN(0,1)$, $\sigma^2\le 4a$ and $\hat{X}^x_t$ be the scheme~\eqref{Alfonsi_scheme}.  Let $m\in \N,L \in \N^*$ and $f \in \CpolKL{m+6}{L}$. Then, we have for $t\in[0,T]$,
        $$ \E[f(\hat{X}^x_t)]=f(x) +t \cL f(x)+\frac{t^2}{2} \cL^2f(x) +\bar{R}f(t,x),$$
 with $\|\bar{R}f(t,\cdot)\|_{m,L+3}\le C t^3 \|f\|_{m+6,L}$.
    \end{prop}
    \begin{proof} 
 The proof is identical to that used to prove Proposition \ref{prop_H1sch}, but now instead of using the estimate found in Lemma \ref{lem_expan_V} for the remainder of the expansion of the functional of $f(X_1(\sqrt{t}Y,\cdot))$ we use \eqref{estim_expan_X1_regular_density}.  
    \end{proof}

 We state the improved version of Theorem \ref{thm_main}
    \begin{theorem}\label{thm_main_CIR_mod}
 Let $\hat{X}^x_t$ be the scheme defined by~\eqref{NV_scheme} for $\sigma^2\le 4a$ and $Q_lf(x)=\E[f(\hat{X}^x_{h_l})]$, for $l\ge 1$.
 Then, for all $f\in \CpolK{12}$, we have $\hat{P}^{2,n}f(x)-P_Tf(x)=O(1/n^4)$ as $n\to \infty$.\\
 Besides, for  $f\in \CpolK{6\nu}$, we have $\hat{P}^{{\nu,n}}f(x)-P_Tf(x)=O(1/n^{2\nu})$.
    \end{theorem}
    \begin{proof}
 One only needs to replace Proposition \ref{prop_H1sch} with Proposition \ref{prop_H1sch_regular_density} in the proof of Theorem \ref{thm_main}.
    \end{proof}

\section{High order approximation of the CIR semigroup in the high volatility regime}

Here, we present very simple ideas to get high order approximations for CIR semigroup in the high regime $\sigma^2> 4a$. We start stating a simple Lemma.
\begin{lemma}\label{lemmaM-1}
 Let $k\in\N,L\in\N^*$ and $f\in\CpolKL{k+1}{L}$. We define $\cM_{-1}f$ as $\cM_{-1}f(x)=f(x)/x$ and $f_0$ as $f_0(x)= f(x)-f(0)$ for every $y\ge0$. Then $\cM_{-1}f_0\in\CpolKL{k}{L}$ and there exists $C>0$ such that $||\cM_{-1}f_0||_{j,L-1}\le C ||f||_{j,L}$ for every $j\in{0,\ldots,k-1}$.
\end{lemma}
\begin{proof}
 Let $f\in\CpolKL{k+1}{L}$, it is easy to show that $\cM_{-1}f$ is $k+1$ times differentiable in 0, then it is $k$-times continuously differentiable in 0, as in the rest of the domain. Then, thanks to the Taylor formula centered in 0, for all $j\in\{0,\ldots,k\}$ one has
  $$
 f(x)=\sum_{i=0}^{j} f^{i}(0)x^i + \int_0^x \frac{(x-t)^j}{j!}f^{(j+1)}(t)dt,
  $$
 and so subtracting $f(0)$ and dividing for $x$ in both sides
  $$
 \cM_{-1}f_0(x) = \sum_{i=0}^{j-1} f^{j+1}(0)x^i + \frac{1}{x} \int_0^x \frac{(x-u)^j}{j!}f^{(j+1)}(u)du.
  $$
 Considering the $j-th$ derivatives, we get

  $$
  \partial^j_x\cM_{-1}f_0(x)= \sum_{i=0}^j \binom{j}{i} \frac{(-1)^i i!}{x^{i+1}} \int_0^x \frac{(x-u)^i}{i!}f^{(j+1)}(u) du.
  $$
 We can give the following estimates
  \begin{align*}
 |\partial^j_x\cM_{-1}f_0(x)| &\le \sum_{i=0}^j \binom{j}{i} \frac{ i!}{x^{i+1}} \int_0^x \frac{(x-u)^i}{i!} ||f||_{j+1,L}(1+u^L) du \\
    &= ||f||_{j+1,L}\sum_{i=0}^j \binom{j}{i} \frac{ 1}{x^{i+1}}\left(\frac{x^{i+1}}{i+1} +  \sum_{l=0}^i \binom{i}{l} (-1)^l x^{i-l}\int_0^xu^{L+l} du \right) \\
    &= ||f||_{j+1,L}\sum_{i=0}^j \binom{j}{i} \frac{ 1}{x^{i+1}}\left(\frac{x^{i+1}}{i+1} +  \sum_{l=0}^i \binom{i}{l} (-1)^l x^{i-l}\frac{x^{L+l+1}}{L+l+1} \right) \\
    &= ||f||_{j+1,L}\sum_{i=0}^j \binom{j}{i} \left(\frac{1}{i+1} +  \sum_{l=0}^i \binom{i}{l} (-1)^l \frac{x^{L}}{L+l+1}  \right) \\
    &\le  ||f||_{j+1,L}\sum_{i=0}^j \binom{j}{i} (1+x^L) \le 2^k||f||_{j+1,L}  (1+x^L).
  \end{align*}
 Hence, for all $j\in\{0,\ldots,k\}$, $||\cM_{-1}f_0||_{j,L}\le C ||f||_{j+1,L}$.
\end{proof}

\begin{prop}\label{representation_semigroup_f0}
 Let $t\in(0,T]$ and consider $(X^{1,x}_t)_{t\in[0,T]}$, $(X^{2,x}_t)_{t\in[0,T]}$ solutions to \eqref{CIR_SDE} with the parameter $a$ replaced respectively by $a_1=a+\sigma^2/2$ and $a_2=a+\sigma^2$.
 Let $f$ such that $\limsup_{x\rightarrow0^+}|f(x)-f(0)|/x<\infty$ and such that $\E[|f(X^x_t)|]<\infty $, Then
  \begin{equation}\label{slow_regime_representation}
 \E[f(X^x_t)] = f(0) + a \psi_k(t) \E[\cM_{-1}f_0(X^{1,x}_t)] + e^{-bt}x \E[\cM_{-1}f_0(X^{2,x}_t)].
  \end{equation}
\end{prop}
\begin{proof}
 Let $f$ such that $\limsup_{x\rightarrow0^+}|f(x)-f(0)|/x<\infty$ and $\E[|f(X^x_t)|]<\infty $. We recall the transition probability density of $X^x_t$
  \begin{equation*}
 p_a(t,x,z)=\sum_{i=0}^\infty \frac{e^{-d_t x/2}(d_t x/2)^i}{i!} \frac{c_t/2}{\Gamma(i+v)}\left(\frac{c_t z}{2}\right)^{i-1+v}e^{-c_t z/2},
  \end{equation*}
 where $c_t=\frac{4\psi_k(t)}{\sigma^2}$, $v=2a/\sigma^2$ and $d_t=c_te^{-bt}$.
 For all $i\in\N$ and $v>0$, $\Gamma(i+1+v) = (i+v)\Gamma(i+v)$, then for all $z>0$
  \begin{align*}
 p_a(t,x,z)&=\sum_{i=0}^\infty \frac{e^{-d_t x/2}(d_t x/2)^i}{i!} \frac{(i+v)c_t/2}{\Gamma(i+1+v)}\left(\frac{c_t z}{2}\right)^{i-1+v}e^{-c_t z/2} \\
    &=\frac{2}{c_t z}\sum_{i=0}^\infty \frac{e^{-d_t x/2}(d_t x/2)^i}{i!} \frac{(i+v)c_t/2}{\Gamma(i+1+v)}\left(\frac{c_t z}{2}\right)^{i+v}e^{-c_t z/2} \\
    &=\frac{d_t x}{c_t z}\sum_{i=0}^\infty \frac{e^{-d_t x/2}(d_t x/2)^i}{i!} \frac{c_t/2}{\Gamma(i+2+v)}\left(\frac{c_t z}{2}\right)^{i+1+v}e^{-c_t z/2} \\
    &\phantom{aa}+\frac{2v}{c_t z}\sum_{i=0}^\infty \frac{e^{-d_t x/2}(d_t x/2)^i}{i!} \frac{c_t/2}{\Gamma(i+1+v)}\left(\frac{c_t z}{2}\right)^{i+v}e^{-c_t z/2}
  \end{align*}
 we call $v_1=1+v=2a_1/\sigma^2$ and $v_2=2+v=2a_2/\sigma^2$, and we remark that we can rewrite the density as follows 
  \begin{equation}
 p_a(t,x,z) = \frac{2v}{c_t z} p_{a_1}(t,x,z) + \frac{d_t x}{c_t z} p_{a_2}(t,x,z) = \frac{a\psi_k(t)}{z} p_{a_1}(t,x,z) + \frac{e^{-bt}x}{z} p_{a_2}(t,x,z)
  \end{equation}
 We rewrite $\E[f(X^x_t)] = f(0) + \E[f_0(X^x_t)]$ and remark that 
  \begin{align*}
 \E[f_0(X^x_t)] &=\int_0^\infty f_0(z)p_a(t,x,z) dz =  \int_0^\infty f_0(z)\left(\frac{a\psi_k(t)}{z} p_{a_1}(t,x,z) + \frac{e^{-bt}y}{z} p_{a_2}(t,x,z)\right) dz\\
    &= a\psi_k(t)\int_0^\infty \cM_{-1}f_0(z)p_{a_1}(t,x,z) dz + e^{-bt}x \int_0^\infty \cM_{-1}f_0(z)p_{a_2}(t,x,z)dz \\
    &= a \psi_k(t) \E[\cM_{-1}f_0(X^{1,x}_t)] + e^{-bt}x \E[\cM_{-1}f_0(X^{2,x}_t)]
  \end{align*}
 that concludes the proof.
\end{proof}

\begin{remark}
 The two diffusions $(X^{i,y}_t)_{t\in[0,T]}$ $i\in\{1,2\}$ have coefficients $\sigma_1=\sigma_2=\sigma$ and $a_1=a+\sigma^2/2$, $a_2=a+\sigma^2$, so both satisfy the well known Feller condition, $\sigma_i^2\le2a_i$.
 So equation \eqref{slow_regime_representation} shows that it is possible to apply approximation techniques that work under the Feller condition, also in the high volatility regime $\sigma^2>4a$. 
\end{remark}

\begin{theorem}
 Let $T>0$ and $\sigma^2>4a$. Let $\hat{X}^{i,x}_t$ be second order Ninomiya Victoir schemes for $(X^{i,x}_t)_{t\in[0,T]}$ $i\in\{1,2\}$, $\hat{P}_i^{\nu,n}$ the linear operator that approximate the CIR semigroup $P^i_T$ with rate $2\nu$. Let $f\in\CpolKL{6\nu+1}{L}$, $L\in\N$ and $\hat{P}_*^{\nu,n}$ the linear operator defined as $\hat{P}_*^{\nu,n}f(x)=f(0) +a\psi_k(t)\hat{P}_1^{\nu,n}\cM_{-1}f_0(x)+e^{-bT}x \hat{P}_2^{\nu,n}\cM_{-1}f_0(x)$, then there exists $C>0$ that does not depend on the function $f$ such that
  \begin{equation}
 ||\hat{P}_*^{\nu,n}f-P_Tf||_{0,L+7} \le \frac{C}{n^{2\nu}} ||f||_{6\nu+1,L}
  \end{equation}
\end{theorem}

\begin{proof}
 Let $f\in 6\nu+1$, by linearity of $\hat{P}_*^{\nu,n}$ and $P_T$ one has $\hat{P}_*^{\nu,n}f-P_Tf=\hat{P}_*^{\nu,n}f_0-P_Tf_0$ for all $x\in\R_+$. One has
    \begin{align*}
 ||\hat{P}_*^{\nu,n}f-P_Tf||_{0,L+7} &= ||\hat{P}_*^{\nu,n}f_0-P_Tf_0||_{0,L+7} \\
        &\le a\psi_k(T)||\hat{P}_1^{\nu,n}\cM_{-1}f_0-P^1_T\cM_{-1}f_0||_{0,L+7}  \\
        &\phantom{ab}+ e^{-bT}||x(\hat{P}_2^{\nu,n}\cM_{-1}f_0-P^2_T\cM_{-1}f_0)||_{0,L+7} \\
        &\le a\psi_k(T)||\hat{P}_1^{\nu,n}\cM_{-1}f_0-P^1_T\cM_{-1}f_0||_{0,L+7}  \\
        &\phantom{ab}+ 3e^{-bT}||\hat{P}_2^{\nu,n}\cM_{-1}f_0-P^2_T\cM_{-1}f_0||_{0,L+6} \\
        &\le a\psi_k(T)C_1n^{-2\nu}||\cM_{-1}f_0||_{6\nu+1,L+1}  \\
        &\phantom{ab}+ 3e^{-bT}C_2n^{-2\nu}||\cM_{-1}f_0||_{6\nu,L} \\
        &\le C^*_1n^{-2\nu}||f||_{6\nu+1,L+1} + C^*_2n^{-2\nu}||f||_{6\nu+1,L} \\
        &\le 2C^*_1n^{-2\nu}||f||_{6\nu+1,L} + C^*_2n^{-2\nu}||f||_{6\nu+1,L} \\
        &\le C n^{-2\nu}||f||_{6\nu+1,L},
    \end{align*}
 where we used to get the first inequality Proposition \ref{representation_semigroup_f0} to rewrite $P_Tf_0$ as $a\psi_k(T)P^1_T\mcM_{-1}f_0+e^{-bT}xP^2_T\mcM_{-1}f_0$, Theorem \ref{thm_main_CIR_mod} to get the third inequality and Lemma \ref{lemmaM-1} to get the fourth one.
\end{proof}

\section{The CIR moment formula and polynomial schemes}
It is well-known in the literature that one can obtain a moment formula for the CIR process using the CIR SDE; for example, as in Lemma \ref{Moments_Formula_CIR}, one can show, for all $L\in\N^*$
$$
 \E\big[(f(X^x_t))^L\big] = \sum_{j=0}^L \tilde{u}_{j,L}(t) x^j,
$$
giving a recursive formula for the coefficients without writing their explicit form.
We present here a proof of the CIR moment formula that gives the explicit form of the coefficients $\tilde{u}_{j,m}(t)$ using the transition density of the CIR process.
Then, we apply the knowledge of the exact form of the coefficients to construct schemes that converge for all polynomial functions $f\in\PRp$.
\subsection{The CIR moment formula}
To prove the moment formula, we start with one Lemma, but first, we define for all $i,j\in\N$,
\begin{equation}\label{i^*_j}
 i^*_j=i(i-1)\cdots(i-j+1)
\end{equation}
with the convention $i^*_0=1$ (and $0^*_j=0$ for all $j\ge1$). We state and prove the following result.
\begin{lemma}
 Let $L\in\N^*$, $i\in\N$ and $v\in\R$, then
    \begin{equation}\label{i_star_representation}
        \prod_{j=0}^{L-1}(i+j+v)=\sum_{j=0}^L{L\choose j}\prod_{q=j}^{L-1}(q+v)\,i^*_j.
        \end{equation}
    \end{lemma}
    \begin{proof}
 For $L=1$ the right hand of \eqref{i_star_representation} is simply ${1\choose 1} 1 i^*_1+{1\choose 0} v i^*_0 =i+v$, so the equality is verified. We consider the representation true for $L\geq 1$, and we prove the representation for $L+1$.
    \begin{align*}
    \prod_{j=0}^{L}(i+j+v)= \prod_{j=0}^{L-1}(i+j+v)(i+L+v) &= \sum_{j=0}^L{L\choose j}\prod_{q=j}^{L-1}(q+v)\,i^*_j  (i-j+j+L+v)
    \end{align*}
 we develop the multiplication by the term $(i-j+j+L+v)$, dividing it in $i-j$, $j$ and we get
    \begin{align*}
    \prod_{j=0}^{L}(i+j+v)= &\sum_{j=0}^L{L\choose j}\prod_{q=j}^{L-1}(q+v)\,i^*_{j+1}  + \sum_{j=0}^L{L\choose j}j\prod_{q=j}^{L-1}(q+v)\,i^*_j \\
    &+ \sum_{j=0}^L{L\choose j}\prod_{q=j}^{L}(q+v)\,i^*_j \\
 = i^*_{L+1} +  &\sum_{j=0}^{L-1}{L\choose j}\prod_{q=j}^{L-1}(q+v)\,i^*_{j+1}  + \sum_{j=1}^L{L\choose j}j\prod_{q=j}^{L-1}(q+v)\,i^*_j \\
    &+ \sum_{j=1}^L{L\choose j}\prod_{q=j}^{L}(q+v)\,i^*_j + \prod_{q=0}^{L}(q+v) \\
 = i^*_{L+1} +  &\underbrace{\sum_{j=0}^{L-1}{L\choose j}\prod_{q=j}^{L-1}(q+v)\,i^*_{j+1}}_{=I}  + \underbrace{\sum_{j=0}^{L-1}{L\choose j+1}(j+1)\prod_{q=j+1}^{L-1}(q+v)\,i^*_{j+1}}_{=II} \\
    &+ \underbrace{\sum_{j=0}^{L-1}{L\choose j+1}\prod_{q=j+1}^{L}(q+v)\,i^*_{j+1}}_{=III} + \prod_{q=0}^{L}(q+v),
    \end{align*}
 where we get the second line taking out the term for $j=L$ from the first addend and the term for $j=0$ from the third one.
 Now we sum the term $I$ with the term $II$
    \begin{align*}
 I+II &= \sum_{j=0}^{L-1}\left( {L\choose j}(j+v) +  {L\choose j+1}(j+1) \right)\prod_{q=j+1}^{L-1}(q+v)\, i^*_{j+1}\\
    &= \sum_{j=0}^{L-1}\left( {L\choose j}(j+v) +  \frac{L!}{j!(L-j)!}(L-j) \right)\prod_{q=j+1}^{L-1}(q+v)\, i^*_{j+1}\\
    &= \sum_{j=0}^{L-1}{L\choose j} (L+v) \prod_{q=j+1}^{L-1}(q+v)i^*_{j+1}\\ 
    &= \sum_{j=0}^{L-1} {L\choose j}  \prod_{q=j+1}^{L}(q+v)\,i^*_{j+1},
    \end{align*}
 and consequently the term $III$, and we get
    \begin{align*}
 I+II+III &= \sum_{j=0}^{L-1} \left({L\choose j} +{L\choose j+1}\right) \prod_{q=j+1}^{L}(q+v)\,i^*_{j+1}\\
    &= \sum_{j=0}^{L-1} {L+1\choose j+1} \prod_{q=j+1}^{L}(q+v)\,i^*_{j+1} \\
    &= \sum_{j=1}^{L} {L+1\choose j} \prod_{q=j}^{L}(q+v)\,i^*_{j}.
    \end{align*}
 Finally,
    \begin{align*}
    \prod_{j=0}^{L}(i+j+v) &= i^*_{L+1} + \sum_{j=1}^{L} {L+1\choose j}\prod_{q=j}^{L}(q+v)\,i^*_{j} + \prod_{q=0}^{L}(q+v) \\
    &= \sum_{j=0}^{L+1} {L+1\choose j}\prod_{q=j}^{L}(q+v)\,i^*_{j}.
    \end{align*}
    \end{proof}
    
    \begin{corollary}\label{gamma_moments_representation}
 Let $i\in\N$, $v\in\R$ such that $i+v>0$, $Z$ a random variable distributes as a $\Gamma(i+v,c)$, and we define 
    \begin{equation}\label{delta_cir_coefficients}
        \delta^{cir}_{j,L}(v)={L\choose j}\prod_{q=j}^{L-1}(q+v).
    \end{equation}
 Then for all $L\in\N^*$, the $L-th$ moment of the random variable $Z$ is
    \begin{equation}
 \E[Z^L] = \bigg(\frac{1}{c}\bigg)^L\sum_{j=0}^{L} \delta^{cir}_{j,L}(v) \,i^*_{j}.
    \end{equation}
    \end{corollary}
    \begin{proof}
 The result follows from a direct application of the identity \eqref{i_star_representation} to the well-known formula of the moments of the gamma law
    $$
 \E[Z^L] = \bigg(\frac{1}{c}\bigg)^L\frac{\Gamma(i+L+v)}{\Gamma(i+v)}=\bigg(\frac{1}{c}\bigg)^L\prod_{j=0}^{L-1}(i+j+v) = \bigg(\frac{1}{c}\bigg)^L\sum_{j=0}^L{L\choose j}\prod_{q=j}^{L-1}(q+v)\,i^*_j.
    $$
    \end{proof}
    
    \begin{prop}
 Let $X^x_t$ the CIR process starting from $x\in\R_+$, $t\in [0,T]$, $a>0$ and $v=2a/\sigma^2$. Then, the moment of $(X^x_t)_{[0,T]}$ are given by the following formula
    \begin{equation}
 \E\big[(X^x_t)^L\big] = \sum_{j=0}^{L} \delta^{cir}_{j,L}(v) \bigg(\frac{\sigma^2\psi_k(t)}{2}\bigg)^{L-j} e^{-jkt}x^j .
    \end{equation}
    \end{prop}
    
    \begin{proof}
 We recall the density of $X^x_t$
    \begin{equation*}
 p(t,x,z)=\sum_{i=0}^\infty \frac{e^{-d_t x/2}(d_t x/2)^i}{i!} \frac{c_t/2}{\Gamma(i+v)}\left(\frac{c_t z}{2}\right)^{i-1+v}e^{-c_t z/2}
    \end{equation*}
 where $c_t=\frac{4}{\sigma^2\psi_k(t)}$, $v=2a/\sigma^2$ and $d_t=c_te^{-kt}$. Then, the $L-th$ moment is given by
    \begin{align*}
 \E\big[(X^x_t)^L\big] =&\int_0^\infty z^L \sum_{i=0}^\infty \frac{e^{-d_t x/2}(d_t x/2)^i}{i!} \frac{c_t/2}{\Gamma(i+v)}\left(\frac{c_t z}{2}\right)^{i-1+v}e^{-c_t z/2}dz\\
 =&\sum_{i=0}^\infty \frac{e^{-d_t x/2}(d_t x/2)^i}{i!} \int_0^\infty z^L\frac{c_t/2}{\Gamma(i+v)}\left(\frac{c_t z}{2}\right)^{i-1+v}e^{-c_t z/2}dz.
    \end{align*}
 We use now the formula of the $L-th$ moment of a random variable $Z_i\sim\Gamma(i+v,2/c_t)$ in Corollary \ref{gamma_moments_representation} to get
    $$
    \int_0^\infty z^L\frac{c_t/2}{\Gamma(i+v)}\left(\frac{c_t z}{2}\right)^{i-1+v}e^{-c_t z/2}= \E[Z_i^L]= \bigg(\frac{2}{c_t}\bigg)^L\sum_{j=0}^{L} \delta^{cir}_{j,L}(v) \,i^*_{j},
    $$
 and so
    \begin{align*}
 \E\big[(X^x_t)^L\big] &=\sum_{i=0}^\infty \frac{e^{-d_t x/2}(d_t x/2)^i}{i!}  \bigg(\frac{2}{c_t}\bigg)^L\sum_{j=0}^{L} \delta^{cir}_{j,L}(v) \,i^*_{j}\\
        &=\bigg(\frac{2}{c_t}\bigg)^L \sum_{j=0}^{L} \delta^{cir}_{j,L}(v) \sum_{i=0}^\infty \frac{e^{-d_t x/2}(d_t x/2)^i}{i!}   \,i^*_{j}\\
        &=\bigg(\frac{2}{c_t}\bigg)^L \sum_{j=0}^{L} \delta^{cir}_{j,L}(v) \bigg(\frac{d_t x}{2}\bigg)^j\\
        &=\sum_{j=0}^{L} \delta^{cir}_{j,L}(v) \bigg(\frac{\sigma^2 \psi_k(t)}{2}\bigg)^{L-j} (e^{-kt}x)^j,
    \end{align*}
 where we used the fact that $d_t=e^{-kt}c_t$ and $c_t=\frac{4}{\sigma^2\psi_k(t)}$ to get the last equality.
    \end{proof}
\subsection{Polynomial schemes}
We show now how we can take advantage of knowing the explicit form of the coefficients in the moment formula to build weak approximation schemes when the test function $f\in\PRp$.

\begin{prop}
 Let $v>0$, $\nu\in\N^*$ and $(\tilde{Z}_i)_{i\in\N}$ a family of positive random variable such that
    \begin{equation}\label{tildeZ_Lmoments}
 \E[\tilde{Z}_i^L]=\sum_{j=0}^{L} \delta^{Z}_{j,L}(v) \,i^*_{j} \quad{where}\quad
        \delta^Z_{j,L}=\delta^{cir}_{j,L} \quad\text{ for all } j\in\{(L-\nu)\vee 0,\ldots,L\}. 
    \end{equation}
 We call $Z_i=2/c_t \tilde{Z_i}$ for all $i\in\N$ and $\hX^{x,Z}_t$ the random variable distributed has $Z_P$ where $P\sim\mathcal{P}(d_t x/2)$ (i.e. a Poisson random variable with parameter $d_t x/2$).
 Then for all $L\in\N^*$ and $f\in\PLRp{L}$ one has
    \begin{equation}
        \|\E[f(\hX^{\cdot,Z}_t)]-\E[f(X^{\cdot}_t)]\|\le Ct^{\nu+1}\|f\|.
    \end{equation}
    \begin{proof}
 Let $m\in\N$, $m>\nu$ one can easily use \eqref{tildeZ_Lmoments} to show that
        \begin{align*}
 \E\Big[\big(\hX^{\cdot,Z}_t\big)^m\Big]-\E[f(X^{\cdot}_t)^m] &= \sum_{j=0}^m (\delta^{Z}_{j,m}(v)-\delta^{cir}_{j,m}(v)) \Big(\frac{\sigma^2 \psi_k(t)}{2}\Big)^{m-j} e^{-jkt}x^j\\
            &=\Big(\frac{\sigma^2 \psi_k(t)}{2}\Big)^{\nu+1}\sum_{j=0}^{m-\nu-1} (\delta^{Z}_{j,m}(v)-\delta^{cir}_{j,m}(v)) \Big(\frac{\sigma^2 \psi_k(t)}{2}\Big)^{m-\nu-1-j} e^{-jkt}x^j
        \end{align*}
 whereas on the other hand if $m\le \nu$ then $\E[\hX^{\cdot,Z}_t]-\E[X^{\cdot}_t]$ vanishes. Then, let $L\in\N^*$ and $f\in\PLRp{L}$, using the previous step one has
        \begin{align*}
 \E[f(\hX^{x,Z}_t)]-\E[f(X^{x}_t&)] = \sum_{m=0}^L a_m\Big(\E\Big[\big(\hX^{x,Z}_t\big)^m\Big]-\E[f(X^{x}_t)^m]\Big)\\
 =&\Big(\frac{\sigma^2 \psi_k(t)}{2}\Big)^{\nu+1}\sum_{m=\nu+1}^L a_m \sum_{j=0}^{m-\nu-1} (\delta^{Z}_{j,m}(v)-\delta^{cir}_{j,m}(v)) \Big(\frac{\sigma^2 \psi_k(t)}{2}\Big)^{m-\nu-1-j} e^{-jkt}x^j
        \end{align*}
 then $\E[f(\hX^{\cdot,Z}_t)]-\E[f(X^{\cdot}_t)]\in\PLRp{L}$ and
    \begin{align*}
 ||\E[f(\hX^{\cdot,Z}_t)]-\E[f(&X^{\cdot}_t)]|| \\
 \le&\Big(\frac{\sigma^2 \psi_k(t)}{2}\Big)^{\nu+1}\sum_{m=\nu+1}^L |a_m| \sum_{j=0}^{m-\nu-1} |\delta^{Z}_{j,m}(v)-\delta^{cir}_{j,m}(v)| \Big(\frac{\sigma^2 \psi_k(t)}{2}\Big)^{m-\nu-1-j} e^{-jkt}\\
 \le& Ct^{\nu+1}\|f\|,
    \end{align*}
 where we used that $\psi_k(t)\le e^{(-b_+)T}t$ for all $0\le t\le T$ and $C=(\sigma^2e^{(-b_+)T}/2)^{\nu+1}C_L$ and $$C_L=\max_{\substack{t\le T\\m\in\{\nu+1,\ldots,L\}}} \sum_{j=0}^{m-\nu-1} |\delta^{Z}_{j,m}(v)-\delta^{cir}_{j,m}(v)| \Big(\frac{\sigma^2 \psi_k(t)}{2}\Big)^{m-\nu-1-j} e^{-jkt}.$$ 
    \end{proof}
\end{prop}
We now provide an example of a first-order scheme that works without parameter restrictions.
\begin{ex}
 Let $\tilde{Z}_i\sim\mathcal{P}(i+v)$, then
    $$\E[\tilde{Z}_i^L]=\sum_{j=0}^{L} \delta^{Z}_{j,L}(v) \,i^*_{j} \quad{where}\quad
    \delta^Z_{j,L}=\delta^{cir}_{j,L} \quad\text{ for all } j\in\{(L-1)\vee 0,L\}.$$
 So $\hX^{x,Z}_t = Z_P$ where $Z_i= 2/c_t \tilde{Z_i}$ and $P\sim\mathcal{P}(d_tx/2)$ is a first order scheme.
\end{ex}
\begin{proof}
 Let $L\in\N^*$, the moments of a Poisson random variable $Z_\lambda$ of parameter $\lambda$ have the following formula
    \begin{equation*}
 \E[\tilde{Z}_\lambda^L]= \sum_{j=0}^L {L\brace j}\lambda^j 
    \end{equation*} 
 where ${L\brace j}$ are Stirling numbers of the second kind. It is well known that ${L\brace L}=1$, ${L\brace L-1}=\binom{L}{2}=L(L-1)/2$, and one could prove for $L\ge 2$, ${L\brace L-2}=L(L-1)(L-2)(3L-5)/24$. Using the definition of $Z_i$ and the formula of the Stirling coefficients.
    \begin{align*}
 \E[\tilde{Z}_i^L]= \sum_{j=0}^L (i+v)^j {L\brace j},
    \end{align*}
 we just want to rewrite the right-hand side in terms of basis $\{i^*_0,\ldots,i^*_L\}$, to explicit the coefficients of the terms $i^*_{L}, i^*_{L-1}$ and to check that are equal respectively to $\delta^{cir}_{L,L}=1$ and $\delta^{cir}_{L-1,L}=L(L-1)+Lv$. The key remark in order to do this calculation is that 
    $$
 i^j=i^*_j+\frac{j(j-1)}{2}i^{j-1}+\sum_{l=0}^{j-2}c_l i^{l},
    $$
 so iterating the formula, one gets
    \begin{equation}\label{i_to_i*}
 i^j=i^*_j+\frac{j(j-1)}{2}i^*_{j-1}+\sum_{l=0}^{j-2}c^*_l i^*_{l}.
    \end{equation} 
 Then one has
    \begin{align*}
 \E[\tilde{Z}_i^L]&= (i+v)^L + \frac{L(L-1)}{2} (i+v)^{L-1}+\sum_{j=0}^{L-2} (i+v)^j {L\brace j},\\
        &= i^L + Lv i^{L-1} +  \frac{L(L-1)}{2} i^{L-1} +\cdots  \\
        &= i^*_L + (L(L-1) +Lv)  i^*_{L-1} +\sum_{j=0}^{L-2} \delta^Z_{j,L}i^*_j 
    \end{align*}
 where we expanded $(i+v)^L$ and $(i+v)^{L-1}$ to get the second equality and used \eqref{i_to_i*} to pass from the base $\{i^0,\ldots,i^L\}$ to $\{i^*_0,\ldots,i^*_L\}$. To check $\delta^Z_{j,2}=\delta^{cir}_{j,2}$ for $j\in\{0,1,2\}$ is trivial.
\end{proof}

One could check that even   $\delta^Z_{L-2,L}=\delta^{cir}_{L-2,L}$ when $L=2$ that proves that when $x$ is smaller than a certain threshold $Ct$, the scheme is indeed of order two (because matches the first two moments of CIR distribution). Unluckily, one can verify using the same calculations used in the proof and the formula for ${L\brace L-2}$ that  $\delta^Z_{L-2,L}<\delta^{cir}_{L-2,L}$ for all $L\ge 3$.
Now, we show using numerical tests what we have proved for $f\in\PRp$: $\hX^{x,Z}_t$ is a scheme of order one.
\begin{figure}[h]
    \centering
    \begin{subfigure}[h]{0.49\textwidth}
      \centering
      \includegraphics[width=\textwidth]{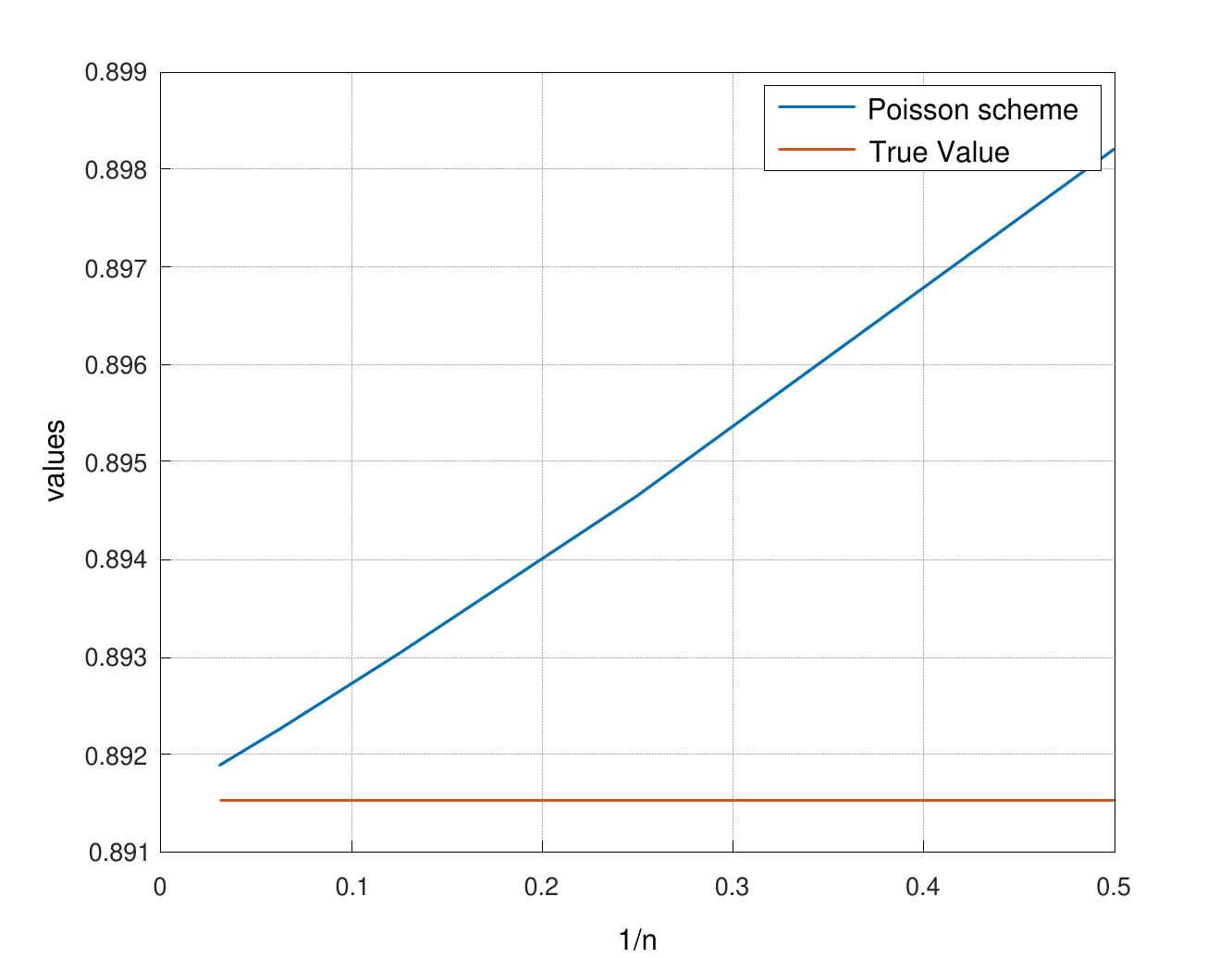}
      \caption{Values plot}
      \label{fig:values_C1S1}
    \end{subfigure}
    \hfill
    \begin{subfigure}[h]{0.49\textwidth}
      \centering
      \includegraphics[width=\textwidth]{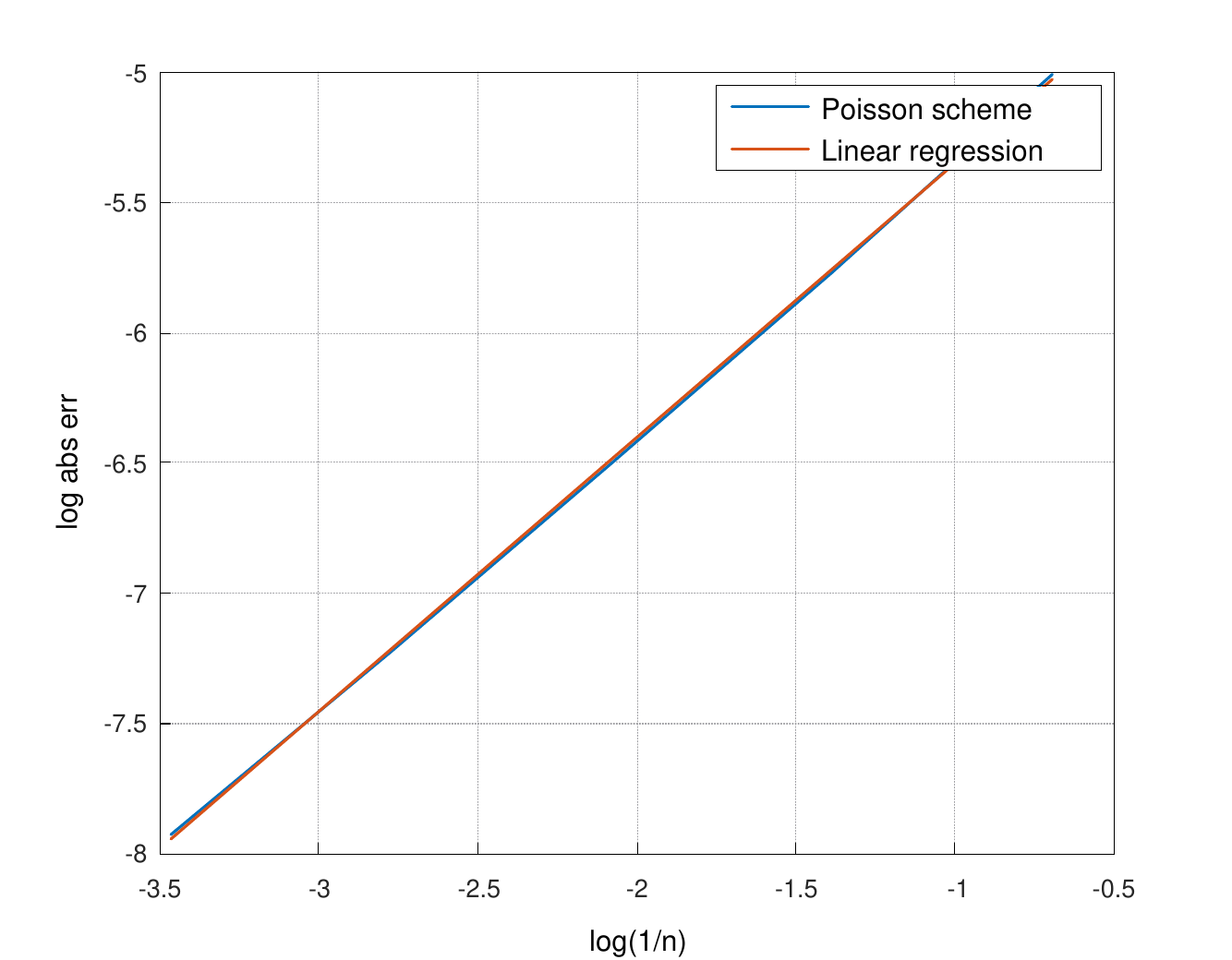}
      \caption{Log-log plot}
      \label{fig:log-log_C1S1}
    \end{subfigure}
    \caption{Parameters: $x=0.3$, $a=0.04$, $k=0.1$, $\sigma=2.0$, $f(z)=\exp(-z)$ and $T=1$ ($\frac{\sigma^2}{2a}= 50$). Graphic~({\sc a}) shows the values of $\hat{P}^{1,n}f$, as a function of the time step $1/n$ and the exact value. Graphic~({\sc b}) draws $\log(|\hat{P}^{1,n}f-P_Tf|)$ in function of $\log(1/n)$ and a regressed line. The slope of the regressed line is 1.05.}\label{CIR_Poiss1}
  \end{figure}

  \begin{figure}[h]
    \centering
    \begin{subfigure}[h]{0.49\textwidth}
      \centering
      \includegraphics[width=\textwidth]{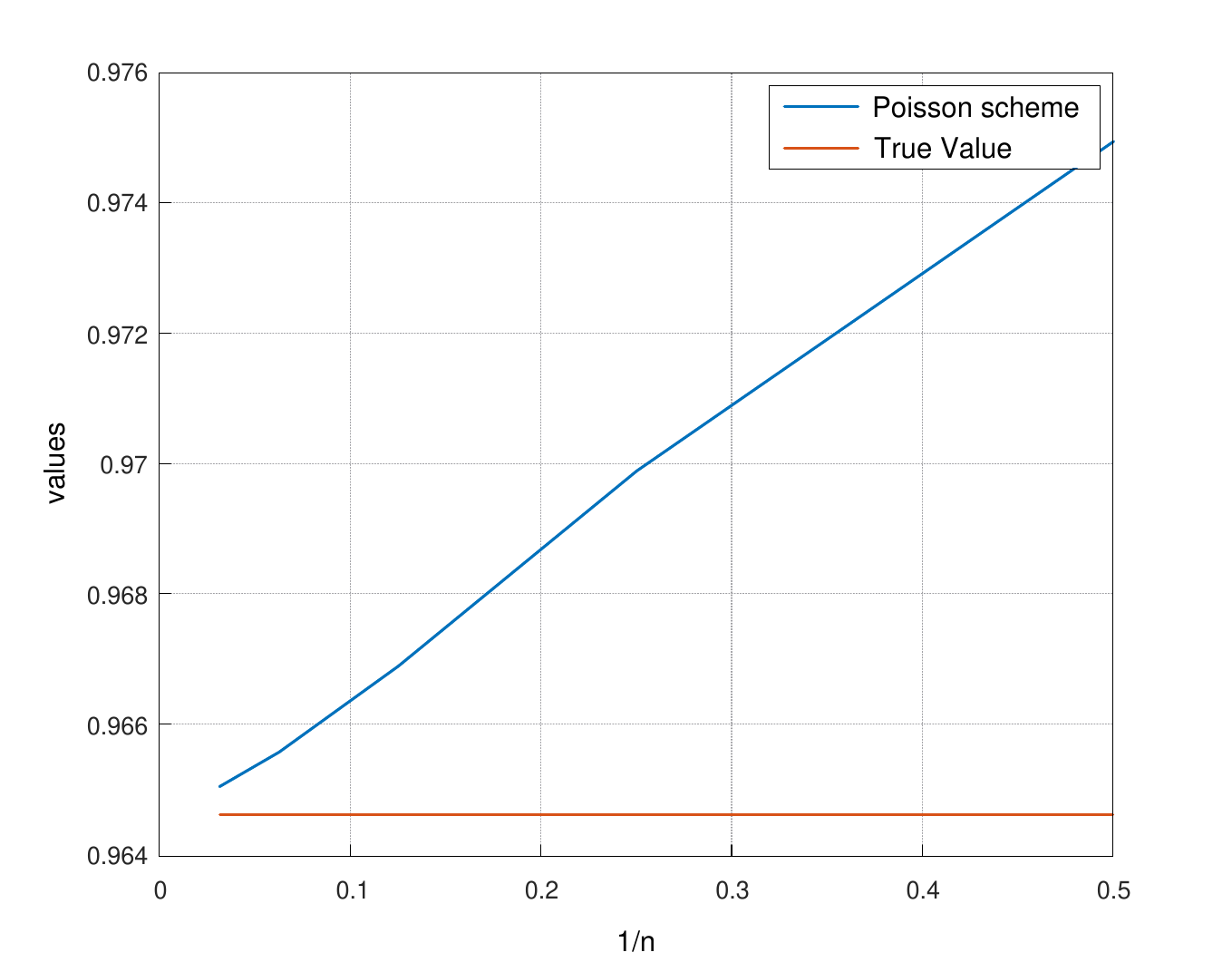}
      \caption{Values plot}
      \label{fig:values_C1S2}
    \end{subfigure}
    \hfill
    \begin{subfigure}[h]{0.49\textwidth}
      \centering
      \includegraphics[width=\textwidth]{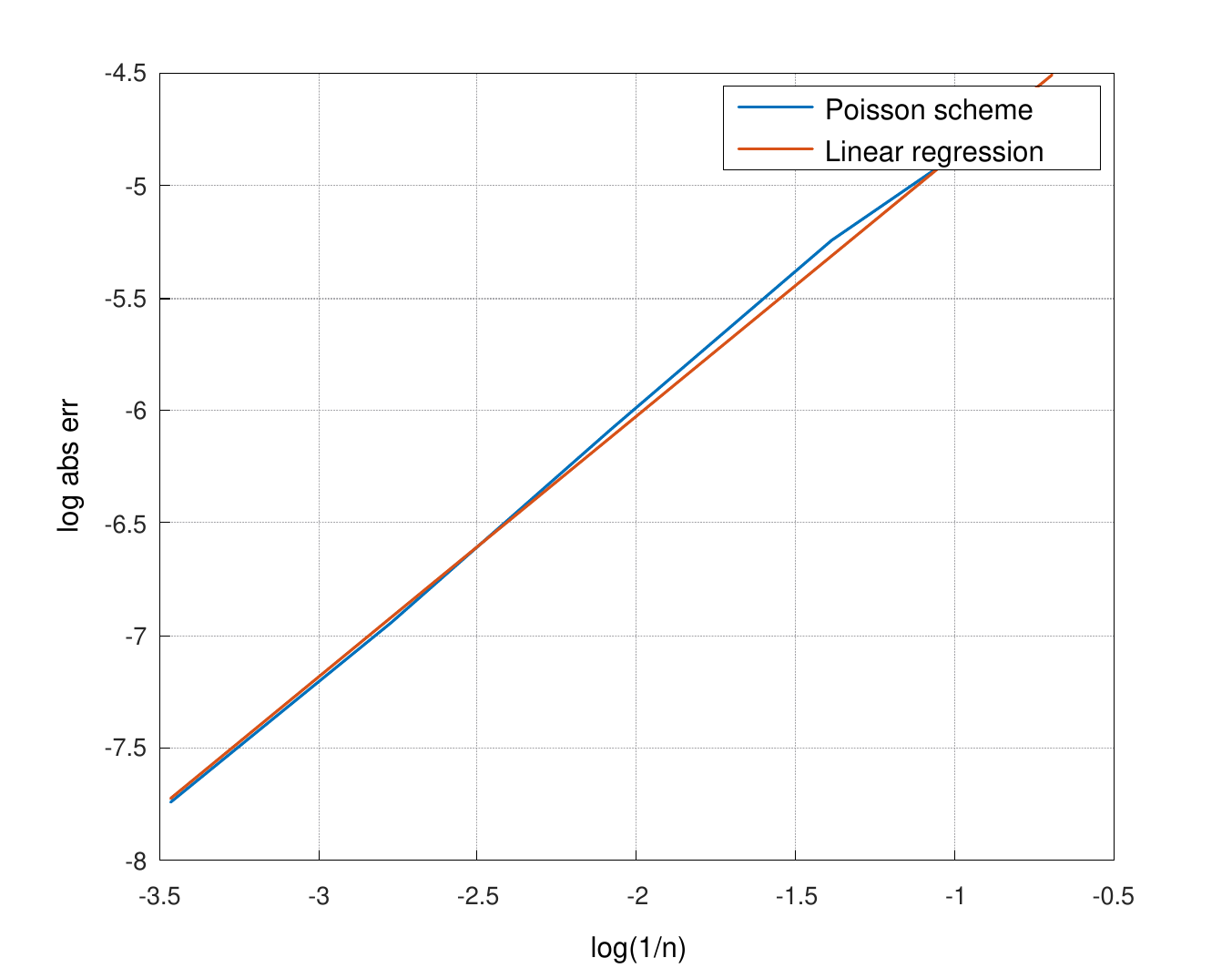}
      \caption{Log-log plot}
      \label{fig:log-log_C1S2}
    \end{subfigure}
    \caption{Parameters: $x=0.0$, $a=0.04$, $k=1.0$, $\sigma=2.0$, $f(z)=\exp(-4z)$ and $T=1$ ($\frac{\sigma^2}{2a}= 50$). Graphic~({\sc a}) shows the values of $\hat{P}^{1,n}f$, as a function of the time step $1/n$ and the exact value. Graphic~({\sc b}) draws $\log(|\hat{P}^{1,n}f-P_Tf|)$ in function of $\log(1/n)$ and a regressed line. The slope of the regressed line is 1.16.}\label{CIR_Poiss2}
  \end{figure}

\cleardoublepage
\phantomsection 
\addcontentsline{toc}{chapter}{Bibliography} 

\printbibliography


\end{document}